\documentclass[10pt]{article}
\usepackage[francais,english]{babel}
\usepackage{amsmath}
\usepackage{amsfonts}
\usepackage{amssymb}
\usepackage{amsthm}
\usepackage{amscd}
\usepackage{fullpage}
\usepackage[final]{showkeys}
\usepackage[matrix,arrow,ps]{xy}

\usepackage{hyperref}
\usepackage[sort=def]{glossaries}   %

\usepackage[pdftex]{graphicx,color} 

\newcommand{\Z}{\mathbb{Z}}
\newcommand{\R}{\mathbb{R}}
\newcommand{\N}{\mathbb{N}}

\newcommand{\C}{\mathbb{C}}

\newcommand{\U}{\mathbb{U}}
\newcommand{\E}{\mathbb{E}}

\newcommand{\mc}{\mathcal}

\newcommand{\mf}{\mathfrak}
\newcommand{\eps}{\varepsilon}
\newcommand{\ind}{{\bf 1}}
\renewcommand{\P}{\mathbb{P}}

\newcommand{\rK}{{\sf K}}
\newcommand{\rR}{{\sf R}}
\newcommand{\uK}{{\underline {\sf K}}}
\newcommand{\rS}{{\sf S}}
\DeclareMathOperator{\dist}{dist}
\newcommand{\dg}{\dagger}
\newcommand{\dmd}{\diamondsuit}
\newcommand{\mfm}{{\mathfrak m}}
\newcommand{\Mon}{{\sf Mon}}

\newcommand{\Lap}{\Delta\!}
\newcommand{\dbar}{\bar\partial}

\DeclareMathOperator{\Cov}{Cov}

\DeclareMathOperator{\Tr}{Tr}
\DeclareMathOperator{\Id}{Id}

\DeclareMathOperator{\sgn}{sgn}

\DeclareMathOperator{\Res}{Res}

\newcommand{\thetaf}[2]{\vartheta\left [\begin{array}{c}#1\\#2\end{array}\right]}

\title{Dimers and families of Cauchy-Riemann operators I}
\author{Julien Dub\'edat\footnote{Partially supported by NSF grant DMS-1005749 and the Alfred P. Sloan Foundation. MSC: 60G15, 82B20.}}
\newtheorem{thm}{Theorem}

\newtheorem{Thm}[thm]{Theorem}

\newtheorem{Prop}[thm]{Proposition}
\newtheorem{Lem}[thm]{Lemma}
\newtheorem{Cor}[thm]{Corollary}

\makeglossaries

\newglossaryentry{Sigma}
	{name={\ensuremath{\Sigma}}, 
	description={a flat torus $\C/(\Z+\tau\Z)$}
}

\newglossaryentry{tau}
	{name={\ensuremath{\tau}}, 
	description={modulus of the torus $\Sigma$}
}

\newglossaryentry{phi}
	{name={\ensuremath{\phi}}, 
	description={realisation of a (scalar) free field}
}

\newglossaryentry{J}
	{name={\ensuremath{J}}, 
	description={free field current}
}

\newglossaryentry{Wkp}
	{name={\ensuremath{W^{k,p}}}, 
	description={Sobolev space}
}

\newglossaryentry{Hs}
	{name={\ensuremath{H^s}}, 
	description={Fractional Sobolev space}
}

\newglossaryentry{dA}
	{name={\ensuremath{dA}}, 
	description={area measure}
}

\newglossaryentry{g0}
	{name={\ensuremath{g_0}}, 
	description={coupling constant of the free field}
}

\newglossaryentry{Lap}
	{name={\ensuremath{\Lap}}, 
	description={positive Laplacian}
}

\newglossaryentry{Om}
	{name={\ensuremath{{\mc O}_m}}, 
	description={magnetic ``operator"}
}

\newglossaryentry{ast}
	{name={\ensuremath{\ast}}, 
	description={Hodge star operator}
}

\newglossaryentry{Lambda}
	{name={\ensuremath{\Lambda}}, 
	description={rhombus tiling (with edge length $\delta$)}
}

\newglossaryentry{Gamma}
	{name={\ensuremath{\Gamma}}, 
	description={isoradial graph}
}

\newglossaryentry{Gammadg}
	{name={\ensuremath{\Gamma^\dg}}, 
	description={isoradial graph dual to $\Gamma$}
}

\newglossaryentry{spadesuit}
	{name={\ensuremath{(\spadesuit)}}, 
	description={rhombus angles are bounded away from 0}
}

\newglossaryentry{M}
	{name={\ensuremath{M}}, 
	description={a bipartite graph, superposition of $\Lambda$ and its dual. (Black vertices: $M_B$; white vertices: $M_W$)}
}

\newglossaryentry{K}
	{name={\ensuremath{K}}, 
	description={$\C$-linear Kasteleyn operator}
}

\newglossaryentry{rK}
	{name={\ensuremath{\rK}}, 
	description={$\C$-linear Kasteleyn operator}
}

\newglossaryentry{nu}
	{name={\ensuremath{\nu}}, 
	description={an angle assigned to a vertex of $M$}
}

\newglossaryentry{uKinv}
	{name={\ensuremath{\protect\underline{\rK}^{-1}}}, 
	description={kernel inverting the full-plane $\rK$}
}

\newglossaryentry{LapGamma}
	{name={\ensuremath{\Lap_\Gamma}}, 
	description={graph Laplacian on $\Gamma$}
}

\newglossaryentry{Xi}
	{name={\ensuremath{\Xi}}, 
	description={a subgraph of $M$}
}

\newglossaryentry{m}
	{name={\ensuremath{{\mfm}}}, 
	description={a perfect matching on $\Xi$}
}

\newglossaryentry{Kalpha}
	{name={\ensuremath{K_\alpha}}, 
	description={perturbed real Kasteleyn operator}
}

\newglossaryentry{rKalpha}
	{name={\ensuremath{\rK_\alpha}}, 
	description={perturbed complex Kasteleyn operator}
}

\newglossaryentry{mcZ}
	{name={\ensuremath{{\mathcal Z}}}, 
	description={dimer partition function}
}

\newglossaryentry{h}
	{name={\ensuremath{h}}, 
	description={dimer height function, defined on $\Xi^\dg$}
}

\newglossaryentry{pwb}
	{name={\ensuremath{p(w,b)}}, 
	description={$\P((bw)\in\mfm)$, under the full-plane measure}
}

\newglossaryentry{RB}
	{name={\ensuremath{R_B}}, 
	description={restrictions from continuous functions to $\C^{M_B}$}
}

\newglossaryentry{RW}
	{name={\ensuremath{R_W}}, 
	description={restrictions from $(0,1)$-forms to $\C^{M_W}$}
}

\newglossaryentry{Upsilon}
	{name={\ensuremath{\Upsilon}}, 
	description={the lattice $\Z+\tau\Z$}
}

\newglossaryentry{TSigma}
	{name={\ensuremath{T_\Sigma}}, 
	description={analytic torsion on the torus $\Sigma$}
}

\newglossaryentry{CMBchi}
	{name={\ensuremath{(\C^{M_B})_\chi}}, 
	description={$\chi$-multivalued functions on $M_B$}
}

\newglossaryentry{Gchi}
	{name={\ensuremath{G_\chi}}, 
	description={Chiral Green kernel}
}

\newglossaryentry{wedge}
	{name={\ensuremath{\wedge}}, 
	description={infix minimum}
}

\newglossaryentry{fkchi}
	{name={\ensuremath{f_{k,\chi}}}, 
	description={a special $\chi$-multivalued discrete holomorphic function}
}

\newglossaryentry{gw}
	{name={\ensuremath{g_w}}, 
	description={a discrete meromorphic $\chi$-multivalued (with a single pole at $w$ adjacent to the singularity}
}

\newglossaryentry{uKchi}
	{name={\ensuremath{\protect\underline{\rK}^{-1}_\chi}}, 
	description={chiral Cauchy kernel}
}

\newglossaryentry{Srho}
	{name={\ensuremath{S_\rho}}, 
	description={inverting kernel for $\rK$ operating on $\rho$-multivalued functions}
}

\newglossaryentry{MonM}
	{name={\ensuremath{\Mon_M}}, 
	description={Monomer correlation}
}

\begin{document}
\maketitle
\begin{abstract}
In the dimer model, a configuration consists of a perfect matching of a fixed graph. If the underlying graph is planar and bipartite, such a configuration is associated to a height function. For appropriate ``critical" (weighted) graphs, this height function is known to converge in the fine mesh limit to a Gaussian free field, following in particular Kenyon's work.
 
In the present article, we study the asymptotics of smoothed and local field observables from the point of view of families of Cauchy-Riemann operators and their determinants. This allows in particular to obtain a functional invariance principle for the field; characterise completely the limiting field on toroidal graphs as a compactified free field; analyse electric correlators; and settle the Fisher-Stephenson conjecture on monomer correlators.

The analysis is based on comparing the variation of determinants of families of (continuous) CR operators with that of their discrete (finite dimensional) approximations. This relies in turn on estimating precisely inverting kernels, in particular near singularities. In order to treat correlators of ``singular" local operators, elements of (multiplicatively) multi-valued discrete holomorphic functions are discussed.  

\end{abstract}

\tableofcontents

\section{Introduction}

The dimer model is a classical model of statistical mechanics; it consists in sampling uniformly (or with weights) perfect matchings of a bipartite graph. For planar graphs, Kasteleyn showed that the partition function can be expressed as the Pfaffian of a properly signed (weighted) adjacency matrix for the graph, the Kasteleyn matrix (or operator). In the small mesh limit, this operator can be interpreted as a finite difference version of a $\dbar$-operator. In the ``Temperleyan" case we will be considering, dimer configurations are in measure-preserving bijection with spanning trees on a related graph.

Following Thurston, one may associate a height function to a dimer configuration on the square lattice (or another bipartite graph). As suggested by Benjamini, this height function can be understood in terms of windings of the associated spanning tree.

In \cite{Ken_domino_conformal}, Kenyon considers the scaling limit of the height function of dimers on the square lattice for planar domains with appropriate boundary conditions, and proves that in the scaling limit, it converges in distribution to a conformally invariant object for $n$-connected domains. This is based on precise asymptotics of the inverse Kasteleyn matrix. In \cite{Ken_domino_GFF}, he identifies this limiting distribution as the classical massless Gaussian free field, in the simply connected case. 

In the present article, we are interested in several extensions of these results. The main point of view is that the invariance principles under consideration can be related in a natural way to Quillen's theory of families of Cauchy-Riemann (CR) operators (\cite{Qui}).

The Kasteleyn matrix may be thought of as a discretisation of the Cauchy-Riemann $\dbar$ operator:
$$\dbar f=\frac{\partial f}{\partial \bar z}d\bar z$$
and functions in its kernel (at least in some region of space) are discrete holomorphic functions. The question of discretisation of various objects and properties of discrete complex and harmonic analysis, going back to \cite{Duffin_rhombic}, has proved instrumental in the asymptotic analysis of combinatorial models such as the dimer and Ising model (\cite{Ken_domino_conformal,Smi_ising,Smi_ICM}).

We will consider broadly two types of perturbation of the ``standard" situation where the Kasteleyn operator approximates the $\dbar$ operator.

Firstly, we will study regular perturbations, corresponding in the continuous limit to $\dbar+\alpha$, $\alpha$ a smooth ``potential". This will lead in particular to a complete description of the scaling limit of the dimer height field on the torus as a compactified free field (Theorem \ref{Thm:toruscompact}).

Secondly, we will turn to singular perturbations. In the continuum, these may be thought alternatively as $\dbar+\alpha$, with $\alpha$ a potential with simple poles at prescribed singularities; or $\dbar$ operating on multiplicatively multivalued functions (or sections of a unitary line bundle over a punctured Riemann sphere). In this situation, the operators are parameterised by the position of the punctures and the monodromy data. The study of these families will lead to precise asymptotics for ``electric" correlators of the dimer height field (Theorem \ref{Thm:electr}), partially answering the Open problem 1 \cite{Ken_IAS}; and for monomer correlators (Theorem \ref{Thm:FS}), in fulfilment of the Fisher-Stephenson conjecture (\cite{FisSte2}, see also \cite{Ciucu_PNAS}). At some general level, all the end results in the article (on compactification, electric and monomer/magnetic correlators) naturally fit in the classical Coulomb Gas heuristics \cite{Nien_Coulomb,DiFSalZub_coulomb}.

The article is organised as follows. Section 2 provides brief background on the continuous free field, the discrete dimer height field, and families of CR operators. In Section 3, we review the necessary results on isoradial dimers. Invariance principles in the plane are discussed in Section 4. The case of toroidal graphs is treated in Section 5. Section 6 describes a general surgery principle for (converging sequences) of discrete CR operators. Discrete multi-valued holomorphic functions and electric correlators are described in Section 7. Monomers correlators are studied in Section 8.

\section{Background and overview}

The main goal of this article is to express asymptotics of natural dimer observables in terms of a (compactified) free field; a key tool will be the study of families of (discrete and continuous) Cauchy-Riemann operators and their variational properties. In this section we provide a brief introduction to these objects.

\subsection{Free field}

The Gaussian (or massless) free field is a Gaussian field with covariance kernel given by the Laplacian Green's function (or a multiple thereof). For a general introduction to the free field, see e.g. \cite{Simon_Pphi}, Chapter 6 in \cite{Glimm_Jaffe}, the survey \cite{She_GFFmath}, and Section 4 in \cite{Dub_SLEGFF}; for the general theory of Gaussian spaces, see \cite{Janson}. Here we will consider the free field on the plane $\C$ or a torus $\gls{Sigma}=\C/(\Z+\tau\Z)$. In both cases, the Laplacian has a non-trivial kernel consisting of constant functions. Consequently, the free field is only defined there up to an additive constant.

Following Gross' abstract Wiener space approach, one may regard the free field as a random element $\gls{phi}$ of a Banach space, usually a Sobolev space modulo additive constants: $H^{-\eps}(\Sigma)/\R$ (for some positive $\eps$). The elements of this space, and thus realisations of the free field, are distributions (in the sense of Schwarz, i.e. generalised functions). In order to avoid quotienting by constant functions, one may consider its distributional total derivative, the {\em current} 
\begin{equation}\label{eq:currentdef}
\gls{J}=d\phi
\end{equation}

After brief reminders on Sobolev spaces, we will discuss the scalar free field (in the plane) and the compactified free field.

\subsubsection{Sobolev Spaces}\label{ssec:Sobolev}

The results described here are used mostly to state and obtain sharper results in Sections \ref{sec:planar},\ref{sec:torus} and have no bearing on Sections \ref{Sec:surgery}, \ref{sec:electr} and \ref{sec:mono}.

For notational simplicity (and to avoid difficulties for domains with boundaries or unbounded domains), we discuss here the case of the square torus $\Sigma=\C/(\Z+i\Z)$. See e.g. Chapter 7.1 in \cite{Stroock_PDEs}, and, for a comprehensive treatment, see Chapters 8-9 in \cite{Brezis} or \cite{Adams_Sobolev}. For $k\in\N$, the space $H^k(\Sigma)$ of functions with square integrable $k$-th order derivatives is the completion of smooth functions on $\Sigma$ w.r.t. to the norm
$$\|f\|_{H^k}=\sum_{\alpha:|\alpha|\leq k}\|\partial^{\alpha}f\|_{L^2}$$
where $\alpha$ denotes a multi-index (possibly empty) and $|\alpha|$ its order. For example, $\|f\|_{H^1}=\|f\|_{L^2}+\|f_x\|_{L^2}+\|f_y\|_{L^2}$. 

More generally, for $k\in\N$ and $p\geq 1$, one may consider the Sobolev space $\gls{Wkp}$ obtained as the completion of smooth functions w.r.t.
$$\|f\|_{W^{k,p}}=\sum_{\alpha:|\alpha|\leq k}\|\partial^{\alpha}f\|_{L^p}$$
For $k\geq 0$, one can then define $H^{-k}(\Sigma)$ from $H^k(\Sigma)$ by duality, using $L^2(\Sigma)$ as a pivot (see e.g. Remark 3 in Section 5.2 of \cite{Brezis}). Since we have a natural inclusion $H^0=L^2\subset H^k$ with $\|.\|_{H^0}\leq\|.\|_{H^k}$, we may write 
\begin{align*}
H^{-k}(\Sigma)&\simeq\{\varphi\in L(H^k(\Sigma),\C):\sup_{f\in H^k(\Sigma)}\frac{|\varphi(f)|}{\|f\|_{H^k}}<\infty\}\\
&\supset \{\varphi\in L(H^0(\Sigma),\C):\sup_{f\in H^0(\Sigma)}\frac{|\varphi(f)|}{\|f\|_{H^0}}<\infty\}=L^2(\Sigma)\supset H^k(\Sigma)
\end{align*}
Here $L(U,V)$ denotes the space of linear maps from $U$ to $V$. Then one can interpolate between the $H^k$, $k\in\Z$, as follows. Take the Fourier basis of $L^2(\Sigma)$ given by 
$$e_{mn}(x+iy)=\exp(2i\pi(mx+ny))$$
for $m,n\in\Z^2$, so that (Parseval)
$$L^2(\Sigma)=\{f:f=\sum_{m,n\in\Z}a_{mn}e_{mn}, \sum |a_{mn}|^2<\infty\}$$
and
$$H^k(\Sigma)=\{f:f=\sum_{m,n\in\Z}a_{mn}e_{mn}, \sum (m^2+n^2)^k|a_{mn}|^2<\infty\}$$
For a ``fractional" (not necessarily integer) Sobolev index $s\in\R$, we may define
\begin{equation}\label{eq:defHs}
\gls{Hs}(\Sigma)=\{f:f=\sum_{m,n\in\Z}a_{mn}e_{mn}, \sum (1+m^2+n^2)^s|a_{mn}|^2<\infty\}
\end{equation}
which is the closure of smooth functions on $\Sigma$ w.r.t. the norm
$$\|f\|_s^2=\sum (1+m^2+n^2)^s|a_{mn}|^2$$
More intrinsically, set $H^0(\Sigma)=L^2(\Sigma)$, and for $s\in (0,1)$, one can define an equivalent Sobolev norm by
$$\|f\|_s^2=\|f\|_{L^2}^2+\int_{\Sigma^2}\frac{|f(x)-f(y)|^2}{|x-y|^{2s+2}}dA(x)dA(y)$$
where $\gls{dA}$ is the area measure (see 7.48 in \cite{Adams_Sobolev}; with a slight abuse of notation, $|x-y|$ is the distance on the torus). For $k\geq 0$, set e.g.
$$\|f\|_{k+s}=\sum_{\alpha:|\alpha|\leq k}\|\partial^\alpha f\|_s$$
Then $H^{k+s}(\Sigma)$ is the completion of smooth functions w.r.t. $\|.\|_{k+s}$. As for integer indices, the negative index spaces may be defined from the positive index ones by duality.

We list below a few key properties of Sobolev spaces (specialised to the plane or compact surfaces). 

\begin{enumerate}
\item For $s<s'$, the natural inclusion $H^{s'}\subset H^{s}$ is a compact embedding (clear e.g. from \eqref{eq:defHs}).
\item $\frac{\partial}{\partial x}$ (resp. $\frac{\partial}{\partial y}$) defines a bounded operator $H^{s+1}\rightarrow H^s$ (also clear from the definitions.)
\item Sobolev embedding: there is a natural bounded embedding $H^{k+s}\rightarrow W^{k,p}$, $k\in\N$, $s\in(0,1)$, $2\leq p\leq\frac 2{1-s}$ (see 7.58 in \cite{Adams_Sobolev}).
\item Morrey's inequality. If $f\in W^{1,p}$, $p>2$, $\alpha=1-\frac 2p$, then $f$ is $\alpha$-H\"older and $\|f\|_{C^\alpha}\leq c\|f\|_{W^{1,p}}$ (see Theorem 9.12 in \cite{Brezis}).
\item $C^\infty(\Sigma)=\cap_{s\in\R}H^s(\Sigma)$ (follows from Morrey's inequality).
\item The space of distributions $C^\infty(\Sigma)'$ is given by $C^\infty(\Sigma)'=\cup_{s\in\R}H^s(\Sigma)$ (if $T$ is a distribution, $\phi\mapsto T(\phi)$ extends to a bounded linear map on $H^s$ for $s$ large enough).
\end{enumerate}
Remark that combining the Sobolev embedding with Morrey's inequality shows that $f\in H^{1+s}$ is $\alpha$-H\"older for $\alpha<s<1$.

\subsubsection{Scalar free field}

The (centred) free field on the plane $\C$ (resp. on a torus $\Sigma$) (see \cite{Simon_Pphi}, Chapter 6 of \cite{Glimm_Jaffe}, or \cite{She_GFFmath}) is the Gaussian field with action functional
\begin{equation}\label{eq:GFFaction}
S(\phi)=\frac{g_0}{4\pi}\int|\nabla\phi|^2dA
\end{equation}
i.e. it is the law on distributions $\phi\in\C^\infty(\C)'$ (modulo additive constant) characterised by the fact that for any zero-mean test function $\psi\in C^\infty_c(\C)$ (resp. $C^\infty_c(\Sigma)$), $\int (\psi\phi) dA$ is a centered Gaussian variable with variance $\frac{2\pi}{g_0}\int \psi\Lap^{-1}\psi dA$, where $\gls{g0}>0$ is the coupling constant and $\gls{Lap}$ is the positive Laplacian:
$$\Lap=-\frac{\partial^2}{\partial x^2}-\frac{\partial^2}{\partial y^2}$$
(Note that $\int\psi\Lap^{-1}\psi dA$ is well-defined for a zero-mean function.)

Instead of working on the Fr\'echet space of distributions, it is sometimes convenient to work on a Banach space $H^s$. In dimension two, we can take $s$ to be any negative number: $\phi\in H^{s}$ for any $s<0$. Notice that almost surely $\phi$ is not defined pointwise (and not even almost everywhere). Then the current $J=d\phi$ is well-defined as a (random) $1$-form with coefficients in $H^{s-1}$; it is closed in the sense that $dJ=0$ in $H^{s-2}$. 

In terms of the characteristic functional,
\begin{equation}\label{eq:GFFchar}
\langle\exp(i\int(\psi\phi)dA\rangle=\exp(-\frac \pi {g_0}\int \psi\Lap^{-1}\psi dA)
\end{equation}
where $\langle\rangle$ denotes the expectation with respect to the free field. 

Instead of integrating against a smooth test function, it is often useful and natural to consider local ``operators", involving the local behaviour of the field at a certain number of selected points. For a detailed conceptual discussion, we refer the reader to \cite{MakKan_GFFCFT}; we will need here only a few simple cases.

\paragraph{Current correlations.\\}
As a basic example, one may consider current correlations, given by the Wick formula:
$$\langle J(z_1)\dots J({z_{2n}})\rangle =(2\pi {g_0}^{-1})^n\sum_{\{(\alpha_1,\beta_1),\dots,(\alpha_n,\beta_n)\}}\prod_{i=1}^n d_{z_{\alpha_i}}d_{z_{\beta_i}}G(z_{\alpha_i},z_{\beta_i})$$
where the sum bears on the pairings $\{(\alpha_1,\beta_1),\dots,(\alpha_n,\beta_n)\}$ of indices in $\{1,\dots,2n\}$ (for an odd number of currents, the correlator vanishes due to the symmetry in law $\phi\leftrightarrow-\phi$). Here $G_\C$ is the Green kernel inverting the Laplacian:
$$(\Lap^{-1} f)(z_1)=\int G_\C(z_1,z_2)f(z_2)dA(z_2)$$
with 
$$G_\C(z_1,z_2)=-\frac{1}{2\pi}\log|z_2-z_1|$$
in the plane. 

This expression makes sense when integrated against regular enough test functions $\psi_1(z_1),\dots,\psi_{2n}(z_{2n})$; for instance,
$$\left\langle\left(\int\psi_1\partial_{x}\phi dA\right)\left(\int\psi_2\partial_{y}\phi dA\right)\right\rangle=\int\psi_1(z_1)\partial^2_{x_1y_2}G_\C(z_1,z_2)\psi_2(z_2)dA(z_1)dA(z_2)$$
where $z_j=x_j+iy_j$ and $\psi_1,\psi_2$ are smooth test functions with disjoint supports.

More generally, for any multi-index $\alpha$, $\partial^\alpha\phi$ is a random distribution (which we can realise in $H^{-s-|\alpha|}$ for any $s>0$). The correlator
$$\langle \partial^{\alpha_1}\phi(z_1)\dots\partial^{\alpha_{n}}\phi(z_n)\rangle$$
is specified by
$$\left\langle\left(\prod_j\int\psi_j\partial^{\alpha_j}\phi dA\right)\right\rangle=\int \langle \partial^{\alpha_1}\phi(z_1)\dots\partial^{\alpha_{n}}\phi(z_n)\rangle\prod_j\psi_j(z_j)dA(z_j)$$
 for any $n$-tuple $\psi_1,\dots,\psi_n$ of test functions with disjoint support. (Again, this is only relevant when $n$ is even).

From this expression, we see that we can also use approximations of the identity. Let $\eta$ be a compactly supported smooth function with $\eta\geq 0$ and $\int\eta dA=1$, and for $\delta$ small set $\eta_\delta(z)=\delta^{-2}\eta(\delta^{-1}z)$. Then
\begin{equation}\label{eq:mollcurrentcorr}
\langle \partial^{\alpha_1}\phi(z_1)\dots\partial^{\alpha_{n}}\phi(z_n)\rangle=\lim_{\delta\searrow 0}\left\langle\prod_j\int \eta_\delta(-z_j+\cdot)
\partial^{\alpha_j}\phi
\right\rangle
\end{equation}
Notice that for any $\delta>0$, the RHS is a well-defined random variable. More generally, let $U$ be an open set and $z\notin\overline U$. Let $F(\phi_{|U})$ be a, say, bounded r.v. depending only on the field in $U$ (eg $F(\phi_{|U})=f(\int\psi_1\phi,\dots,\int\psi_n\phi)$, for test functions $\psi_1,\dots,\psi_n$ supported in $U$). Then one can define
$$\langle \partial^{\alpha}\phi(z)F(\phi_{|U})\rangle=\lim_{\delta\searrow 0}\left\langle F(\phi_{|U})\int \eta_\delta(z+\cdot)
\partial^{\alpha}\phi
\right\rangle$$

\paragraph{Electric correlators.\\}
After distributional derivatives of the field, the next local operators to consider are {\em electric} vertex operators (e.g. Lecture 1.6 in \cite{Gaw_CFT}, 10.4.1 and 12.6.2 in\cite{DiF}), which we write formally as $\exp(i\alpha \phi(z))$ where the constant $\alpha$ is the {\em charge}. The difficulty is that, as $\phi$ is not defined pointwise but as a distribution, it may not {\em a priori} be composed with the nonlinear function $\exp$. %
A standard normalisation scheme (e.g. Section 1.6 in \cite{Gaw_CFT}) consists, as in \eqref{eq:mollcurrentcorr}, in mollifying the field to obtain well-defined variables and take the limit of the resulting correlators as $\delta\searrow 0$. However, in order to obtain a non-degenerate limit, we now need a multiplicative counterterm. 

Let us specialise to the simplest case, which is the one relevant in this article, where we want to define a regularised ``$\langle\exp(i\sum_j\alpha\phi(z_j))\rangle$". Here $z_1,\dots,z_n$ are distinct points in the plane and $\alpha_1,\dots,\alpha_n$ are charges adding up to zero (which we need since $\phi$ is defined modulo an additive constant). We start with the well-defined, mollified correlator
$$\left\langle\exp\left(i\sum_j\alpha_j\int \phi\eta_\delta(-z_j+\cdot)\right)\right\rangle$$
which, by \eqref{eq:GFFchar}, equals 
$$\exp\left(-\frac{\pi}{g_0}\sum_{j,k}\alpha_j\alpha_k\int\eta_\delta(-z_j+u)G_\C(u,v)\eta_\delta(-z_k+v)dA(u)dA(v)\right)$$
In this sum, the off-diagonal terms have a finite limit (as $\delta\searrow 0$), while the diagonal terms have a logarithmic but purely local (i.e. independent of the position and nature of the other insertions) divergence. Indeed, by scaling we have
$$\int\eta_\delta(u)G_\C(u,v)\eta_\delta(v)dA(u)dA(v)=-\frac 1{2\pi}\log\delta+\int\eta(u)G_\C(u,v)\eta(v)dA(u)dA(v)$$
Let us choose $\eta$ so that the RHS is $-\frac 1{2\pi}\log\delta$. Then
\begin{equation}\label{eq:electcorrGFF}
\begin{array}{rl}
\langle :\exp(i\sum_j\alpha_j\phi(z_j)):\rangle_\C&\stackrel{def}{=}\lim_{\delta\searrow 0}\delta^{-\frac1{2g_0}\sum_j\alpha_j^2}\langle \exp(i\sum_j\alpha_j\int\eta_\delta(-z_j+\cdot)\phi)\rangle_\C\\
&=\prod_{j<k}|z_j-z_k|^{\frac {\alpha_j\alpha_k}{g_0}}
\end{array}
\end{equation}
where the colons recall the implicit multiplicative counterterm. More generally, if $F(\phi_{|U})$ is a test function as below \eqref{eq:mollcurrentcorr}, $z_1,\dots,z_n\notin\bar U$, we may define
$$\langle :\exp(i\sum_j\alpha_j\phi(z_j): F(\phi_{|U})\rangle\stackrel{def}{=}\lim_{\delta\searrow 0}\delta^{-\frac1{2g_0}\sum_j\alpha_j^2}\langle \exp(i\sum_j\alpha_j\int\eta_\delta(-z_j+\cdot)\phi)F(\phi_{|U})\rangle$$

\paragraph{Magnetic correlators.\\}
Another set of ``local operators" is given by {\em magnetic} correlators.  
These are disorder variables, i.e. they represent a modification of the state space rather than a modification of the state weights. Specifically, given marked points $z_1,\dots,z_{n}$ and ``magnetic charges" $m_1,\dots,m_n$, one considers additively multivalued functions which increase by $2\pi m_j$ when cycling clockwise around $z_j$; one requires $\sum_jm_j=0$ so that the function is single-valued near infinity. This defines an affine state space with a single element of minimal (at least after normalisation) Dirichlet energy: 
$$\phi_0=\sum_jm_j\Im\log(z-z_j)$$
In the Gaussian formalism, offsetting fields by a harmonic function $\phi_0$ results in multiplying the partition function by $\exp(-\frac{g_0}{4\pi}\int_\Sigma |\nabla\phi_0|^2dA)$ (see e.g. Section 5 in \cite{Dub_SLEGFF} for a discussion of free field partition functions). Here the Dirichlet energy is infinite, due to logarithmic singularities at the $z_j$'s. A regularised Dirichlet energy is obtained by discarding the leading part (as $\delta\searrow 0$) of $\int_{(\cup_j D(z_j,\delta))^c}|\nabla \phi_0|^2dA$. Representing by $\gls{Om}(z)$ a magnetic charge $m$ at $z$, we will thus consider the following regularised correlator:
$$\langle :\prod_j{\mc O}_{m_j}(z_j):\rangle\stackrel{def}{=}\exp\left(-\frac{g_0}{4\pi}\int^{reg}|\nabla\phi_0|^2dA\right)$$
where
$$\int^{reg}|\nabla\phi_0|^2dA\stackrel{def}{=}\lim_{\delta\searrow 0}\left(\int_{(\cup_j D(z_j,\delta))^c}|\nabla\phi_0|^2dA+2\pi\sum_jm_j^2\log|\delta|\right)$$
Again, in the planar case this correlator has a simple evaluation. Indeed, note that $\phi_0$ has the same Dirichlet energy as its (single valued) harmonic conjugate $\tilde\phi_0=\sum_j m_j\Re\log(z-z_j)$. By Green's formula,
\begin{align*}
\int_{\Sigma\setminus(\cup_j D(z_j,\delta))}|\nabla \tilde\phi_0|^2dA
&=\sum_j\int_{C(z_j,\delta)}\tilde\phi_0\partial_n\tilde\phi_0 d\ell=-2\pi\sum_j (m_j+O(\delta))(m_j\log|\delta|+\sum_{k\neq j}m_k\log|z_k-z_j|+O(\delta))\\
&=-{2\pi}(\sum_j m_j^2)\log|\delta|-{2\pi}\sum_{j\neq k}m_jm_k\log|z_j-z_k|
\end{align*}
so that
\begin{equation}\label{eq:GFFmagncorr}
\langle :\prod_j{\mc O}_{m_j}(z_j):\rangle_\C=\prod_{j<k}|z_j-z_k|^{g_0m_jm_k}
\end{equation}

\subsubsection{Compactified free field}

Up to now, we have discussed the {\em scalar} free field, taking real values. An important variant is the {\em compactified} free field (Lecture 1.4 in \cite{Gaw_CFT}, 6.3.5 and 10.4.1 \cite{DiF}), taking values in the circle $\R/2\pi r\Z$; let us begin with an informal discussion. Given two Riemannian manifolds $S$ and $T$, the classical harmonic mapping problem consists in finding a mapping $\phi:S\rightarrow T$ which minimises Dirichlet energy, for instance within a homotopy class. For example, the harmonic mappings $\C/\Upsilon\rightarrow\R/2\pi r\Z$ are written $z\mapsto\Re(z\bar w)$, where $\Re(\ell\bar w)\in 2\pi r\Z$ for all $\ell\in \Upsilon$ ($\Upsilon$ a lattice). In the quantised version of the problem, one considers mappings $\phi:\Sigma\rightarrow\R/2\pi r\Z$ ($\Sigma$ a surface) with action functional given by the Dirichlet energy
$$S(\phi)=\frac {g_0}{4\pi}\int_\Sigma|\nabla\phi|^2dA$$ 
and distribution formally given by $``e^{-S(\phi)}\prod_{x\in\Sigma}d\phi(x)"$ (the ``flat measure" $``\prod_{x\in\Sigma}d\phi(x)"$ does not make sense in and of itself).

In order to reduce to the well-understood scalar free field discussed above, observe that $\phi$ can be decomposed as the sum of a harmonic mapping $\phi_h$ and a scalar part $\phi_s:\Sigma\rightarrow\R$, projected on $\R/2\pi r\Z$; moreover the decomposition is unique up to additive constants. 

Indeed, the current $J=d\phi$ is well defined as a closed 1-form on $\Sigma$: $dJ=0$. Its periods $\int_\gamma J$, $\gamma$ a cycle on $\Sigma$, are integer multiple of $2\pi r$. Manifestly there is an equivalence
\begin{align*}
\{{\rm mappings\ }\Sigma\rightarrow\R/2\pi r\Z\}{\rm\ modulo\ additive\ constant} \leftrightarrow
\{J\in\Omega^1(\Sigma), dJ=0, \int_\gamma J\in 2\pi r\Z{\rm\ \ }\forall\gamma\in H_1(\Sigma,\Z)\}
\end{align*}

On a compact Riemannian manifold, closed forms may be decomposed uniquely as the sum of an exact form and a harmonic form (i.e. the de Rham cohomology groups are represented by harmonic forms). Here it means we may write the Hodge decomposition (e.g. Theorem 5.2 in \cite{Voisin_Hodge}, specialised here to closed 1-forms on surfaces) as
\begin{equation}\label{eq:hodgedec}
J=d\phi=\omega_h+d\phi_s
\end{equation}
where $\phi_s:\Sigma\rightarrow\R$ is well defined up to additive constant and $\omega_h$ is closed and harmonic: $d\omega_h=d\ast\omega_h=0$, where $\gls{ast}$ is the Hodge star operator on 1-forms. (If $x,y$ are isothermal coordinates on the surface $\Sigma$, i.e. if the metric is given locally by $dx\otimes dx+dy\otimes dy$, then $\ast(fdx+gdy)=(fdy-gdx)$). The harmonic 1-form $\omega_h$ integrates to an additively multivalued harmonic function $\phi_h$ with periods in $2\pi r\Z$.

In the case of interest to us (the torus $\Sigma=\C/(\Z+\tau\Z)$), it is very easy to justify directly \eqref{eq:hodgedec}. The closed 1-form $J$ lifts to a periodic (under action of the lattice $\Z+\tau\Z$) 1-form $\tilde J$ on $\C$. By the Poincar\'e lemma, it is the differential of a function: $\tilde J=d\tilde\phi$. Since the derivatives of $\tilde\phi$ are periodic, it readily appears that $\tilde\phi$ is additively quasiperiodic: there are $\alpha,\beta$ s.t.
$$\tilde\phi(z+m+n\tau)=\tilde\phi(z)+m\alpha+n\beta$$
for any $m,n\in\Z$, $z\in\C$. Then there is a unique affine function of type $A:z\mapsto a\Re(z)+b\Im(z)$ s.t.
$\tilde\phi-A$ is periodic. Setting $\omega_h=adx+bdy$, $\phi_s=\tilde\phi-A$ gives the decomposition \eqref{eq:hodgedec}. Uniqueness is also clear from that argument.

On the space $\Omega^1(\Sigma)$ of $1$-forms, there is a natural scalar product given by
$$(\omega_1,\omega_2)=\int \omega_1\wedge\ast\omega_2$$
so that $S(\phi)=\frac{g_0}{4\pi}(J,J)$. Moreover the decomposition \eqref{eq:hodgedec} is orthogonal w.r.t. this scalar product (as follows easily from Stokes' formula) and
$$S(\phi)=S(\phi_s)+S(\phi_h)$$
Given this discussion for smooth (or at least $C^1$) mappings, it is rather natural to define a free field on $\Sigma$ with compactification radius $r$ as the data $(\omega_h,\phi_s)$ where the two components are independent. The harmonic form $\omega_h$  is the {\em instanton} component and $\phi_s$ is the {\em scalar} component.

The 1-form $\omega_h$ has unnormalised distribution
$$\exp(-\frac{g_0}{4\pi}\int_\Sigma\omega_h\wedge\ast\omega_h)d\mu(\omega_h)$$
where $\mu$ is the counting measure on the lattice of harmonic 1-forms with periods in $2\pi r\Z$: $\int_\gamma\omega_h\in 2\pi r\Z$ for any closed cycle $\gamma$. The scalar component $\phi_s$ is a scalar free field with covariance kernel $2\pi g_0^{-1}G_\Sigma$ (and is defined modulo additive constant; here $G_\Sigma$ is the Green kernel inverting the Laplacian on $\Sigma$).

Finally, let us mention that, similarly to Kramers-Wannier duality for the Ising model, the free field possesses an abelian duality which is a simple example of $T$-{\em duality} (Lecture 1.5 in \cite {Gaw_CFT}, 10.4.1 in \cite{DiF}). For a well-chosen normalisation of the field, $T$-duality inverts the compactification radius and exchanges electric (order) operators with magnetic (disorder) operators. This gives (additional) motivation for the introduction of magnetic operators, as well as a conceptual interpretation of the similarly between the expressions \eqref{eq:electcorrGFF} and \eqref{eq:GFFmagncorr}.

\subsection{Dimers}

For simplicity, we discuss here dimers on the square lattice $\Z^2$; however we shall subsequently consider the more general framework of isoradial graphs. For background on dimers, see \cite{Ken_IAS} and references therein.

We will consider the square lattice itself, or some portion of it, or its quotient by a large scale lattice (yielding a graph embedded on a torus). A {\em perfect matching} or {\em dimer configuration} consists in a selection of edges such that each vertex abuts exactly one selected edge. For a finite graph, one may consider the uniform measure on such matchings. A Gibbs measure on matchings of the infinite volume graph $\Z^2$ may be obtained as the weak limit of finite volume measures (\cite{CKP,KOS,dT_isoradial}). The analysis of dimers, in particular in the fine mesh limit, relies crucially on Kasteleyn's Pfaffian enumeration of dimer configurations (\cite{Kas_square}) for planar graphs.

Following Thurston, to every dimer configuration on such a graph derived from $\Z^2$, one may associate an integer-valued (in some normalisation) height function $h$ on its dual graph. Equivalently (with some care given to boundary conditions for bounded domains and to Gibbs measures in the plane), one can consider the $6$-vertex model at the free-fermionic point and the corresponding height function, introduced in \cite{Beijeren}; see e.g. Section 5 in \cite{Dub_abelian} for a discussion of this formulation.

Kenyon established in several set-ups that in the small mesh limit this height function converges to a free field (\cite{Ken_domino_GFF}); this has been extended in \cite{dT_isoradial,KOS,Ken_honeycomb}. Specifically (for the square lattice), one can obtain a local central limit theorem, showing convergence of discrete current correlators to their free field limit. By integration against a test function, this gives the correct limit in law for the discrete height field, in some finite dimensional marginal sense. An approach of the GFF invariance principle for the dimer height function (on the hexagonal lattice, for certain boundary conditions including facets) and related models can be found in \cite{BorFer_growth} and subsequent articles.

The main goal of the article is to study, in the toroidal case, convergence of the height (current) to the compactified free field (current); and in the planar case, to express the asymptotics of suitable dimer observables to regularised correlators of electric and magnetic operators for the free field.

While the relationship between dimer height functions and the {\em scalar} free field was highlighted by these works, the relation with the {\em compactified} free field becomes apparent when considering toroidal graphs. There, only the height current is well-defined, and the height function is additively multivalued (when tracing it along non-contractible cycles on the torus): as explained earlier, it splits into an instanton and a scalar component. In the case of the hexagonal lattice, 
the limiting distribution of the instanton marginal was identified in \cite{BdT_loop}, in agreement with what is predicted by the compactified free field. This will be extended in the present article.

By analogy with vertex correlators for the free field, one may consider the asymptotics of discrete observables of the type: $\langle\exp(i\sum_j\alpha_j h(z_j)\rangle$, where $h$ denotes the height function and the weights $\alpha_j$ sum up to zero. The systematic study of these electric vertex correlators for dimer height functions has been initiated by Pinson (\cite{Pinson_vertex}) and will be further discussed here.

A classical question for dimers, introduced by Fisher and Stephenson in \cite{FisSte2}, is the problem of {\em monomer correlations}. Specifically, the question is to estimate the variation of the partition function when a certain number of vertices are removed from the graph. A motivation for this question is a close analogy with Ising correlations, as illustrated by Hartwig's results (\cite{Hartwig_monomer}) for monomer pair correlations on the diagonal. More recently, a series of articles by Ciucu (see \cite{Ciucu_PNAS,Ciucu_TAMS} and references therein) has born on several variants of monomer correlations, in particular correlations of appropriate ``islets" on the hexagonal lattice. The case of monomer insertions on the boundary of the half-square lattice $\Z\times\N$ is addressed in \cite{PriRue_monomer}. In terms of height function, monomer insertions may be thought of as magnetic charges. Indeed, in the presence of monomers (in the bulk), the height function becomes additively multivalued, picking fixed additive constants when cycling around monomers (or larger defects).

\subsection{Families of Cauchy-Riemann operators}

We shall be mostly concerned with establishing convergence of dimer observables (height field, vertex correlators) to the corresponding (compactified) free field quantities. In all cases, the arguments will be based on the analysis of families of Cauchy-Riemann operators and their discrete counterparts. While not logically needed later on, the analogies with existing results on family of CR operators in the continuum guide many of our arguments, especially in Section \ref{sec:torus}.

A Cauchy-Riemann operator in the sense of Quillen (\cite{Qui}) is a first-order differential operator $D:\Omega^0(L)\rightarrow\Omega^{0,1}(L)$ of type: 
$$D=\dbar+\alpha=d\bar z\left(\frac{\partial}{\partial{\bar z}}+a\right)$$ 
operating on sections of a complex line bundle $L$ over a compact Riemann surface $\Sigma$, say; here $\alpha=ad\bar z$ is a $(0,1)$-form. The operator $D$ maps smooth sections of $L$ to smooth $(0,1)$-forms with values in $L$ (i.e. something of type $sd\bar z$, where $s$ is a smooth section of $L$).

For our purposes (working on the plane and on the torus), we can trivialise $L$ and the cotangent bundle, and it will be enough to simply consider first-order differential operators of type:
$$\varphi\longmapsto\left(z\mapsto\frac{\partial\varphi}{\partial{\bar z}}+\alpha(z)\varphi(z)\right)$$ 
operating on smooth functions (on the plane or the torus).

These operators constitute an affine (infinite-dimensional) complex space parameterised by the ``potentials" $\alpha$. Given a metric on $\Sigma$ and $L$ (so that tangent spaces to $\Sigma$ are Euclidean and fibers of $L$ are Hermitian), one can define an adjoint operator $D^*$ and the Laplacian-type operator $D^*D$. The latter has a $\zeta$-regularised determinant $\det_\zeta(D^*D)$ (e.g. Section 9.6 in \cite{BGV_dirac}). In \cite{Qui}, it is shown in particular how to evaluate the logarithmic variation of this determinant when $D$ varies among Cauchy-Riemann operators. This is controlled by the asymptotic expansion of the inverting kernel of $D$ near the diagonal. More precisely, if $D$ is invertible, we have the near diagonal expansion (in the standard local coordinate):
$$D^{-1}(z,w)=D^{-1}_w(z,w)+r_\alpha(w)+O(|z-w|)$$
where $D_w^{-1}$ is the parametrix: 
$$D_w^{-1}(u,v)=\frac{1}{\pi(u-v)}\exp(2i\Im(a(w)\overline{(u-v)}))$$
and $\alpha(w)=a(w)d\bar w$ (here $D_w^{-1}$ inverts the translation invariant operator that agrees with $D$ at $w$). The curvature formula in \cite{Qui} is based on the variational identity (specialised here to the simple case of line bundles on a flat torus):
\begin{equation}\label{eq:Quillenvar}
\frac {d}{dt}\log{\det}_\zeta(D^*_tD_t)=\frac i2\int_\Sigma  r_{\alpha_t}(w)dw\wedge\dot \alpha_t=\int_\Sigma r_{\alpha_t}(w)\dot a_t(w)dA
\end{equation}
where $D_t=(\dbar+\alpha_t(z))$ is a smooth parametric family of CR operators. This is itself a specialisation of the general variational formula for $\zeta$-determinants of Laplacian-type operators (e.g. Proposition 9.38 in \cite{BGV_dirac}).

For instance, consider a flat torus $\Sigma=\C/(\Z+\tau\Z)$ and a flat unitary line bundle corresponding to a unitary character $\chi:\pi_1(\Sigma)\simeq \Z+\tau\Z\rightarrow\U$. Global sections of this bundle are identified with multiplicatively multivalued functions on $\Sigma$ corresponding to the character $\chi$. These sections may be written as 
$$s(z)=f(z)\exp(2i\Re(z\bar\lambda))$$
where $f$ is a (single-valued) function on $\Sigma$ and $\lambda$ is s.t. $\chi(1)=\exp(2i\Re(\lambda))$ and $\chi(\tau)=\exp(2i\Re(\lambda\bar\tau))$. If we write the $\dbar$ operator on this line bundle with this trivialisation, we get the (translation invariant) Cauchy-Riemann operator 
$$D=\dbar+\lambda d\bar z$$
In this case, $\det_\zeta(D^*D)$ is the (square of) the Ray-Singer {\em analytic torsion} (\cite{RS_analytic}). This is a numerical invariant depending on the complex structure of the torus (the modulus $\tau$) and the character $\chi$. The spectrum and semigroup for the Laplacian $D^*D$ are completely explicit, and the evaluation of $\det_\zeta(D^*D)$ reduces to a problem in analytic number theory known as the Second Kronecker Limit formula (see Section 4 in \cite{RS_analytic}).

An alternative approach to the evaluation of $\det_\zeta(D^*D)$ is based on a variational analysis (when varying $\chi$, $\tau$ being fixed). 
Applying the Quillen variational formula \eqref{eq:Quillenvar} to the family $(\dbar+\lambda d\bar z)_\lambda$, one obtains an alternative evaluation (up to multiplicative constant depending on $\tau$) of the analytic torsion; this is carried out in particular by Fay in \cite{Fay_kernel} (Chapter 3). (As we shall see in Section \ref{sec:torus}, mimicking this approach at the discrete level leads to a characterisation of the instanton component of the dimer height field).

More generally, one may also consider CR operators written as 
$$D=\dbar+g_{\bar z}d\bar z$$
which corresponds to changing the metric on the trivial line bundle. Indeed, if we define a metric on the trivial bundle in such a way that $e^g$ is a unitary section, and write a generic section as $s=fe^g$, then with respect to this trivialisation the $\dbar$ operator is written as $D=\dbar+g_{\bar z}d\bar z$. (Remark that the definition of the adjoint $\dbar^*$ involves the metric on the line bundle). Again \eqref{eq:Quillenvar} quantifies the variation of $\det_\zeta(\dbar^*\dbar)$ under change of the metric on the trivial line bundle (``anomaly"), or in terms of this trivialisation the variation
of $\det_\zeta(D^*D)$ under change of $g$. (In the asymptotic analysis, this will correspond to the scalar component of the height/compactified free field).

Then one can consider varying jointly the complex structure of the bundle and its metric; this corresponds to a family
$$D=\dbar+(\lambda+g_{\bar z})d\bar z$$
of Cauchy-Riemann operators. 

In order to deal with vertex correlators, we will need to control more singular situations. For instance in the plane with $n$ punctures $\Sigma=\C\setminus\{z_1,\dots,z_n\}$, consider a fixed unitary character $\chi$ of $\pi_1(\Sigma)$ and the associated complex line bundle $L_{(z_j)}^\chi$ over $\Sigma$ and $\dbar$ operator. (A section of this bundle may be identified with a multiplicatively multivalued function on $\Sigma$, that gets multiplied by $\chi(\ell)$ when tracked along a non-contractible loop $\ell$). 
Here the punctures induce a logarithmic divergence in the evaluation of $\det_\zeta(\dbar_L^*\dbar_L)$ and an additional principal part regularisation is needed (see \cite{Palmer_CR}). One may distinguish two types of variations: a moduli space (isomonodromic) variation (moving the punctures and keeping the character fixed) and a Jacobian variation (changing the character/complex structure of the line bundle with fixed punctures). 

On the discrete side, Kasteleyn (\cite{Kas_square}) has shown how to evaluate the partition function of dimer configurations on a planar graph as the determinant (in the bipartite case) of a nearest neighbour linear operator. This operator may be interpreted as a finite difference version of $\dbar$, which is reflected in the asymptotic expansion of its inverse (\cite{Ken_domino_conformal,Ken_isoradial,KOS}) and underpins the analysis of the scaling limit of dimers.

We will be relying systematically on this interpretation: in order to evaluate variations of the partition function, we will consider the associated modified Kasteleyn operator as a discrete version of a CR operator varying in a family. Variational formulae, both at the discrete and continuous level, involve the behaviour of the inverting kernel near the diagonal (after substracting the leading singularity, which is that of a translation-invariant operator). Convergence of observables will follow from convergence of these leading non-trivial coefficients in the short-range expansion of inverting kernels. For vertex correlators, this will involve a rather detailed study of discrete holomorphic functions with prescribed monodromy.

\subsection{Main results}

We are interested in asymptotic properties of the dimer height function on a suitable graph (e.g. the square lattice $\Z^2$) either for a fixed mesh $\delta=1$ and at large enough scale, or at fixed scale as the mesh $\delta$ goes to zero; $h$ denotes a dimer height function on a fixed lattice, and $h_\delta$ denotes the height on a mesh $\delta$ lattice (infinite volume measures are discussed in Section \ref{subsec:planecharfun}).

In order to outline the general philosophy and structure of the manuscript, we state (somewhat informal) versions of the main results, together with the relevant family of CR operators. The correspondence between dimer observables and families of CR operators (or Hermitian line bundles) goes as follows:
\begin{itemize}
\item scalar field $\leftrightarrow$ variation of the Hermitian norm of the line bundle\\
In terms of operators, this corresponds (in the continuum) to $\dbar+\alpha$, $\alpha$ a smooth potential; this yields e.g. relatively sharp results on height fluctuations in the plane (Corollary \ref{Cor:FFplane}), refining earlier results of \cite{Ken_domino_GFF,dT_isoradial}. This may be stated as follows: as the mesh $\delta$ goes to zero, the distribution of the height ``current" $2\pi dh_\delta$, seen as a probability measure on $H^{-2-\eps}_{loc}(\C)$, converges weakly to the distribution of a free field current $J=d\phi$ with coupling constant $g_0=\frac 12$.
\item instanton component $\leftrightarrow$ variation of the complex structure of the line bundle\\
On a torus, the law of the periods $(\int_A dh,\int_B dh)$ of the dimer height current converges to that of the periods of a compactified free field (current). This was known in the case of the hexagonal lattice \cite{BdT_loop}.\\
Moreover, we have (Theorem \ref{Thm:toruscompact}): as the mesh goes to zero, the dimer height current $2\pi dh_\delta$ converges in law to a compactified free field current with radius of compactification 1 and coupling constant $g_0=\frac 12$.\\
The result is stated as a functional invariance principle (convergence of probability measures on abstract Wiener space); a strong invariance principle (quantitative coupling between $dh_\delta$ and the limiting current $J$) could also be obtained from these methods.
\item variation of electric vertex correlators $\langle:\exp(i\sum_j\alpha_j\phi(z_j):\rangle$ w.r.t. position of insertions $\leftrightarrow$ 
family $(z_1,\dots,z_n)\mapsto L^\chi_{(z_j)}$ \\
In the plane (say in the fixed mesh, large separation regime), if $x_1,y_1,\dots,x_n,y_n$ are at large and comparable pairwise distances, and $s_1,\dots,s_n$ are small enough, then (Proposition \ref{Prop:electrsmalls})
$$\langle\exp(2i\pi\sum s_j(h(y_j)-h(x_j)))\rangle_{\rm dimer}=\langle:\exp(i\sum s_j(\phi(y_j)-\phi(x_j)):\rangle_{\rm GFF}(c+o(1))$$
where $\phi$ is a free field with $g_0=\frac 12$, and the positive constant $c$ depends on the exponents $s_j$ and (possibly) on the lattice (recall \eqref{eq:electcorrGFF}).
\item  variation of electric vertex correlators $\langle:\exp(i\sum_j\alpha_j\phi(z_j):\rangle$ w.r.t. charges $\leftrightarrow$ 
family $\chi\mapsto L^\chi_{(z_j)}$ \\
This allows to state Theorem \ref{Thm:electr}, which extends the previous result to the case $s_j\in (-\frac 12,\frac 12)$ (this condition is sharp, at least in the case of one pair).
\item Monomer correlations.\\
See Section \ref{ssec:FSintro} for a precise definition of the monomer correlation
$$\Mon_M(b_1,\dots,b_p,w_1,\dots,w_p)$$
where $M$ is the graph carrying dimer configurations, the $b_j$'s (resp. $w_j$'s) are black (resp. white) vertices; it may be thought of as the ratio of the dimer partition function in $M\setminus\{b_1,\dots,w_p\}$ by the partition function in $M$.

This corresponds to the family $(b_1,w_1,\dots,b_p,w_p)\mapsto L^\chi_{(b_1,\dots,w_p)}$ with $\chi(\ell)=-1$ for any loop $\ell$ encircling a single puncture; this is the limiting case $s\nearrow\frac 12$ of the objects considered in the analysis of electric correlators. However, the situation is now complicated by the appearance of growth conditions at the punctures, that distinguish black from white monomers. For instance, when $p=1$, the relevant line bundle over $\hat \C\setminus\{x,y\}$ has local sections $s$ with monodromy $(-1)$ around $x,y$ and $s(z)=O(\sqrt{\frac{z-x}{z-y}})$ near $x,y$.\\ 
The conclusion of that analysis is Theorem \ref{Thm:FS}: 
$$\Mon(b_1,\dots,w_p)=\langle:{\mc O}_1(w_1)\dots{\mc O}_1(w_p){\mc O}_{-1}(b_1)\dots{\mc O}_{-1}(b_p):\rangle_{\rm GFF}(c+o(1))$$
for large and mutually comparable pairwise distances (recall \eqref{eq:GFFmagncorr}, where again $g_0=\frac 12$). 
\end{itemize}

Remark that in recent work (\cite{Ken_double}), Kenyon employs related variational argument for rank 2 bundles in his study of the double dimer model.

\section{Dimers on isoradial graphs}

In this section, we are following the formalism of \cite{Duffin_rhombic,Mercat,Ken_isoradial,SmiChe_isoradial} for isoradial (critical) graphs. Notations and conventions are mostly as in \cite{SmiChe_isoradial}.

\subsection{Kasteleyn operator}

\subsubsection{Critical graphs}\label{sss:critgraphs}

Consider a tiling $\gls{Lambda}$ of the plane by rhombi with edge length $\delta$. As faces have even degree, it is bipartite (ie there is a 2-colouring of vertices, say red and blue). Thus one may obtain two graphs from this tiling: the vertices of $\gls{Gamma}$ are the blue vertices and are joined by an edge if they are vertices of the same rhombus; similarly the vertices $\Gamma^\dg$ are the red vertices of $\Lambda$. The graphs $\Gamma$ and $\gls{Gammadg}$ are dual. They are also {\em isoradial} in the sense that each face of $\Gamma$ is inscribed in a circle of radius $\delta$, the center of which is the corresponding vertex of the dual $\Gamma^\dg$. The dual of $\Lambda$ is denoted $\dmd$. By abuse of terminology, we sometimes identify a graph with the set of its vertices.

We will work under the following assumption (see \cite{SmiChe_isoradial}):
\begin{equation}\label{eq:spade}
\gls{spadesuit}:\textrm{\ \ the lozenge angles belong to }[\eta_0,\pi-\eta_0]\textrm{\ for some fixed\ }\eta_0\in (0,\pi)\end{equation}

{\em Here $\eta_0>0$ is fixed once and for all; throughout, ``absolute" constants may depend on $\eta_0$ (but not on $\Lambda$ etc). In particular, all error estimates on Green kernels, Cauchy kernels and their variants are uniform in $\Lambda$ under $(\spadesuit)$.}

One can form a new bipartite graph $\gls{M}$ as follows: black (resp. white) vertices of $M$ are the vertices (resp. centers) of rhombi in $\Lambda$. Edges of $M$ are half-diagonals of the rhombi in $\Lambda$. We will be interested in perfect matchings of $M$ (or subgraphs of $M$). We refer to vertices of $M$ as {\em nodes}. Black nodes corresponding to vertices of $\Gamma$ (resp. vertices of $\Gamma^\dg$) are {\em vertex nodes} (resp. {face nodes}); white nodes are {\em edge nodes}. Denote $M_V\simeq \Gamma$, $M_F\simeq \Gamma^\dg$, $M_W\simeq\dmd$ the sets of vertex nodes, face nodes and edge nodes (ie white nodes) respectively, and $M_B=M_V\sqcup M_F\simeq\Lambda$ (black nodes). See Figure \ref{fig:loztil}. Notice that $M$ is itself isoradial, with all faces inscribed in circles of radius $\delta/2$.

\begin{figure}[htb]
\begin{center}
\leavevmode
\includegraphics[width=0.8\textwidth]{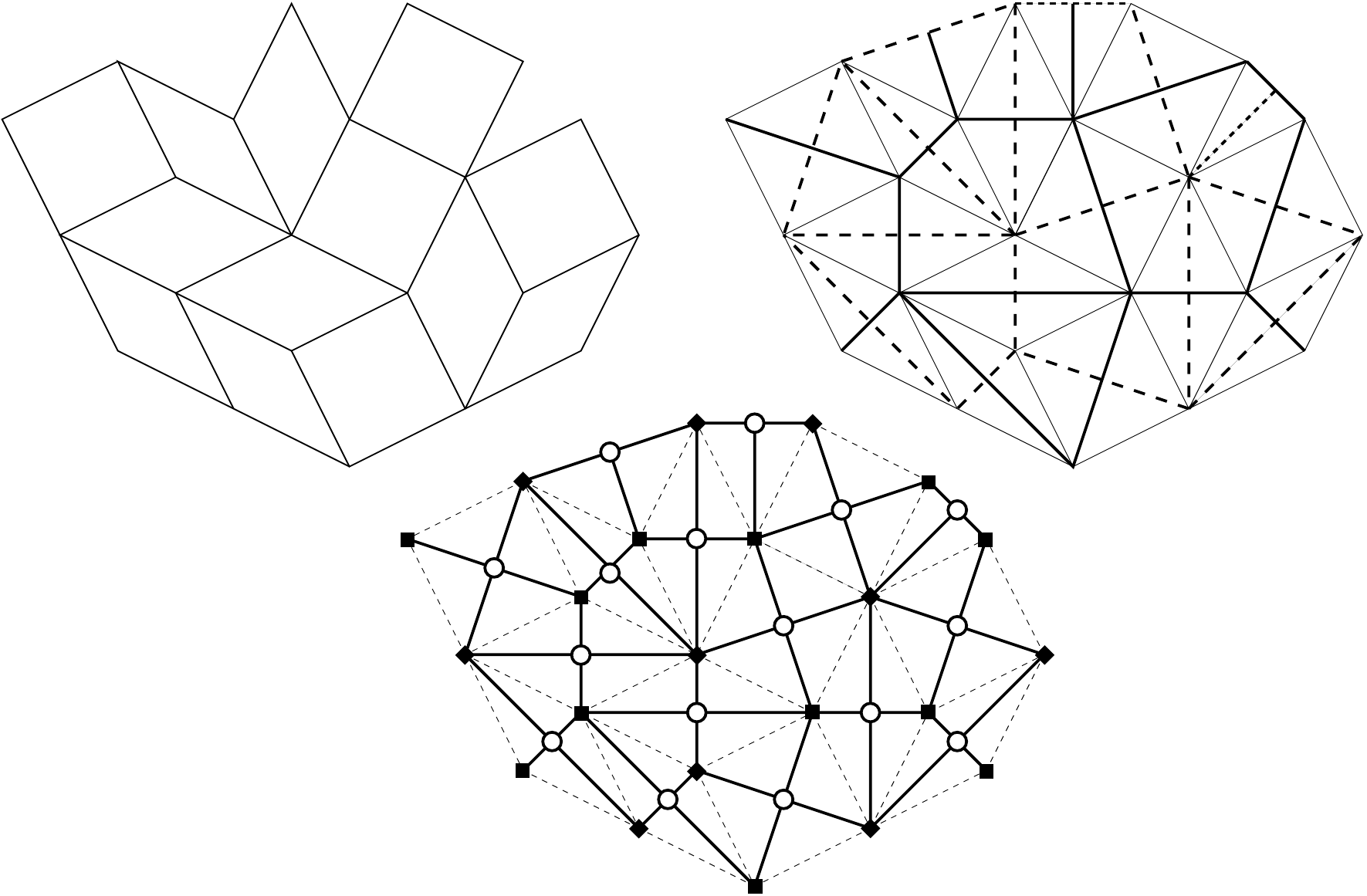}
\end{center}
\caption{(1) A portion of a rhombus tiling $\Lambda$. (2) Dual graphs $\Gamma$ (solid) and $\Gamma^\dg$ (dashed) obtained from $\Lambda$. (3) Corresponding bipartite graph $M$.}
\label{fig:loztil}
\end{figure}

Let $M_B$ (resp. $M_W$) be the set of black (resp. white) vertices of $M$. The {\em Kasteleyn operator} $\gls{K}:\C^{M_B}\rightarrow\C^{M_W}$ is given by (see Section 3.1 of \cite{Ken_isoradial}, where it is denoted by $K_{BW}$):
$$(Kf)(w)=\sum_{b\sim w}K(w,b)f(b)$$
where $K(w,b_0)=\frac i2(b_3-b_1)$ if $b_0,b_1,b_2,b_3$ are the black neighbours of $w$ listed in counterclockwise order; $\sim$ denotes adjacency in the relevant graph, here $M$. (Remark that in the present article, we consider only the special subclass of isoradial graphs, the ``Temperleyan" ones).

For our purposes, it will be convenient to consider a real operator $\gls{rK}:\R^{M_B}\rightarrow\R^{M_W}$ which differs from $K$ by the argument of its entries: $\rK(w,b)=\pm |K(w,b)|$, where the sign depends on a chosen orientation of $(w,b)\in E_M$, which now we describe (see \cite{dT_quadri}). Pick an arbitrary orientation of edges of $\Gamma$; then orient edges of $\Gamma^\dg$ in such a way that if $(xx')$ is an oriented edge of $\Gamma$, $(yy')$ the dual oriented edge of $\Gamma^\dg$, then $((xx'),(yy'))$ is a direct frame. Each edge of $M$ is a half-edge of an edge of either $\Gamma$ or $\Gamma^\dg$ and inherits
its orientation. Then $\rK(w,b)=+|K(w,b)|$ (resp. $\rK(w,b)=-|K(w,b)|$) if $(wb)$ is positively (resp. negatively) oriented. One readily checks that this yields a {\em Kasteleyn orientation} of $M$, i.e. an orientation such that each face has an odd number of clockwise oriented edges on its boundary.

For brevity of notations it will also be convenient to assign an argument $\gls{nu}$ to each vertex of $M$. 
If $w\in M_W$ is on the oriented edge $(xx')$ of $\Gamma$, set $e^{i\nu(w)}=\frac{x'-x}{|x'-x|}$; also set $e^{i\nu(b)}=1$ if $b\in\Gamma$ is a vertex node and $e^{i\nu(b)}=i$ if $b\in\Gamma^\dg$ is a face node. Then it is easy to check that
\begin{equation}\label{eq:Kastgauge}
K(w,b)=e^{i\nu(w)}\rK(w,b)e^{i\nu(b)}
\end{equation}
for all $b,w$, i.e. $\rK$ is obtained by composing $K$ with diagonal operators.

Translating results of \cite{Ken_isoradial} obtained for $K$ in terms of $\rK$ (see also \cite{SmiChe_isoradial}), we have the following fundamental results on the inverting kernel $\uK^{-1}$:

\begin{Thm}[\cite{Ken_isoradial}]\label{Thm:Kencrit}
\begin{enumerate}
\item %
For any $w\in M_W$, there is a unique $f_w\in\R^{M_B}$ satisfying $\rK f_w=\delta_w$ and $f_w(b)\xrightarrow[|b|\rightarrow\infty]{} 0$. Set $\gls{uKinv}(b,w)=f_w(b)$.
\item If $x,y,x',y'$ are black vertices in ccwise order around $w$, 
$$\rK(w,x)\uK^{-1}(x,w)=\frac{1}{2\pi}\arg\left(\frac{y'-x}{y-x}\right).$$\\
\item As $|b-w|\rightarrow\infty$,
$$\uK^{-1}(b,w)=\Re\left(\frac{e^{i\nu(w)}e^{i\nu(b)}}{\pi(b-w)}\right)+O\left(\frac{\delta}{|b-w|^{2}}\right)$$
\end{enumerate}
\end{Thm}
\begin{proof}
See respectively Theorem 4.1, Theorem 4.2 and Corollary 7.4 in \cite{Ken_isoradial}, and also Theorem 2.14 in \cite{SmiChe_isoradial} for the general case.
\end{proof}

{\bf Remarks.}\begin{enumerate}
\item
If $E,F$ are (say, finite) sets, $T:E\times F\rightarrow\C$ (a ``kernel"), we may consider the corresponding linear operator (also denoted by $T$ by a slight abuse of notation) $T:\C^F\rightarrow\C^E$ defined by
$$(T\varphi)(e)=\sum_{f\in F}T(e,f)\varphi(f)$$
Throughout the article we will identity the kernel $T:E\times F\rightarrow\C$ with the corresponding ``integral" operator $T:\C^F\rightarrow\C^E$; this applies to $\rK,\uK^{-1}$ and their variants. A useful heuristic is to think of $\uK^{-1}$ as a discretisation of the classical integral operator
$$\varphi\longmapsto\left(w\mapsto\iint\varphi(z)\frac{dA(z)}{\pi(z-w)}\right)$$
(for, say, $\varphi$ smooth and compactly supported on $\C$), which inverts $\frac{\partial}{\partial \bar z}$ ($\dbar$-Poincar\'e Lemma, see e.g. Theorem 1.28 in \cite{Voisin_Hodge}).

\item Remark the following simple scaling property. If $\Lambda$ is a lozenge tiling with mesh $1$, with associated operator $\rK_\Lambda$ (and inverting kernel $\uK_\Lambda^{-1}$), one may consider the tiling $\delta\Lambda$ with mesh $\delta$ and associated operator $\rK_{\delta\Lambda}$ (and inverting kernel $\uK^{-1}_{\delta\Lambda}$). Then for all $b\in M_B,w\in M_W$, 
\begin{align*}
\rK_{\delta\Lambda}(\delta b,\delta w)&=\delta\rK_{\Lambda}(b,w)\\
\uK^{-1}_{\delta\Lambda}(\delta b,\delta w)&=\delta^{-1}\uK_{\Lambda}(b,w)
\end{align*}
Observe that the asymptotic expansion of $\uK^{-1}$ given in Theorem \ref{Thm:Kencrit} is invariant under scaling (replacing $\Lambda$ with $\delta\Lambda$); fixing the position of $b,w$ and letting $\delta\searrow 0$ or fixing $\delta=1$ and letting $|b-w|\rightarrow\infty$ are equivalent points of view.

\item As mentioned below \eqref{eq:spade}, the error is uniform in the tiling $\Lambda$ under $(\spadesuit)$. Let us phrase this uniformity extensively (once and for all): there is $C=C(\eta_0)>0$ s.t. for all tiling $\Lambda$ with mesh $\delta$ satisfying $(\spadesuit)$, and all $b\in M_B$, $w\in M_W$, we have
$$\left|\uK^{-1}(b,w)-\Re\left(\frac{e^{i\nu(w)}e^{i\nu(b)}}{\pi(b-w)}\right)\right|\leq C\left(\frac{\delta}{|b-w|^{2}}\right)$$
\end{enumerate}

\subsubsection{Interpretation as finite difference operators}\label{sss:finitediff}

\paragraph{First-order operators.\\}

Define restriction operators $\rR_B: C^0(\R^2,\C)\rightarrow \R^{ M_B}$ and $\rR_W: C^{(0,1)}(\R^2,\C)\rightarrow\R^{ M_W}$ by:
\begin{equation}\label{eq:Rrestrop}
\begin{array}{rlll}
(\rR_B\varphi)(b)&=\Re(e^{i\nu(b)}\varphi(b))&&b\in M_B\\
(\rR_W\alpha)(w)&=\Re(e^{-i\nu(w)}\alpha/d\bar z) &&w\in M_W
\end{array}
\end{equation}
that restrict complex-valued functions (resp. $(0,1)$-forms) to black (resp. white) nodes. Note that these operators are only $\R$-linear. 
For $\varphi$ a function on $\C$, we define restriction operators $\gls{RB},\bar R_B: C^0(\R^2,\C)\rightarrow \C^{ M_B}$ and $\gls{RW}: C^{(0,1)}(\R^2,\C)\rightarrow\C^{ M_W}$, $\bar R_W: C^{(1,0)}(\R^2,\C)\rightarrow\C^{ M_W}$,  by:
\begin{equation}\label{eq:Crestrop}
\begin{array}{rlll}
(R_B\varphi)(b)=(\bar R_B\varphi)(b)&=\varphi(b)&&b\in \Gamma\\
(R_B\varphi)(b)=-(\bar R_B\varphi)(b)&=i\varphi(b)&&b\in \Gamma^\dg\\
(R_W(fd\bar z))(w)&=e^{-i\nu(w)}f(w) &&w\in M_W\\
(\bar R_W(fdz))(w)&=e^{i\nu(w)}f(w) &&w\in M_W
\end{array}
\end{equation}
In other words, $(R_B\varphi)(b)=e^{i\nu(b)}\varphi(b)$ and $\bar R_B\bar\varphi=\overline{R_B\varphi}$. Similarly $\bar R_W\bar\alpha=\overline{R_W\alpha}$.

If $\varphi$ is a $C^1$ function on $\C$, a direct computation shows that:
\begin{align*}
(\rK(\rR_B\varphi))(w)&=
\frac 12\Re\left( |y-y'|(\varphi(w+u.\frac{|x'-x|}2)-\varphi(w-u.\frac{|x'-x|}2)+i|x-x'|(\varphi(w+iu.\frac{|y'-y|}2)-\varphi(w-iu.\frac{|y'-y|}2)\right)\\
&=2\mu_\dmd(w)\left(\Re\left(\frac{\partial\varphi}{\partial{\bar w}}(w)\bar u\right)+o(1)\right)=2\mu_\dmd(w)(\rR_W(\dbar\varphi)(w)+o(1))
\end{align*}
where $(xx')$ is the oriented edge of $\Gamma$ corresponding to $w$, $(yy')$ the dual edge, $u=e^{i\nu(w)}$ and $\mu_\dmd(w)$ is the area of the rhombus of $\Lambda$ containing $w$ (and thus of order $O(\delta^2)$).

Similarly, we have
\begin{align}\label{eq:findiff}
\begin{split}
\rK(R_B\phi)&=2\mu_\dmd R_W(\dbar\phi)+O(\delta^2\omega_{\phi'}(\delta))\\
&=2\mu_\dmd R_W(\dbar\phi)+O(\delta^3\omega_{\phi''}(\delta))
\end{split}
\end{align}
where $\omega$ designates the modulus of continuity, $\phi'$ and $\phi''$ the gradient and Hessian of $\phi$, and $\mu_\dmd$ is seen as a diagonal operator on $\C^{M_W}$. The second line uses symmetry of a lozenge around its centre $w$ and shows that $\rK$ is a ``surconvergent" first-order difference scheme (first-order significant term and third-order remainder). This explains ``better than expected" error terms later on, as e.g. in Lemma \ref{Lem:monholom}.

We will also consider $\rK_\alpha,K_\alpha$, perturbations of $\rK,K$ by a $(0,1)$-form $\alpha$. Define
\begin{align}\label{eq:Kastpert}
\begin{split}
\gls{Kalpha}(w,b)&=\rK(w,b)\exp(2i\Im\int_w^b\alpha)\\
\gls{rKalpha}(w,b)&=\rK(w,b)\exp(2\Re\int_w^b\alpha)
\end{split}
\end{align}
where $\alpha$ is a smooth $(0,1)$-form: $\alpha=a(z)d\bar z$. 
Observe that
$$K_\alpha(R_B\varphi)(w)=\rK(R_B(\tilde\varphi))(w)$$
where $\tilde\varphi(z)=\varphi(z)\exp(2i\Im\int_w^z\alpha)$. We deduce (here $\alpha=ad\bar z$)
\begin{equation}\label{eq:findiffpert}
\begin{array}{rl}
K_\alpha(R_B\phi)&=2\mu_\dmd R_W(\dbar\varphi+\varphi\alpha)+\delta^2O(\omega_{\varphi'}(\delta)+\delta\|\varphi'\|_\infty\|a\|_{\infty}+\|\varphi\|_\infty\omega_a(\delta)+\|\varphi\|_\infty\delta\|a\|^2_\infty)\\
&=2\mu_\dmd R_W(\dbar\varphi+\varphi\alpha)+O(\delta^3\|\varphi\|_{C^2}(1+\|a\|_{C^1})^2)\\
&=2\mu_\dmd R_W(\dbar\varphi+\varphi\alpha)+O(\delta^4\|\varphi\|_{C^3}(1+\|a\|_{C^2})^3)
\end{array}
\end{equation}

In the case $\alpha=\lambda d\bar z$ for some fixed $\lambda\in\C$, we denote simply $\rK_\lambda=\rK_{\lambda d\bar z}$. We have $\rK_\lambda(w,b)=e^{-2\Re(\bar\lambda w)}\rK(w,b)e^{2\Re(\bar\lambda b)}$ and consequently
$$\uK_\lambda^{-1}(b,w)\stackrel{def}{=}e^{-2\Re(\bar\lambda b)}\uK^{-1}(b,w)e^{2\Re(\bar\lambda w)}$$
satisfies $\rK_\lambda\uK_\lambda^{-1}=\Id$.

Similarly, we obtain
$$K_\alpha(\bar R_B\varphi)=2\mu_\dmd \bar R_W(\partial\varphi-\varphi\bar\alpha)
+O(\delta^4\|\varphi\|_{C^3}(1+\|a\|_{C^2})^3)$$

Fix $w_0\in M_W$ and set $\lambda=a(w_0)$. Then:
\begin{equation}\label{eq:close}
K_\alpha(w,b)=e^{2i\Im(\bar\lambda w)}\rK(w,b)e^{-2i\Im(\bar\lambda b)}+O(\delta^2\omega_a(|b-w_0|))
\end{equation}

\paragraph{Laplacian.\\}

On $\R^{M_B}$, take the inner product given by counting measure and on $\R^{M_W}$:
$$\langle f,g\rangle_{\R^{M_W}}=\sum_{w}\mu_\dmd(w)f(w)g(w)$$
where $\mu_\dmd(w)$ is the area of the rhombus of $\Lambda$ containing $w$. With respect to these inner products, the adjoint operator $\rK^*:\R^{M_W}\rightarrow\R^{M_B}$ has matrix elements:
$$\rK^*(b,w)=\mu_\dmd(w)^{-1} \rK(w,b)$$
Then it is easy to check that $(\rK^*\rK)(b,b')=0$ for all $b\in \Gamma,b'\in \Gamma^\dg$, and if $x\sim x'$ in $\Gamma$ and $(yy')$ is the edge of $\Gamma^\dg$ dual to $(xx')$, then
\begin{align*}
(\rK^*\rK)(x,x')=\frac{|y'-y|}{2|x'-x|}&&(\rK^*\rK)(x,x)=-\sum_{x'\sim x}(\rK^*\rK)(x,x')\\
 (\rK^*\rK)(y,y')=\frac{|x'-x|}{2|y'-y|}&&(\rK^*\rK)(y,y)=-\sum_{y'\sim y}(\rK^*\rK)(y,y')
\end{align*}
so that with respect to the decomposition $\R^{M_B}=\R^{M_V}\bigoplus\R^{M_F}$, 
\begin{equation}\label{eq:KstarK}
\rK^*\rK=\Lap_\Gamma\bigoplus\Lap_{\Gamma^\dg}
\end{equation}
where $\gls{LapGamma}$ is a graph Laplacian for appropriate edge weights (\cite{Ken_isoradial}).

Similarly, 
\begin{align*}
\sum_{k=1}^n\Lap_\Gamma(x,x_k)(x_k-x)&=\sum_{k=1}^n\frac{i(y_{k-1}-y_k)}{2(x_k-x)}(x_k-x)=0\\
\sum_{k=1}^n\Lap_\Gamma(x,x_k)(x_k-x)^2&=\sum_{k=1}^n\frac{i(y_{k-1}-y_k)}{2(x_k-x)}(x_k-x)^2=0
\end{align*}
which shows that the random walk on $\Gamma$ associated to $\Lap_\Gamma$ converges to isotropic Brownian motion.

\subsubsection{Basic estimates}\label{sss:basic}

We have seen that $\Lap_\Gamma$ and $\rK$ may be seen as finite-difference approximations of $\Lap$ and $\dbar$ respectively. We list a few estimates and results relating discrete harmonic/holomorphic functions with their continuous counterparts, and refer the reader to \cite{SmiChe_isoradial} for a thorough treatment.

\paragraph{Maximum principle.\\}

Let $\Omega_\delta$ be a bounded subset of $\Gamma$ ($\delta$ is the mesh of $\Lambda$), $\partial\Omega_\delta$ its boundary (vertices in $\Gamma\setminus\Omega_\delta$ adjacent to a vertex in $\Omega_\delta$), and $\bar\Omega_\delta=\Omega_\delta\sqcup\partial\Omega_\delta$. If $h:\bar\Omega_\delta\rightarrow\R$
is harmonic on $\Omega_\delta$ (in the sense that $\Lap_\Gamma h=0$ on $\Omega_\delta$), then we have the classical maximum principle:
\begin{equation}\label{eq:maxprincip}
\sup_{\bar\Omega_\delta}|h|=\sup_{\partial\Omega_\delta}|h|
\end{equation}
(this is valid for the generator of any irreducible nearest-neighbour Markov chain on a graph). It follows that $h$ is uniquely determined by its values on $\partial\Omega_\delta$ (uniqueness in the Dirichlet boundary value problem).

\paragraph{Harnack estimate.\\}

With the same notations and assumptions, let $v\sim v'$ be adjacent vertices in $\Gamma$ such that $\partial\Omega_\delta$ is at (Euclidean) distance at least $R\delta$ of $v,v'$ ($\delta$ is the mesh of $\Lambda$). Then
\begin{equation}\label{eq:harnack}
|h(v)-h(v')|\leq cR^{-1}\sup_{\partial\Omega_\delta}|h|
\end{equation}
where $c$ depends only on $\eta_0$ in $(\spadesuit)$ (see Proposition 2.7 in \cite{SmiChe_isoradial}). 

This is stronger than the Liouville property: if $h:\Gamma\rightarrow\R$ is bounded and harmonic, it is constant.

\paragraph{Dirichlet boundary value problem.\\}

Assume furthermore that $\partial\Omega_\delta$ is within $O(\delta)$ of, say, a piecewise $C^1$ simple loop enclosing a simply-connected domain $\Omega$; and that $\tilde h$ is a continuous function defined in and near $\Omega$, $1$-Lipschitz near $\partial\Omega$, and harmonic in $\Omega$: $\Lap\tilde h=0$ in $\Omega$. Let $F$ be a compact subset of $\Omega$. Then (see Theorem 3.10 in \cite{SmiChe_isoradial}) there is $c=c(\Omega,F,\eta_0)$ s.t.
\begin{equation}\label{eq:dirprob}
\begin{split}
\sup_{\Gamma\cap F}|h-\tilde h|&\leq c\sup_{\partial\Omega_\delta}|h-\tilde h|\\
\sup_{v\in\Gamma\cap F,v'\sim v}\left|(h(v')-h(v))-(\tilde h(v')-\tilde h(v))\right|&\leq c\delta\sup_{\partial\Omega_\delta}|h-\tilde h|
\end{split}
\end{equation}
The first line may be interpreted as $C^0$ convergence of the discrete Dirichlet boundary value problem (BVP) to its continuous limit; the second line is a version of $C^1$ convergence.

\paragraph{Cauchy boundary value problem.\\}

We turn to discrete holomorphic functions. Let $f:M_B\rightarrow\C$ be s.t. $\rK f=0$ on a subset $\Omega_W$ of $M_W$. Then by \eqref{eq:KstarK}, $f_{|\Gamma}$ is harmonic at every $b\in\Gamma$ the white neighbours of which are all in $\Omega_W$ (and similarly for $f_{|\Gamma^\dg}$). Let $\Omega_B$ be this set of black vertices.

Let $\tilde f_1$ (resp. $\tilde f_2$) be a holomorphic (resp. antiholomorphic) function extending continuously to $\bar\Omega$, where $\Omega$ is a domain with $\Omega_W=\Omega\cap M_W$. Let $F$ be a compact subset of $\Omega$; then from \eqref{eq:Crestrop} and \eqref{eq:dirprob} we have
\begin{equation}\label{eq:dirprobcauchy}
\sup_F\left|f-R_B(\tilde f_1)-\bar R_B(\tilde f_2)\right|\leq c\sup_{\partial\Omega_B}
\left|f-R_B(\tilde f_1)-\bar R_B(\tilde f_2)\right|  
\end{equation}
which may be combined with the Cauchy integral formula
$$\tilde f_1(w)=\frac{1}{2i\pi}\oint_{\partial\Omega}\frac{\tilde f_1(z)dz}{z-w}$$
(for $w\in\Omega$, a simply connected domain with Jordan boundary $\partial\Omega$) in order to approximately evaluate a discrete holomorphic function near $F$ given its values near $\partial\Omega$.

One can also obtain an exact formula as follows (see Proposition 2.22 in \cite{SmiChe_isoradial}). Let $f:M_W\rightarrow\C$ be s.t. $\rK f=0$ on $\Omega_W\subset M_W$, $\Omega_W$ finite. Let $\Omega'_B$ denotes the set of black neighbours of vertices in $\Omega_W$. Then
$$f\ind_{\Omega'_B}=\sum_{w\in M_W}\uK^{-1}(.,w)\rK(f\ind_{\Omega'_B})(w)=\uK^{-1}(\rK(f\ind_{\Omega'_B}))$$
Indeed, the RHS is well-defined (it is actually a finite sum), and the difference between the two sides is a discrete holomorphic function vanishing at infinity, hence (Liouville) is identically zero. Let $\gamma_W$ be the set of white vertices the black neighbours of which are not all in $\Omega'_B$ or not all in $M_B\setminus\Omega'_B$. Then $\rK(f\ind_{\Omega'_B})$ is supported on $\gamma_W$ and
$$f\ind_{\Omega'_B}=\sum_{w\in \gamma_W}\uK^{-1}(.,w)\rK(f\ind_{\Omega'_B})(w)$$
In particular when $\Omega_W=\Omega\cap M_W$, the RHS can be seen as a discretisation of a contour integral $\oint_{\partial\Omega}$, and this may be combined with Theorem \ref{Thm:Kencrit} and an examination of the local geometry near $\gamma_W$ to obtain an approximation of $f(b)$ for $b$ inside of and far from $\partial\Omega$ (in graph distance).

Remark that this argument requires only existence and uniqueness of the inverting kernel $\uK^{-1}$, which will allow for useful variants. The general idea is that controlling the inverting kernel allows to transfer information on discrete holomorphic functions from a contour $\partial\Omega$ to the bulk $\Omega$ (an argument we will use repeatedly in Sections \ref{Sec:surgery}, \ref{sec:electr} and \ref{sec:mono}).

\paragraph{Integration.\\}

Let $x_1,\dots,x_n$ be the neighbours of a vertex $x$ in $\Gamma$, in ccwise order (with cyclical indexing $x_n=x_0$). Let $y_k$ be the vertex of $\Gamma^\dg$ corresponding to the face of $\Gamma$ on the left-hand side of the oriented edge $(xx_k)$ and $w_k\in M_W$ the vertex corresponding to $(xx_k)$. Then for $\alpha=\lambda d\bar z$, $\lambda\in\C$ constant,
\begin{align*}
\sum_{k=1}^n\mu_\dmd(w_k)e^{i\nu(w_k)}(R_W\alpha)(w_k)&=\sum_{k=1}^n\frac i2(y_{k-1}-y_{k})\Re(\lambda(\overline{x_k-x}))\\
&=\sum_{k=1}^n\frac i4\left(\lambda (y_{k-1}-y_{k})(\overline{x_{k}-x})+\overline\lambda (y_{k-1}-y_{k})(x_k-x)\right)=\frac\lambda 2\sum_{k=1}^n\mu_\dmd(w_k)
\end{align*}
taking into account $x_k-x=(y_{k-1}-x)+(y_k-x)$ and $\mu_\dmd(w_k)=\frac i2(y_{k-1}-y_k)(\overline {x_k-x})$. Hence if $\varphi$ is continuous on a region $D\subset\C$, 
\begin{equation}\label{eq:int}
\sum_{w\in M_B\cap D}\mu_\dmd(w)e^{i\nu(w_k)}(\rR_W\varphi)(w)=\frac 12\int_D\frac{\alpha}{d\bar z} dA+o(1)=-\frac i4\int_D\alpha\wedge dz+o(1)
\end{equation}
since $dA=dx\wedge dy=-\frac i2 dz\wedge d\bar z$.

\subsection{Height function}\label{ss:height}

Consider a finite bipartite graph $\gls{Xi}$ which is a subgraph of $M$ bounded by a simple closed cycle on $M$. 
A {\em perfect matching} $\gls{m}$ of the bipartite graph $\Xi$ is a subset of edges of $\Xi$ such that every vertex of $\Xi$ is incident to exactly one edge in $\mfm$. The weight of a matching $\mfm$ is the product of the weights of the edges present in ${\mfm}$:
$$w(\mfm)=\prod_{(bw)\in\mfm}|\rK(b,w)|$$
The {\em partition function} of the model is the sum of these weights over all possible perfect matchings of $\Xi$:
\begin{equation}\label{eq:dimerpartfun}
\gls{mcZ}=\sum_{\mfm}\prod_{(bw)\in\mfm}|\rK(b,w)|
\end{equation}
We will consider the probability measure $\P$ on perfect matchings $\mfm$ of $\Xi$ (if they exist) given by 
$$\P(\mfm=\mfm_0)=\frac{w(\mfm_0)}{{\mathcal Z}}$$

In order to have at least one matching, it is necessary that $\Xi_B=\Xi\cap M_B$ and $\Xi_W=\Xi\cap M_W$ have the same number of vertices. Let us also denote by $\rK:\R^{\Xi_B}\rightarrow\R^{\Xi_W}$ the restriction of $\rK:\R^{M_B}\rightarrow\R^{M_W}$, ie 
$$(\rK f)(w)=\sum_{b\in \Xi_B}\rK(b,w)f(b)$$

The fundamental result, due to Kasteleyn \cite{Kas_square} (see also \cite{Per_onemore}), is the following determinantal enumeration formula:
$${\mathcal Z}=\pm\det(\rK)$$
(this is where the Kasteleyn orientation condition is required). Notice that $\rK:\R^{\Xi_B}\rightarrow\R^{\Xi_W}$ is not an endomorphism; however, $\R^{\Xi_B}$ and $\R^{\Xi_W}$ have canonical bases (up to permutation) with respect to which this determinant is evaluated (up to sign). Also, 
$$\det(K)=\det(\rK)\prod_{w\in\Xi_W} e^{i\nu(w)}\prod_{b\in\Xi_B} e^{i\nu(b)}$$
so that ${\mathcal Z}=|\det(K)|$.

Moreover, we can also evaluate the partition function for arbitrary (complex) edge weights. If we replace the positive weight $|\rK(b,w)|$ with $|\rK(b,w)|u(b,w)$, $u(b,w)\in\C$, then:
$${\mathcal Z}'=\sum_\mfm\prod_{(bw)\in\mfm}|\rK(b,w)|u(b,w)=\pm\det(\rK')$$
where $\rK':\C^{\Xi_B}\rightarrow\C^{\Xi_W}$ is given by $(\rK'f)(w)=\sum_{b\in\Xi_B}\rK(b,w)u(b,w)f(b)$. The undetermined sign is the same as before, and thus:
\begin{equation}\label{eq:ratiopart}
\E(\prod_{(bw)\in\mfm}u(b,w))=\frac{{\mathcal Z'}}{{\mathcal Z}}=\frac{{\det \rK'}}{\det \rK}=\det(\rK'\rK^{-1})=\det(K'K^{-1})
\end{equation}
where $(K'f)(w)=\sum_{b\in\Xi_B}K(b,w)u(b,w)f(b)$.

Let us now discuss the height function, see Section 5 in \cite{KPW} or Section 2 in \cite{Ken_IAS}. The height function is defined on the ``dual" $\Xi^\dg$, defined here as the subgraph of $M^\dg$ whose vertices correspond to faces of $M$ that are adjacent to vertices of $\Xi$; we can take the vertices of $\Xi^\dg$ to be the midpoints of edges of $\Lambda$. Given a perfect matching $\mfm$ of the bipartite graph $\Xi$, one can define a closed 1-form on its dual $\Xi^\dagger$ (ie. an antisymmetric function on oriented edges of $\Xi^\dagger$) as follows. Consider the 1-form $\omega\in C^1(\Xi^\dagger)$ defined by:
$$\omega((bw)^\dagger)=\left\{   \begin{array}{ll} 
      1 &\textrm{if } $(bw)$ \textrm{ matched, oriented from black to white}\\
      0 &\textrm{otherwise} \\
   \end{array}\right.$$ 
Then $d\omega\in C^2(\Xi^\dagger)\sim C^0(\Xi)$ is $1$ on black vertices, $-1$ on white vertices, where $(d\omega)(f)$ is the sum of $\omega$ on the edges bounding the face $f$ of $\Xi^\dagger$, oriented counterclockwise. 
Given the embedding of $\Xi$ in the plane, one can construct a fixed 1-form $\omega_0\in C^1(\Xi^\dagger)$ as follows (here $(bw)$ is an edge of $\Xi$; the black neighbours of $w$ are $x,y,x',y'$ in this order):
$$\omega_0((bw)^\dagger)=\frac 1{2\pi}\arg\frac{y'-x}{y-x}=\rK(w,b)\uK^{-1}(b,w)\stackrel{def}{=}\gls{pwb}$$
by Theorem \ref{Thm:Kencrit}. Then $d\omega_0=d\omega$ for any $\omega$ defined from a perfect matching as above. This is clear 
from the local geometry.

Since $d(\omega_0-\omega)=0$, we can write 
\begin{equation}\label{eq:heightdef}
\omega_0-\omega=d\gls{h}
\end{equation}
for some function $h$ on $\Xi^\dg$, which is uniquely
defined up to an additive constant. (Remark that if $h$ and $h'$ are the height functions corresponding to two matchings $\mfm,\mfm'$, then $h'-h$ takes values in $\Z$ modulo a global additive constant). In the case of graphs in multiply connected domains and on a torus, $h$  is additively multivalued (when tracing $h$ along a non-contractible cycle).

We now seek an interpretation of perturbed operators $K_\alpha,\rK_\alpha$ in terms of height functions. 
From \eqref{eq:ratiopart} we get
\begin{align*}
\E\left(\exp(2\sum_{(bw)\in\mfm}\Re\int_w^b\alpha)\right)&=\det(\rK_\alpha\rK^{-1})\\
\E\left(\exp(2i\sum_{(bw)\in\mfm}\Im\int_w^b\alpha)\right)&=\det(K_\alpha \rK^{-1})
\end{align*}
In the simply-connected case (ie when $h$ is single valued), one can write:
$$\ind_{(bw)\in\mfm}=\omega((ff'))=h(f)-h(f')+\omega_0((ff'))$$
where $(ff')=(bw)^\dagger$. Then
\begin{align*}
\sum_{(bw)\in\mfm}\Re\int_w^b\alpha&=\sum_{f\in\Xi^\dagger}h(f)(\Re\int_{\partial f}\alpha)+P(\alpha)
\end{align*}
where $\partial f$ is the boundary of the face $f$ (taken counterclockwise) and $P(\alpha)$ is the $\R$-linear form:
$$P(\alpha)=\sum_{(bw)\in E_\Xi}p(w,b)\Re\int_w^b\alpha$$
Let us extend $h$ (initially defined on $\Xi^\dg$) to a piecewise constant function, ie constant in each face of $\Xi$. Then by Stokes' formula
$$\sum_{f\in\Xi^\dagger}h(f)(\Re\int_{\partial f}\alpha)=\Re\int_\Xi h\partial\alpha$$
(here $\Xi$ designates the domain of $\C$ covered by the graph; $d\alpha=\partial\alpha$ as $\alpha$ is a $(0,1)$-form). Finally we obtain the following expression for the Laplace and Fourier functional of the height field:
\begin{equation}\label{eq:FourLap}
\begin{array}{rl}
\E\left(\exp(2\Re\int_\Xi h\partial\alpha)\right)&=\det(\rK_\alpha\rK^{-1})\exp(-2P(\alpha))\\
\E\left(\exp(2i\Im\int_\Xi h\partial\alpha)\right)&=\det(K_\alpha \rK^{-1})\exp(2iP(i\alpha))
\end{array}
\end{equation}
The characteristic functional is bounded and is more practical is the absence of boundary. The Laplace functional preserves the real structure ($\rK_\alpha:\R^{\Xi_B}\rightarrow\R^{\Xi_W}$), which is useful for domains with boundary.

\section{Planar graphs}\label{sec:planar}

In order to illustrate the general method while avoiding complications related to non-trivial homology or boundaries, we now discuss the case of the plane. The analysis is based on the family of operators $(K_\alpha)_\alpha$ (see \eqref{eq:Kastpert}), which may be thought of as discretisations of the continuous CR operators $(\dbar+\alpha)_\alpha$. 

In Section \ref{subsec:planecharfun}, we relate the $K_\alpha$'s with the characteristic functional of the (infinite volume) dimer height field. In Section \ref{subsec:planarinvker}, we construct and analyse (in particular near the diagonal) an inverting kernel $\rS_\alpha$. In Section \ref{subsec:varplane}, we apply these estimates to obtain a precise version of the convergence of the characteristic functional.

\subsection{Characteristic functional}\label{subsec:planecharfun}

Let $\Lambda=\Lambda_\delta$ be a lozenge tiling of the complex plane $\C$ with edge length $\delta$, with $\delta$ going to zero along some sequence; $\Gamma,\Gamma^\dg$ is the pair of associated isoradial graphs, $M$ the corresponding bipartite graph (as in Figure \ref{fig:loztil}).

In this section, the graph carrying dimer configurations is $\Xi=M$. We assume that condition $(\spadesuit)$ is satisfied for all $\delta$ (for a fixed $\eta_0$). 

From \cite{Ken_locstat,KOS,dT_quadri}, we know that there is a Gibbs measure on perfect matchings of $M$ such that, for any finite subset $\{(b_i,w_i),1\leq i\leq n\}$ of edges of $\Xi$, 
\begin{equation}\label{eq:locstats}
\P((b_1,w_1)\in\mfm,\dots,(b_n,w_n)\in\mfm)=\left(\prod_{i=1}^n\rK(w_i,b_i)\right)\det_{1\leq i,j\leq n}(\uK^{-1}(b_i,w_j))
\end{equation}
(manifestly, these {\em local statistics} completely specify the measure). When $\Lambda$ is realised as the local limit of a sequence of biperiodic lattices, this measure is the limit of uniform dimer covers of (finite volume) toroidal graphs. In particular ($n=1$),
$$\P((b,w)\in\mfm)=p(w,b)=\rK(w,b)\uK^{-1}(b,w)$$
Thus we have a probability measure on matchings of $\Xi=\Xi_\delta$; $\E=\E_\delta$ is the expectation under this measure, $h$ the height function, seen as a function constant on faces of $\Xi$ and well-defined modulo a global additive constant.   

Remark that $\E(h)$ is constant; indeed, if $(ff')$ is the edge of $M^\dg$ dual to $(bw)$, then
$$\E(h(f')-h(f))=\E((\omega_0-\omega)(ff'))=p(w,b)-\P((bw)\in\mfm)=0$$
which explains the choice of $\omega_0$ as a reference 1-form in \eqref{eq:heightdef}.

Consider $g\in C^1_c(\C)$ a function with compact support, continuous first derivatives and second derivatives in $L^1_{loc}$; set $\alpha=\bar\partial g=g_{\bar z}d\bar z$; we consider again the perturbation \eqref{eq:Kastpert}
$$K_\alpha(w,b)=\rK(w,b)\exp(2i\Im\int_w^b\alpha)$$
Since $g$ has compact support, $K_\alpha\uK^{-1}$ is a finite rank perturbation of the identity, and consequently $\det(K_\alpha\uK^{-1})$ is well-defined as a Fredholm determinant (see e.g. Chapter 3 in \cite{SimTrace}). Concretely, let $S\subset\Xi_W$ be finite and s.t. $K_\alpha(w,.)=\rK(w,.)$ whenever $w\notin S$. Then the determinant of $((K_\alpha\uK^{-1})(w,w'))_{w,w'\in S}$ does not depend on the choice of $S$ and is denoted by $\det(K_\alpha\uK^{-1})$. We may think of $K_\alpha\uK^{-1}$ as an (infinite) block triangular matrix, with a finite nontrivial diagonal block (corresponding to $S$) and an infinite identity block (corresponding to $S^c$), so that the problem is effectively finite dimensional.

\begin{Lem}\label{Lem:detplane}
Let $h$ be the height function (constant on faces); then
\begin{align*}
\det(K_\alpha \uK^{-1})&=\E(\exp(2i\Im\int_\C h\partial\dbar g))\exp(2i\sum_{w\sim b}p(w,b)\Im\int_w^b\alpha)\\
&=\E(\exp(-i\Re\int_\C h(\Lap g)dA))\exp(-2iP(i\alpha))
\end{align*}
\end{Lem}
\begin{proof}
The issue is that we cannot apply directly the finite volume identity we observed earlier \eqref{eq:FourLap}. Let us consider $v$ a finitely supported function on edges of $\Xi$. Then
$$\prod_{(bw)\in E_\Xi}(1+v(w,b)\ind_{(bw)\in\mfm})=\sum_{S\subset E_\Xi}\prod_{(b,w)\in S}\ind_{(bw)\in\mfm}v(w,b)$$
(with finitely many nonzero summands on the right hand side and by convention $\prod_{\emptyset}=1$) and thus by \eqref{eq:locstats}
$$\E\left(\prod_{(bw)\in E_\Xi}(1+v(w,b)\ind_{(bw)\in\mfm})\right)
=\sum_{S=\{(b_iw_i),1\leq i\leq n\}\subset E_\Xi}\left(\prod_{i=1}^n\rK(w_i,b_i)v(w_i,b_i)\right)\det_{1\leq i,j\leq n}(\rK^{-1}(b_i,w_j))$$
On the other hand, if $K':\C^{\Xi_B}\rightarrow\C^{\Xi_W}$ is given by its matrix elements $K'(w,b)=(1+v(w,b))\rK(w,b)$, the Fredholm expansion (e.g. Lemma 3.3 \cite{SimTrace}) reads:
$$\det(\Id+(K'-K)\rK^{-1})
=\sum_{\{w_i,1\leq i\leq n\}\subset \Xi_W}\det_{1\leq i,j\leq n}(((K'-K)\rK^{-1})(w_i,w_j))$$
The Cauchy-Binet formula gives
$$\det_{1\leq i,j\leq n}(((K'-K)\rK^{-1})(w_i,w_j))
=\sum_{\{b_k,1\leq k\leq n\}\subset \Xi_B}\det_{1\leq i,k\leq n}((K'-K)(w_i,b_k))\det_{1\leq k,j\leq n}(\rK^{-1}(b_k,w_j))
$$
and besides
$$\det_{1\leq i,k\leq n}((K'-K)(w_i,b_k))=\sum_{\sigma\in{\mf S}_n}\sgn(\sigma)\prod_{i=1}^n\rK(w_i,b_{\sigma(i)})v(w_i,b_{\sigma(i)})$$ 
This completes the identification (reordering rows of $\det_{1\leq i,j\leq n}(\rK^{-1}(b_i,w_j))$ absorbs $\sgn(\sigma)$):
$$\E\left(\prod_{(bw)\in E_\Xi}(1+v(w,b)\ind_{(bw)\in\mfm})\right)=\det(K'\rK^{-1})$$
Specialising to $K'=K_\alpha$ concludes (the rest of the argument being as in the finite volume case \eqref{eq:FourLap}).
\end{proof}

\subsection{Inverting kernels}\label{subsec:planarinvker}

The next step is to construct and estimate (in particular near the diagonal) a kernel $\rS_\alpha$ inverting $K_\alpha$.

The finite difference approximation \eqref{eq:findiffpert} leads us to think of $K_\alpha$ as a finite difference version of (simultaneously) $\dbar+\alpha$ and the adjoint operator $\partial-\bar\alpha$. In order to construct a kernel $\rS_\alpha$ inverting $K_\alpha$, we are going to construct an approximate kernel $\tilde\rS_\alpha$ using at large scale inverting kernels for these continuous Cauchy-Riemann operators, and at small scale ($|b-w|\ll1$) ``standard" discrete holomorphic functions; and finally control the error between $\tilde\rS_\alpha$ and $\rS_\alpha$.

We have $\alpha=\dbar g$. Thus $\dbar+\alpha=e^{-g}(\dbar)e^{g}$, and consequently
$$S_\alpha(z,w)=\frac{e^{g(w)-g(z)}}{\pi(z-w)}$$
is a kernel inverting $(\dbar+\alpha)$ (on the right): $(\dbar+\alpha)S_\alpha(.,w)=\delta_wd\bar z$ as distributions. Among such kernels, it is uniquely characterized by $S_\alpha(z,w)\rightarrow 0$ as $z\rightarrow\infty$.

Set 
$$\tilde \rS_\alpha(b,w)=e^{2i\Im(\bar\lambda (b-w))}(\uK^{-1}(b,w)+R_B(\mu)+\bar R_B(\mu'))$$ 
for $|b-w|<\eta$ and
\begin{equation}\label{eq:planekerasymp}
\tilde \rS_\alpha(b,w)=\frac 12 \left(R_B(e^{i\nu(w)}S_\alpha(b,w))+\bar R_B(e^{-i\nu(w)}\overline {S_{-\alpha}(b,w)})\right)
\end{equation}
for $|b-w|\geq\eta$, where $\eta,\mu,\mu'$ are parameters to be specified.

We have
\begin{align*}
S_\alpha(b,w)e^{2i\Im\lambda\overline{(b-w)}}&=\frac{1}{\pi(b-w)}e^{2i\Im g_{\bar z}(w)\overline{(b-w)}-(g(b)-g(w))}\\
&=\frac{1}{\pi(b-w)}-\frac{1}\pi(\overline {g_{\bar z}}(w)+g_z(w))+O(b-w)
=\frac{1}{\pi(b-w)}-\frac{2}{\pi}(\partial_z\Re g)(w)+O(e^{2\|g\|_\infty}\omega_{g'}(|b-w|))
\end{align*}

We can now estimate $K_\alpha\tilde\rS_\alpha(.,w)$. 
If $w'$ is such that all its black neighbours are in the ball $B(w,\eta)=\{z:|z-w|<\eta\}$, given that $\uK^{-1}(b,w)=O(|b-w|^{-1})$, we get from \eqref{eq:close}:
$$(K_\alpha\tilde\rS_\alpha(.,w))(w')=\delta_w(w')+O\left(\delta^2\frac{\omega_{g'}(|w'-w|+\delta)}{|w'-w|+\delta}\right)$$
If all black neighbours of $w'$ are outside $B(w,\eta)$, we have as in \eqref{eq:findiffpert}
$$(K_\alpha\tilde\rS_\alpha(.,w))(w')=\delta^2e^{2\|g\|_\infty}O\left(
\frac{\omega_{g'}(\delta)}{|w'-w|}+\frac{\delta\|g'\|_\infty}{|w'-w|^2}+\frac{\delta}{|w'-w|^3}+\frac{\delta\|g'\|_\infty^2}{|w'-w|}\right)$$
and for $w'$ outside of $B(w,\eta)\cup supp(g)$, 
$$(K_\alpha\tilde\rS_\alpha(.,w))(w')=e^{\|g\|_\infty}O\left(
\frac{\delta^4}{|w'-w|^4}\right)$$
We may rephrase Theorem \ref{Thm:Kencrit} as
$$\uK^{-1}(b,w)=\frac 12 R_B\left(\frac{e^{i\nu(w)}}{\pi(b-w)}\right)+
\frac 12 \bar R_B\left(\frac{e^{-i\nu(w)}}{\pi\overline{(b-w)}}\right)+O\left(\frac{\delta}{|b-w|^2}\right)$$
Set
\begin{align*}
\mu&=\frac 12 e^{i\nu(w)}(-\frac {2}\pi(\partial_z\Re g)(w))\\
\mu'&=\frac 12 e^{-i\nu(w)}(\frac {2}\pi(\partial_{\bar z}\Re g)(w))=-\bar\mu
\end{align*}
Then if $\dist(b,w)$ is of order $\eta$, the difference between the short and long distance definitions of $\tilde\rS_\alpha$ is of order $O(\omega_{g'}(\eta)e^{2\|g\|_\infty}+\delta/\eta^2)$. Consequently, if $w'$ has neighbours both inside and outside of $B(w,\eta)$, $(K_\alpha\tilde\rS_\alpha(.,w))(w')=O(\delta\omega_{g'}(\eta)e^{2\|g\|_\infty}+\delta^2/\eta^2)$. 

Assume now that $g$ is in the Sobolev space $W^{2,p}$ with $p>2$ and has support in $B(0,r)$. In what follows constants may depend on $p,r$. By Morrey's inequality (see Section \ref{ssec:Sobolev}), $g'$ is $\eps$-H\"older with H\"older norm less than $c\|g''\|_p$ for $\eps=1-2/p$. It follows that $\omega_{g'}(s)\leq c\|g''\|_ps^\eps$ and $\|g\|_{C^1}\leq c\|g''\|_p$. Set $C=\|g''\|_p$.

Thus, the $L^1$ norm of $K_\alpha\tilde\rS_\alpha(.,w)-\delta_w$ (w.r.t. counting measure) is less than:
$$c\left(\delta^2\sum_{k=1}^{\eta/\delta}kC(k\delta)^{\eps-1}+(Ce^{cC}\delta\eta^\eps+\frac{\delta^2}{\eta^2})\frac{\eta}{\delta}+\delta^2e^{cC}\sum_{k=\eta/\delta}^{r/\delta}k.\left(\frac{C\delta^{\eps}+C^2\delta}{(k\delta)}+\frac{\delta C}{(k\delta)^2}+\frac{\delta}{(k\delta)^3}\right)
+e^{cC}\sum_{k=r/\delta}^\infty k\frac{\delta^4}{(k\delta)^4}
\right)$$
hence less than
$$ce^{cC}(\eta^{\eps+1}+\frac{\delta}{\eta}+\delta^{\eps}+\delta|\log\eta|+\frac{\delta}{\eta}+\delta^2)$$
Here we simply collect errors stemming from: replacing $K_\alpha$ with an operator conjugate to $\rK$ at short distance; gluing the short and long distance approximation; and using the limiting continuous kernel at long distance.

Setting now $\eta=\delta^{\beta}$ with $\beta=\frac{1}{\eps+2}$, and $T=\Id-K_\alpha\tilde\rS_\alpha$, we get that 
$$\lVert T\rVert_{L^1}\leq ce^{cC}(\delta^{(\eps+1)/(\eps+2)}+\delta^\eps)$$ 
where $\|.\|_{L^1}$ is the $L^1\rightarrow L^1$ operator norm, and consequently $(\Id-T):L^1\rightarrow L^1$ is invertible for $\delta$ small enough, and we may set:
$$\rS_\alpha\stackrel{def}{=}\tilde\rS_\alpha(\sum_{k=0}^\infty T^k)$$
so that $K_\alpha\rS_\alpha=\Id$.

We now want to estimate $\lVert\tilde\rS_\alpha T\rVert_{L^1\rightarrow L^\infty}$. We simply expand 
$$(\tilde\rS_\alpha T)(b,w)=\sum_{w'}\tilde S_\alpha(b,w')T(w',w)$$ 
and as before we split $T$ in a short range and long range part and notice that $\tilde\rS_\alpha(b,w)=O(e^{cC}/\dist(b,w))$ to obtain:
\begin{align*}
\|(\tilde\rS_\alpha T)(.,w)\|_\infty
&\leq 
ce^{cC}\left(\delta^2\sum_{k=0}^{\eta/\delta}k(k\delta)^{\eps-2}
+(\delta\eta^\eps+\frac{\delta^2}{\eta^2})\sum_{i=1}^{c\eta/\delta}\frac 1{\delta i}
+\delta^2\sum_{k=\eta/\delta}^{r/\delta}k.\left(\frac{\delta^{\eps}}{(k\delta)^2}+\frac{\delta}{(k\delta)^3}\right)+\sum_{k=1}^{\eta/\delta} \frac{k}{k\delta}\left(\frac{\delta^{2+\eps}}{\eta^2}+\frac{\delta^{3}}{\eta^3}\right)\right.\\
&{\rm\ \ \ \ \ \ \ }
\left.+\sum_{k=r/\delta}^\infty k\frac{\delta^4}{(k\delta)^5}+\sum_{k=1}^{r/\delta}\frac{k}{k\delta}\delta^{4}
\right)\\
&\leq ce^{cC}\left(\eta^\eps+(\eta^\eps+\delta^{1-2\beta})|\log(\delta)|+(\delta^\eps|\log(\delta)|+\delta^{1-\beta}+\delta^{\eps-\beta}+\delta^{1-2\beta})
+\delta^2\right)=O(e^{cC}\delta^{\eps'})
\end{align*}
The various terms correspond to the possible relative positions of $b,w',w$. Notice however that the estimate is simpler when $b\sim w$, which is the most useful case.

Since $\rS_\alpha-\tilde\rS_\alpha=(\tilde\rS_\alpha T)(\Id-T)^{-1}$, $\lVert(\Id-T)^{-1}\rVert_{L^1}=O(1)$, $\lVert\tilde\rS_\alpha T\rVert_{L^1\rightarrow L^\infty}=O(e^{cC}\delta^{\eps'})$, we conclude that  
$$\rS_\alpha(b,w)-\tilde\rS_\alpha(b,w)=O(e^{c\|g''\|_p}\delta^{\eps'})$$
uniformly in $b,w$ (for $p,r$ fixed; $\eps'>0$ depends on $p>2$). In particular for $b\sim w$ we have
\begin{equation}\label{eq:diagestplane}
\rS_\alpha(b,w)-e^{2i\Im\bar\lambda(b-w)}\uK^{-1}(b,w)=i\Im(e^{i\nu}r_\alpha)+O(e^{c\|g''\|_p}\delta^{\eps'})
\end{equation}
where $r_\alpha=-\frac 2\pi(\partial_z\Re g)(w)$ and $\nu=\nu(b)+\nu(w)$.

\subsection{Variational analysis}\label{subsec:varplane}

We now apply the asymptotic analysis of the inverting kernel $\rS_\alpha$ to the original problem, i.e. the convergence of the characteristic functional of the height field over a large class of test functions.

\begin{Lem}\label{Lem:planevariation}
The following estimate holds
$$\log\det(K_\alpha\uK^{-1})=2i\sum_{b\sim w}p(w,b)\Im\int_w^b\alpha-\frac 1{2\pi}\int |\nabla\Re g|^2+O(\delta^{\eps'})$$
\end{Lem}
\begin{proof}
Observe that $K_\alpha\uK^{-1}$ and $\rK\rS_\alpha$ are finite rank (hence trace class) perturbations of the identity, and $(K_\alpha\uK^{-1})(\rK\rS_\alpha)=\Id$ (since $\uK^{-1}$ is also a left inverse of $\rK$). If $\alpha$ depends smoothly on a parameter $t$ (say $\alpha(t)=t\alpha$), we have the variational formula (eg \cite{GohKre}, IV.1)
$$\frac {d}{dt}\log\det(K_{\alpha}\rS)=\Tr((\frac{d}{dt}K_{\alpha})\rS_{\alpha})$$
as long as $K_{\alpha}\uK^{-1}$ is invertible, which is at least true for small enough $t$. From
\begin{align*}
\dot{K_{\alpha}}(w,b)&=\rK(w,b)\exp(2i\Im\int_w^b\alpha)2i\Im\int_w^b\dot\alpha\\
\rS_{\alpha}(b,w)&=\exp(-2i\Im\int_w^b\alpha)\uK^{-1}(b,w)+i\Im(e^{i(\nu(b)+\nu(w))}r_\alpha)+O(\delta^{\eps'})
\end{align*}
(for $b\sim w$), we obtain (if $x,y,x',y'$ denote the black neighbours of $w$ in ccwise order)
\begin{align*}
\dot{K_{\alpha}}\rS_{\alpha}(w,w)&=2i\sum_{b\sim w}p(w,b)\Im\int_w^b\dot\alpha-|y'-y|.|x'-x|\left(\Im(\dot\lambda e^{-i\nu(w)})\Im(e^{i\nu(w)}r_\alpha)+\Im(\dot\lambda e^{-i\nu(w)}(-i))\Im(e^{i\nu(w)}ir_\alpha)
\right)+O(\delta^{2+\eps'})\\
&=2i\sum_{b\sim w}p(w,b)\Im\int_w^b\dot\alpha+2\mu_{\dmd}(w)\Re(\dot\lambda r_\alpha)+O(\delta^{2+\eps'})
\end{align*}
and
$$\frac {d}{dt}\log\det(K_{\alpha}\uK^{-1})=
2i\sum_{b\sim w}p(w,b)\Im\int_w^b\dot\alpha+\Im\int \dot\alpha\wedge(r_\alpha dz)+O(\delta^{\eps'})
$$
(since $dz\wedge d\bar z=-2idA$). Moreover, with $\alpha=\dbar g$, 
$$\Im\int \dot\alpha\wedge(r_\alpha dz)=-\frac 2\pi\Im\int \dbar \dot g\wedge\partial(\Re g) =\frac 2\pi\Im\int \dot g\dbar\partial(\Re g)=\frac 1\pi\Re\int \dot g\Lap g dA= -\frac 1{2\pi}\frac {d}{dt}\int |\nabla \Re g|^2$$
In particular, for $\delta$ small enough $\det(K_\alpha\uK^{-1})$ does not vanish along the interpolation path, thus (see e.g. Theorem 3.5 in \cite{SimTrace})  $K_\alpha\uK^{-1}$ stays invertible and the variational formula is legitimate.
\end{proof}

From Lemmas \ref{Lem:detplane} and \ref{Lem:planevariation}, we immediately conclude:

\begin{Cor}\label{Cor:FFplane}
If $g\in W^{2,p}\cap C^0_c$, $p>2$, then $\int h\Lap g$ converges in distribution to a centered normal variable with variance $\frac 1\pi\int |\nabla g|^2$, as $\delta\searrow 0$.
\end{Cor}

Thus we can integrate the discrete height field $h$ against a test function in $L^p$, $p>2$ (the compact support assumption is for technical convenience). An optimal statement (given the free field limit) would involve a test function in $H^{-1}$; however $h$ (as defined here) is not in $H^1$, due to jump discontinuities on edges of $\Xi$. It is unclear whether one can define a better interpolation of the discrete height function for which one could relax substantially the condition $g\in W^{2,p}$ (to $g\in H^{1+\eps}$, say). 

This identifies the coupling constant of the limiting free field as $g_0=2\pi^2$. In the ``solid-on-solid" normalisation, the discrete height field considered is $2\pi h$, with the coupling constant for the limiting field given by $g_0=\frac 12$.

\section{Toroidal graphs}\label{sec:torus}

As before (Section \ref{sss:critgraphs}), we start from a rhombus tiling $\Lambda=\Lambda_\delta$ of the plane with edge length $\delta\ll1$, from which we construct $\Gamma$, $\Gamma^\dg$, $\dmd=\Lambda^\dg$, $M$. Additionally, we assume that these structures are biperiodic, with periods $1$ and $\tau=\tau_\delta$ ($\Im\tau>0$). Denote by $\gls{Upsilon}$ the lattice $\Z+\tau\Z$; by quotienting, we obtain a rhombus tiling of $\Sigma=\C/\Upsilon$. We are interested in perfect matchings of $\Xi=M/\Upsilon$.
Notice that Euler's formula applied to $\Gamma$ shows that $|\Xi_B|=|\Xi_W|$.

The height ``function" $h$ is associated to a perfect matching ${\mf m}$ on $\Xi$ by a local rule (see Section \ref{ss:height}). Since the torus is non contractible, the height function is now additively multivalued. The current $J=dh$ is a well-defined closed 1-form on $\Sigma$ (the discrete height function on $\Xi^\dg$ is extended to a piecewise continuous function on $\Sigma$, constant on faces of $\Xi$; then $J$ is a distributional 1-form). The Hodge decomposition \eqref{eq:hodgedec} of $J$ reads:
\begin{equation}\label{eq:hodgedecdimer}
J_\delta=\omega_h+dh_0
\end{equation}
where $\omega_h$ is a closed harmonic form (hence, in flat metric, $\omega_h=adx+bdy$ for some $a,b\in\R$) and $h_0$ is a (single-valued) function on $\Sigma$, well-defined up to an additive constant. The ``topological" (or ``instanton") component $\omega_0$ is uniquely specified by its periods, which equal those of the current: $\int_\gamma\omega_h=\int_\gamma J_\delta$, where $\gamma\in H_1(\Sigma,\Z)$ is a basic cycle on $\Sigma$.

We are interested in the asymptotic distribution of $J_\delta$, or equivalently the asymptotic joint distribution of $(\omega_h,h_0)$. For clarity we will be discussing first the topological component of the current in Section \ref{ss:linebundles}. Section \ref{subsec:bosonid} recalls useful Poisson summation arguments (bosonisation identities). Finally, Section \ref{ssec:torusgeneral} handles jointly the topological and scalar components.

\subsection{Flat line bundles}\label{ss:linebundles}

For $\lambda\in\C$, denote:
$$K_\lambda(w,b)=\rK(w,b)\exp(2i\Im\int_w^b\lambda d\bar z)=\rK(w,b)\exp(2i\Im\lambda\overline{(b-w)})$$
so that $K_\lambda$ defines an operator $\C^{M_B}\rightarrow\C^{M_W}$. Plainly $K_\lambda$ commutes with the action of  and by quotienting an operator $\C^{\Xi_B}\rightarrow\C^{\Xi_W}$. We are concerned with the inverse of that last operator. 
Let $\chi:\Upsilon\rightarrow \U=\{z\in\C:|z|=1\}$ be a unitary character; consider the finite-dimensional space 
$$(\C^{M_B})_\chi=\{f\in\C^{M_B}:\forall z\in M_B, \omega\in \Upsilon, f(z+\omega)=\chi(\omega)f(z)\}\subset \C^{M_B}$$
and $(\C^{M_W})_\chi$ is defined similarly; note that $(\C^{M_.})_{\Id}\simeq \C^{\Xi_.}$ (here $\Id$ denotes the trivial character). Plainly, as $K_\lambda$ commutes with the action of $\Upsilon$ by translation, it defines an operator $K_\lambda:(\C^{M_B})_\chi\rightarrow(\C^{M_W})_\chi$. Set :
$$(T_\lambda f)(z)=f(z)e^{2i\Im\lambda\overline z}$$
If $\chi_\lambda$ given by 
\begin{equation}\label{eq:chilambda}
\chi_\lambda(\omega)=e^{2i\Im\lambda\overline\omega}
\end{equation}
denotes the character associated to $\lambda$ , $T_\lambda$ maps $(\C^{M_.})_{\Id}$ to $(\C^{M_.})_{\chi_\lambda}$, and we have the commutative diagram
$$ \xymatrix{
 (\C^{M_B})_{\chi_\lambda} \ar@{->}[rr]^{\displaystyle \rK} \ar@{->}[dd]_{\displaystyle T_{-\lambda}}
      && (\C^{M_W})_{\chi_\lambda} \ar@{->}[dd]^{\displaystyle T_{-\lambda}}    \\ \\
 (\C^{M_B})_{\Id} \ar@{->}[rr]_{\displaystyle K_\lambda} && (\C^{M_W})_{\Id}
 }
 $$
so that we may focus on $\rK:(\C^{M_B})_{\chi}\rightarrow(\C^{M_W})_{\chi}$. An element in the kernel restricts to a bounded harmonic function on $\Gamma$ and $\Gamma^\dg$; then (Liouville) these restrictions are constant, hence the kernel is trivial iff $\chi\neq\Id$. As $|\Xi_B|=|\Xi_W|$, $\rK:(\C^{M_B})_{\chi}\rightarrow(\C^{M_W})_{\chi}$ is invertible iff $\chi\neq\Id$. Let $\rS_\chi$ denote the inverse operator ($\chi\neq\Id$). We now want to relate $\rS_\chi$ to kernels inverting (continuous) Cauchy-Riemann operators. In this subsection, we describe and discuss the continuous inverting kernels and their relation to analytic torsion.

\subsubsection{Inverting kernels and $\theta$-functions}

Denote by $L_\chi$ the holomorphic, unitary line bundle over $\Sigma$ obtained by twisting the trivial line bundle $L_{\Id}$ by the unitary character $\chi:\pi_1(\Sigma)\rightarrow\U$. Sections of $L_\chi$ may be identified with multiplicatively multi-valued functions on $\Sigma$, or multiplicatively quasi-periodic functions $\varphi$ on $\C$ with the transformation rule prescribed by $\chi$: 
\begin{equation}\label{eq:chiquasiper}
\varphi(\omega+z)=\chi(\omega)\varphi(z)
\end{equation}
for any $z\in\C$, $\omega\in\Upsilon$. We may consider the operator
$$
\begin{array}{rl}
\dbar_\chi: L_\chi&\longrightarrow\Omega^{(0,1)}(L_\chi)\\
s&\longmapsto\frac{\partial s}{\partial \bar z}d\bar z
\end{array}$$
where %
$\Omega^{(0,1)}(L_\chi)$ denotes $(0,1)$-forms with values in $L_\chi$. The $\dbar_\chi$ operator is canonically associated to the complex structure on $L_\chi$. It is well known that this is a zero index operator which is invertible iff $\chi\neq\Id$. Moreover its inverse has an explicit expression in terms of $\theta$-functions, which we now recall (see e.g. Chapter I \cite{Fay_Theta}).

Consider the $\theta$ function with characteristics:
\begin{equation}\label{eq:thetachardef}
\begin{array}{rl}
\thetaf{2\eps}{2\eps'}(z)&\stackrel{def}{=}\sum_{n\in\Z}\exp
2i\pi\left(\frac 12 \tau(n+\eps)^2+(n+\eps)(z+\eps')
\right)\\
&=\exp(2i\pi(\frac\tau 2\eps^2+\eps z+\eps\eps'))\thetaf{0}{0}(z+\eps'+\eps\tau)
\end{array}
\end{equation}
which is easily seen to transform as:
\begin{equation}\label{eq:thetatransf}
\begin{array}{rl}
\thetaf{2\eps}{2\eps'}(z+1)&=\exp(2i\pi\eps)\thetaf{2\eps}{2\eps'}(z)\\
\thetaf{2\eps}{2\eps'}(z+\tau)&=\exp(-2i\pi(z+\eps'+\frac\tau 2))\thetaf{2\eps}{2\eps'}(z)
\end{array}
\end{equation}
(see e.g. Chapter VI in \cite{FarKra}). For concision, let us denote $\vartheta=\thetaf{2\eps}{2\eps'}$, $\theta_3=\thetaf{0}{0}$ and $\theta=\thetaf{1}{1}$ (these last notations are as in \cite{Chandra}). Classically (e.g. Chapter V, Theorem 1 in \cite{Chandra}), $\theta$ is odd with simple zeroes on $\Z+\tau\Z$ and no zeroes elsewhere.

Then consider the meromorphic function:
$$T(z)=\frac{\vartheta(z)\theta'(0)}{\vartheta(0)\theta(z)}$$
From \eqref{eq:thetatransf} we see $T(z+1)=-e^{2i\pi\eps}T(z)$, $T(z+\tau)=-e^{-2i\pi\eps'}T(z)$. Moreover $T(z)$ has a simple pole at $z=0$ with residue 1 (if $(\eps,\eps')\neq(\frac 12,\frac 12)\mod\Z^2$). This follows from the statement on zeroes of $\theta$ and \eqref{eq:thetachardef}. It is also clear that $T$ is uniquely specified by these properties. 

Thus:
\begin{equation}\label{eq:Schi}
S_\chi(z,w)=\frac 1\pi\cdot\frac{\vartheta(z-w)\theta'(0)}{\vartheta(0)\theta(z-w)}
\end{equation}
is the kernel inverting $\dbar_\chi$, where $\chi$ is the character given by $\chi(1)=-e^{2i\pi\eps}$, $\chi(\tau)=-e^{-2i\pi\eps'}$, $\chi\neq\Id$. We may take $\Im\lambda=\pi(\eps+\frac 12)$, $\Im(\lambda\bar\tau)=-\pi(\eps'+\frac 12)$, i.e. 
\begin{equation}\label{eq:lambdachar}
\lambda=\frac\pi {\Im\tau}\left(\eps'+\eps\tau+\frac{\tau+1}2\right)
\end{equation}
so that $\chi=\chi_\lambda$ (see \eqref{eq:chilambda}).

In concrete terms, let $\varphi:\C\rightarrow\C$ be a quasiperiodic function as in \eqref{eq:chiquasiper}, then
$$(S_\chi\varphi)(z)=\iint_\Sigma S_\chi(z,w)\varphi(w)dA(w)$$
has the same quasiperiodicity (remark that the integrand is periodic) and satisfies
$$\frac{\partial}{\partial\bar z}(S_\chi\varphi)(z)=\varphi(z)$$
(see e.g. Theorem 1.28 in \cite{Voisin_Hodge}).

Since $\theta$ is odd, we have the asymptotic expansion near the diagonal:
\begin{equation}\label{eq:Schidiag}
S_\chi(z,w)=\frac{1}{\pi(z-w)}+r_\chi(w)+O(|z-w|/\vartheta(0))
\end{equation}
as $z\rightarrow w$, where
$$\pi r_\chi(w)=\lim_{z\rightarrow w}(\pi S_\chi(z-w)-\frac{1}{z-w})=\frac{\vartheta'}{\vartheta}(0)=2i\pi\eps+\frac{\theta_3'}{\theta_3}(\eps'+\eps\tau)$$
(which in this case depends only on $\chi$, by translation invariance). If we set $c_\chi=2-\Re\chi(1)-\Re\chi(\tau)$, which is comparable to $(\eps-1/2)^2+(\eps'-1/2)^2$ for $\eps,\eps'\in [0,1]$, from \eqref{eq:thetachardef} we have $\vartheta(0)=O(\sqrt {c_\chi})$.

\subsubsection{Variational analysis}\label{ssec:vartorus}

If $\eps,\eps'$ depend smoothly on a parameter $u$, we have from \eqref{eq:thetachardef}:
\begin{align*}
\frac{d}{du}\left(\log\thetaf{2\eps}{2\eps'}(0)\right)&=\frac{d}{du}(i\pi\tau\eps^2+2i\pi\eps\eps')+\frac{d}{du}\log\theta_3(\eps'+\eps\tau)\\
&=2i\pi\eps(\dot\eps\tau+\dot\eps')+2i\pi\dot\eps\eps'+(\dot\eps'+\dot\eps\tau)\frac{\theta_3'}{\theta_3}(\eps'+\eps\tau)\\
&=(\dot\lambda\Im\tau) r_\chi(0)+2i\pi\dot\eps\eps'
\end{align*}
the last term being pure imaginary. 

Following Ray and Singer (\cite{RS_analytic}),  we consider the {\em analytic torsion} defined by
$$\gls{TSigma}(\chi)=\exp(-\frac 12\zeta_\chi'(0))={\det}_\zeta((\dbar_\chi)^*\dbar_\chi)^{1/2}$$
where $\zeta_\chi$ is the $\zeta$ function defined from the Laplacian-type operator $(\dbar_\chi)^*\dbar_\chi$. Then an explicit diagonalisation of $(\dbar_\chi)^*\dbar_\chi$ (in flat metric) and Kronecker's second limit formula imply that (see Section 4 in \cite{RS_analytic})  
$$T_\Sigma(\chi)=e^{-\pi\nu^2\Im\tau}|\theta(\mu-\tau\nu)/\eta(\tau)|$$
where $\chi(m\tau+n)=\exp(2i\pi(m\mu+n\nu))$ and $\eta$ is the  Dedekind $\eta$ function (e.g. VIII.1 in \cite{Chandra}). Here $\mu=\frac 12-\eps'$, $\nu=\eps-\frac 12$, so that
$$T_\Sigma(\chi)=e^{-\pi(\Im z)^2/\Im\tau}|\theta_3(z)/\eta(\tau)|=\left|\eta(\tau)^{-1}\thetaf{2\eps}{2\eps'}(0)\right|$$
with $z=\eps'+\eps\tau$. Note that $\left|\thetaf{2\eps}{2\eps'}(0)\right|=e^{-\pi(\Im z)^2/\Im\tau}|\theta_3(z)|$, where $z=\eps'+\eps\tau$.

We admit for now the following near diagonal estimate for $\rS_\chi$ (see \eqref{eq:torusdiagest} for a more general statement): 
$$\rS_\chi(b,w)-\uK^{-1}(b,w)=i\Im(e^{i\nu}r_\chi(0))+O(\delta^{\eps_0})$$
for $|b-w|=O(\delta)$ (in particular if $b\sim w$ in $M$), for some positive constant $\eps_0$; we will also see that the estimate is uniform in $\chi$ for $\chi$ in a compact set of characters not containing the trivial one (see \eqref{eq:torusdiagest}). An alternative argument can be given based on the representation
$$\rS_\chi(b,w)=\sum_{\omega\in\Upsilon}\bar\chi(\omega)\uK^{-1}(b,w+\omega)$$
where summability can be justified by using an Abel integration by part argument, for $\chi$ non-trivial.

Let us consider now $\lambda$ as a differentiable function of a parameter $u$ and analyse the variation of $\det K_\lambda$, $K_\lambda:\C^{\Xi_B}\rightarrow\C^{\Xi_W}$. Assume that $\chi_\lambda\neq\Id$, so that $K_\lambda$ is invertible. Then:
$$\frac{d}{du}\log\det(K_\lambda)=\Tr(\dot K_\lambda K_\lambda^{-1})$$
where for $b\sim w$,
\begin{align*}
\dot K_\lambda(w,b)&=2i\Im(\dot\lambda\overline{(b-w)})e^{2i\Im(\lambda\overline{(b-w)})}\rK(b,w)\\
e^{-2i\Im(\lambda\overline{(b-w)})}K_\lambda^{-1}(b,w)&=\rS_{\chi}(b,w)\\
&=\uK^{-1}(b,w)+i\Im(e^{i\nu}r_\chi(0))+O(\delta^{\eps_0})
\end{align*}
where $\chi=\chi_\lambda$. If $b\sim b'$ in $\Gamma$ (resp. $\Gamma^\dg$), $w$ the white vertex corresponding to $(bb')$, we get $\rK(w,b)\uK^{-1}(b,w)=\rK(w,b')\uK^{-1}(b',w)$ from Theorem \ref{Thm:Kencrit}. Thus the contribution of $\uK^{-1}(b,w)$ in the trace cancels exactly; besides,
$$2i\Im(\dot\lambda\overline{(b-w)}))\rK(b,w)(i\Im(e^{i\nu}r_\chi(0)))=
-\mu_\dmd(w)\Im(\dot\lambda e^{-i\nu})\Im(e^{i\nu}r_\chi(0))$$
(where $\nu=\nu(w)+\nu(b)$) and we are left with:
\begin{align*}
\Tr(\dot K_\lambda K_\lambda^{-1})&=2\sum_{w\in\Xi_w}\mu_\dmd(w)\Re(\dot\lambda r_\chi(0))+O(\delta^{\eps_0})=2.{\rm Area}(\C/\Upsilon)\Re(\dot\lambda r_\chi(0))+O(\delta^\eps)\\
&=2\Re\frac{d}{du}\log\left|\thetaf{2\eps}{2\eps'}(0)\right|+O(\delta^\eps)=2\frac{d}{du}\log T_\Sigma(\chi)+O(\delta^{\eps_0})
\end{align*}
since around $w$, $e^{i\nu(b)}$ is alternatively $1$ and $i$; and ${\rm Area}(\C/\Upsilon)=\Im\tau$. We conclude (based on the near diagonal estimate \eqref{eq:torusdiagest} below) that if $\lambda_i\leftrightarrow\chi_i$ (as in \eqref{eq:chilambda}) non trivial ($i=1,2$), then
\begin{equation}\label{eq:discrtors}
\frac{\det(K_{\lambda_2})}{\det(K_{\lambda_1})}\longrightarrow\left(\frac{T_\Sigma(\chi_2)}{T_\Sigma(\chi_1)}\right)^2=\frac{\det_\zeta((\dbar_{\chi_2})^*\dbar_{\chi_2})}{\det_\zeta((\dbar_{\chi_1})^*\dbar_{\chi_1})}
\end{equation}
as the mesh $\delta\searrow 0$.

\subsection{Bosonisation identity}\label{subsec:bosonid}

Before describing the full limit of the dimer height current in the next subsection, we show here how to identify the limit of the topological component from \eqref{eq:discrtors}. 

Let $\phi_{nm}$ be the harmonic differential on $\Sigma$ with half-integer periods $n=\int_A\phi_{nm}$ and $m=\int_B\phi_{nm}$. Let
\begin{align*}
S(\phi)&=2\pi\int_\Sigma|\nabla\phi|^2dA+\frac{4i\pi}{\Im\tau}\Im((m-\bar\tau n)z)+4i\pi nm\\
&=\frac{2\pi}{\Im\tau}|m-\bar\tau n|^2+4i\pi(m\eps+n\eps')+4i\pi nm
\end{align*}
where $z=\eps'+\eps\tau=x+iy$, and
\begin{equation}\label{eq:Zinstdef}
Z_{inst}=Z_{inst}(z)\stackrel{def}{=}\sum_{n,m\in\frac 12\Z}e^{-S(\phi_{nm})}
\end{equation}
Then a Poisson summation argument (\cite{ABMNV}, Section 4.C) shows that
\begin{equation}\label{eq:bosonid}
Z_{inst}=(2\Im\tau)^{\frac 12}e^{-2\pi (\Im z)^2/\Im\tau}|\theta_3(z)|^2
\end{equation}

The determinants $\det(K_\lambda)$ count dimer configurations with some unitary weight, which depends only on the periods of the current. This is originally due to Kasteleyn (\cite{Kas_square}); see \cite{CimRes} for a recent (and exhaustive) treatment.
 
We begin with $\det\rK$ (which is 0 as the kernel of $\rK$ contains constant functions). Then (for a proper choice of ordering of vertices), each term in the determinant expansion corresponds to a matching of $\Xi$ (with associated current $J$) counted with a positive sign if $(\int_A J,\int_B J)=(0,0)\mod 2$, and negative sign otherwise. Let us denote this sign by $Q({\mf m})=Q(J)$.
In the expansion of $\det K_\lambda$, each matching is counted with an additional phase:
$$2\Im\sum_{(bw)\in{\mf m}}\int_w^b\alpha$$
where $\alpha=\lambda d\bar z$, i.e.
\begin{align*}
{\mc Z}&=\sum_\mfm w(\mfm)\\
\det\rK&=\sum_{\mfm}Q(\mfm)w(\mfm)\\
\det K_\lambda&=\sum_{\mfm}Q(\mfm)\exp\left(2i\Im\sum_{(bw)\in{\mf m}}\int_w^b\alpha\right)w(\mfm)
\end{align*}
As before (see \eqref{eq:FourLap}), we can integrate by parts over a fundamental domain $C$ bounded by cycles $A,B$ drawn on $\Gamma$ to obtain:
\begin{align*}
\sum_{(bw)\in{\mf m}}\int_w^b\alpha&=-\sum_{(bw)\in E_\Xi, (ff')=(bw)^\dg}(h(f')-h(f)-\omega_0((ff')))\int_w^b\alpha\\
&=\int_Bdh\int_A\alpha-\int_Adh\int_B\alpha
\end{align*}
(Note that in this case, $\sum\omega_0((ff'))\int_b^w\alpha=0$). Thus:
$${\mc Z}^{-1}\det(K_\lambda)=\E\left(Q(J)\exp(2i\Im(-\lambda\bar\tau\int_A J+\lambda\int_B J))\right)$$
where ${\mc Z}$ is the partition function of the dimer model on $\Xi$ (see \eqref{eq:dimerpartfun}). Let us set 
\begin{equation}\label{eq:halfpers}
(n,m)=\frac 12(\int_A J_\delta,\int_B J_\delta),
\end{equation}
the half-periods of the current (this somewhat awkward convention is used to connect with \eqref{eq:bosonid}), and 
$$\lambda=\lambda(\eps,\eps')=\frac{\pi}{\Im\tau}\left(\eps'+\eps\tau+\frac{\tau+1}2\right)$$
as before (see \eqref{eq:lambdachar}). Then:
$$2i\Im(-\lambda\bar\tau\int_A J_\delta+\lambda\int_B J_\delta)=4i\pi(m\eps+n\eps')+2i\pi(m+n)$$
Let us remark that
$$Q(J)=-e^{4i\pi (m+\frac 12)(n+\frac 12)}=e^{4i\pi mn}e^{2i\pi(m+n)}$$
so that
\begin{equation}\label{eq:torusabelianpert}
{\mc Z}^{-1}\det(K_\lambda)=\E\left(e^{4i\pi(m\eps+n\eps')+4i\pi mn}\right)
\end{equation}
which is 1-periodic in $\eps,\eps'$. 

Let $\E_0$ be the expectation relative to the probability measure on $(\frac 12\Z)^2$ with weights proportional to $e^{-\frac{2\pi}{\Im\tau}|m-\bar\tau n|^2}$. Then by definition \eqref{eq:Zinstdef}
$$Z_{inst}(\eps'+\eps\tau)\propto\E_0\left(e^{4i\pi(m\eps+n\eps')+4i\pi mn}\right)$$
as a function of $\eps,\eps'$.

From \eqref{eq:discrtors} and \eqref{eq:bosonid} we have
$$\frac{\det{K_{\lambda_2}}}{\det{K_{\lambda_1}}}\longrightarrow\left(\frac{T_\Sigma(\chi_2)}{T_\Sigma(\chi_1)}\right)^2=\frac{Z_{inst}(z_2)}{Z_{inst}(z_1)}$$
for $\chi_1,\chi_2\neq\Id$ (as we have seen, this follows from near diagonal estimates \eqref{eq:torusdiagest}).

This is enough to conclude that $(m,n)$ converges in distribution. Indeed, if we write $K_{\eps,\eps'}=K_{\lambda(\eps,\eps')}$, we have:
$${\mc Z}^{-1}\left(\det K_{\eps,\eps'}+\det K_{\eps+\frac 12,\eps'}+\det K_{\eps,\eps'+\frac 12}-\det K_{\eps+\frac 12,\eps'+\frac 12}\right)=2\E\left(e^{4i\pi(m\eps+n\eps')}\right)$$
by virtue of 
$$1+e^{2i\pi m}+e^{2i\pi n}-e^{2i\pi(m+n)}=2-(1-e^{2i\pi m})(1-e^{2i\pi n})=2e^{4i\pi mn}$$
for $m,n\in \frac 12\Z$. 

When $(\eps,\eps')=(\frac 12,\frac 12)$, $\det K_{\frac12,\frac 12}=\det\rK=0$ (as constant functions are in the kernel of $\rK$) and $Z_{inst}(\frac 12+\frac 12\tau)=0$ by \eqref{eq:bosonid}. Consequently,
$${\mc Z}^{-1}(\det K_{0,0}+\det K_{\frac 12,0}+\det K_{0,\frac 12})=2\E(1)=2$$
or (\cite{Kas_square}):
\begin{equation}\label{eq:toruspartfun}
{\mc Z}=\frac 12(\det K_{0,0}+\det K_{\frac 12,0}+\det K_{0,\frac 12})
\end{equation}
Similarly,
$$\frac{Z_{inst}(\eps'+\eps\tau)+Z_{inst}(\eps'+(\eps+\frac 12)\tau)+Z_{inst}(\eps'+\frac 12+\eps\tau)-Z_{inst}(\eps'+\frac 12+(\eps+\frac 12)\tau)}
{Z_{inst}(0)+Z_{inst}(\frac\tau 2)+Z_{inst}(\frac 12)-Z_{inst}(\frac 12+\frac\tau 2)}
=\E_0\left(e^{4i\pi(m\eps+n\eps')}\right)$$

It follows that, for any $(\eps,\eps')\in\R^2$ (distinguishing the case $(\eps,\eps')\in (\frac 12\Z)^2$, where we use $\det(\rK)=0$), we have
\begin{equation}\label{eq:periodconv}
\E\left(e^{4i\pi(m\eps+n\eps')}\right)\longrightarrow\E_0\left(e^{4i\pi(m\eps+n\eps')}\right)
\end{equation}
Convergence of the characteristic function then implies convergence in distribution of the half-periods $(n,m)$.

\subsection{General Cauchy-Riemann operators}\label{ssec:torusgeneral}

Here we are going to combine the arguments of the previous subsection - for the topological component of the field - with the framework explained in the planar case (leading to Corollary \ref{Cor:FFplane}) - for the scalar component, in order to describe fully the limit of the dimer height current.

\paragraph{Estimates near the diagonal.\\} 
Let $\alpha=a(z)d\bar z$ be an $\Upsilon$-periodic $(0,1)$-form on $\C$ (equivalently, a $(0,1)$-form on $\Sigma$). We can define a perturbed operator (see \eqref{eq:Kastpert})
$$K_\alpha(w,b)=\rK(w,b)\exp(2i\Im\int_w^b\alpha)$$
so that $K_\alpha$ defines an operator $\C^{M_B}\rightarrow\C^{M_W}$ and by quotienting an operator $\C^{\Xi_B}\rightarrow\C^{\Xi_W}$. We are interested in the inverse of that last operator, especially near the diagonal. We first consider the limiting continuous Cauchy-Riemann operators. 

As a $(0,1)$-form on $\Sigma$, $\alpha$ can be decomposed uniquely as
\begin{equation}\label{eq:Dolbdec}
\alpha=\lambda_0 d\bar z+\dbar g
\end{equation}
where $g$ is a function on $\Sigma$ (modulo additive constant), and $\lambda_0\in\C$ constant (Dolbeault decomposition, see e.g. Section 2.B in \cite{ABMNV}). Thus $\dbar+\alpha=e^{-g}(\dbar+\lambda_0d\bar z)e^{g}$. 

In the case of interest here ($(0,1)$-forms on the torus), it is easy to justify \eqref{eq:Dolbdec} directly (although this is not logically needed for our arguments). Write $\alpha=ad\bar z$, $a$ a smooth function, so that $\partial a=a_zdz\wedge d\bar z$. Since $a_z$ has zero mean on the torus (integration by parts), it is the Laplacian of a smooth function: we can find $g$ so that $\partial a=\partial\dbar g$. Then $(a-g_{\bar z})$ is antiholomorphic on $\Sigma$, hence a constant $\lambda_0$. For uniqueness, observe that if $\lambda_0d\bar z+\dbar g=0$, then $\lambda_0dz\wedge d\bar z=\dbar(gdz)$, so that $\lambda_0=0$ by Stokes' formula.

In the case $\chi=\chi_{\lambda_0}\neq\Id$ (see \eqref{eq:chilambda}), this is invertible, with inverting kernel given by (see \eqref{eq:Schi})
$$S_\alpha(b,w)=S_{\chi}(b,w)e^{g(w)-g(b)}$$
From the expansion \eqref{eq:Schidiag}
$$S_{\chi}(b,w)e^{2i\Im\lambda_0\overline{(b-w)}}=1/\pi(b-w)+r_{\chi}+O(|z-w|/\sqrt{c_\chi})$$
we obtain the asymptotic expansion as $b\rightarrow w$, $w\in\Sigma$ fixed, $\lambda=a(w)$:
\begin{align*}
S_\alpha(b,w)e^{2i\Im\lambda\overline{(b-w)}}&=\left(\frac{1}{\pi(b-w)}+r_{\chi}+O(|b-w|/\sqrt{c_\chi})\right)e^{2i\Im g_{\bar z}(w)\overline{(b-w)}-(g(b)-g(w))}\\
&=\frac{1}{\pi(b-w)}+r_{\chi}-\frac{1}\pi(\bar g_z(w)+g_z(w))+O(|b-w|.\|g\|_{C^2}/{\sqrt{c_\chi}})
\end{align*}

Let us turn now to the discrete operator $K_\alpha$, seen as a finite difference operator. For a function $\varphi$ on $\Sigma$, we define as in the planar case (see \eqref{eq:Crestrop}) restriction operators $R_B,\bar R_B$ by: $(R_B\varphi)(b)=(\bar R_B\varphi)(b)=\varphi(b)$ for $b\in\Gamma$ and $(R_B\varphi)(b)=-(\bar R_B\varphi)(b)=i\varphi(b)$ for $b\in\Gamma^\dg$. Let $w\in M_W$ be an edge node, with black neighbours $x,y,x',y'$ in ccwise order labelled in such a way that $(xx')$ is an oriented edge of $\Gamma$. We denote $(R_W\beta)(w)=b(w)e^{-i\nu(w)}$ for $\beta=b(z)d\bar z$; then (see \eqref{eq:findiffpert})
$$K_\alpha(R_B\varphi)=2\mu_\dmd R_W(\dbar\varphi+\varphi\alpha)+O(\delta^4\|\varphi\|_{C^3}(1+\|a\|_{C^2})^3)$$
Similarly, with $(\bar R_{W}\beta)=b(w)e^{i\nu(w)}$ for $\beta=b(z)dz$, we have:
$$K_\alpha(\bar R_B\varphi)=2\mu_\dmd \bar R_W(\partial\varphi-\varphi\bar\alpha)
+O(\delta^4\|\varphi\|_{C^3}(1+\|a\|_{C^2})^3)$$

Fix $w_0\in M_W$ and set $\lambda=a(w_0)$. Then:
$$K_\alpha(w,b)=e^{2i\Im(\bar\lambda w)}\rK(w,b)e^{-2i\Im(\bar\lambda b)}+O(\delta^2\dist(b,w_0)\|a\|_{C^1})$$

As in the planar case (see Section \ref{subsec:planarinvker}), the argument consists in defining an approximate inverting kernel $\tilde\rS_\alpha$ using the continuous inverting kernel $S_\alpha$ away from the diagonal and translation-invariant inverting kernels mesoscopically  close to the diagonal.
 
Set 
$$\tilde \rS_\alpha(b,w)=e^{2i\Im(\bar\lambda (b-w))}(\uK^{-1}(b,w)+R_B(\mu)+\bar R_B(\mu'))$$ 
for $|b-w|<\eta$ and
$$\tilde \rS_\alpha(b,w)=\frac 12 \left(R_B(e^{i\nu(w)}S_\alpha(b,w))+\bar R_B(e^{-i\nu(w)}\overline {S_{-\alpha}(b,w)})\right)$$
for $|b-w|\geq\eta$, where $\eta$ is a mesoscopic scale and
\begin{align*}
\mu&=\frac 12 e^{i\nu(w)}(r_{\chi}-\frac {2}\pi(\partial_z\Re g)(w))\\
\mu'&=\frac 12 e^{-i\nu(w)}(\overline{r_{\bar\chi}}+\frac {2}\pi(\partial_{\bar z}\Re g)(w))=-\bar\mu
\end{align*}
Reasoning as in the planar case (see Section \ref{subsec:planarinvker}; it is a bit simpler here as $\Sigma$ is compact), we may use this approximate inverse to show that $K_\alpha$ is invertible for $\delta$ small enough and its inverse $\rS_\alpha$ satisfies the following near diagonal estimate:
\begin{equation}\label{eq:torusdiagest}
\rS_\alpha(b,w)-e^{2i\Im\bar\lambda(b-w)}\uK^{-1}(b,w)=i\Im(e^{i\nu}r_\alpha)+O(\delta^{\eps_0})
\end{equation}
where $\eps_0$ is a positive constant, 
\begin{equation}\label{eq:ralphator}
r_\alpha=r_{\chi}-\frac 2\pi(\partial_z\Re g)(w)
\end{equation}
and the error term is uniform in $(\lambda_0,g)$ for $\|g\|_{C^2}$ bounded and $\chi=\chi_{\lambda_0}$ in a compact subset of the group of characters not containing the identity ($r_\chi$ blows up as $\chi\rightarrow\Id$). (Compare with \eqref{eq:diagestplane}).

\paragraph{Characteristic functional of the current.\\}

Similarly to the variational analysis of Sections \ref{subsec:varplane} (Lemma \ref{Lem:planevariation}) and \ref{ssec:vartorus} (leading to \eqref{eq:discrtors}), the near-diagonal estimate \eqref{eq:torusdiagest} implies:
$$\frac{\det{K_{\alpha_2}}}{\det{K_{\alpha_1}}}\exp(2i(P(i\alpha_2)-P(i\alpha_1)))\longrightarrow
\exp\left(-\frac 1{2\pi}\left(\int_\Sigma |\nabla\Re g_2|^2dA-\int_\Sigma |\nabla\Re g_1|^2dA\right)\right)
\left(\frac{T_\Sigma(\chi_2)}{T_\Sigma(\chi_1)}\right)^2$$
where $\alpha_i=\lambda_id\bar z+\dbar g_i$, $\chi_i$ is the associated character, and convergence is uniform for $\|g_i\|_{C^2}$ bounded and $\chi_i$ away from the trivial character. Let us point out at this stage that the additive decomposition of $r_\alpha$ (see \ref{eq:ralphator}) reflecting the Dolbeault decomposition of $\alpha$ (see \ref{eq:Dolbdec}) leads to the multiplicative decomposition of the relative determinant $\det(K_{\alpha_2})/\det(K_{\alpha_1})$ (in the fine mesh limit), which in turns yields the independence of the instanton and scalar components of the limiting field.

In order to complete the identification, we notice that, by combining Lemma \ref{Lem:detplane} and \eqref{eq:torusabelianpert}, we have
$$\det(K_\alpha)\exp(2iP(i\alpha))={\mc Z}\E\left(\exp(4i\pi(m\eps+n\eps')+4i\pi mn)\exp(-i\Re\int_\Sigma h_0(\Lap g)dA)\right)$$
with $h_0$ as in \eqref{eq:hodgedecdimer}, $\alpha=\lambda d\bar z+\dbar g$, $(\eps,\eps')$ as in \eqref{eq:lambdachar}, $(n,m)$ as in \eqref{eq:halfpers}. Classically (see \eqref{eq:toruspartfun})
$${\mc Z}=\frac 12(\det K_{0,0}+\det K_{\frac 12,0}+\det K_{0,\frac 12})$$
which implies
$$\E\left(\exp(4i\pi(m\eps+n\eps'))\exp(-i\Re\int_\Sigma h_0(\Lap g)dA)\right)\longrightarrow\E_0\left(\exp(4i\pi(m\eps+n\eps'))\right)\exp\left(
-\frac 1{2\pi}\int_\Sigma |\nabla\Re g|^2dA\right)$$
in the generic case $(\eps,\eps')\notin(\frac 12\Z)^2$. In the case $g=0$, we already noticed that this ensures convergence in distribution of the half-periods $(n,m)$ (Section \ref{subsec:bosonid}, \eqref{eq:periodconv}). Then for $(n_0,m_0)\in (\frac 12\Z)^2$, we write
$$\delta_{(n_0,m_0)}(n,m)=\int_{[0,1]^2}e^{4i\pi((m-m_0)\eps+(n-n_0)\eps')}d\eps d\eps'$$
and 
$$\E\left(\ind_{\{(n,m)=(n_0,m_0)\}}\exp(-i\Re\int_\Sigma h_0(\Lap g)dA)\right)=\int_{[0,1]^2}e^{-4i\pi(m_0\eps+n_0\eps')} 
\E\left(e^{4i\pi(m\eps+n\eps')}\exp(-i\Re\int_\Sigma h_0(\Lap g)dA)\right)d\eps d\eps'
$$
Since $\P((n,m)=(n_0,m_0))\rightarrow\P_0((n,m)=(n_0,m_0))>0$, we deduce by dominated convergence:
\begin{equation}\label{eq:fluctcondtor}
\E\left(\left.\exp(-i\Re\int_\Sigma h_0(\Lap g)dA)\right|(n,m)=(n_0,m_0)\right)\longrightarrow \exp\left(
-\frac 1{2\pi}\int_\Sigma |\nabla\Re g|^2dA\right)
\end{equation}
for any $(n_0,m_0)\in (\frac 12\Z)^2$.

From this we deduce:

\begin{Prop}[Finite dimensional marginals]\label{Prop:findimmargtor}
Let $h_\delta$ be the height function of a dimer configuration on the graph $\Xi$ and $J_\delta=dh_\delta$ be its current. Let $\alpha_1,\dots,\alpha_k$ be $C^1$ 1-forms on $\Sigma$. As $\delta\searrow 0$, the joint distribution of 
$$\left(2\pi\int_\Sigma J_\delta\wedge \ast\alpha_1,\dots,2\pi\int_\Sigma J_\delta\wedge\ast\alpha_k\right)$$
converges to the joint distribution of 
$$\left(\int_\Sigma J\wedge\ast \alpha_1,\dots,\int_\Sigma J\wedge\ast\alpha_k\right)$$
where $J$ is the current of the compactified free field on $\Sigma$ with compactification radius $1$ and coupling constant $g_0=\frac 12$ (ie action functional $S(J)=-\frac 1{8\pi}\int_\Sigma J\wedge\ast J$).
\end{Prop}
\begin{proof}
We have the Hodge decomposition $\alpha_j=(a_jdx+b_jdy)+dg_j$, $g_j\in C^2$ (see \eqref{eq:hodgedec}). We know that the marginal distribution of the current periods converges (see \eqref{eq:periodconv}, Section \ref{subsec:bosonid}), and that the conditional characteristic function
$$\E(\exp(i\sum_j\lambda_j\int_\Sigma J\wedge \ast dg_j)|(n,m))$$
converges pointwise by \eqref{eq:fluctcondtor}. This is enough to ensure convergence of the joint distribution.
\end{proof}

\paragraph{Tightness.\\}
In order to obtain a functional CLT from this finite dimensional CLT, we need a tightness estimate.

\begin{Lem}
For any $\epsilon>0$, the family of probability measures induced on $H^{-2-\epsilon}(\Omega^1(\Sigma))$ by the current $J_\delta$ is tight.
\end{Lem}
\begin{proof}

Let us remark that for a r.v. $X$,
$$\E(X^2)\leq \E(Q_{0,0}X^2)+\E(Q_{0,\frac 12}X^2)+\E(Q_{\frac 12,0}X^2)$$
where $Q_{\eps,\eps'}=e^{4i\pi(m\eps+n\eps')+4i\pi mn}$ (a random sign), since $1\leq Q_{0,0}+Q_{0,\frac 12}+Q_{\frac 12,0}\leq 3$. 

Set $\alpha_{\eps,\eps'}(t)=\lambda(\eps,\eps')d\bar z+t\dbar g$ (see \eqref{eq:lambdachar}) and $X=\Re\int h\Lap g dA$. We get
$$\E((\Re\int h_0\Lap gdA)^2)\leq -{\mc Z}^{-1}{\frac{d^2}{dt^2}}_{|t=0}(\E(Q_{0,0}e^{itX})+\E(Q_{0,1}e^{itX})+\E(Q_{1,0}e^{itX}))$$
We have:
$$\det K_{\alpha_{\eps,\eps'}(t)}={\mc Z}\E(Q_{\eps,\eps'}e^{-itX})\exp(-2itP(i\dbar g))$$
and thus 
$$\E(Q_{\eps,\eps'}X^2)=-2{\mc Z}^{-1}{\frac{d^2}{dt^2}}_{|t=0}(\det K_{\alpha_{\eps,\eps'}(t)}\exp(2itP(i\dbar g)))$$
We now assume that $(\eps,\eps')\in\{(0,0),(0,\frac 12),(\frac 12,0)\}$, so that in particular $\det K_{\alpha_{\eps,\eps'}(t)}$ is invertible at $t=0$, and
\begin{align*}
\frac {d}{dt}\log\E(Q_{\eps,\eps'}e^{itX})&=\Tr(K_\alpha^{-1}\dot{K_\alpha})+2iP(i\dbar g)\\
\frac {d^2}{dt^2}\log\E(Q_{\eps,\eps'}e^{itX})&=\Tr(K_\alpha^{-1}\ddot{K_\alpha})-\Tr((K_\alpha^{-1}\dot {K_\alpha})^2)
\end{align*}
where $\alpha=\alpha_{\eps,\eps'}(t)$ for brevity.

From the short distance asymptotics for $K_\alpha^{-1}$ (see \eqref{eq:torusdiagest}), we get
$$\Tr(K_\alpha^{-1}\dot{K_\alpha})+2iP(i\dbar g)=O(\delta^{\eps_0} \|g\|_{C^1})$$
at $t=0$ for some $\eps_0>0$. We also have
$$\Tr(K_\alpha^{-1}\ddot{K_\alpha})=O(\|g\|_{C^1}^2)
$$
We are left with estimating $\Tr((K_\alpha^{-1}\dot {K_\alpha})^2)$ (at $t=0$). Since $K_\alpha^{-1}(b,w)=\uK^{-1}(b,w)+O(1)$ and $\dot K_\alpha=O(\delta^2(1+\|g\|_{C^1}))$, by isolating the leading singularity we may write
$$\Tr((K_\alpha^{-1}\dot {K_\alpha})^2)=\sum_{|w-w'|\leq \eta_0}(\dot K_\alpha\uK^{-1})(w,w')(\dot K_\alpha\uK^{-1})(w',w)+O((1+\|g\|_{C^1})^2)$$
where $\eta_0\leq \min(1,\Im\tau)/10$, say. Fix $w$; replacing $K_\alpha$ with $K_\beta$ where $\beta=tg_{\bar z}(w)d\bar z$ (a constant $(0,1)$-form) induces an error of order $O(\|g\|_{C^{1,\rho}}(1+\|g\|_{C^1})$ (here $C^{1,\rho}$ designates functions with $\rho$-H\"older derivative, $\rho>0$ arbitrarily small but fixed). Thus we simply need to estimate 
$$\sum_{w':|w-w'|\leq \eta_0}(\dot K_\beta\uK^{-1})(w,w')(\dot K_\beta\uK^{-1})(w',w)$$
which may be thought of as a discrete version of a principal value integral of type $p.v.\iint \frac{dA(w')}{(w'-w)^2}$. A probabilistic interpretation of this quantity goes as follows: a product $\uK^{-1}(b,w')\uK^{-1}(b',w)$ (for $b\sim w$, $b'\sim w'$) is, up to multiplicative local factors, the covariance $\Cov(\ind_{(bw)\in{\mf m}},\ind_{(b'w')\in{\mf m}})$ under the appropriate Gibbs measure on tilings of the full plane (see \eqref{eq:locstats}). By linearity ($w$ is fixed and we sum over $w'$), we are left with estimating $\Cov(\ind_{(bw)\in{\mf m}},\ell)$, where $\ell$ is a linear function of the heights in $B(w,\eps_0)$. Since $\beta$ is constant, $\partial\beta=0$
and it follows that $\ell$ depends only on the heights on $\partial B(w,\eps_0)$. Since $\Cov(\ind_{(bw)\in{\mf m}},\ind_{(b'w')\in{\mf m}})=O(1/|w'-w|^2)$, we conclude:
$$\sum_{w':|w-w'|\leq \eps_0}(\dot K_\beta\uK^{-1})(w,w')(\dot K_\beta\uK^{-1})(w',w)=O(\delta^2\|g\|_{C^1}^2)$$

We finally get the estimate
\begin{equation}\label{eq:2ndmomentbound}
\E((\Re\int h_0\Lap gdA)^2)=O((1+\|g\|_{C^1})(1+\|g\|_{C^{1,\rho}}))
\end{equation}
which is uniform in $\delta$ for $\delta$ small enough (from Proposition \ref{Prop:findimmargtor} we know that the limit as $\delta\searrow 0$ is of order at least $\|g\|_{C^1}^2$).

Consider an eigenbasis for the Laplacian on $\Sigma=\C/\Upsilon$: 
set
$$g_u(z)=\exp(i\Re(z\bar u))$$
where $u\in\check{\Upsilon}=\{v\in\C: \forall z\in\Upsilon,\Re(z\bar v)\in 2\pi\Z\}$. Note that $\|g_u\|_{C^k}=O(1+|u|^k)$ and $\|g_u\|_{C^{1,\rho}}=O(1+|u|^{1+\rho})$. We may define (see Section \ref{ssec:Sobolev})
$$\|\sum_{u\in\check{\Upsilon}} a_ug_u\|^2_{H^s}=\sum_{u\in\check{\Upsilon}} |a_u|^2(1+|u|^2)^s$$
If we choose $h_0$ (which is given modulo an additive constant) so that $\int_\Sigma h_0dA=0$, we may write $h_0=\sum_{u\in\check{\Upsilon}\setminus\{0\}}a_ug_u$, where from \eqref{eq:2ndmomentbound}
$$\E((a_u)^2)\leq c|u|^{\rho-2}$$
for ${u\in\check{\Upsilon}\setminus\{0\}}$, and $c$ is uniform in $\delta, u$. Consequently, the Chebychev inequality yields:
$$\P(\forall{u\in\check{\Upsilon}\setminus\{0\}}, |a_u|\leq C |u|^{-\gamma})\geq 1-\frac{c}{C^2}\sum_{u\in\check{\Upsilon}\setminus\{0\}}
|u|^{\eps-2-2\gamma}$$
where the sum converges if $2\gamma+\eps<0$. On this event, 
$$\|\sum_u a_ug_u\|^2_{H^s}\leq C\sum_{u\in\check{\Upsilon}\setminus\{0\}} |u|^{2s-2\gamma}$$
which converges if $2s-2\gamma<-2$. By taking $\gamma\in(-\rho,-\rho/2)$, and observing that $H^{s_1}(\Sigma)$ is compactly embedded in $H^{s_2}(\Sigma)$ for $s_1>s_2$,
we conclude that the probability measures induced by $h_0$ on $H^{-1-\epsilon}(\Sigma)$ are tight for $\delta$ small enough.

The current can be decomposed as $J_\delta=dh_0+\omega_h$. Since $\omega_h$ takes values in a two-dimensional lattice of $H^s(\Omega^1(\Sigma))$ and converges in distribution, it induces a tight family of probability measures ($s$ arbitrary). The lemma follows.

\end{proof}

\paragraph{Functional invariance principle.\\}

We may now state a functional limit theorem for the current.

\begin{Thm}\label{Thm:toruscompact}
For any $\epsilon>0$, the probability measures induced on $H^{-2-\epsilon}(\Omega^1(\Sigma))$ by the current $2\pi J_\delta$ converge as $\delta\searrow 0$ to the distribution of the current of the compactified free field on $\Sigma$ with compactification radius $1$ and coupling constant $g_0=\frac 12$ (ie action functional $S(J)=\frac \pi 2\int J\wedge\ast J$).
\end{Thm}
\begin{proof}
We have obtained tightness of the measures (by the previous lemma) and convergence of the characteristic functional for smooth enough test functions by Proposition \ref{Prop:findimmargtor}. The dual of $H^{-2-\epsilon}(\Omega^1(\Sigma))$ may be identified with $H^{2+\epsilon}(\Omega^1(\Sigma))$ (via the $L^2$ product on $1$-forms), which consists of 1-forms with $C^1$ coefficients (see Section \ref{ssec:Sobolev}), so that we may apply Proposition \ref{Prop:findimmargtor}. %
For Banach space-valued variables, tightness and pointwise convergence of the characteristic functional ensures weak convergence (eg \cite{Ledoux_Talagrand}, 0.2.1).
\end{proof}

As in Corollary \ref{Cor:FFplane}, the result is one order of differentiability below the notional optimal result (convergence of the current in $H^{-2-\epsilon}$ rather than $H^{-1-\epsilon}$).

\section{Surgery}\label{Sec:surgery}

\paragraph{Introduction.}
Many problems in the asymptotic analysis of dimers boil down to estimating the inverse of the Kasteleyn operator $\rK$ or a perturbation or modification thereof. In this section, we develop a general surgery argument which allows to localise the analysis in the following (rough) sense. Assume that a graph $\Xi_g$, along with a Kasteleyn-type operator $K_g$, is obtained from the standard full-plane graph $M$ and operator $\rK$ by two disjoint local modifications. Possible local modifications include: adding a boundary component; adding a (compactly supported) smooth test function (as in \eqref{eq:Kastpert}); and, as we shall see in later sections, adding pairs of electric or magnetic insertions.

Let $(\Xi_i,K_i)$ and $(\Xi_o,K_o)$ be the graph/operator pairs corresponding to just one of these modifications - so that $(\Xi_g,K_g)$ is obtained by gluing $(\Xi_i,K_i)$ with $(\Xi_o,K_o)$. Our main contention is that, in the small mesh limit, controlling $K_i^{-1}$ and $K_o^{-1}$ allows to control $K_g^{-1}$; this is the content of Lemma \ref{Lem:surgery}.  

In order to build up some intuition, let us discuss a few elementary facts about (continuous) boundary value problems. Consider the Riemann sphere $\hat\C$ split into two analytic discs $D_+,D_-$ by the unit circle $\U$ (where $0\in D_+$ and $\infty\in D_-$). Consider the boundary value problems (BVPs):
$$\left\{\begin{array}{rll}
\dbar f_\pm&=0&{\rm\ in\ }D_{\pm}\\
f_{\pm}&=g_\pm&{\rm\ on\ }\U
\end{array}
\right.$$
say for $C^1$ functions up to the boundary, vanishing at infinity. Writing functions on the circle as $z\mapsto\sum_{n\in\Z}a_nz^n$, the BVP is solvable if $g_{\pm}\in C_{\pm}$, where $C_\pm$ are the {\em Cauchy data spaces} given by
\begin{align*}
C_+&=\{g:g=\sum_{n\geq 0}a_nz^n\}\\
C_-&=\{g:g=\sum_{n<0}a_nz^n\}
\end{align*}
where we disregard convergence and regularity issues. Given $g_+\in C_+$, the unique solution $f$ of the BVP is given by
$$f_+(w)=\frac 1{2i\pi}\oint_\U\frac{g_+(z)}{z-w}dz$$
The map $g_+\mapsto f_+$ is the Poisson operator for the BVP; its composition with the restriction to (or rather trace on) $\U$ is the Calder\'on projector $P_+$ given by
$$P_+(g)(z_0)=\lim_{r\nearrow 1}\frac 1{2i\pi}\oint\frac{g_+(z)}{z-rz_0}dz$$
This is a projection onto the Cauchy data space $C_+$; it is a pseudodifferential operator and extends to a bounded operator $L^2(\U)\rightarrow L^2(\U)$. 

Now if we consider a perturbation of the problem, e.g. by replacing $\dbar$ with $\dbar+\mu\partial$, $\mu$ supported away from $\U$, one may consider the modified Cauchy data spaces $C_\pm$ or projectors $P_\pm$. (These projectors need no longer be orthogonal in $L^2(\U)$). Information on the modified operator in the full space $\hat \C$ (such as its index) may be recovered from this data (``glueing"). We shall be concerned with natural discrete versions of these constructions and their convergence to their continuous counterparts.

\paragraph{Set-up.} The conditions given here will be assumptions in Lemma \ref{Lem:surgery}.

We start from a family of rhombi tilings $(\Lambda_\delta)_\delta$ with edge mesh $\delta$ going to zero along some sequence. Throughout the section, the dependence on $\delta$ will be omitted when there is no risk of ambiguity. Let $D_i$ (resp. $D_o$) be a simply connected neigbourhood of $0$ (resp. $\infty$) in $\hat\C$ such that $D_i\cap D_o$ is an annulus $A$ separating $0$ from $\infty$. 

Let $\gamma$ be a simple closed loop (say, piecewise $C^1$) such that $A$ is a tubular neighbourhood of $\gamma$. We assume that the distance between $\gamma$ and $\partial A$ is large enough compared with $\delta$.

From $\Lambda$ we construct a planar graph $M$ as in Section \ref{sss:critgraphs}. Let $\Xi_i=\Xi_i(\delta)$ (resp. $\Xi_o$) be a bipartite graph that agrees with $M$ in $D_o$ (resp. $D_i$); $K_s:\C^{\Xi_s^B}\rightarrow\C^{\Xi^W_s}$, $s\in\{i,o\}$, is a linear operator such that $K_i$ (resp. $K_o$) agrees with $\rK$ in $D_o$ (resp. $D_i$); in other words, $\Xi_i,K_i$ are obtained by modifying $M,\rK$ in $D_i\setminus D_o$, and vice versa. The subscripts $i,o,g$ stand for (modified) inside, (modified) outside, and glued. See Figure \ref{Fig:surgery}.
\begin{figure}[htb]
\begin{center}
\leavevmode
\includegraphics[width=0.8\textwidth]{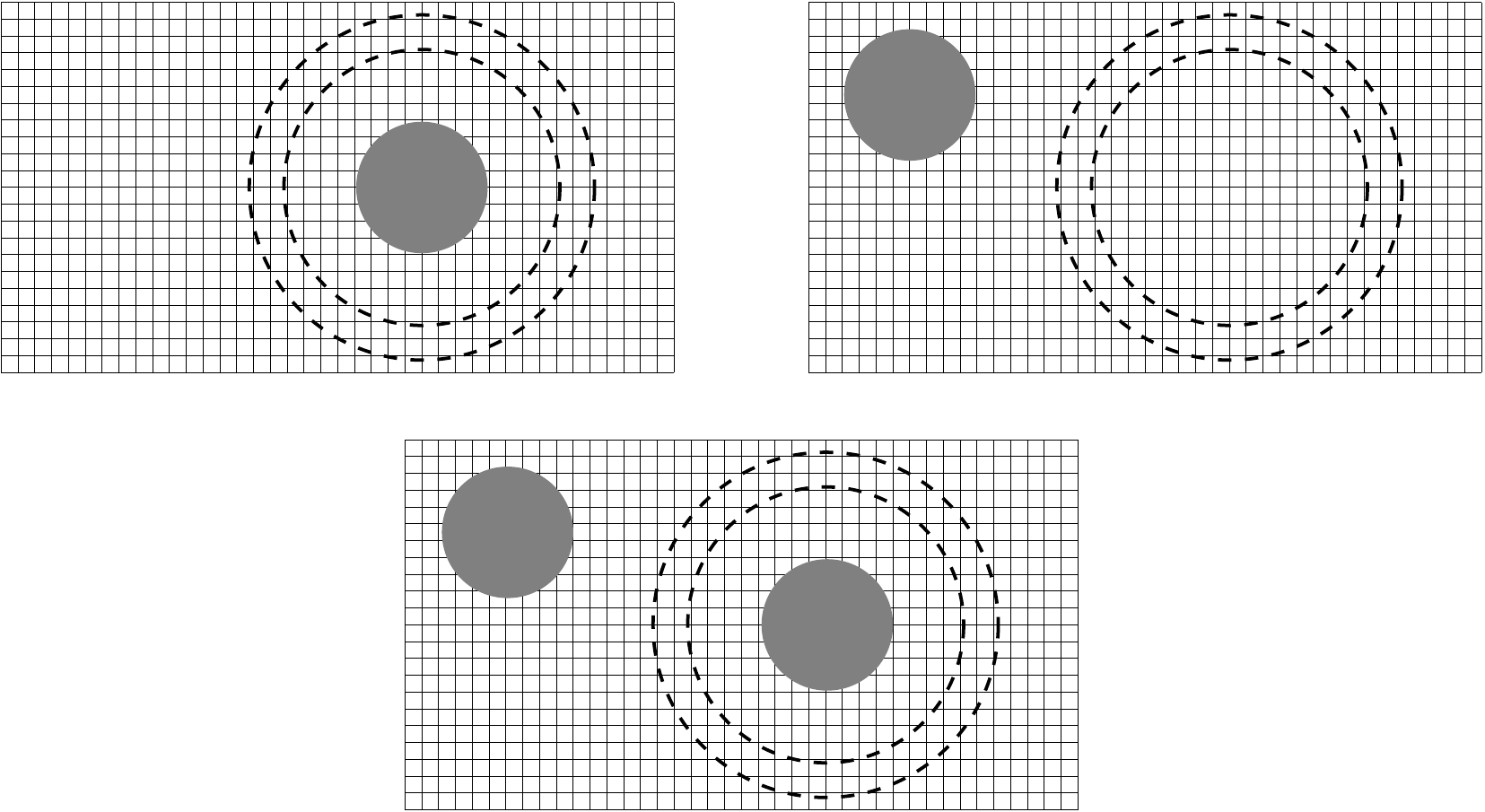}
\end{center}
\caption{Set-up (schematic). Top left: $\Xi_i$, obtained from $M$ by a modification of the graph and/or operator $\rK$ in the shaded area, inside the annulus $A$ (dashed). Top right: $\Xi_o$, modified in the shaded area (outside $A$). Bottom: the glued
graph $\Xi_g$, agreeing with $\Xi_o$ outside $A$ and with $\Xi_i$ inside $A$.
}
\label{Fig:surgery}
\end{figure}

We assume that $K_s$ is invertible in the sense that for each $w\in \Xi^W_s$, there is a unique function $f\in\C^{\Xi^B_s}$ vanishing at infinity such that $K_sf=\delta_w$, $s\in\{i,o\}$; it is denoted $K_s^{-1}(.,w)$.  This implies in particular that $K_sf=0$ and $f$ vanishes at infinity iff $f=0$. (Remark that $\Xi_o$ can be bounded, in which case the ``vanishing at infinity" condition is void.)

We assume that, for $s\in\{i,o\}$, there are kernels $(z,w)\mapsto S_s(z,w)$, $(z,w)\mapsto \bar S_s(z,w)$, and $\eta$ is an error rate: $\lim_{\delta\searrow 0}\eta(\delta)=0$ such that:
\begin{equation}\label{eq:surgkerconvassum}
K_s^{-1}(b,w)=\frac 12 R_B(e^{i\nu(w)}S_s(b,w))+
\frac 12\bar R_B(e^{-i\nu(w)}\bar S_s(b,w))+O(\eta(\delta))
\end{equation}
on $A^2\setminus\Delta_A$, uniformly on compact subsets, where $\Delta_A=\{(x,y)\in A^2:x= y\}$ (see \eqref{eq:Crestrop} and compare e.g. with \eqref{eq:planekerasymp}). 

Note that we do not require that $\bar S_s(z,w)=\overline{S_s(z,w)}$, which is the case when $K_s$ is real. More precisely, we assume that 
\begin{equation}\label{eq:surgregassum}
(z,w)\mapsto T_s(z,w)=S_s(z,w)-\frac{1}{\pi(z-w)}
\end{equation}
is $C^1$ in $A^2$ and holomorphic in $z$ in $A$ (in the presence of boundaries, it will be harmonic, rather than holomorphic, in $w$). The holomorphicity condition is actually superfluous but will always be obvious in applications. Correspondingly, we assume
$$(z,w)\mapsto \bar T_s(z,w)=\bar S_s(z,w)-\frac{1}{\pi\overline{(z-w)}}$$ 
is $C^1$ in $(z,w)$ and antiholomorphic in $z$.

A discrete Cauchy integral formula argument (see \eqref{eq:dirprobcauchy}, applied to $K_s^{-1}(.,w)-\uK^{-1}(.,w)$) shows that the convergence assumption \eqref{eq:surgkerconvassum} is equivalent to assuming the seemingly stronger condition:
$$K_s^{-1}(b,w)=\uK^{-1}(b,w)+\frac 12 R_B(e^{i\nu(w)}T_s(b,w))+
\frac 12\bar R_B(e^{-i\nu(w)}\bar T_s(b,w))+O(\eta(\delta))$$
on $A^2$, uniformly on compact subsets (including on the diagonal $\Delta_A$).

We now consider the glued data: $\Xi_g$ agrees with $\Xi_s$ in $D_s$, $s\in\{i,o\}$; $K_g: \C^{\Xi_g^B}\rightarrow\C^{\Xi^W_g}$ agrees with $K_s$ in $D_s$, $s\in\{i,o\}$.  Our goal to estimate $K_g^{-1}$ (if defined) by an expression of the type of \eqref{eq:surgkerconvassum}.

\paragraph{Discrete Cauchy integral formula.}

Let us consider $\gamma_\delta$ a simple cycle on $\Gamma$ (and thus on $M$) which approximates $\gamma$ in the sense that each arc of $\gamma$ of length $\ell$ is at Hausdorff distance $\leq C\delta$ of an arc of $\gamma_\delta$ of length at most $C\ell$, $C>0$ fixed.

Let $f\in\C^{\Xi_i^B}$ be such that $K_if(w)=0$ for any $w\in\Xi_i^W$ which is (strictly) inside $\gamma_\delta$. We want to express $f$ in terms of its boundary values and the kernel $K_i^{-1}$. This is a well-known argument, see e.g. Section 2.6 in \cite{SmiChe_isoradial} or the discussion after \eqref{eq:dirprobcauchy}.

Let $\hat f(b)=f(b)$ if $b$ is on or inside $\gamma_\delta$ and $0$ otherwise. Then $K_i\hat f$ is supported on white vertices on $\gamma_\delta$ or adjacent to a vertex on $\gamma_\delta$. Moreover $\hat f-K_i^{-1}(K_i\hat f)=0$ (since it is in the kernel of $K_i$ and vanishes at infinity). This yields the Cauchy integral formula:
$$f(b)=\sum_w K_i^{-1}(b,w)(K_i\hat f)(w)$$
for $b$ on or within $\gamma_\delta$. Let $\gamma^B$ be the set of black vertices which are either on $\gamma_\delta$ (and thus on $\Gamma$) or inside of $\gamma_\delta$ and adjacent to a white vertex on $\gamma_\delta$ (and thus on $\Gamma^\dg$). Correspondingly, let $\gamma^W$ be the set of white vertices which are either on $\gamma_\delta$ or outside of $\gamma_\delta$ and adjacent to a black vertex on $\gamma_\delta$. Let $\rK_{\gamma}(w,b)=\rK(w,b)$ if $w\in\gamma^W$, $b\in\gamma^B$ and $\rK_\gamma(w,b)=0$ otherwise. We can rephrase the Cauchy formula as:
\begin{equation}\label{eq:discrcauchform}
f(b)=\sum_{w\in\gamma^W}K_i^{-1}(b,w)(\rK_\gamma f_{|\gamma^B})(w)
\end{equation}
for $b\in\Xi_i^B$ on or inside $\gamma_\delta$ and $f$ such that $(K_if)(w)=0$ for $w$ inside $\gamma_\delta$.

\paragraph{Inner Cauchy data space and projector.} 
The (discrete) {\em Cauchy data space} (on $\gamma$, for $K_i$) is the subspace $C_i^\delta\subset\C^{\gamma^B}$ consisting of restrictions to $\gamma^B$ of functions $f$ such that $K_if=0$ strictly inside $\gamma_\delta$. Clearly, 
\begin{equation}\label{eq:surgdiscrcald}
\begin{array}{rl}
P_i^\delta:\C^{\gamma^B}&\longrightarrow \C^{\gamma^B}\\
g&\longmapsto\sum_{w\in\gamma^W}K_i^{-1}(.,w)(\rK_\gamma g)(w)
\end{array}
\end{equation}
is a projector onto the Cauchy data space $C_i^\delta$. 

We want to relate this to ``limiting" continuous Cauchy data spaces. Given our data \eqref{eq:surgkerconvassum}, we can define them in terms of the following operators on functions on $\gamma$:
\begin{equation}\label{eq:surgcontcald}
\begin{array}{rl}
P_i(f)(z_0)&=\lim_{z\rightarrow z_0} \frac 1{2i}\oint_\gamma S_i(z,w)f(w)dw\\
\bar P_i(f)(z_0)&=\lim_{z\rightarrow z_0} -\frac 1{2i}\oint_\gamma\bar S_i(z,w)f(w)d\bar w
\end{array}
\end{equation}
where limits are taken from inside $\gamma$. Writing
$$\frac 1{2i}\oint_\gamma S_i(z,w)f(w)dw=\frac 1{2i}\oint T_i(z,w)f(w)dw+f(z)+\oint_\gamma \frac{f(w)-f(z)}{2i\pi(z-w)}dw$$
shows that $P_i$ is a bounded operator ${\rm Lip}(\gamma)\rightarrow C^0(\gamma)$. This defines Cauchy data spaces by 
 \begin{equation}\label{eg:surgcontdata}
 \begin{array}{rl}
 C_i&=\{f\in{\rm Lip}(\gamma):f=P_if\}\\
 \bar C_i&=\{f\in{\rm Lip}(\gamma):f=\bar P_if\}
 \end{array}
 \end{equation} 
 (more classically one would consider the $L^2$ closure of these; Lipschitz functions are sufficient for our purposes).
 
\paragraph{Outer Cauchy data spaces.} 
While we will focus the discussion on {\em inner} Cauchy data spaces, one may repeat the argument for {\em outer} Cauchy data spaces. 
Let $\gamma^B_o$ be the black vertices which are on $\gamma_\delta$ or are outside of $\gamma_\delta$ and adjacent to a white vertex of $\gamma_\delta$. The outer Cauchy data space $C_o^\delta\subset\C^{\gamma^B_o}$ consists of restrictions to $\gamma^B_o$ of functions $f\in\C^{\Xi^B_o}$ such that $K_o f(w)=0$ for any $w\in \Xi^W_o$ strictly outside of $\gamma_\delta$ and $f$ vanishes at infinity. The continuous outer Cauchy data space $C_o$, $\bar C_o$ are defined as fixed points of the operators $P_o, \bar P_o$:
\begin{align*}
P_o(f)(z_0)&=\lim_{z\rightarrow z_0} \frac 1{2i}\oint_{\gamma}S_o(z,w)f(w)dw\\
\bar P_o(f)(z_0)&=\lim_{z\rightarrow z_0} -\frac 1{2i}\oint_{\gamma}\bar S_o(z,w)f(w)d\bar w
\end{align*} 
where limits are taken from outside $\gamma$.

\paragraph{Convergence of projectors.} 
Let $f$ be Lipschitz in a neigbourhood of $\gamma$. We wish to estimate the discrete ``contour integral" 
$$\sum_{w\in\gamma^W}K_i^{-1}(b,w)(\rK_\gamma (R_Bf))(w)$$ where $(R_Bf)(b)=f(b)$ if $b\in\Gamma\cap\gamma^B$ and  $(R_Bf)(b)=if(b)$ if $b\in\Gamma^\dg\cap\gamma^B$, as in \eqref{eq:Crestrop}. For this we note that $\gamma^W$ consists of white vertices on $\gamma_\delta$ (a simple cycle on $\Gamma$) and white vertices on $\gamma_\delta^\dg$ (a cycle on $\Gamma^\dg$). The path $\gamma_\delta^\dg$ may be described as follows. Let $(b_0,b_1,\dots,b_n=b_0)$ be the black vertices on $\gamma_\delta$ (taken counterclockwise). For each $i$, enumerate (in counterclockwise order) the faces of $\Gamma$ which are adjacent to $b_i$ and outside of $\gamma_\delta$: $b_{i,1}^\dg,\dots,b_{i,k_i}^\dg$. Concatenating these lists, one gets $\gamma_\delta^\dg=(b_{0,1}^\dg,\dots,b_{0,k_0}^\dg,b_{1,1}^\dg,\dots,b_{n-1,k_{n-1}}^\dg,\dots,b_{n,1}^\dg=b_{0,1}^\dg)$, a cycle on $\Gamma^\dg$ (which may involve some backtracking). Without loss of generality, one may assume that the reference orientation (used to define the Kasteleyn orientation, see Section \ref{sss:critgraphs}) of edges of $\Gamma$ on $\gamma_\delta$ agrees with the direct orientation of $\gamma_\delta$. Then if $w$ on $\gamma_\delta$ corresponds to the edge $(bb')$ of $\Gamma$,
\begin{equation}\label{eq:discrtcontint}
\begin{array}{rl}
(\rK_\gamma (R_Bf))(w)&=\frac 12|b'-b|if(b)+O(\delta^2 \|f\|_{{\rm Lip}})\\
e^{i\nu(w)}(\rK_\gamma (R_Bf))(w)&=\frac i2\int_b^{b'}f(z)dz+O(\delta^2 \|f\|_{{\rm Lip}})
\end{array}
\end{equation}
(with the integral taken on the segment $[b,b']$). Similarly, if $w$ on $\gamma_\delta^\dg$ corresponds to the edge $(bb')$ of $\Gamma^\dg$, and $b_0$ is the black neighbour of $w$ which is on $\gamma$, then
\begin{equation*}%
\begin{array}{rl}
(\rK_\gamma (R_Bf))(w)&=\sgn(\rK(b_0,w))\frac 12|b'-b|f(b)+O(\delta^2 \|f\|_{{\rm Lip}})\\
e^{i\nu(w)}(\rK_\gamma (R_Bf))(w)&=\frac i2\int_b^{b'}f(z)dz+O(\delta^2 \|f\|_{{\rm Lip}})
\end{array}
\end{equation*}
(if $w$ has two neighbours on $\gamma$,  $(\rK_\gamma (R_Bf))(w)=O(\delta^2 \|f\|_{{\rm Lip}})$). Then
$$\sum_{w\in\gamma^W}(e^{i\nu(w)}T_i(z,w))(\rK_\gamma (R_Bf))(w)=\frac i2\oint_{\gamma_\delta} T_i(z,w)f(w)dw
+\frac i2\oint_{\gamma^\dg_\delta} T_i(z,w)f(w)dw
+O(\delta\|f\|_{{\rm Lip}}\|S_i\|_\infty)
$$ 
and by Stokes' formula, $\oint_{\gamma_\delta} T_i(z,w)f(w)dw
-\oint_{\gamma^\dg_\delta} T_i(z,w)f(w)dw=O(\delta\|f\|_{{\rm Lip}}\|T_i\|_{C^1})$ (as the area of the annulus between $\gamma_\delta,\gamma_\delta^\dg$ is $O(\delta)$). (Estimates are uniform for $z$ in a compact subset of $A$). Similarly,
\begin{align*}
\sum_{w\in\gamma^W}(e^{i\nu(w)}T_i(z,w))(\rK_\gamma (R_Bf))(w)&=i\oint_{\gamma} T_i(z,w)f(w)dw+O(\delta\|f\|_{{\rm Lip}}\|T_i\|_{C^1})\\
\sum_{w\in\gamma^W}(e^{i\nu(w)}T_i(z,w))(\rK_\gamma (\bar R_Bf))(w)&=O(\delta\|f\|_{{\rm Lip}}\|T_i\|_{C^1})\\
\sum_{w\in\gamma^W}(e^{-i\nu(w)}\bar T_i(z,w))(\rK_\gamma (R_Bf))(w)&=O(\delta\|f\|_{{\rm Lip}}\|T_i\|_{C^1})\\
\sum_{w\in\gamma^W}(e^{-i\nu(w)}\bar T_i(z,w))(\rK_\gamma (\bar R_Bf))(w)&=-i\oint_{\gamma} \bar T_i(z,w)f(w)d\bar w+O(\delta\|f\|_{{\rm Lip}}\|T_i\|_{C^1})
\end{align*}
(One could replace $\|T\|_{C^1}$ with $\|T_i\|_\infty$ using biharmonicity of $T_i$; these norms are taken on a compact neighbourhood of $\gamma$). 
In order to deal with the singular part, we observe that the constant functions $R_B(\mu)$, $\bar R_B(\mu)$ are discrete holomorphic and consequently by replication \eqref{eq:discrcauchform}:
\begin{align*}
\sum_{w\in\gamma^W}\uK^{-1}(b,w)(\rK_\gamma (R_B(\mu)))(w)&=R_B(\mu)\ind_{\gamma^i}(b)\\
\sum_{w\in\gamma^W}\uK^{-1}(b,w)(\rK_\gamma (\bar R_B(\mu)))(w)&=\bar R_B(\mu)\ind_{\gamma^i}(b)
\end{align*}
 where $\gamma^i$ is the set of vertices on or inside $\gamma_\delta$. Then by Theorem \ref{Thm:Kencrit} and \eqref{eq:discrtcontint}

 \begin{align*}
\sum_{w\in\gamma^W}\uK^{-1}(b,w)(\rK_\gamma (R_Bf))(w)&=
R_B(f(b))\ind_{\gamma^i}(b)+\sum_{w\in\gamma^W}\uK^{-1}(b,w)(\rK_\gamma (R_B(f-f(b)))(w)\\
&=R_B(f(b))\ind_{\gamma^i}(b)+\frac{i}{2\pi}R_B\left(\oint_\gamma\frac{f(w)-f(b)}{b-w}dw\right)+O(\delta|\log\delta|.\|f\|_{{\rm Lip}})
\end{align*}
 and similarly
 $$\sum_{w\in\gamma^W}\uK^{-1}(b,w)(\rK_\gamma (\bar R_Bf))(w)=\bar R_B(f(b))\ind_{\gamma^i}(b)-\frac{i}{2\pi}\bar R_B\left(\oint_\gamma\frac{f(w)-f(b)}{\overline{b-w}}d\bar w\right)+O(\delta|\log\delta|.\|f\|_{{\rm Lip}})
 $$
We conclude that if $f$ is Lipschitz around $\gamma$ ($\|f\|_{{\rm Lip}}$ its Lipschitz norm in a neighbourhood of $\gamma$), then 
$$\|P_i^\delta(R_B(f))-R_B(P_if)\|_\infty=O(\tilde\eta(\delta)\|f\|_{{\rm Lip}})$$
(on $\gamma$), where 
$$\tilde\eta(\delta)=\eta(\delta)+(\delta|\log\delta|).$$ 
In particular, if $f$ is Lipschitz around $\gamma$, then it is in the (continuous) Cauchy data space $C_i$ iff $\|P_i^\delta(R_B(f))-R_B(f)\|_\infty$ goes to zero as $\delta\searrow 0$ (uniform norm on $\gamma^B$). Correspondingly, with the same assumptions, 
\begin{equation}\label{eq:surgprojconv}
\|P_i^\delta(\bar R_B(f))-\bar R_B(\bar P_if)\|_\infty=O(\tilde\eta(\delta)\|f\|_{{\rm Lip}})
\end{equation}
again on $\gamma$.

\paragraph{Glueing - uniqueness.} 
Let us address uniqueness for $K_g$, ie $K_gf=0$, $f$ vanishing at infinity implies that $f=0$. This will follow (for small enough $\delta$) from the following natural assumption on continuous Cauchy data spaces: 
\begin{equation}\label{eq:surgzeroind}
C_i\cap C_0=\{0\},{\rm\ \ \ }\bar C_i\cap \bar C_o=\{0\}.
\end{equation}
 Indeed, assume by contradiction that for some sequence $\delta_n\searrow 0$, there is $f_n\in\C^{\Xi^{B,\delta_n}_g}\neq 0$ such that $K_g^{\delta_n}f_n=0$, $f_n$ vanishes at infinity (the line of argument here is similar to some arguments in Section 3 of\cite{SmiChe_isoradial}). Let $\gamma$ be a simple cycle in $A$ that disconnects $0$ from $\infty$. We normalise $f_n$ so that $\|(f_n)_{|\gamma}\|_\infty=1$. Let $\gamma_1$,$\gamma_2$ be two disjoint cycles in $A$ bounding an open annulus $A'$ that contains $\gamma$. By replication (as in \eqref{eq:discrcauchform}), we see that $\|(f_n)_{|\gamma_1\cup\gamma_2}\|_\infty$ is bounded and consequently the Lipschitz norm of $f_n$ on compact subsets of $A'$ is bounded. Up to extracting a subsequence, one may assume that (a suitable interpolation of) $(f_n)$ converges uniformly on compact subsets of $A'$ to a non-vanishing (since it has uniform norm $1$ on $\gamma$) Lipschitz function $f$. More precisely, we may write $f_n=R_B(g_n)+\bar R_B(h_n)$, where $(g_n)$ and $(h_n)$ converge in Lispchitz norm on compact subsets of $A'$ (again by replication). Consequently, $g=\lim g_n  $ is in $C_i\cap C_0=\{0\}$,  $h=\lim h_n  $ is in $\bar C_i\cap \bar C_0=\{0\}$, yielding the needed contradiction.

\paragraph{Glueing - convergence.}

We may now address the central question of this section, i.e. convergence of the glued inverting kernel $K_g^{-1}$ assuming \eqref{eq:surgkerconvassum}.

Assume that $S_g:A^2\rightarrow\C$ is a kernel with the same regularity conditions as $S_i,S_o$ (see around \eqref{eq:surgregassum}) and such that: for any simple cycle $\gamma$ in $A$ disconnecting $0$ from $\infty$, if $w$ is outside $\gamma$, $S_g(.,w)\in C_i$ and $S_g(.,w)-S_o(.,w)\in C_o$ ; and if $w$ is inside $\gamma$, $S_g(.,w)\in C_o$ and $S_g(.,w)-S_i(.,w)\in C_i$ (where $C_i,C_o$ are the inner and outer Cauchy data space on $\gamma$). Similarly, we assume given $\bar S_g:A^2\rightarrow\C$ a kernel compatible in the same way with inner and outer Cauchy data spaces $\bar C_i$, $\bar C_o$, and with the same regularity as $\bar S_i,\bar S_o$. Remark that this uniquely specifies $S_g$ by \eqref{eq:surgzeroind}.

Given this data, we start with constructing $\tilde S_g$, an approximate inverting kernel for $K_g:\C^{\Xi_g^B}\rightarrow\C^{\Xi_g^W}$ (at least for some $w$'s). First let us consider two disjoint simple cycles $\gamma_1$, $\gamma_2$ in $A$ that disconnect $0$ from $\infty$, with $\gamma_1$ inside $\gamma_2$; $\gamma_i^\delta$ is an approximation of $\gamma_i$ on $\Gamma_\delta$. For $w\in M_W^\delta$ within $O(\delta)$ of $\gamma_2$, we set:
\begin{align*}
\tilde S_g(b,w)&=K_o^{-1}(b,w)+P_o^{\gamma^\delta_1}\left(\frac 12R_B(e^{i\nu(w)}(S_g(.,w)-S_o(.,w)))+\frac 12\bar R_B\left(e^{-i\nu(w)}(\bar S_g(.,w)-\bar S_o(.,w))\right)\right)\\
&{\rm\ \ \ \ \ for\ }b{\rm\ outside\ }\gamma_1\\
&=P_i^{\gamma^\delta_1}\left(\frac 12R_B\left(e^{i\nu(w)}S_g(.,w)\right)+\frac 12\bar R_B\left(e^{-i\nu(w)}\bar S_g(.,w)\right)\right)\\
&{\rm\ \ \ \ \ for\ }b{\rm\ inside\ }\gamma_1
\end{align*}
and for $b$ on $\gamma^\delta_1$ one may use either definition (here $P_i^{\gamma_\delta}$, $P_o^{\gamma_\delta}$ denote the inner and outer discrete Cauchy data space projectors described earlier, see \eqref{eq:surgdiscrcald}). Symmetrically, for $w\in M_W^\delta$ within $O(\delta)$ of $\gamma_1$, we set:
\begin{align*}
\tilde S_g(b,w)&=K_i^{-1}(b,w)+P_i^{\gamma^\delta_2}\left(\frac 12R_B(e^{i\nu(w)}(S_g(.,w)-S_i(.,w)))+\frac 12\bar R_B\left(e^{-i\nu(w)}(\bar S_g(.,w)-\bar S_i(.,w)))\right)\right)\\
&{\rm\ \ \ \ \ for\ }b{\rm\ inside\ }\gamma_2\\
&=P_o^{\gamma^\delta_2}\left(\frac 12R_B\left(e^{i\nu(w)}S_g(.,w)\right)+\frac 12\bar R_B\left(e^{-i\nu(w)}\bar S_g(.,w)\right)\right)\\
&{\rm\ \ \ \ \ for\ }b{\rm\ outside\ }\gamma_2
\end{align*}
Let us observe that if $w$ is within $O(\delta)$ of $\gamma_i$, $K_g\tilde S_g(.,w)-\delta_w$ is supported on white vertices with graph distance $\leq 1$ to $\gamma_{3-i}^\delta$, $i\in\{1,2\}$; let us denote $\hat\gamma^W_i$ these sets of white vertices, $i\in\{1,2\}$.

In the continuous limit, the assumptions on compatibility of $S_g$ with the Cauchy data spaces specified by $S_i$, $S_o$ translate into the replication identities: 
$$
P^{\gamma_1}_i(S_g(.,w))=S_g(.,w),{\rm\ \ \ \ }
P^{\gamma_1}_o(S_g(.,w)-S_o(.,w))=S_g(.,w)-S_o(.,w)
$$
on $\gamma_1$ if $w$ is on $\gamma_2$; and symmetrically
$$
P^{\gamma_2}_o(S_g(.,w))=S_g(.,w),{\rm\ \ \ \ }
P^{\gamma_2}_i(S_g(.,w)-S_i(.,w))=S_g(.,w)-S_i(.,w)
$$
on $\gamma_2$ if $w$ is on $\gamma_1$. The corresponding identities for $\bar S_g$ also hold. Together with the earlier convergence result for $P^{\gamma_\delta}_s$, $s\in\{i,o\}$ (see \eqref{eq:surgprojconv}), we get for instance that 
$$\|P^{\gamma^\delta_1}_i(R_B(S_g(.,w))-R_B(S_g(.,w))\|_\infty=O(\tilde\eta(\delta))$$ 
(uniform norm on $\gamma^B_1$), uniformly in $w\in\gamma_2$.

Consequently, if $w$ is near $\gamma_2$, we have
$$\tilde S_g(b,w)=\frac 12 R_B\left(e^{i\nu(w)}S_g(b,w)\right)+\frac 12 \bar R_B\left(e^{-i\nu(w)}\bar S_g(b,w)\right)+O(\tilde\eta(\delta))
$$
for $b$ a black vertex in a compact subset of $A\setminus \gamma_1$, and in particular for $b$ on either side of $\gamma_1$ (ie both definitions of $\tilde S_g$ agree up to $O(\tilde\eta(\delta))$ near $\gamma_1$). It follows that $K_g\tilde S_g(.,w)-\delta_w$ is supported on $\hat\gamma_1^W$ and is $O(\delta\tilde\eta(\delta))$ there. Thus
$(K_g\tilde S_g)$, seen as an operator $\C^{\hat\gamma^W}\rightarrow\C^{\hat\gamma^W}$ (where $\hat\gamma^W=\hat\gamma^W_1\sqcup\hat\gamma^W_2$) is such that $\||K_g\tilde S_g-\Id\||_{L^1(\gamma^W)}=O(\tilde\eta(\delta))$ (since $|\gamma^W|=O(\delta^{-1})$). Thus for $\delta$ small enough, $K_g\tilde S_g:\C^{\hat\gamma^W}\rightarrow\C^{\hat\gamma^W}$ is invertible, and $\||(K_g\tilde S_g)^{-1}-\Id\||_{L^1(\gamma^W)}=O(\tilde\eta(\delta))$. Then for $w\in\gamma^W$, we set
$$S^\delta_g(.,w)=\sum_{w'\in\gamma^W}(K_g\tilde S_g)^{-1}(w,w')\tilde S_g(.,w')$$
so that $K_gS^\delta_g(.,w)=\delta_w$. It follows that
$$S^\delta_g(b,w)=\frac 12 R_B\left(e^{i\nu(w)}S_g(b,w)\right)+\frac 12 R_B\left(e^{-i\nu(w)}\bar S_g(b,w)\right)+O(\tilde\eta(\delta))
$$
uniformly in $w\in\gamma^W$ and $b$ in a compact subset of $A\setminus(\gamma_1\cup\gamma_2)$. 

For $w$ in general position in $A$, $w$ is in the exterior of $\gamma_1$ or in the interior of $\gamma_2$ (or both). The two cases are similar, so assume that $w$ is in a compact set of $A\setminus\gamma_1$, on the exterior of $\gamma_1$. Then we define $\tilde S_g(.,w)$ as we did for $w$ near $\gamma_2$. Then $K_g\tilde S_g(.,w)-\delta_w$ is supported on $\hat\gamma^W_1$, is $O(\delta\tilde\eta(\delta))$ there, and we may set:
$$S_g^\delta(.,w)=\tilde S_g(.,w)+\sum_{w'\in\hat\gamma^W_1}(K_g(\tilde S_g(.,w)))(w')S^\delta_g(.,w')$$
so that $K_g S_g^\delta(.,w)=\delta_w$. We know that there is a most one $f$ vanishing at infinity such that $K_gf=\delta_w$ (for $\delta$ small enough, under \eqref{eq:surgzeroind}). Thus $S^\delta_g(.,w)$ does not depend on the choice of $\gamma_1,\gamma_2$. By moving $\gamma_1$, $\gamma_2$ towards the boundary cycles of $A$, we can extend estimates of $S^\delta_g(b,w)$ for $(b,w)$ in a compact subset of $A^2\setminus\Delta_A$. 

Finally for $w$ outside of $A$, it is easy to construct $S_g^\delta(.,w)$. For instance if $w$ is in $D_o\setminus D_i$, one starts with a truncation of $K_o^{-1}(.,w)$ and set
$$S_g^\delta(.,w)=\ind_{\gamma^{int}}K_o^{-1}(.,w)-\sum_{w'\neq w}K_g(\ind_{\gamma^{int}}K_o^{-1}(.,w))S^\delta_g(.,w')$$
where $\gamma^{int}$ denotes the inside of the closed cycle $\gamma$. 

\paragraph{Conclusion.} 
Let us summarise the results of this section. Recall that $M$ is a bipartite graph derived from a rhombi tiling $\Lambda_\delta$ with edge length $\delta$, $\delta$ going to zero along some sequence; it is equipped with a linear operator $\rK:\C^{M_B}\rightarrow\C^{M_W}$. We consider $D_i$ (resp. $D_o$) a simply connected neighbourhood of $0$ (resp. $\infty$) in $\hat\C$, such that $A=D_i\cap D_o$ is an annulus separating $0$ from $\infty$. The (sequences of) graphs $\Xi_s=\Xi_s^\delta$, $s\in\{i,o\}$, are bipartite graphs equipped with $K_s:\C^{\Xi_s^B}\rightarrow\C^{\Xi_s^W}$, a nearest neighbour (or finite range) linear operator. The pair $(
\Xi_i,K_i)$ is obtained by modifying $(M,\rK)$ in $D_i\setminus D_o$, and vice versa for $(\Xi_o,K_o)$. In particular $(\Xi_i,K_i)$ and $(\Xi_o,K_o)$ agree with $(M,\rK)$ in $A$. The glued data $(\Xi_g,K_g)$ agrees with $(\Xi_s,K_s)$ in $D_s$, $s\in\{i,o\}$. We also assume that for any $w\in\Xi_i^W$, there is a unique $K_s^{-1}(.,w)\in\C^{\Xi^B_i}$ vanishing at infinity such that $K_s(K_s^{-1}(.,w))=\delta_w$, $s\in\{i,o\}$. 

\begin{Lem}\label{Lem:surgery}
Assume that $S_i,\bar S_i,S_o,\bar S_o:A^2\setminus \Delta_A\rightarrow\C$ are such that
$$K_s^{-1}(b,w)=\frac 12 R_B(e^{i\nu(w)}S_s(b,w))+
\frac 12\bar R_B(e^{-i\nu(w)}\bar S_s(b,w))+O(\eta(\delta))$$
uniformly in compact sets of $A^2\setminus\Delta_A$, $s\in\{i,o\}$, with $(z,w)\mapsto S_i(z,w)-\frac 1{\pi(z-w)}$ (resp. $\bar S_i(z,w)-\frac 1{\pi\overline{(z-w)}}$) $C^1$ in $A^2$, and $\lim_{\delta\searrow 0}\eta(\delta)=0$. Assume that there are $S_g,\bar S_g:A^2\setminus \Delta_A\rightarrow\C$ with same regularity such that: if $\gamma$ is a fixed simple cycle in $A$ disconnecting $0$ from $\infty$, then for $w\in A$ outside $\gamma$, $S_g(.,w)_{|\gamma}$ is in the inner Cauchy data space $C_i$ defined by $S_i$ and $(S_g(.,w)-S_o(.,w))_{|\gamma}$ is in the outer Cauchy data space $C_o$ defined by $S_o$ (see \eqref{eg:surgcontdata}); and the corresponding conditions for $\bar S_g$ and for $w\in A$ inside $\gamma$ also hold. Assume that $C_i\cap C_o=\{0\}$, $\bar C_i\cap\bar C_o=\{0\}$.

Then for $\delta$ small enough, for each $w\in\Xi_g^W$ there is a unique $K_g^{-1}(.,w)\in\C^{\Xi^B_g}$ vanishing at infinity such that $K_g(K_g^{-1}(.,w))=\delta_w$, and 
$$K_g^{-1}(b,w)=\frac 12 R_B(e^{i\nu(w)}S_g(b,w))+
\frac 12\bar R_B(e^{-i\nu(w)}\bar S_g(b,w))+O(\eta(\delta)+\delta|\log \delta|)$$
uniformly in compact sets of $A^2\setminus\Delta_A$, $s\in\{i,o\}$.
\end{Lem}

Let us remark that in the case where $\Xi_o$ is actually bounded, the ``vanishing at $\infty$" condition is void. Conceivably, for unbounded $\Xi_o$ (agreeing with $M$ far enough), one might find use for ``divisor" type boundary conditions at infinity: $f(z)=O(|z|^n)$ as $z\rightarrow\infty$, $n\in\Z$ fixed. In this case the previous surgery argument is still valid.

\section{Electric correlators}\label{sec:electr}

We are now interested in scaling limits of vertex correlators. For instance in the plane, one may consider asymptotics of
$$\langle e^{2i\pi\sum_js_j h(z_j)}\rangle$$
where $\sum s_j=0$. Here $h$ is the height field of a dimer configuration and $\langle\rangle$ is the expectation under the full-plane measure specified by \eqref{eq:locstats}. Heuristically, we expect these asymptotics to be governed by electric correlators for a free field 
$$\langle :\exp(i\sum_js _j\phi(z_j):\rangle$$
 (see \eqref{eq:electcorrGFF}), with coupling constant $g_0=\frac 12$ (the value read eg from Corollary \ref{Cor:FFplane}, with $2\pi h$ converging to $\phi$ in the small mesh limit). However, the discrete height function is at every point deterministic modulo $\Z$ (recall the choice of normalisation for the height from Section \ref{ss:height}). Thus
$$(s_j)_j\mapsto \langle e^{2i\pi\sum_js_j h(z_j)}\rangle$$
is $1$-periodic in each variable, which is not the case for the scalar free field electric correlators:
$$\langle :\exp(2i\pi\sum_js_j\phi(z_j):\rangle_\C
\propto \prod_{i<j}|z_i-z_j|^{2s_is_j}$$
This may be seen as a manifestation of the compactified nature of the height field.

The relevant Cauchy-Riemann operators are those associated to the line bundle $L_\rho$ over the punctured sphere $\Sigma=\hat\C\setminus\{z_1,\dots,z_n\}$, where $\rho:\pi(\Sigma)\rightarrow\U$ is a unitary character. The two types of variations we shall consider are the isomonodromic family $(z_1,\dots,z_n)\mapsto L_\rho(z_1,\dots,z_n)$, and Jacobian family $\rho\mapsto L_\rho$.

The analysis relies on a rather precise description of the corresponding discrete operators and their inverting kernels, in particular near the diagonal. 

In Section \ref{ssec:monodromy}, we study in some details discrete (harmonic and) holomorphic functions with monodromy around a given face of $M$ (in terms of the Riemann sphere, there is another singularity at infinity). Building on these results, we analyse in Section \ref{ssec:monoinv} the case of discrete holomorphic functions (and the associated inverting kernels) with monodromy around one or several pairs of points in $\C$ (and the point at infinity is regular). In Section \ref{ssec:monovariat}, we vary the position of singularities and the monodromy exponents in order to analyse asymptotically the electric correlators, yielding the main result (Theorem \ref{Thm:electr}).

For notational simplicity (and without loss of generality), throughout this section we take $\delta=1$.

\subsection{Discrete holomorphic functions with monodromy around a point}\label{ssec:monodromy}

We proceed with a local study of discrete holomorphic functions and inverting kernels in the presence of a singularity. As described in Section \ref{sss:critgraphs}, we consider a rhombus tiling $\Lambda$ of the plane. Let us mark the midpoint $v_0$ of an edge of $\Lambda$ (that is, the center of a face of $M$). Up to scaling and centering we may assume 
$$\delta=1{\rm,\ \ }v_0=0.$$

\begin{figure}[htb]
\begin{center}
\leavevmode
\includegraphics[width=0.4\textwidth]{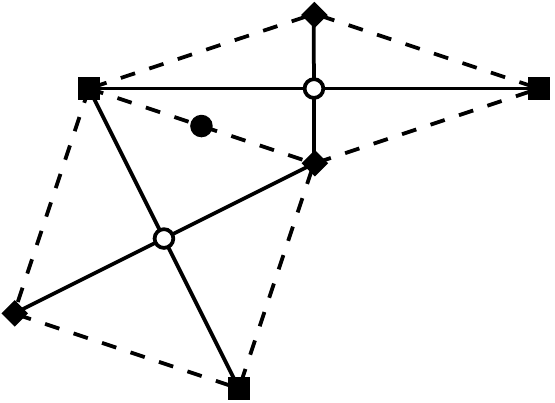}
\end{center}
\caption{Local geometry of $M$ near the singularity $v_0$ (black disk)}
\label{fig:monloc}
\end{figure}

The unitary characters of $\pi_1(\C\setminus\{v_0\})\simeq\Z$ are identified with the unit circle
$\U$. Fix a {\em non-trivial} character $\chi$ (identified to an element of $\U\neq\{1\}$). 

We consider $(\C^{M_B})_\chi$, the space of functions on the lift of $M_B$ to the universal cover of the punctured plane $\C\setminus\{v_0\}$ which belong to the character $\chi$ (recall Section \ref{ss:linebundles}, where similar constructions for the torus are discussed). Explicitly, if $\hat M$ is the lift of $M$ to the universal cover of $\C\setminus\{v_0\}$, and $\theta:M\rightarrow M$ is the deck transform corresponding to a simple counterclockwise loop around $v_0$,
$$\gls{CMBchi}\simeq\{f\in\C^{\hat M_B}:f\circ\theta=\chi f\}$$
We can choose the universal cover to be $\C$, the covering map to be $z\mapsto v_0+e^z$, and $\theta(z)=z+2i\pi$. Alternatively, we can choose a branch cut $\gamma$ (running from $v_0$ to $\infty$ on $M^\dg$) and realise $\hat M$ as $M\times\Z$ (as vertex sets); $M\times\{n\}$ is the $n$-th sheet, and one moves to the $(n\pm 1)$-th sheet when crossing $\gamma$. With this identification, $f\in(\C^{M_B})_\chi$ if $f((b,n+1))=\chi f((b,n))$.

Elements of $(\C^{M_B})_\chi$ may also be seen as multiplicatively multivalued elements of $\C^{M_B}$ (that get multiplied by $\chi$ when tracked along a counterclockwise loop around $v_0$). The space $(\C^{M_W})_\chi$ is defined similarly; we may also consider the restrictions $(\C^{M_V})_\chi$ and $(\C^{M_F})_\chi$. We are concerned with the operators
\begin{align*}
\Lap_\Gamma: &(\C^{M_V})_\chi\longrightarrow (\C^{M_V})_\chi\\
\rK:& (\C^{M_B})_\chi\longrightarrow (\C^{M_W})_\chi
\end{align*}
and functions $f$ s.t. $\Lap_\Gamma f=0$ or $\rK f=0$ on all or part of $M$ (discrete harmonic and holomorphic functions).

\subsubsection{Multivalued harmonic functions: {\em a priori} estimates}

We begin with a few basic estimates on harmonic functions with monodromy (i.e. multiplicatively multivalued). In what follows, 
$$B(0,R)=\{z\in M_V: |z|\leq R\}.$$

\begin{Lem}\label{Lem:chirharmest}
\begin{enumerate}
\item There is $\eps=\eps(\chi)>0$ such that for $n$ large enough, if $f\in(\C^{M_V})_\chi$ is harmonic in $B(0,2n)$, then
$$\sup_{x\in B(0,n)} |f(x)|\leq (1-\eps)\sup_{x\in \partial B(0,2n)}|f(x)|$$
Similarly,  if $f\in(\C^{M_V})_\chi$ is harmonic in $A(n,3n)=B(0,3n)\setminus B(0,n)$, then
$$\sup_{x\in A(\frac 32n,\frac 52n)} |f(x)|\leq (1-\eps)\sup_{x\in \partial A(n,3n)}|f(x)|$$
\item If $f\in (\C^{M_V})_\chi$ is  bounded and harmonic, $f\equiv 0$.
\item If $y\in M_V$, there is at most one function $\gls{Gchi}(.,y)\in (\C^{M_V})_\chi$ vanishing at infinity such that $\Lap_\Gamma G_\chi(.,y)=\delta_y$ (here $\delta_y$ designates an element of $(\C^{M_V})_\chi$ such that $\delta_y(y)\in\{\chi^n:n\in\Z\}$ and $\delta_y(x)=0$ otherwise).
\item  If $y\in M_V$, $G_\chi(.,y)$ exists and satisfies
$$G_\chi(x,y)=O(\left|\frac xy\right|^\eps\wedge \left|\frac yx\right|^\eps)$$
if $|x-y|\geq|y|/2$, for some $\eps=\eps(\chi)>0$.
\item For all $x,y\in M_V$, $G_\chi(x,y)=G_{\bar\chi}(y,x)$.
\item If $y\sim y'$ ($y,y'$ adjacent in $\Gamma\simeq M_V$),
$$G_\chi(x,y')-G_\chi(x,y)=O(|y|^{-1}(\left|\frac xy\right|^\eps\wedge \left|\frac yx\right|^\eps))$$
if $|x-y|\geq |y|/4$ and $G_\chi(x,y')-G_\chi(x,y)=O(|x-y|^{-1})$ otherwise.
\end{enumerate}
\end{Lem}
1. is an improved maximum principle; 2. is a Liouville-type result; 3,4. state the existence and uniqueness of a chiral Green kernel $G_\chi$; $C^0$ and $C^1$ estimates for $G_\chi$ are given in 4., 6.. We use $\wedge$ to denote the infix minimum:
$$a\gls{wedge} b\stackrel{def}{=}\min (a,b)$$

\begin{proof}
\begin{enumerate}
\item Take $x\in\partial B(0,n)$; up to rotation, we may assume $\arg(x)=0$. Consider 
$$C=A(\frac 23 n,\frac 43n)\cap\{z:|\arg(z)|\leq \pi-\eps_0\},$$ 
$\eps_0>0$ small enough and fixed. From the convergence of discrete harmonic measure (see \cite{SmiChe_isoradial}, Theorem 3.8), we may deduce that there exists $\eta>0$ such that for $n$ large enough, the probability that the random walk on $\Gamma$ started from $x$ exits $C$ on either the top or bottom side of the cone $\{z:|\arg(z)|\leq \pi-\eps_0\}$ is at least $\eta$; let us denote $T$ and $B$ these events. If $y$  is a point on the top side, a random walk starting from $y$ disconnects the bottom side from $\partial B(-x,\frac n2)$ before exiting $B(-x,\frac n2)$ with probability at least $\eta'>0$ (uniformly in $n$ large enough, $y\in A(\frac 23 n,\frac 43n)$; this may also be seen easily using harmonic measure estimates).

Hence with probability at least $\eta'>0$, we may couple the random walk started from $x$ conditional on $T$ with the random walk conditional on $B$ in such a way that they couple before exiting $B(0,2n)$ and their winding around 0 differs by $2\pi$. This shows that
$$|f(x)|\leq \left(\eta\eta'|1+\chi|+(1-2\eta\eta')\right)\sup_{x\in \partial B(0,2n)}|f(x)|$$
and since $\chi\neq 1$, we have $|1+\chi|<2$. The same argument works in the annular case.

\item By iterating 1., we get 
$$|f(x)|\leq (1-\eps)^n\sup_{x\in \partial B(0,|x|2^n)}|f(x)|)\leq (1-\eps)^n\|f\|_\infty$$ 
and consequently $f(x)=0$.

\item Follows from 2.

\item 
If $y\in M_V$, set $h_y(x)=\E_x(\delta_y(X_{\tau_y}))$ where is a $\chi$-multivalued Dirac mass at $y$ and $\tau_y$ is the first time $y$ is attained. Clearly $h_y$ is harmonic on $M_V\setminus\{y\}$, $\chi$-multivalued and bounded by $1$. Since it is not identically $0$, by 2. $\Lap_\Gamma h_y(y)\neq 0$. Thus we may set
$$G_\chi(.,y)=\frac{h_y}{\Lap_\Gamma h_y(y)}$$
Due to the denominator, it is delicate to estimate $G_\chi$ using this representation. We are going to give another representation, based on a Markovian decomposition.

Up to rotation we may assume $\arg(y)=0$. Let $C_e=\partial B(y,|y|/2)$ and $C_i=\partial B(y,|y|/4)$. For $x\in C_e$, $y\in C_i$, let $\mu_e(y,\{x\})$ be the harmonic measure on $C_e$ seen from $y\in C_i$ and $\mu_i(x,\{y\})$ be the chiral harmonic measure on $C_i$ seen from $x\in C_e$. Explicitly, if $C^n_i$ (resp. $y_n$) is the lift of $C_i$ (resp. $y$) to the $n$-th sheet of the universal cover of $\C\setminus\{0\}$, 
$$\mu_i(x,\{y\})=\sum_{m\in \Z}\chi^m {\rm Harm}_{\sqcup C_i^n}(x,\{y_m\})$$
where ${\rm Harm}_I(x,J)$ is the probability that the random walk started from $x$ on $\hat M_V$ (the lift of $M_V$ to the universal cover of $\C\setminus\{0\}$) first hits $I$ on $J\subset I$.

Reasoning as in 1., we can couple a random walk starting from $x$ which approaches the crosscut $(-\infty,0)$ from above before reaching $C_i$ with a random walk approaching it from below in such a way that, with probability bounded away from 0, if the lift of the first random walk exits at $y_m$, the second one exits at $y_{m-1}$. This shows that $\|\mu_i(x,.)\|_{TV}\leq 1-\eps$ for some constant $\eps>0$ ($\|.\|_{TV}$ denotes the total variation norm).

Consider the (continuous-time) random walk $(X_t)_{t\geq 0}$ on $M_V$ (lifted to the universal cover) started from, say, $x\in B(y,|y|/4)$; let $\tau_e$ be the first time the RW touches $C_e$, and $\tau_i$ be the first time it touches $\sqcup_n C_i^n$ after $\tau_e$. Then by Dynkin's formula
\begin{align*}
G_\chi(x,y)&=\E^x\left(G_\chi(X_{\tau_i},y)+\int_0^{\tau_i}\ind_{X_t=y}dt\right)\\
&=\E^x\left(\int_0^{\tau_e}\ind_{X_t=y}dt\right)+\E^x\left(\E^x\left(G_\chi(X_{\tau_i},y)|X_{\tau_e}\right)\right)\\
&=G_{B(y,|y|/2)}(x,y)+\sum_{a\in C_e,b\in C_i}\mu_e(x,\{a\})\mu_i(a,\{b\})G_\chi(b,y)
\end{align*}
Here $G_B$ is the Green kernel for the RW on $B$ with Dirichlet boundary conditions on $\partial B$ (standard arguments show that $\E^x(\tau_e)<\infty$). 

Denoting by $T:L^\infty(C_i)\rightarrow L^\infty(C_i)$ the operator:
$$(Tf)(x)=\sum_{a\in C_e,b\in C_i}\mu_e(x,\{a\})\mu_i(a,\{b\})f(b)$$
we have $\lVert T\rVert_{L^\infty\rightarrow L^\infty}\leq 1-\eps$ since $\|\mu_i(x,.)\|_{TV}\leq 1-\eps$ for all $x$. Consequently, on $C_i$ we have
$$G_\chi(.,y)=(\Id-T)^{-1}G_{B(y,|y|/2)}(.,y)_{|C_i}$$
The values of $G_\chi(.,y)$ on $C_i$ determine its values elsewhere: by harmonic extension outside of $C_i$, and by harmonic extension inside $C_i$ after substracting $G_{B(y,|y|/4)}(.,y)$. 

Having expressed $G_\chi$ in terms of the Green kernel in balls, we may used known estimates for the latter. We have $G_{B(y,R)}(x,y)=O(1)$ for $x\in B(y,R)\setminus B(y,R/2)$
and $ G_{B(y,R)}(x,y)=O(\log R)$ for $|x-y|\ll R$ (see Sections 2.2. and 3.3 in \cite{SmiChe_isoradial}).
Thus $G_\chi(x,y)=O(1)$ for $x\in C_i$. Consequently, by 1. we have 
$$G_\chi(x,y)=O(\left|\frac xy\right|^\eps\wedge \left|\frac yx\right|^\eps)
$$
if $|x-y|\leq |y|/2$.

\item By uniqueness 3., $G_\chi(x,y)$ is $\bar\chi$-multivalued in $y$ for $x$ fixed (reasoning on the universal cover). To evaluate
$$\Lap^y_\Gamma G_\chi(x,y)$$
we fix $y$ and consider it as a $\chi$-multivalued function in $x$. Then observe that it decays at infinity and has the same Laplacian as a lifted Dirac mass $\delta_y(.)$. Consequently by 2., 
$$\Lap^y_\Gamma G_\chi(x,y)=\delta_y(x)=\delta_x(y)$$
which concludes.

\item We write $G_\chi(x,y')-G_\chi(x,y)=G_{\bar\chi}(y',x)-G_{\bar\chi}(y,x)$. If $|x-y|\geq |y|/2$, a combination of the estimate in 4. (for $G_{\bar\chi})$ and a discrete Harnack estimate (see \eqref{eq:harnack}) concludes. If $x\in B(y,|y|/4)$, as in 4. we may write 
$$G_{\bar\chi}(y,x)=G_\Gamma(y,x)-\frac 1{2\pi}\log|x|+h(y)$$
where $G_\Gamma$ is the ``free Green function" (e.g. Theorem 2.5 in \cite{SmiChe_isoradial}, Theorem 2.5) and $h$ is harmonic in $B(x,|x|/2)$ and uniformly bounded. We conclude with the Harnack estimate \eqref{eq:harnack} and asymptotics for first differences of $G_\Gamma$.

\end{enumerate}
\end{proof}

\subsubsection{Classification and asymptotic behaviour}

We now turn to discrete holomorphic functions. We are interested in characterising and estimating discrete holomorphic functions in $(\C^{M_B})_\chi$. More precisely, we shall describe bounded discrete holomorphic functions in $(\C^{M_B})_\chi$ (Lemma \ref{Lem:monholom}); discrete meromorphic functions with a single pole adjacent to the singularity (Lemma \ref{Lem:monpole}); and the expansion at infinity of $\chi$-multivalued discrete holomorphic functions (Lemma \ref{Lem:moninf}).

\paragraph{Classification.}

Let us start with reminding a few constructive elements on discrete holomorphic functions in the absence of monodromy. The bounded holomorphic functions in $\C^{M_B}$ are spanned (over $\C$) by $R_B(1)$, $\bar R_B(1)$. Then one verifies directly that $z\mapsto R_B(z)$, $z\mapsto R_B(z^2)$ and their conjugates are also discrete holomorphic. 

There is a notion of integration of discrete holomorphic functions (see \cite{Merc_poly} and references therein) which allows to construct by induction a sequence $(P_n)_{n\geq 0}$ of discrete holomorphic functions in $\C^{M_B}$ s.t. $P_n(z)=R_B(z^n)(1+o(1))$ as $z\rightarrow\infty$. The special discrete holomorphic functions (see the proof of Lemma \ref{Lem:monholom}) may be seen as the (exponential) generating series of the sequence $(P_n)$. Furthermore, one may check that a discrete holomorphic function with at most polynomial growth is a linear combination of the $P_n$'s and their conjugates. 

Using special holomorphic functions and an integral representation, Kenyon obtained Theorem \ref{Thm:Kencrit}, which shows e.g. that one can find a function $f$ in $\C^{M_B}$ which is discrete holomorphic except at two points s.t. $f(z)=R_B(z^{-1})(1+o(1))$ as $z\rightarrow\infty$.

In the presence of monodromy, if we set $\chi=e^{2i\pi s}$, then $\chi$-multivalued (continuous) holomorphic functions in $\C\setminus\{0\}$ may be written as
$$z\longmapsto z^s\sum_{n\in\Z}a_nz^n$$
by standard results on Laurent series. Similarly, antiholomorphic $\chi$-multivalued functions may be written
$$z\longmapsto \bar z^{-s}\sum_{n\in\Z}b_n\bar z^n$$
By analogy with the discrete polynomials $P_n$, a natural question is whether there are basic discrete functions given asymptotically by $z\mapsto R_B(z^{s+n})(1+o(1))$. This is addressed by the following lemma.

\begin{Lem}\label{Lem:monholom}
\begin{enumerate}
\item
The space of $\chi$-multivalued bounded discrete holomorphic functions:
$$\{f\in(\C^{M_B})_\chi: \|f\|_\infty<\infty, \rK f=0\}$$
is one-dimensional and is spanned by a function $f_\chi$ with the following asymptotic expansion:
$$f_\chi(u)=R_B\left(2^{s-1}\Gamma(1-s)(u-v_0)^{s-1}\right)+\bar R_B\left(\bar\tau 2^{-s}\Gamma(s)\overline{(u-v_0)}^{-s}\right)+O(|u-v_0|^{-s-2}+|u-v_0|^{s-3})$$
where $s\in (0,1)$, $\chi=e^{2i\pi s}$, and $\tau=(x_0-v_0)/|x_0-v_0|$ if $x_0$ is the vertex of $\Gamma$ adjacent to the singularity $v_0$.
\item More generally, for $s\in (0,1)$, $k\in\N$, there exists $f_{k,\chi}$ discrete holomorphic and $\chi$-multivalued s.t.
$$\gls{fkchi}(u)=\bar R_B\left(\bar\tau 2^{k-s}\Gamma(s-k)\overline{(u-v_0)}^{k-s}\right)+O(|u-v_0|^{s-k-1}+|u-v_0|^{k-s-2})$$
with $f_{0,\chi}=f_\chi$ and for $b$ a black vertex adjacent to the singularity,
$$f_{k,\chi}(b)=\frac{2i\pi e^{i\nu(b)}}{(b-v_0)^{k+1-s}}\frac{2^{s-k-1}(-1)^ke^{-i\pi s}}{1-e^{-2i\pi s}}=(-1)^ke^{i\nu(b)}\frac{\pi}{\sin(\pi s)}\overline{(2(b-v_0))}^{k+1-s}$$
\end{enumerate}
\end{Lem}
Remark that $\overline{f_{k,\bar\chi}}$ is also discrete holomorphic and $\chi$-multivalued, and that $\overline{f_{\bar\chi}}=\tau f_\chi$. 
Let us also point out that if $b$ is a black vertex adjacent to the singularity,
$$f_{k,\chi}(b)=\frac{(-1)^k\pi}{\Gamma(s-k)\sin(\pi s)}\bar R_B\left(\bar\tau 2^{k-s}\Gamma(s-k)\overline{(b-v_0)}^{k-s}\right)$$
ie plugging $b$ in the leading term of the asymptotic expansion of $f_{k,\chi}$ yields the correct value up to the positive multiplicative constant $\frac{(-1)^k\pi}{\Gamma(s-k)\sin(\pi s)}=\Gamma(k+1-s)$.

\begin{proof}
\begin{enumerate}
\item
Let $f\in(\C^{M_B})_\chi$ be bounded with $\rK f=0$. Then a local computation (see Section \ref{sss:finitediff}) shows that $f_{|M_V}$ is harmonic except at $x_0$, the vertex of $\Gamma$ adjacent to the marked edge of $\Lambda$. Consequently (Lemma \ref{Lem:chirharmest}), $f_{|M_V}$ decays at infinity and is proportional to $G_\chi(.,x_0)$. Besides $f_{|M_F}$ is locally given as the harmonic conjugate of $f_{|M_F}$ and is thus given modulo a global additive constant (on the lift to the universal cover). This additive constant is determined by the fact that $f_{|M_F}$ is $\chi$-multivalued. Thus the space of bounded holomorphic functions in $(\C^{M_B})_\chi$ is at most one-dimensional.

We then simply need to exhibit a non-trivial holomorphic function. Adapting the argument for the construction of $\uK^{-1}$ in Section 4.2 of \cite{Ken_isoradial} (see also \cite{SmiChe_isoradial}, Appendix A, which we follow more closely), we use an integral representation involving the ``special" discrete holomorphic functions (or {\em discrete exponentials}). Recall that the singularity is located at $v_0=0$, the midpoint of an edge of $\Lambda$ which abuts $x_0\in\Gamma$ and $y_0\in\Gamma^\dg$ (thus $x_0=-y_0$). For a parameter $\lambda\in\C$, we set
\begin{align*}
e_\lambda(x_0)&=\frac 1{1-\frac \lambda 2(x_0-y_0)}\\
e_\lambda(y_0)&=\frac 1{1-\frac \lambda 2(y_0-x_0)}
\end{align*}
 and for $x\in\Gamma$, $y\in\Gamma^\dg$ adjacent, 
 $$\frac{e_\lambda(x)}{e_\lambda(y)}=\frac {1-\frac \lambda 2(y-x)}
{1-\frac \lambda 2(x-y)}$$
It may be checked that $e_\lambda\in\C^{M_B}$ is well-defined and discrete holomorphic in the sense that $K e_\lambda=0$ (recall that $\rK$ is obtained from $K$ by a gauge change). For $u\in M_B$, we have $e_\lambda(u)=O(1/|\lambda|)$ as $\lambda\rightarrow\infty$. Besides, one can choose a path $(u_0\dots u_n)$ on $\Lambda$ from $u_0=x_0$ or $y_0$ to $u_n=u$ in such a way that $u_0,u_1-u_0,\dots,u_{n}-u_{n-1}$ lie in the cone:
$$\{z:|\arg(z)-\arg(u)|\leq \pi-\eps_0\}$$
for some $\eps_0>0$. Consequently the poles of $e_\lambda(u)$ are in $\{\lambda:|\arg(\lambda)-\arg(\bar u)|\leq\pi-\eps_0\}$ (and have norm $\delta/2$).

For $s\in(0,1)$, set
$$f_s(u)=\int_0^{\bar u\infty}\lambda^{-s} e_{-\lambda}(u)d\lambda$$
which is $\chi=e^{2i\pi s}$-multivalued. Here $\int_0^{\bar u\infty}$ represents integration along a half-line from 0 to infinity in the direction $\bar u$. If $w\in M_W$, $u_1,\dots,u_4$ its four black neighbours, one can show that the same integration path may be used for $u_1,\dots,u_4$ and consequently by linearity $(K f_s)(w)=0$. 

For small $\lambda$ we have
$$e_\lambda(u)=\exp(\lambda u+O(|u\lambda^3|+|\lambda^2|))$$
and for large $\lambda$
$$e_\lambda(u)=-\frac{\epsilon}{\lambda(x_0-v_0)}\exp(4\frac{\bar u}{\lambda}+O(|u\lambda^{-3}|+|\lambda^{-2}|)$$
where $\epsilon=1$ on $\Gamma$ and $\epsilon=-1$ on $\Gamma^\dg$ (see A.1 in  \cite{SmiChe_isoradial}). (Notice that $\frac{1+\eps}{1-\eps}=\exp(2\eps+O(\eps^3))$). The asymptotic expansion comes from small values of $\lambda$ (say $|\lambda|\leq1/\sqrt{|u|}$) and large values (say $|\lambda|\geq\sqrt{|u|}$). The intermediate values of $\lambda$ contribute to an exponentially small error in the asymptotic analysis. We have
\begin{align*}
\int_0^{\bar u\infty}\lambda^{-s} \exp(-\lambda u)d\lambda&=\Gamma(1-s)u^{s-1}\\
\int_0^{\bar u\infty}\lambda^{-s}\left(\frac{\epsilon}{\lambda(x_0-v_0)} \exp(-4\bar u\lambda^{-1})\right)d\lambda
&=\frac\epsilon{x_0-v_0}\Gamma(s)(4\bar u)^{-s}
\end{align*}
The error terms are estimated via
$$\int_0^{1/\sqrt{|u|}}t^{-s}\exp(-|u|t)O(|u|t^3+t^2)dt=\int_0^{\sqrt{|u|}}t^{-s}\exp(-t)O(t^3+t^2)\frac{dt}{|u|^{3-s}}=O(|u|^{s-3})$$
and similarly for the other integral. This shows in particular that $f_s$ is not identically zero. We then obtain $f_\chi$ by multiplying by a constant and applying a gauge transformation (recall \eqref{eq:Kastgauge}).

\item For $k\in\N$, $s\in (0,1)$, set 
$$f_{k,s}(u)=\int_\gamma \lambda^{k-s} e_{-\lambda}(u)d\lambda$$
where $\gamma=\gamma(u)$ is an integration contour consisting of a large counterclockwise circle around zero (encircling all poles of the $e_.(u)$'s); two connecting segments on $\bar u(0,\infty)$; and a small clockwise circle around 0 (with all poles on its exterior). As before it may be checked that $Kf_{k,s}=0$, since for a given $w\in M_W$, the same integration contour may be used for all black neighbours of $w$.\\
We choose a lift of the mapping $u\mapsto\lambda=-u^{-1}$ to the covers of $\{u:u\neq v_0\}$ and $\{\lambda:\lambda\neq 0\}$. We choose determinations of $\log$ on these pointed covers in such a way that for $\lambda=u^{-1}$, $\log(\lambda)=-\log(u)$.
Then $\gamma(u)$ is a simple contour on the cover of $\{\lambda:\lambda\neq 0\}$, for instance by specifying that the outward segment is in the direction $u^{-1}$; then the inward segment is on the next sheet of the cover.  

Notice that, up to rotating the lattice around $v_0$, one may assume that $u-v_0\in (0,\infty)$, in which case one may a fixed branch of $\log$, which will be convenient for local computations. More precisely, if $M'=e^{i\theta}M$ (the graph rotated by $\theta$), one may choose branches in the definitions of $f^{M}_{k,s}$, $f^{M'}_{k,s}$ in such a way that $f^{M'}_{k,s}(e^{i\theta}u)=e^{is\theta}f^{M}_{k,s}(u)$. 

In order to obtain an asymptotic expansion for large $R=|u|$, one may take the outer circle of the integration contour with radius $R$ and the inner circle with radius $R^{-1}$.  The contribution of the inner half of the contour $\gamma$ is, modulo exponentially small error terms, 
$$R^{s-k-1}\int_{\gamma_i}\lambda^{k-s}\exp(-\lambda u/R+O(u\lambda^2/R^2))d\lambda=O(|u|^{s-k-1})$$
where $\gamma_i$ is a contour consisting of a clockwise unit circle around 0 connected to infinity by two half-lines along $\bar u(0,\infty)$.
Similarly, the contribution of the outer half of $\gamma$ is:
$$R^{(k-s)}\int_{\gamma_o}\lambda^{k-s}\left(\frac\epsilon{\lambda (x_0-v_0)} \exp(-4\bar u(R\lambda)^{-1}+O(u(R\lambda)^{-3}+(R\lambda)^{-2})\right)d\lambda$$
where $\gamma_o$ consists of the unit circle connected to zero by two rays along $\bar u(0,\infty)$.\\
Up to rotation and inversion, we have to evaluate 
$$\int_{\gamma_0}t^{s'-1}e^{-t}dt$$
where $\gamma_0$ is a contour consisting of a clockwise circle around 0 connected to infinity by two half-lines along $(0,\infty)$. (For definiteness, set $\log(re^{i\theta})=\log r+i\theta$ for $\theta\in [0,2\pi)$, and accordingly $t^{s'-1}=\exp((s'-1)\log(t))$). We observe that, by contour deformation, 
$$\int_{\gamma_0}t^{s'-1}e^{-t}dt=(1-e^{2i\pi s'})\Gamma(s')$$
when $s'\in(0,\infty)$; and that both sides are analytic in $s'$ (the RHS has removable singularities on $-\N$). Consequently they agree for all $s'\in\C$. Thus
$$f_{k,s}(u)=\epsilon\frac{(1-\chi^{-1})}{x_0-v_0}\Gamma(s-k)(4\bar u)^{k-s}+O(|u|^{k-s-2}+|u|^{s-k-1})$$
On the other hand, for $u\in\{x_0,y_0\}$, the residue formula yields:
$$f_{k,s}(u)=\int_\gamma\lambda^{k-s} \frac{1}{1+\lambda(u-v_0)}d\lambda=2i\pi e^{i\pi(k-s)}(u-v_0)^{s-k-1}
$$

\end{enumerate}
\end{proof}

\paragraph{Discrete meromorphic functions with a pole adjacent to the singularity.\\}

We proceed with another constructive result on discrete holomorphic/meromorphic functions with monodromy, where we now allow a single ``pole" at one of the white vertices adjacent to the singularity (see Figure \ref{fig:monloc}). The argument is based on special discrete holomorphic functions, as in Lemma \ref{Lem:monholom}. A possible alternative route involves displacing the singularity across an edge of $M$, see the discussion after Lemma \ref{Lem:monCauchy}.

\begin{Lem}\label{Lem:monpole}
Let $w\in M_W$ be one of the two white vertices adjacent to the singularity $v_0$. Then there is $\gls{gw}\in (\C^{M_B})_\chi$ which is discrete holomorphic except at $w$, has asymptotic expansion
$$g_w(u)=R_B\left(2^{s}e^{i\nu(w)}\Gamma(1-s)(u-v_0)^{s-1}\right)+\bar R_B\left(2^{-s}e^{-i\nu(w)}\Gamma(1+s)\overline{(u-v_0)}^{-s-1}\right)+O(|u|^{-s-3}+|u|^{s-3})$$
and such that
\begin{align*}
(\rK g_w)(w)&=2^{1+s}\pi(w-v_0)^{s}
\end{align*}
\end{Lem}
Remark that for $s=0$, $g_w=2\pi\uK^{-1}(.,w)$; compare with Theorem \ref{Thm:Kencrit}.
\begin{proof}
We proceed as in the construction of $f_\chi$, simply changing the normalisation of the discrete exponentials (or rather reverting to the normalisation used in \cite{Ken_isoradial}, \cite{SmiChe_isoradial}; again we follow closely Appendix A of \cite{SmiChe_isoradial} here). Notations are as in the proof of Lemma \ref{Lem:monholom}. Let $x_0\in\Gamma$, $y_0\in\Gamma^\dg$ be the two black vertices adjacent to the singularity; let $x_0'\in\Gamma$, $y_0'\in\Gamma^\dg$ be such that $w$ corresponds to the edges $(x_0x_0')$ and $(y_0y_0')$ (see FIgure \ref{fig:monloc}). For $\lambda\in\C$, set
$$e_\lambda(x_0)=\frac 1{(1-\frac\lambda 2(x_0-y_0))(1-\frac\lambda 2(x_0-y_0'))}$$
and for $x\in\Gamma$, $y\in\Gamma^\dg$ adjacent, 
 $$\frac{e_\lambda(x)}{e_\lambda(y)}=\frac {1-\frac \lambda 2(y-x)}
{1-\frac \lambda 2(x-y)}$$
Then $K e_\lambda=0$ and we have the following estimates:
\begin{align*}
e_\lambda(u)&=\exp(\lambda (u-w)+O(|u\lambda^3|+|\lambda^2|))\\
e_\lambda(u)&=\frac{4\epsilon}{(x_0-y_0)(x_0-y_0')\lambda^2}\exp(4\frac{\overline{u-w}}\lambda+O(|u\lambda^{-3}|+|\lambda^{-2}|)
\end{align*}
For notational simplicity, we now assume that $w=0$.

Fix a ray $\bar u_0(0,\infty)$ from $w=0$ to infinity which is disjoint of the rays $\bar u(0,\infty)$ for all $u\in M$, and such that $\arg(u_0-v_0)\in(-\pi,\pi)$. We also fix a branch of $\log$ in $\C\setminus \bar u_0[0,\infty)$ and set $z^s=\exp(s\log(z))$ there. By rotating the lattice we may assume that $u_0=-1$ and $\log$ is the usual branch in $\C\setminus (-\infty,0]$ (so that $(z^{-1})^s=z^{-s}$ outside of the branch cut).

Then set
$$\tilde g(u)=\int_{-\bar u\infty}^0(-\lambda)^{-s}e_\lambda(u)d\lambda$$
which defines a single valued function on $M_B$. It can be identified with an element of $(\C^{M_B})_\chi$ by taking $u_0(0,\infty)$ as a branch cut. As before (see Lemma \ref{Lem:monholom}), one then checks that $K\tilde g$ vanishes except at $w$, by deformation of the integration path.

We have
\begin{align*}
\int_{-\bar u\infty}^0(-\lambda)^{-s} \exp(\lambda u)d\lambda&=\Gamma(1-s)u^{s-1}\\
\int_{-\bar u\infty}^0(-\lambda)^{-s}\left(\frac{4\epsilon}{(x_0-y_0)(x_0-y_0')\lambda^2} \exp(4\bar u\lambda^{-1})\right)d\lambda
&=\frac {4\epsilon}{(x_0-y_0)(x_0-y_0')}\Gamma(1+s)(4\bar u)^{-s-1}
\end{align*}
Observe that $(y_0-x_0)(y_0'-x_0)=e^{2i\nu(w)}$ (if $\delta=1$). We need to evaluate $\tilde g$ at neighbors of $w$. 
In order to evaluate $\int_{-\bar u\infty}^0(-\lambda)^{-s}f(\lambda)d\lambda$, where $f$ is rational of degree $-2$ without poles on the integration path, we may write (here $\chi=e^{2i\pi s}$)
$$(\chi^{-1}-1)\int_{-\bar u\infty}^0(-\lambda)^{-s}f(\lambda)d\lambda=\oint_\gamma\widetilde{(-\lambda)^{-s}}f(\lambda)d\lambda$$
where $\gamma$ is a closed counterclockwise contour containing all poles of $f$ and not crossing the ray $-\bar u(0,\infty)$ and $\lambda\mapsto \widetilde{(-\lambda)^{-s}}$ is the determination of $\lambda\mapsto {(-\lambda)^{-s}}$ with a branch cut on $-\bar u(0,\infty)$ which agrees with the reference determination (i.e. the one with branch cut on the negative half-line) on the right handside of $-\bar u(\infty,0)$. Specifying to $f(\lambda)=\frac{1}{(1-\lambda z_1)(1-\lambda z_2)}$, the residue formula yields
$$\int_{-\bar u\infty}^0 (-\lambda)^s\frac{1}{(1-\lambda z_1)(1-\lambda z_2)}d\lambda=\frac{2i\pi}{\chi^{-1}-1}\sum_z\widetilde{(-z)^{-s}}\Res_z(f)=\frac{2i\pi}{(1-\chi^{-1})(z_2-z_1)}\left(\widetilde{(-z_2^{-1})^{-s}}-\widetilde{(-z_1^{-1})^{-s}}\right)$$
Let us denote $v_0,v_1,v_2,v_3$ the midpoints of edges of the rhombus of $\Lambda$ corresponding to $w$ (where $v_0$ is the midpoint of $(x_0y_0)$) listed counterclockwise, and also $(b_0,b_1,b_2,b_3)=(x_0,y_0,x_0',y_0')$. We have
$$e_\lambda(b_i)=\frac{1}{(1-\lambda v_{i-1})(1-\lambda v_{i})}$$
for $i=0,\dots,3$, with cyclical indexing. 
For definiteness, let us assume that $u_0$ was chosen in the interior of the cone generated by $x_0$ and $v_0$. Then we obtain (in terms of the reference determination)
\begin{align*}
\tilde g(b_0)&=\frac{2i\pi}{(1-\chi^{-1})(v_0-v_3)}\left((-v_0)^{s}-(-v_3)^{s}\right)\\
\tilde g(b_1)&=\frac{2i\pi}{(1-\chi^{-1})(v_1-v_0)}\left(\chi^{-1}(-v_1)^{s}-\chi^{-1}{(-v_0)^{s}}\right)\\
\tilde g(b_2)&=\frac{2i\pi}{(1-\chi^{-1})(v_2-v_1)}\left((-v_2)^{s}-\chi^{-1}{(-v_1)^{s}}\right)\\
\tilde g(b_3)&=\frac{2i\pi}{(1-\chi^{-1})(v_3-v_2)}\left((-v_3)^{s}-{(-v_2)^{s}}\right)
\end{align*}
Taking into account $K(w,b_j)=i(v_{j-1}-v_{j})$, we get
$$(K\tilde g)(w)=2\pi(-v_0)^s$$
In order to obtain $g_w$ we multiply $\tilde g$ by $2^s e^{i\nu(w)}$ and apply a gauge transformation (see \eqref{eq:Kastgauge}).
\end{proof}

\paragraph{Discrete holomorphic functions with monodromy: asymptotic behaviour at infinity.\\}

Consider a bounded discrete function $f\in\C^{M_B}$ which is (discrete) holomorphic in $B(0,r)^c$ ($r\gg 1$ a large radius); from harmonic function arguments one sees that
$$f(z)=R_B(a)+\bar R_B(a')+O((r/z))$$
for some $a,a'\in\C$. By Theorem \ref{Thm:Kencrit} and a Cauchy integral formula 
(e.g. \eqref{eq:dirprobcauchy}) it is easy to see that
$$f(z)=R_B(a+b\frac{r}z)+\bar R_B(a'+b'\frac{r}{\bar z})+O((r/z)^2)$$
and one could expand at higher order terms with additional arguments. In the following Lemma, we identify the leading and subleading term at infinity for a discrete holomorphic function with monodromy.

\begin{Lem}\label{Lem:moninf}
Let $\chi=e^{2i\pi s}$, $s\in (0,\frac 12)$. Let $f\in (\C^{M_B})_\chi$ be bounded and s.t. $\rK f=0$ outside of $B(0,R)$. Then there exists $c=c(f)$ such that for $\eps_0>0$, $|z|\geq R$,
$$f(z)=cf_\chi(z)+O(\|f_{|B(0,R)}\|_\infty |z/R|^{s-1+\eps_0})$$
More precisely, for $s\in (0,\frac 12]$, there exists $c_1(f)=O(\|f_{|B(0,R)}\|_\infty R^{s})$, $c_2(f)=O(\|f_{|B(0,R)}\|_\infty R^{1-s})$ such that for $\eps_0>0$, $|z|\geq R$,
\begin{align*}
f(z)&=c_1f_\chi(z)+O(\|f_{|B(0,R)}\|_\infty|z/R|^{s-1})\\
&=c_1f_\chi(z)+c_2g_\chi(z)+O(\|f_{|B(0,R)}\|_\infty|z/R|^{-s-1+\eps_0})
\end{align*}
where $g_\chi=g_w$ for a white vertex $w\in M_W$ adjacent to $v_0$.
\end{Lem}
\begin{proof}
Fix $\eps,\eta>0$. Then there exists $R_0>0$ such that for all $g\in (\C^{M_V})_\chi$ discrete harmonic outside of $B(0,R_0)$, there exists a continuous harmonic, $\chi$-multivalued function $\hat g$ s.t. for all $z\in M_V$, $|z|\geq (1+\eta)R_0$
\begin{equation}\label{eq:Lemmoninf1}
|g-\hat g|(z)\leq \eps \|g_{|{\partial B(0,R_0)}}\|_\infty
\end{equation}
Moreover $R_0$ is uniform in the choice of rhombus tiling $\Lambda$ 
under \eqref{eq:spade}. This is obtained by contradiction using equicontinuity and Ascoli-Arzel\`a (see e.g. Section 3.1 in \cite{SmiChe_isoradial} for similar arguments). Similarly, if $g$ is discrete holomorphic, then $g_{|M_F}$ converges to a continuous harmonic function $\hat g^\dg$, which is conjugate to $\hat g$.\\

Let $\hat g$ be bounded, harmonic, $\chi$-multivalued on $\{z:|z|>1\}$. Then $\hat g$ can be written locally as the sum of a holomorphic function $h$ and an antiholomorphic function $\tilde h$. The decomposition $\hat g=h+\tilde h$ is unique up to additive constants. It is easy to see that the additive constant can be set so that $h,\tilde h$ are $\chi$-multivalued. From the Laurent expansion of $h(z)z^{-s}$, $\tilde h(z)\bar z^{s-1}$, we get that
$$\hat g(z)=\sum_{n\geq 1}a_n z^{s-n}+\sum_{n\geq 0}b_n \bar z^{-s-n}$$
for coefficients $(a_n)$, $(b_n)$ which are bounded by $\|\hat g_{\|\partial B(0,1)}\|_\infty$. Thus for all $|z|\geq C>1$,
$$|\hat g(z)-b_0\bar z^{-s}|\leq |z|^{s-1}\frac {C(1+C^{-2s})}{C-1}\|\hat g_{|\partial B(0,1)}\|_\infty$$

Let us now consider $f$ (restricted to $M_V$ for notational simplicity) satisfying the assumptions of the lemma. Let $\hat g_n$ be, as in \eqref{eq:Lemmoninf1}, a continuous, harmonic such that for $|z|\geq (1+\eta)R_n$
$$|f-\hat g_n|(z)\leq \eps\|f_{\partial B(0,R_n)}\|_\infty$$
($\eps,R_n$ to be fixed later). There is $c>0$, $\alpha_n=O(\|g_{|\partial B(0,R_n)}\|_\infty R^s)$ such that for $C>1$ large enough, $|z|\geq CR_n$
$$|\hat g_n(z)-\alpha_n \bar z^{-s}|\leq |z|^{s-1}(1+cC^{-1})\|\hat g_{|\partial B(0,(1+\eta)R_n)}\| \leq |z/R_n|^{s-1}(1+cC^{-1})(1+\eps)\|f_{|\partial B(0,R_n)}\|_\infty$$
Given that $f_\chi(z)=\overline R_B(\bar z^{-s})+O(|z|^{s-1})$ (up to multiplicative constant, see Lemma \ref{Lem:monholom}), we get that for $z\geq CR_n$,
$$|f-\alpha_nf_\chi|(z)\leq (\eps+|z/R_n|^{s-1}(1+cC^{-1})(1+\eps)(1+cR_n^{2s-1}))\|f_{|\partial B(0,R_n)}\|_\infty$$
By iterating, one can find coefficients $\beta_n$ such that
$$|f-\beta_{n+1} f_\chi|(z)\leq (\eps+|z/R_n|^{s-1}(1+\eps'))\|(f-\beta_n f_\chi)_{|\partial B(0,R_n)}\|_\infty$$
for $|z|\geq R_{n+1}=CR_n$. Choose $\eps$ small enough and $C,R_0$ large enough so that $\eps+C^{s-1}(1+\eps')\leq C^{s-1+\eps_0}$ (and $R_0\geq R$ in any case). It follows that
$$\|(f-\beta_n f_\chi)_{|\partial B(0,R_n)}\|_\infty=O(\|(f_{|\partial B(0,R_0)}\|_\infty(R_n/R_0)^{s-1+\eps_0})$$
Since $\|(f-\beta_n f_\chi)_{|\partial B(0,R_{n+1})}\|_\infty$ is of the same order, we get
$$\beta_{n+1}-\beta_n=O(\|(f_{|\partial B(0,R_0)}\|_\infty(R_n/R_0)^{s-1+\eps_0}R_n^s)$$
If we fix $\eps_0>0$ small enough such that $2s-1+\eps_0<0$, this is summable. Moreover, if $\beta=\lim\beta_n$, we have
$$\beta_n=\beta+O(\|(f_{|\partial B(0,R_0)}\|_\infty(R_n/R_0)^{s-1+\eps_0}R_n^s)$$
and finally
$$\|(f-\beta f_\chi)_{|\partial B(0,R_n)}\|_\infty=O(\|(f_{|\partial B(0,R_0)}\|_\infty(R_n/R_0)^{s-1+\eps_0})$$

This gives the first statement. In order to obtain the more precise second statement, we write $\hat g(z)=b_0\bar z^{-s}+a_1 z^{s-1}+O(|z|^{-s-1})$ and
we proceed similarly to show the existence of coefficients $(\beta_n,\gamma_n)$ such that
 $$\|(f-\beta_n f_\chi-\gamma_n g_\chi)_{|\partial B(0,R_n)}\|_\infty=O(\|(f_{|\partial B(0,R_0)}\|_\infty(R_n/R_0)^{-s-1+\eps_0})$$
Thus
$$(\beta_{n+1}-\beta_n)f_\chi+(\gamma_{n+1}-\gamma_n)g_\chi=O(\|(f_{|\partial B(0,R_0)}\|_\infty(R_n/R_0)^{-s-1+\eps_0})$$
on $B(0,R_{n+1})^c$. This shows (by specialising at points close to the positive and negative half-lines respectively, say) that
$$\beta_{n+1}-\beta_n=O(\|(f_{|\partial B(0,R_0)}\|_\infty(R_n/R_0)^{-s-1+\eps_0}R_n^s)$$
and subsequently
$$\gamma_{n+1}-\gamma_n=O(\|(f_{|\partial B(0,R_0)}\|_\infty(R_n/R_0)^{-s-1+\eps_0}R_n^{1-s})$$
and one concludes as before.
\end{proof}

\subsubsection{Chiral Cauchy kernel}

The Cauchy kernel $\uK^{-1}$ (see Theorem \ref{Thm:Kencrit}), which inverts $\rK$, and its asymptotics are central to the analysis of dimers for, say, smooth test functions (e.g. Corollary \ref{Cor:FFplane}). In this section, we construct and estimate $\uK_\chi^{-1}$, an inverting kernel for $\rK:(\C^{M_B})_\chi\rightarrow(\C^{M_W})_\chi$.

In what follows, we assume without loss of generality that $s\in (0,\frac 12)$ (excluding the delicate case $s=\frac 12$). If $s\in(-\frac 12,0)$, set $\uK_\chi^{-1}(b,w)=\overline{\uK^{-1}_{\bar\chi}(b,w)}$.

\paragraph{Chiral Cauchy kernel: construction.\\}

We begin with a construction of the chiral Cauchy kernel $\uK_\chi^{-1}$, and give basic {\em a priori} estimates.

\begin{Lem}\label{Lem:chirkernel0}
Let $\chi=e^{2i\pi s}$, $s\in(0,\frac 12)$. For $w\in M_W$, there is a unique function $\gls{uKchi}(.,w)\in (\C^{M_B})_\chi$ 
such that $\rK\uK^{-1}_\chi(.,w)=\delta_w$ and 
$$\uK^{-1}_\chi(b,w)=O\left(\frac {1}{|w|}\left(\frac{|b|}{|w|}\right)^{s-1}\right)$$
for $|b|\geq 2|w|$. Moreover,
$\uK^{-1}_\chi(b,w)=O(\frac {1}{|w|}(|w|/|b|)^{s}))$ for $|b|\leq |w|/2$.
\end{Lem}
\begin{proof}
Uniqueness follows from the fact that there is no non-zero holomorphic function $f$ s.t. $f(z)=o(|z|^{-s})$ (Lemma \ref{Lem:monholom}).\\
If $w$ corresponds to the oriented edge $(x_0x_0')$ of $\Gamma$, we may set
$$S(b)=G_\chi(b,x_0')-G_\chi(b,x_0)$$
for $b\in\Gamma$. By harmonic conjugation (see e.g. Theorem 11 in \cite{Ken_domino_conformal} for similar arguments), one may extend $S$ to $\Gamma^\dg$ in such a way that $S(y_0)=0$ ($y_0$ is the vertex of $\Gamma^\dg$ adjacent to the vertex singularity at $v_0=0$) and $\rK S=\delta_w$. By Lemma \ref{Lem:chirharmest}, we know that
$$S(b)=O(\frac 1{|w|}\left((|b|/|w|)^\eps\wedge (|w|/|b|)^\eps\right))$$
for $b\in\Gamma$, $|b-w|\geq |w|/4$.\\
Consider a branch cut running from $0$ to infinity in the half-space $\{z:\Re(z\bar w)\leq 0\}$. Applying the discrete Green's formula in the slit plane (first intersecting it with a large disc) shows that the sum of (weighted) discrete derivatives of $S$ across the slit vanishes. By harmonic conjugation, this is saying that $S(b)\rightarrow 0$ as $b\rightarrow\infty$ on $\Gamma^\dg$.\\
By harmonic conjugation and the Harnack inequality (see \eqref{eq:harnack}), we get the following estimates on $S$ for $|b-w|\geq |w|/4$, $b\in\Gamma^\dg$, by integrating along a ray from $y_0$ to $b$ if $|b|\leq |w|$ and from $\infty$ to $|b|$ if $|b|\geq |w|$:
$$S(b)=O(\frac {\eps^{-1}}{|w|}\left((|b|/|w|)^\eps\wedge(|w|/|b|)^\eps\right))$$
This shows that there is $S\in(\C^{M_B})_\chi$ bounded s.t. $\rK S=\delta_w$ and $S(b)=O(|w|^{-1})$ on the boundary of the annulus $\{z:\frac{|w|}2<|z|<2|w|\}$, say. Because of Lemma \ref{Lem:monholom}, this does not specify $S$ uniquely.

By Lemma \ref{Lem:moninf}, there exists a unique $\beta=O(\|S\|_{\partial B(0,2|w|)}|w|^s)=O(|w|^{s-1})$ such that 
$$\uK_\chi^{-1}(.,w)\stackrel{def}{=}S(.)-\beta f_\chi$$
satisfies
$$\uK_\chi^{-1}(b,w)=O(|w|^{-1}(|b|/|w|)^{s-1})$$
for $|b|\geq 2|w|$, say.

For $|b|\ll |w|$, we observe that $S(b)=O(|w|^{-1})$ and $\beta f_\chi(b)=O(|w|^{s-1}|b|^{-s})$, so that $\uK_\chi^{-1}(b,w)=O(|w|^{-1}(|w|/|b|)^s)$ for $|b|\leq |w|$, $|b-w|\geq|w|/4$.
\end{proof}

\paragraph{Chiral Cauchy kernel: estimates.\\}

We now give estimates for the chiral Cauchy kernel constructed in Lemma \ref{Lem:chirkernel0} (compare with Theorem \ref{Thm:Kencrit}).

\begin{Lem}\label{Lem:monCauchy}
Let $\chi=e^{2i\pi s}$, $s\in(0,\frac 12)$. 
\begin{enumerate}
\item The kernel $\uK_\chi^{-1}$ satisfies:
$$\uK^{-1}_\chi(b,w)=\frac 12R_B\left(\frac{e^{i\nu(w)}}{\pi(b-w)}\left(\frac bw\right)^s\right)+\frac 12\bar R_B\left(\frac{e^{-i\nu(w)}}{\pi\overline{(b-w)}}\left(\frac {\bar b}{\bar w}\right)^{-s}\right)+O(|w|^{2s-2})$$
for $\frac 12|w|\leq|b|\leq 2|w|$, $|b-w|\geq\frac 14|w|$.
\item 
If $w$ is adjacent to the singularity $v_0=0$, we have
$$\uK^{-1}_\chi(b,w)=\frac {\Gamma(1-s)}2R_B\left(\frac{e^{i\nu(w)}}{\pi(b-w)}\left(\frac bw\right)^s\right)+\frac {\Gamma(1+s)}2\bar R_B\left(\frac{e^{-i\nu(w)}}{\pi\overline{(b-w)}}\left(\frac {\bar b}{\bar w}\right)^{-s}\right)+O(|b|^{s-3})$$
and if moreover the edge $(bw)$ is adjacent to the singularity, and $v_1$ is the symmetric image of $v_0$ across $(bw)$, then for $b\in\Gamma$, $v_1,b,b_0$ in counterclockwise order around $w$, 
$$\rK(w,b)\uK^{-1}_\chi(b,w)=\frac{1}{1-\chi^{-1}}\left(1-\left(\frac{w-v_1}{w-v_0}\right)^s\right)$$
\end{enumerate}
\end{Lem}
Note that the conditions in 2. are not restrictive, by exchanging the roles of $\Gamma,\Gamma^\dg$ and/or conjugating the lattice. Also, remark that
$$\lim_{s\searrow 0}\frac{1-e^{is(\arg(w-v_1)-\arg(w-v_0))}}{1-e^{-2i\pi s}}=\frac{1}{2\pi}(\arg(w-v_0)-\arg(w-v_1))=\rK(w,b)\uK^{-1}(b,w)$$
Before proceeding with the proof, let us sketch an alternative construction for $\uK_\chi^{-1}(.,w)$, for $w$ adjacent to $v_0$, from which one can deduce the evaluation $\uK_\chi^{-1}(b,w)$ for $b,w$ adjacent to $v_0$. Consider a branch cut $\gamma$ (a simple path on $M^\dg$) running from infinity to $v_0$ and $v_1$: $\gamma=(\dots,v_{-2},v_{-1},v_0,v_1)$. We have constructed basic discrete holomorphic functions $f_{\chi,v_i}$ with monodromy at $v_i$, $i=0,1$ (see Lemma \ref{Lem:monholom}). Using the branch cut $\gamma$, we can look at them as single-valued on the same fundamental domain. Then $f_{\chi,v_1}$, seen as a function with singularity at $v_0$, is discrete holomorphic except at $w$. From the asymptotic expansions (Lemma \ref{Lem:monholom}), we can find a (non-trivial) linear combination of $f_{\chi,v_0}$, $f_{\chi,v_1}$ which is $o(|z|^{-s})$ as $z\rightarrow\infty$. Thus, by the characterisation of $\uK_{\chi}^{-1}(.,w)$, it can be written as a linear combination of $f_{\chi,v_0}$, $f_{\chi,v_1}$. From the exact evaluation of these functions near their singularity, one recovers 2.

\begin{proof}
\begin{enumerate}
\item
Let us fix $w\in M_W$ and set $R=|w|$. We first construct an approximate kernel $\tilde S\in(\C^{M_B})_\chi$. If $b\in B(0,r_0)$, set $\tilde S(b)=cf_\chi(b)$. If $b\in B(w,r_1)$, set 
$$\tilde S(b)=\uK^{-1}(b,w)+
\frac 12R_B\left(\frac{e^{i\nu(w)}}{\pi}\cdot\frac sw\right)+\frac 12\bar R_B\left(\frac{e^{-i\nu(w)}}{\pi}\cdot\frac{-s}{\bar w}\right)$$
and otherwise
$$\tilde S(b)=\frac 12R_B\left(\frac{e^{i\nu(w)}}{\pi(b-w)}\left(\frac bw\right)^s\right)+\frac 12\bar R_B\left(\frac{e^{-i\nu(w)}}{\pi\overline{(b-w)}}\left(\frac {\bar b}{\bar w}\right)^{-s}\right)$$
Here $c,r_0,r_1$ are parameters yet to be specified. Then $\rK \tilde S(.,w)=\delta_w$ in $B(0,r_0)\sqcup B(w,r_1)$. Set $c$ s.t.
$$c\tau 2^{-s}\Gamma(s)=-\frac {e^{-i\nu(w)}}{2\pi \bar w^{1-s}}$$
Then the discontinuity of $\tilde S$ on $\partial B(0,r_0)$ is of order $O(r_0^s/|w|^{s+1}+r_0^{s-1}/|w|^{1-s})$.\\
The discontinuity of $\tilde S$ on $\partial B(w,r_1)$ is of order $O(r_1/|w|^2)$. Elsewhere $\rK\tilde S(.,w)$ is controlled by the third derivative of the continuous limit (see \eqref{eq:findiff}).

Next we use the {\em a priori} estimate (Lemma \ref{Lem:chirkernel0}) on the kernel $\uK_\chi^{-1}$ to correct $\tilde S$. 
First let us observe that by Lemma \ref{Lem:monholom}
$$\uK_\chi^{-1}(w,.)=\tilde S-\uK_\chi^{-1}(\rK\tilde S).$$ 
For this we need only check that $(\uK_\chi^{-1}(\rK\tilde S))(b)=O(|b|^{s-1})$ as $b\rightarrow\infty$. The contribution from $(\rK\tilde S)$ in $B(0,2R)$ is easily seen to be $O(|b|^{s-1})$. The other terms are handled as below .

Let us estimate the correction $\uK_\chi^{-1}(\rK\tilde S)$ for $b\in A(R/2,2R)$, $b\notin B(w,R/2)$ (where $R=|w|$, say). We simply add upper bounds (up to multiplicative constants) stemming from $(\rK\tilde S)(w')$ for respectively: $w'\in B(b,R/4)$; $w'\in \partial B(w,r_1)$; $w'-w\in A(r_1,R/4)$;  $w'\in A(R/2,2R)\setminus(B(b,R/4)\cup B(w,R/4))$; $w'\notin B(0,2R)$; $w'\in\partial B(0,r_0)$; and $w'\in A(r_0,R/2)$:
 $$\sum_{k=1}^R\frac{k}{kR^4}+\frac{r_1}{R^2}\cdot\frac{r_1}R+\sum_{k=r_1}^R\frac{k}{k^4}\cdot\frac 1R+\frac{R^2}{R^4.R}+\sum_{k=2R}^\infty k\cdot\frac{k^{s-4}}{R^s}\cdot\frac{k^s}{kR^s}+r_0(\frac{r_0^s}{R^{s+1}}+(r_0R)^{s-1})\frac{R^{s-1}}{r_0^s}
 $$
 Thus, up to a constant, we get the upper bound
 $$\frac{r_1^2}{R^3}+\frac 1{r_1^2R}+\frac{r_0}{R^2}+R^{2s-2}=O(R^{2s-2})$$
with $r_1=O(\sqrt R)$, $r_0=O(1)$. 
 
 \item Here $\uK_\chi^{-1}(.,w)$ is proportional to $g_w$ (see Lemma \ref{Lem:monpole}). Since $\delta=1$, we have $(v_0-w)^s\overline{(v_0-w)}^s=2^{-2s}$.
\end{enumerate} 
\end{proof}

\subsubsection{Behaviour near the singularity}

If $f\in\C^{M_B}$ is discrete holomorphic in $B(0,R)$, then its values near the origin can be estimated in terms of its values on the boundary $\partial B(0,R)$. At order 0, if $f=R_B(a)+\bar R_B(a')+o(1)$ on $\partial B(0,R)$, then the estimate also holds near $0$ (simply by the maximum principle \eqref{eq:maxprincip}). Using e.g. a Cauchy integral formula (see \eqref{eq:dirprobcauchy}), one can obtain a first-order estimate for $f$ near the origin of type:
$$f(z)=R_B\left(a+b\frac zR\right)+\bar R_B\left(a'+b'\frac{\bar z}R\right)+o(\|f_{|B(0,R)}\|_\infty R^{-1})$$

In this section, we address the problem of estimating $f$ near the singularity given its values at distance $R$, where $f$ is $\chi$-multivalued and (discrete) holomorphic in $B(0,R)$. This is the counterpart of Lemma \ref{Lem:moninf} (under inversion $z\leftrightarrow z^{-1}$ exchanging the singularities at $0$ and $\infty$).

\paragraph{Discrete holomorphic functions with monodromy: growth near the singularity.\\}

As an application of Lemma \ref{Lem:chirkernel0}, we start with the following simple growth estimate for chiral discrete holomorphic functions near the singularity.

\begin{Lem}\label{Lem:monzerogr}
Let $\chi=e^{2i\pi s}$, $s\in (0,\frac 12)$. Let $f\in (\C^{M_B})_\chi$ be bounded and s.t. $\rK f=0$ in $B(0,R)$. Then
$$f(z)=O(\|f_{|\partial B(0,R)}\|_\infty|z/R|^{-s})$$
\end{Lem}
\begin{proof}
Let $\tilde f=f\ind_{B(0,R)}$. Then $\rK \tilde f=O(\|f_{|\partial B(0,R)}\|_\infty)$ and the support of $\rK\tilde f$ is contained in white vertices adjacent to black vertices vertices belonging to $\partial B(0,R)\subset M_B$; let us use the same notation for this subset of $M_W$. Then consider
$$\tilde f-\sum_{w\in\partial B(0,R)}(\rK\tilde f)(w)\rK_\chi(w,.)$$
This function in $(\C^{M_B})_\chi$ is discrete holomorphic and $O(|z|^{s-1})$ at infinity, hence (Lemma \ref{Lem:monholom}) vanishes identically. Consequently, we have the (chiral) Cauchy integral formula (compare e.g. with \eqref{eq:discrcauchform}):
$$f(b)=\sum_{w\in\partial B(0,R)}(\rK\tilde f)(w)\rK_\chi^{-1}(w,b)=O(\lVert f_{|\partial B(0,R)}\rVert_\infty|b/R|^{-s})$$
in $B(0,R)$.
\end{proof}

\paragraph{Discrete holomorphic functions with monodromy: expansion near the singularity.\\}

We are now in position to state the analogue of Lemma \ref{Lem:moninf} (concerning the asymptotic expansion of a discrete $\chi$-multivalued holomorphic function in a neighbourhood of $\infty$) for the singularity at zero. The arguments are somewhat similar but more involved (roughly speaking, when zooming on the singularity at infinity, the graph mesh gets finer, and the mesh gets coarser when zooming on the singularity at $0$).

\begin{Lem}\label{Lem:monzero}
Let $\chi=e^{2i\pi s}$, $s\in (0,\frac 12)$. Let $f\in(\C^{M_B})_\chi$ be bounded and s.t. $\rK f=0$ in $B(0,R)$, and so that in $A(\frac R2,R)$,
$$f(z)=R_B(\psi_1(z))+\bar R_B(\psi_2(z))+\xi(z)$$
where $\psi_1(z)=z^s\sum_{n\geq 0}a_nz^n$, $\psi_2(z)=\bar z^{-s}\sum_{n\geq 0}b_n\bar z^n$. Then if $\eta>0$, for $z=O(1)$, 
$$f(z)=\frac{b_0}{\bar\tau 2^{-s}\Gamma(s)} f_{\chi}(z)+\frac{a_0}{\tau 2^s\Gamma(-s)} h_\chi(z)+O(R^{s+\eta}\|\xi\|_{\infty,A(\frac R2,R)}+R^{3s-1+\eta}\|\psi_i\|_{\infty,A(\frac R2,R)})$$
where $h_\chi(z)=\overline{f_{1,\bar\chi}(z)}$, and $\|\psi_i\|=\|\psi_1\|+\|\psi_2\|$.
\end{Lem}
Recall $f_\chi, f_{k,\chi}$ from Lemma \ref{Lem:monholom}.
\begin{proof}
{\bf Step 1.}  Let $\varphi\in(\C^{M_B})_\chi$ be such that $\rK\varphi=0$ in $B(0,r)$,
\begin{equation}\label{eq:lemmonzero1}
\varphi(z)=\alpha f_\chi(z)+\beta h_\chi(z)+R_B(\varphi_1(z))+{\overline R_B}(\varphi_2(z))+\xi(z)
\end{equation}
in $A(\frac r2,r)=B(0,r)\setminus B(0,\frac r2)$, say. Here $\varphi_1$ (resp. $\varphi_2$) is holomorphic (resp. antiholomorphic) and $\chi$-multivalued, with Laurent expansions $\varphi_1(z)=z^s\sum_{n\geq 1}a_n z^n$, $\varphi_2(z)=\bar z^{-s}\sum_{n\geq 1}b_n \bar z^n$. 

Our first task is to show that, when going from scale $r$ to a smaller scale $\eps_0r$, the estimate \eqref{eq:lemmonzero1} improves in the (rough) sense that the ratio $(\alpha f_\chi+\beta h_\chi)/({\rm other\ terms})$ gets larger.

Observe
that
$$\varphi_1(z)=\frac 1{2i\pi}\oint_{C(0,\frac 34 r)}\left(\frac zw\right)^s\frac{\varphi_1(w)dw}{z-w}=\frac 1{2i\pi}\oint_{C(0,\frac 34 r)}\left(\frac zw\right)^s\frac{\varphi_1(w)zdw}{w(z-w)}$$
for $z\in D(0,\frac 34 r)$, and thus $\|\varphi_1\|_{\infty,C(0,\eps_0r)}\leq c_1(\eps_0)\|\varphi_1\|_{\infty,C(0,r)}$, where $c_1(\eps_0)\sim\eps_0^{s+1}$ goes to zero as $\eps_0$ goes to zero ($\eps_0>0$ small, to be fixed later). We have a similar estimate for $\varphi_2$: $\|\varphi_2\|_{\infty,C(0,\eps_0r)}\leq c_2(\eps_0)\|\varphi_2\|_{\infty,C(0,r)}$, $c_2(\eps_0)\sim\eps_0^{-s+1}$ (this is sharp: consider e.g. $\phi_1(z)=z^{s+1}$ and $\phi_2(z)=\bar z^{-s+1}$).

If $\gamma_r$ is a simple cycle on $\Gamma$ approximating $C(0,\frac 34r)$, let $B$ be the set of black vertices on or inside $\gamma_\delta$. We have the replication identity (a discrete Cauchy integral formula, as in the proof of Lemma \ref{Lem:monzerogr})
$$\varphi(z)=\sum_{w\in\gamma^W}\rK(\varphi\ind_B)(w)\uK_\chi^{-1}(z,w)$$
where $\gamma^W$ are white vertices which are on $\gamma_r$ or are outside of $\gamma_r$ and adjacent to a black vertex on $\gamma_r$. For $z\in B(0,r/2)$, say, we have:
$$\varphi(z)=\alpha f_\chi(z)+\beta h_\chi(z)+O(\|\xi\|_{\infty,A(\frac r2,r)}(|z|/r)^{-s})+\sum_{w\in\gamma^W}\rK((R_B\varphi_1+{\overline R_B}\varphi_2)\ind_B)(w)\uK_\chi^{-1}(z,w)$$
As in Section \ref{Sec:surgery} (leading to \eqref{eq:surgprojconv}), for $z\in A(\frac{\eps_0r}2,\eps_0r)$, using Lemma \ref{Lem:monCauchy} in lieu of Theorem \ref{Thm:Kencrit}, we may estimate the discrete Cauchy integral by:
\begin{align*}
\sum_{w\in\gamma^W}\rK((R_B\varphi_1)\ind_B)(w)\uK_\chi^{-1}(z,w)
&=R_B\frac{1}{2i\pi}\oint_{C(0,\frac 34r)}\left(\frac zw\right)^s\frac{\varphi_1(w)dw}{z-w}+O(\|\varphi_1\|_{\infty,A(\frac r2,r)}\eps_0^{-s}r^{2s-1})\\
\sum_{w\in\gamma^W}\rK((\bar R_B\varphi_2)\ind_B)(w)\uK_\chi^{-1}(z,w)
&=\bar R_B\frac{-1}{2i\pi}\oint_{C(0,\frac 34r)}\left(\frac {\bar z}{\bar w}\right)^{-s}\frac{\varphi_2(w)d\bar w}{\overline{z-w}}+O(\|\varphi_2\|_{\infty,A(\frac r2,r)}\eps_0^{-s}r^{2s-1})
\end{align*}
Consequently in $A(\frac{\eps_0r}2,\eps_0r)$, we have
$$\varphi(z)=\alpha f_\chi(z)+\beta h_\chi(z)+R_B(\varphi_1(z))+{\overline R_B}(\varphi_2(z))+\xi(z)$$
with
$$\|\varphi_i\|_{\infty,A(\frac{\eps_0r}2,\eps_0r)}\leq c_i(\eps_0)\|\varphi_i\|_{\infty,A(\frac{r}2,r)}$$
for $i=1,2$, and
$$\|\xi\|_{\infty,A(\frac{\eps_0r}2,\eps_0r)}\leq c\eps_0^{-s}\left(\|\xi\|_{\infty,A(\frac{r}2,r)}+\|\varphi_i\|_{\infty,A(\frac r2,r)}r^{2s-1}\right)$$

\noindent {\bf Step 2.} We now iterate Step 1 along a sequence of scales $\eps_0^kr$, for $\eps_0$ small enough.

Fix $\eta>0$ arbitrarily small. For small enough $\eps_0$, $k\geq 1$, we get
$$\|\varphi_1\|_{\infty,A(\frac{\eps_0^kr}2,\eps_0^kr)}\leq (\eps_0^k)^{s+1-\eta}\|\varphi_1\|_{\infty,A(\frac{r}2,r)},{\rm\ \ \ \ \ }
\|\varphi_2\|_{\infty,A(\frac{\eps_0^kr}2,\eps_0^kr)}\leq (\eps_0^k)^{-s+1-\eta}\|\varphi_2\|_{\infty,A(\frac{r}2,r)}
$$
and
$$\eps_0^{s+\eta}\|\xi\|_{\infty,A(\frac{\eps^{k+1}_0r}2,\eps^{k+1}_0r)}\leq \|\xi\|_{\infty,A(\frac{\eps_0^kr}2,\eps_0^kr)}+r^{2s-1}(\eps_0^k)^{2s-1+1-s-\eta}\|\varphi_i\|_{\infty,A(\frac r2,r)}$$
so that for $\eta$ small enough we get
$$\|\xi\|_{\infty,A(\frac{\eps_0^kr}2,\eps_0^kr)}\leq (\eps_0^k)^{-s-\eta}\left(\|\xi\|_{\infty,A(\frac{r}2,r)}+C.\|\varphi_i\|_{\infty,A(\frac r2,r)}r^{2s-1}\right)$$
and for $z=O(1)$:
$$\varphi(z)=\alpha f_\chi(z)+\beta h_\chi(z)+
O(r^{s+\eta}\|\xi\|_{\infty,A(\frac r2,r)}+r^{3s+\eta-1}\|\varphi_i\|_{\infty,A(\frac r2,r)}+r^{s-1}\|\varphi_i\|_{\infty,A(\frac r2,r)})
$$
{\bf Step 3.} We have by Lemma \ref{Lem:monholom}
\begin{align*}
f_\chi(z)&=\bar R_B(\bar\tau 2^{-s}\Gamma(s)\bar z^{-s})+O(|z|^{s-1})\\
h_\chi(z)&=R_B(\tau 2^{s}\Gamma(-s)z^s)+O(|z|^{-s-1})
\end{align*}
Taking $f$ as in the statement of the lemma, set $\varphi_1(z)=\psi_1(z)-a_0z^s$, $\varphi_2(z)=\psi_2(z)-b_0\bar z^{-s}$, $\alpha=\frac{b_0}{\bar\tau 2^{-s}\Gamma(s)}$, $\beta=\frac{a_0}{\tau 2^s\Gamma(-s)}$. We get
$$f(z)=\alpha f_\chi(z)+\beta h_\chi(z)+R_B(\varphi_1(z))+R_B(\varphi_2(z))+\xi'(z)$$
where $\|\varphi_i\|_{\infty,A(\frac R2,R)}$, $\alpha R^{-s}$ and $\beta R^s$ are of order $\|\psi_i\|_{\infty,A(\frac R2,R)}$ and
$$\|\xi'\|_{\infty,A(\frac R2,R)}=O(\|\xi\|_{\infty,A(\frac R2,R)}+\|\psi_i\|_{\infty,A(\frac R2,R)}R^{2s-1})
$$
and applying Step 2 concludes.
\end{proof}

This applies in particular to $f=\uK_\chi^{-1}(.,w)$, $R=|w|/2$. For  $b=O(1)$, combining the estimate of Lemma \ref{Lem:monCauchy} (for $|b|,|w|,|b-w|$ mutually comparable), this leads to an estimate of $\uK_\chi^{-1}(b,w)$ (which is {\em a priori} of order $O(R^{s-1})$, see Lemma \ref{Lem:chirkernel0}) within $O(R^{3s-2+\eta})$. By replication (as in Lemma \ref{Lem:monzerogr}), for $s<\frac 14$, if $\varphi$ is discrete holomorphic and $\chi$-multivalued in a ball $B(0,R)$, we can estimate $\varphi(b)$ for $b$ adjacent to the singularity based on values taken by $\varphi$ in, say, the neighbourhood of a closed cycle around 0 in $A(\frac R2,R)$.

\subsection{More singularities}\label{ssec:monoinv}

In Section \ref{ssec:monodromy}, we studied in some details discrete holomorphic functions with monodromy around a singularity (in terms of the Riemann sphere $\hat\C$, this corresponds to a pair of singularities, including the one at $\infty$).

Let us now consider the situation where two points (singularities) $x\neq y$ (midpoints of edges of $\Lambda$) are marked. We consider functions with monodromy $\chi=e^{2i\pi s}$, $s\in (0,\frac 12)$, around $x$ and $\bar\chi$ around $y$. We are interested in estimating $S_\rho$, a kernel inverting $\rK$ operating on these functions, in particular near the diagonal.

More specifically, we consider a family of rhombus tilings $(\Lambda_\delta)_\delta$ with respective edge length $\delta$, $\delta\searrow 0$ along some sequence.
Two points $x_\delta,y_\delta$ (midpoints of edges of $\Lambda$) are marked; we assume $x_\delta\rightarrow x$, $y_\delta\rightarrow y$. Around each singularity, the local picture is as in Figure \ref{fig:monloc}.

Let $\rho:\pi_1(\C\setminus\{x,y\})\rightarrow\U$ be the unitary character s.t. $\rho(\gamma_x)=\chi$, $\rho(\gamma_y)=\bar\chi$ where $\gamma_x$ (resp. $\gamma_y$) is a simple loop winding around $x$ (resp. $y$) counterclockwise. 

Let $(\C^{M_B})_\rho$ (resp. $(\C^{M_V})_\rho$) be the space of multiplicatively multi-valued functions on $M_B$ transforming according to $\rho$. As discussed at the beginning of Section \ref{ssec:monodromy}, this can be realised by lifting $M$ to the universal cover of $\hat\C\setminus\{x,y\}$; more concretely (and less canonically), one can use a branch cut running from $x$ to $y$ on $M^\dg$ and glue decks, isomorphic to $M$ and indexed by $\Z\simeq\pi_1(\hat\C\setminus\{x,y\})$, along the branch cut.

\subsubsection{Green kernel}

First we consider the discrete Laplacian operating on $(\C^{M_V})_\rho$. In the following lemma, we construct an inverting (Green) kernel $G_\rho^\delta$, based on arguments similar to those of Lemma \ref{Lem:chirharmest}.

\begin{Lem}\label{Lem:Green2pts}
We assume $x\neq y$, $\rho$ non-trivial.
\begin{enumerate}
\item There is no nonzero bounded harmonic function in $(\C^{M_V})_\rho$.
\item For $w\in M_V$, there is a unique bounded function $G_\rho^\delta(.,w)\in (\C^{M_V})_\rho$ s.t. $\Lap_\Gamma G_\rho(.,w)=\delta_w$.
\item There a unique bounded continuous $\rho$-multivalued function $G_\rho(.,w)$ on $\C\setminus\{x,y\}$ s.t. $\Lap G_\rho(.,w)=\delta_w$ and $G_\rho(z,w)\rightarrow 0$ as $z\rightarrow x,y$.
\item As $\delta\searrow 0$, $G_\rho^\delta(.,.)$ converges uniformly to $G_\rho(.,.)$ uniformly on compact subsets of $\{(z,w)\in (\C\setminus\{x,y\})^2, z\neq w\}$.
\end{enumerate}
\end{Lem}
\begin{proof}
\begin{enumerate}
\item Let $h$ be a bounded harmonic function in $(\C^{M_V})_\rho$. Fix $z\in M_V$. Consider the random walk $(X_n)$ started from $z$ stopped at $\tau$, the time of first return to $z$ or the first exit of $B(0,n)$, whichever comes first. Let $w_x$ (resp. $w_y$) be the winding of $(X_0,\dots,X_\tau)$ around $x$ (resp. $y$). An excursion decomposition of the random walk between successive visits to $z$ before time $\tau$ shows that:
$$h(z)=\frac{1-\P(X_\tau=z)}{1-\P(X_\tau=z)\E(\chi^{w_x-w_y}|X_\tau=z)}\E(h(X_\tau)|X_\tau\neq z)$$
As $n\rightarrow\infty$, $|\E(\chi^{w_x-w_y}|X_\tau=z)|$ stays bounded away from 1 while $\P(X_\tau=z)\rightarrow 1$. Consequently $h(z)=0$.

\item Uniqueness follows from 1. For existence, we set $h(z)=\E_z(\rho(\gamma))$, where $\gamma$ is the concatenation of the sample path of a random walk started from $z$ stopped when it first hits $w$ and a fixed path from $w$ to $z$ (one may also reason on a branched cover of $M_V$, stopping the walk when it hits a lift of $w$). Then $h$ is bounded, $\rho$-multivalued and harmonic except at $w$. It is indeed not harmonic at $w$ by 1. Consequently it is proportional to $G_\rho(.,w)$.

\item By compactifying at infinity, a bounded harmonic function has a removable singularity at infinity. By mapping $\{x,y\}$ to $\{0,\infty\}$ by a homography and lifting to the universal cover of $\C\setminus\{0\}$ via $\log$, the problem it to find $\tilde h$ harmonic except on $w'+2i\pi\Z$ s.t. $\tilde h(\cdot+2i\pi)=\chi \tilde h(\cdot)$ and $\Lap \tilde h=\sum_k\chi^k\delta_{w'+2i\pi k}$. It is then easy to see that the solution to this problem exists and is unique (up to normalisation) and decays as $\Re(z)\rightarrow \pm\infty$ (corresponding to $x,y$ in the original coordinates). Explicitly, we may write
$$\tilde h(.)=\frac{1}{2-\chi-\bar\chi}\sum_{k\in\Z}\chi^k\left(-G_\C(.,w'+2i\pi(k-1))+2G_\C(.,w'+2i\pi k)-G_\C(.,w'+2i\pi(k+1))\right)$$
where $G_\C(z_1,z_2)=-\frac{1}{2\pi}\log|z_1-z_2|$ is the full-plane (continuous) Green kernel. (Remark that for a fixed value of the argument $z$, the summand is $O(1/(k^{2}+\Re(z)^2)$).

\item Fix $\eps>0$; we consider $\{(z,w)\in (B(0,\eps^{-1})\setminus (B(x,3\eps)\cup B(y,3\eps)))^2:|z-w|\geq\eps\}$. We need to show uniform convergence of $G^\delta_\rho$ on this set.\\
Let $\gamma_e$, $\gamma_i$ be simple paths on $\Gamma_\delta$ at distance $O(\delta)$ of the circles $C(w,2\eps)$, $C(w,\eps)$ respectively. Let $B$ be the connected component of $w$ in $\Gamma_\delta\setminus\gamma_e$. Let $\tau_e$ (resp. $\tau_i$) be the time of first exit of $\Gamma_\delta\setminus\gamma_e$ (resp. $\Gamma_\delta\setminus\gamma_i$) by a random walk on $\Gamma$. We define (twisted) Poisson operators $P_e:\C^{\gamma_e}\rightarrow\C^{\gamma_i}$ and $P_i:\C^{\gamma_i}\rightarrow\C^{\gamma_e}$ by $(P_ef)(z)=\E_z(f(X_{\tau_e}))$ and 
$$(P_if)(z)=\E_z(\chi^Nf(X_{\tau_i})),$$
where $N$ is the algebraic number of crossings of a branch cut between $x$ and $y$ (not intersecting $\gamma_e$) by the random walk before $\tau_i$. (One may also reason on a branched cover of $M_V$, lifting $f$ to a $\rho$-multivalued function). We omit the dependence on $\delta$ for lightness of notation.\\
As before (see Lemma \ref{Lem:chirharmest}), one can show that $(P_eP_i)$ is a strict contraction on $L^\infty(\gamma_i)$, uniformly in $w$ in a compact subset of $\C\setminus\{x,y\}$, $\delta$ small enough ($\eps>0$ is fixed). Moreover, by starting a random walk on $\gamma_i$ and stopping it on its first return to $\gamma_i$ after its first visit to $\gamma_e$, we get 
the following identity in $L^\infty(\gamma_i)$:
$$G_\rho^\delta(.,w)_{|\gamma_i}=G^\delta_B(.,w)_{|\gamma_i}+(P_eP_i)G_\rho^\delta(.,w)_{|\gamma_i}$$
Here $G^\delta_B$ denotes the Green kernel with Dirichlet boundary conditions in $B$. Consequently,
$$G_\rho^\delta(.,w)_{|\gamma_i}=(\sum_{k=0}^\infty(P_eP_i)^k)G^\delta_B(.,w)_{|\gamma_i},$$
a summable series in $L^\infty(\gamma_i)$.\\
Clearly, the analogue decomposition holds for the continuous Green kernel $G_\rho$. We know that $G^\delta_B(.,w)_{|\gamma_i}$ converges uniformly to $G_B(.,w)_{|\gamma_i}$ as $\delta\searrow 0$ and in particular is uniformly bounded. Since $(P_eP_i)$ is uniformly strictly contracting, we only need to establish uniform convergence of each term in the expansion.

We note that $P_e^\delta h$ converges to $P_eh$ in $L^\infty(\gamma_i(w))$ uniformly in $h,w$ as $\delta\searrow 0$, where $h$ is (the restriction to $\gamma_i(w)$ of) a 1-Lipschitz function. This may be shown by contradiction, using equicontinuity (of $h_{|\gamma_i}$ and of $P_e^\delta h$ on compact subsets of $B$). \\
Similarly, $P_i^\delta h$ converges to $P_ih$ in $L^\infty(\gamma_e(w))$ uniformly in $h,w$, where $h$ is (the restriction to $\gamma_e(w)$ of) a 1-Lipschitz function. We reason as above, the contradiction coming from the uniqueness in the following Dirichlet problem: find $h_0$ harmonic, bounded, $\rho$-multivalued in $\C\setminus(\{x,y\}\cup D(w,\eps))$, with boundary condition a continuous function $h$ on $C(0,\eps)$. As in 3., this may be recast as a standard Dirichlet problem in $\C$.

Using iteratively these convergence statements for $P_e^\delta$, $P_i^\delta$, we obtain term-wise convergence in the series expansion for $G^\delta_\rho$ above. 
\end{enumerate}
\end{proof}

Remark that since the statement holds for {\em any} sequence of tilings $\Lambda_\delta$ with mesh $\delta$ going to zero, it holds uniformly in the tiling $\Lambda$ (under $(\spadesuit)$).

\subsubsection{Cauchy kernel}\label{sss:moninv}

\paragraph{Two singularities.\\}

We want to construct and estimate an inverting kernel $S_\rho$ for $\rK:(\C^{M_B})_\rho\rightarrow(\C^{M_W})_\rho$. This is complicated by the fact that we do not have (yet) an analogue of the classification result of Lemma \ref{Lem:monholom} and there are distinct possible choices of normalisation for an inverting kernel (namely, vanishing at a singularity or at infinity, on the primary lattice $\Gamma$ or its dual $\Gamma^\dg$).

\begin{Lem}\label{Lem:Srho2pts}
For $s\in (-\frac 12,\frac 12)$, $\delta$ small enough, for $w\in M_W$ there is a unique $\rho$-multivalued function $S_\rho(.,w)$ vanishing at infinity s.t. $\rK S_\rho(.,w)=\delta_w$. Moreover,
$$S_\rho(z,w)=\frac 12R_B\left(\left(\frac{z-x}{w-x}\cdot\frac{w-y}{z-y}\right)^s\frac{e^{i\nu(w)}}{\pi(z-w)}\right)
+\frac 12\bar R_B\left(\overline{\left(\frac{z-x}{w-x}\cdot\frac{w-y}{z-y}\right)}^{-s}\frac{e^{-i\nu(w)}}{\pi\overline{(z-w)}}\right)+o(|x-y|^{-1})$$
with uniform convergence on compact subsets of $\{(z,w)\in (\C\setminus\{x,y\})^2:z\neq w\}$.
\end{Lem}
The continuous holomorphic and antiholomorphic kernels 
\begin{align*}
z&\longmapsto\left(\frac{z-x}{w-x}\cdot\frac{w-y}{z-y}\right)^s\frac{1}{\pi(z-w)}\\
z&\longmapsto\overline{\left(\frac{z-x}{w-x}\cdot\frac{w-y}{z-y}\right)}^{-s}\frac{1}{\pi\overline{(z-w)}}
\end{align*}
are uniquely specified by the fact that they are $\rho$-multivalued, have a simple pole at $w$, vanish as $z\rightarrow\infty$, and have singularities of order at most $|s|$ at $x,y$. 
\begin{proof}
{\bf Step 1.} 
As was the case with one singularity (Lemma \ref{Lem:chirkernel0}), we can use $G^\delta_\rho$ to construct an inverting kernel $\hat S_\rho$ for $\rK$ operating on $(\C^{M_B})_\rho$. Among such kernels, it is characterised by the fact that $\hat S_\rho(.,w)$ restricted to $M_V$ is harmonic except at endpoints of the edge of $\Gamma$ corresponding to $w$. The restriction of $\hat S_\rho$ to $M_V$ is $G^\delta_\rho(.,v')-G^\delta_\rho(.,v)$, where $w\in M_W$ corresponds to the oriented edge $(vv')$ of $\Gamma_\delta$. 
On $M_F$, it is deduced by harmonic conjugation; it vanishes on the vertices of $\Gamma^\dg$ corresponding to $x,y$ respectively. A discrete Green's formula argument (cutting the domain along a branch cut from $x$ to $y$) shows that these two conditions for $\hat S_\rho(.,w)_{|M_F}$ are consistent. Indeed, the variation of the harmonic conjugate between these two vertices is proportional to the flux of $\hat S_\rho$ across a branch cut connecting the vertices; as $\hat S_\rho$ decays at infinity, this flux is the integral of the Laplacian of $\hat S_\rho$ in the fundamental domain defined by the cut, hence zero. Note also  that $\hat S_\rho$ is {\em not} invariant under duality $\Gamma\leftrightarrow\Gamma^\dg$ (i.e. reasoning on the random walk on $\Gamma^\dg$ produces $\hat S_\rho^\dg\neq\hat S_\rho$).

In order to estimate $\hat S_\rho$, let us evaluate $\partial_wG_\rho$, $\partial_{\bar w}G_\rho$ (in the continuous limit). As a function of $z$, $\partial_wG_\rho$ is locally the sum of a meromorphic and an antimeromorphic component; the decomposition is unique up to an additive constant. The additive constant can be specified uniquely in order to make these two components individually $\rho$-multivalued. Moreover, $G_\rho(z,w)+\frac 1{2\pi}\log|z-w|$ is harmonic in $z$ in a neighbourhood of $w$.
If $\chi=e^{2i\pi s}$, $s\in (0,1)$, this leaves the only possibility:
$$\partial_wG_\rho(z,w)=\left(\frac{z-x}{w-x}\right)^s\left(\frac{z-y}{w-y}\right)^{1-s}\frac{1}{\pi(z-w)}$$
as there is no (nonvanishing) $\rho$-multivalued, bounded holomorphic (resp. antiholomorphic) function on $\C\setminus\{x,y\}$. (If $\varphi$ is such an holomorphic function, $z\mapsto\varphi(z)(z-x)^{-s}(z-y)^{s-1}$ is bounded, with removable singularities and vanishing at infinity, hence vanishes identically). Similarly,
$$\partial_{\bar w}G_\rho(z,w)=\overline{\left(\frac{z-x}{w-x}\right)}^{1-s}\overline{\left(\frac{z-y}{w-y}\right)}^{s}\frac{1}{\pi\overline{(z-w)}}$$

We note that $G^\delta_\rho(z,w)=\overline{G^\delta_{\bar\rho}(w,z)}$ (as in Lemma \ref{Lem:chirharmest}). 

From Lemma \ref{Lem:Green2pts}, $G^\delta_\rho$ converges to $G_\rho$ uniformly away from singularities. By general discrete harmonic function arguments (e.g. Section 3.3 in \cite{SmiChe_isoradial}), this yields convergence of discrete derivatives of $G^\delta_\rho(z,w)$ (in either variable), also uniformly away from singularities. More precisely,
$$S_\rho(z,w)=\frac 12R_B\left(e^{i\nu(w)}\partial_wG_\rho(z,w)\right)
+\frac 12\bar R_B\left(e^{-i\nu(w)}\partial_{\bar w}G_\rho(z,w)\right)+o(1)$$
uniformly in compact subsets of $(\C\setminus\{x,y\})^2\setminus\Delta$, where $\Delta$ denotes the diagonal. 

Remark that if $x_\delta,y_\delta$ do not converge to $x,y$, one may apply an affine transformation to $\Lambda_\delta$ so that they converge (as long as the graph distance between the two singularities goes to infinity). By a simple scaling argument, this turns the $o(1)$ error term into $o(|x-y|^{-1})$.

{\bf Step 2.}

We are now seeking an inverting kernel $S_\rho$ which vanishes at infinity: $S_\rho(b,w)\rightarrow 0$ as $b\rightarrow\infty$. Such a kernel differs from $\hat S_\rho$ by a discrete holomorphic function (freezing the second variable) and is invariant under duality $\Gamma\leftrightarrow\Gamma^\dg$. We start by describing bounded $\rho$-multivalued discrete holomorphic functions.

Observing that the restriction to $M_V$ of a $\rho$-multivalued discrete holomorphic function is harmonic except possibly at the two vertices adjacent to $x,y$, we conclude (from Lemma \ref{Lem:Green2pts}) that the space of bounded $\rho$-multivalued discrete holomorphic functions is at most two-dimensional. Let us construct two linearly independent such functions.

Let $\gamma_\delta$ be a simple cycle on $\Gamma$ at distance $O(\delta)$ of $\gamma=C(x,|x-y|/2)$. Let $\gamma_\delta^\dg$ be the outer boundary of the union of faces of $\Gamma^\dg$ corresponding to vertices of $\Gamma$ on $\gamma_\delta$. Set $f(z)=\delta^{1-s}2^s\Gamma(s)^{-1}f_{\chi,x}(z)$ for $z\in M_B$ inside or on $\gamma_\delta$ and $f(z)=0$ otherwise (see Lemma \ref{Lem:monholom}). Then:
$$g_x(z)=f(z)-\sum_{w\in M_W}(\rK f)(w)\hat S_\rho(z,w)$$
is $\rho$-multivalued and discrete holomorphic. Moreover $\rK f$ is supported on vertices of $M_W$ corresponding to edges of $\gamma_\delta$ and $\gamma_\delta^\dg$. Assume here $s\in (\frac 12,1)$ (so that the dominant terms near $x$ are holomorphic rather than antiholomorphic; otherwise conjugate): $f(z)=R_B((z-x)^{s-1})+o(1)$ near $\gamma$.

For $z\in M_V$, we have by construction of $\hat S_\rho$ (see also \eqref{eq:discrtcontint})
\begin{align*}
\sum_{w\in\gamma_\delta}(\rK f)(w)\hat S_\rho(z,w)
&=\int_{\gamma_\delta}i(w-x)^{s-1}d_wG_\rho(z,w)+o(1)\\
\sum_{w\in\gamma_\delta^\dg}(\rK f)(w)\hat S_\rho(z,w)
&=\int_{\gamma_\delta^\dg}(w-x)^{s-1}\ast d_wG_\rho(z,w)+o(1)
\end{align*}
where
\begin{align*}
d_wG_\rho(z,w)&=\partial_wG_\rho(z,w)dw+\partial_{\bar w}G_\rho(z,w)d\bar w\\
\ast d_wG_\rho(z,w)&=-i\partial_wG_\rho(z,w)dw+i\partial_{\bar w}G_\rho(z,w)d\bar w
\end{align*}

Taking into account $(w-x)^{s-1}=(\overline{w-x})^{1-s}r^{2s-2}$ on $C(x,r)$, we may write these integrals as contour integrals of closed forms and deform them to $\gamma$ to obtain
\begin{align*}
g_x(z)&=(z-x)^{s-1}\chi_{|z-x|\leq|x-y|/2}\\
&-\oint_\gamma (w-x)^{s-1}\left(\left(\frac{z-x}{w-x}\right)^s\left(\frac{z-y}{w-y}\right)^{1-s}\frac{dw}{\pi(z-w)}
\right)
+o(1)\\
&=(z-x)^{s-1}\chi_{|z-x|\leq|x-y|/2}-(-(z-x)^{s-1}\left(\frac{z-y}{x-y}\right)^{1-s}+(z-x)^{s-1}\chi_{|z-x|\leq|x-y|/2})+o(1)\\
&=(z-x)^{s-1}\left(\frac{z-y}{x-y}\right)^{1-s}+o(1)
\end{align*}
by the residue formula (still for $z\in M_V$). Symmetrically, one may construct another discrete holomorphic function with:
\begin{equation}\label{eq:Lem:Srho2pts}
g_y(z)=(\overline{z-y})^{s-1}\overline{\left(\frac{z-x}{y-x}\right)}^{1-s}+o(1)
\end{equation}
The asymptotics on $M_F$ are obtained by harmonic conjugation (or by an argument similar to the one for $M_V$), the additive constant being fixed by the condition on $\rho$-multivaluedness. This gives:
\begin{align*}
g_x(z)&=R_B\left((z-x)^{s-1}\left(\frac{z-y}{x-y}\right)^{1-s}\right)+o(1)\\
g_y(z)&=\bar R_B\left((\overline{z-y})^{s-1}\overline{\left(\frac{z-x}{y-x}\right)}^{1-s}\right)+o(1)
\end{align*}
for $z\in M_B$. Convergence is uniform on compact subsets away from $x,y,\gamma$. By writing the Cauchy formula on another circle $\gamma'$ (say close to $C(x,|x-y|/4)$), one obtains uniform convergence on compact subsets of $\C\setminus\{x,y\}$. 

From these asymptotics, we see in particular that for a small enough mesh, $g_x$ and $g_y$ are linearly independent and consequently span the space of bounded discrete holomorphic functions in $(\C^{M_B})_\rho$.

{\bf Step 3.}

We have constructed an inverting kernel $\hat S_\rho$ normalised by: $\hat S_\rho(b,w)=0$ for $b\in M_F\simeq\Gamma^\dg$ adjacent to one of the singularities $x,y$; and described bounded $\rho$-multivalued discrete holomorphic functions.

Thus one may add a linear combination of $g_x$, $g_y$ to $\hat S_\rho(.,w)$ to obtain a kernel vanishing at infinity. Indeed, a function on $M_V$ (resp. $M_F$) which is bounded and (discrete) harmonic in a neighbourhood of infinity has a limit at infinity, as is easily shown e.g. by a random walk coupling argument. Convergence on compact sets is enough to ensure that
$$\lim_{z\rightarrow\infty,z\in M_V} g_x(z)=(x-y)^{s-1}+o(1)$$
Similarly, $\lim_{z\rightarrow\infty,z\in M_F} g_x(z)=i(x-y)^{s-1}+o(1)$, $\lim_{z\rightarrow\infty,z\in M_V} g_y(z)=(\overline{y-x})^{s-1}+o(1)$, $\lim_{z\rightarrow\infty,z\in M_F} g_y(z)=-i(\overline{y-x})^{s-1}+o(1)$. 
Thus, for $\delta$ small enough, one may find $a(w)$, $b(w)$ s.t. 
$$S_\rho(z,w)=\hat S_\rho(z,w)-a(w)g_x(z)-b(w)g_y(z)$$
vanishes at infinity. Combining with previous estimates on $\hat S_\rho(.,w)$, $g_x$, $g_y$, we identify $a(w),b(w)$ up to $o(1)$.
\end{proof}

\paragraph{Multiple pairs of singularities.\\}

Let us consider a more general situation, allowing for $n$ pairs of punctures. This is relevant for correlators with $2n$ electric insertions and will illustrate the formalism of Section \ref{Sec:surgery}.

Let $x_1,\dots,x_n$, $y_1,\dots,y_n$ be marked points, $s_1,\dots,s_n\in (-\frac 12,\frac 12)$. We consider $\rho:\pi_1(\C\setminus\{x_1,\dots,y_n\})\rightarrow\U$ the unitary character such that $\rho(\gamma_{x_j})=\rho(\gamma_{y_j})^{-1}=\chi_j=e^{2i\pi s_j}$ where $\gamma_{z}$ is a counterclockwise loop around $z\in\{x_1,\dots,y_n\}$ with no other marked point in its interior. We are interested in the associated operator $\rK:(\C^{M_B})_\rho\rightarrow (\C^{M_W})_\rho$. This may be realised using $n$ disjoint branch cuts running from $x_j$ to $y_j$, $j=1,\dots,n$.

Let us address the case $n=2$ using the surgery formalism (Lemma \ref{Lem:surgery}). Let $\gamma$ be a simple loop with $\{x_1,y_1\}$ in its interior $U_i$ and $\{x_2,y_2\}$ in its exterior $U_o$. Let $\rho_1$(resp. $\rho_2$) be the character corresponding to weights $(s_1,0)$ (resp. $(0,s_2)$). We have (by Lemma \ref{Lem:Srho2pts})
$$S_{\rho_j}(z,w)=\frac 12R_B\left(\left(\frac{z-x_j}{w-x_j}\cdot\frac{w-y_j}{z-y_j}\right)^{s_j}\frac{e^{i\nu(w)}}{\pi(z-w)}\right)
+\frac 12\bar R_B\left(\overline{\left(\frac{z-x_j}{w-x_j}\cdot\frac{w-y_j}{z-y_j}\right)}^{-{s_j}}\frac{e^{-i\nu(w)}}{\pi\overline{(z-w)}}\right)+o(|x-y|^{-1})$$
where all pairwise distances (between singularities $x_1,y_1,x_2,y_2$ and arguments of the kernel $z,w$) are of order 1. The Cauchy data spaces (see \eqref{eg:surgcontdata}) corresponding to the limiting continuous kernels are easy to identify: 
\begin{align*}
C_i&=\{f_{|\gamma}: f(z)=g(z)\left(\frac{z-x_1}{w-x_1}\cdot\frac{w-y_1}{z-y_1}\right)^{s_1}, (g_{\bar z})_{|U_i}=0\}\cap {\rm Lip}(\gamma)\\
\bar C_i&=\{f_{|\gamma}: f(z)=g(z)\overline{\left(\frac{z-x_1}{w-x_1}\cdot\frac{w-y_1}{z-y_1}\right)}^{-s_1}, (g_{z})_{|U_i}=0\}\cap {\rm Lip}(\gamma)\\
C_o&=\{f_{|\gamma}: f(z)=g(z)\left(\frac{z-x_2}{w-x_2}\cdot\frac{w-y_2}{z-y_2}\right)^{s_2}, (g_{\bar z})_{|U_o}=0, g(\infty)=0\}\cap {\rm Lip}(\gamma)\\
\bar C_o&=\{f_{|\gamma}: f(z)=g(z)\overline{\left(\frac{z-x_2}{w-x_2}\cdot\frac{w-y_2}{z-y_2}\right)}^{-s_2}, (g_{z})_{|U_o}=0, g(\infty)=0\}\cap {\rm Lip}(\gamma)
\end{align*}
Note that in general $\bar C_i\neq \overline{C_i}$. We check that $C_i\cap C_o=\bar C_i\cap \bar C_o=\{0\}$ and find the glued Cauchy kernels:
\begin{align*}
S_g(z,w)&=\frac{1}{\pi(z-w)}\prod_{j=1}^2\left(\frac{z-x_j}{w-x_j}\cdot\frac{w-y_j}{z-y_j}\right)^{s_j}\\
\bar S_g(z,w)&=\frac{1}{\pi\overline{(z-w)}}\prod_{j=1}^2\overline{\left(\frac{z-x_j}{w-x_j}\cdot\frac{w-y_j}{z-y_j}\right)}^{-s_j}
\end{align*}
By induction on $n$ (using surgery at each step to add a pair of insertions), we obtain
\begin{equation}\label{eq:Srhomult}
S_\rho(z,w)=\frac 12R_B\left(\frac{e^{i\nu(w)}}{\pi(z-w)}\prod_{j=1}^n\left(\frac{z-x_j}{w-x_j}\cdot\frac{w-y_j}{z-y_j}\right)^{s_j}\right)
+\frac 12\bar R_B\left(\frac{e^{-i\nu(w)}}{\pi\overline{(z-w)}}\prod_{j=1}^n\overline{\left(\frac{z-x_j}{w-x_j}\cdot\frac{w-y_j}{z-y_j}\right)}^{-{s_j}}\right)+o(|x-y|^{-1})
\end{equation}
with uniform convergence on compact subsets of $\Sigma^2\setminus \Delta_\Sigma$, $\Sigma=\C\setminus\{x_1,\dots,y_n\}$. Through surgery, we also retain invertibility (i.e. $S_\rho$ is uniquely characterised by $\rK S_\rho(.,w)=\delta_w$, $S_\rho(.,w)$ vanishes at infinity), at least for small enough $\delta$.

\paragraph{Error estimates away from singularities.\\}

In Lemma \ref{Lem:Srho2pts} - and in greater generality in \eqref{eq:Srhomult} - we established (existence and) convergence of the kernel $S_\rho$, without a quantitative control on the error (this originates in compactness arguments in Lemma \ref{Lem:Green2pts}).
We would like to gain an explicit error estimate for $S_\rho$. The following is a simple (and rather crude) estimate, which will be enough for our purposes.

\begin{Lem}\label{Lem:Srhoquant}
Let $x_1,y_1,\dots,x_n,y_n$ be $n$ pairs of singularities, $s_1,\dots,s_n\in(-\frac 12,\frac 12)$, $\rho:\pi_1(\C\setminus\{x_1,\dots,y_n\})\rightarrow\U$ the corresponding character. 

For $\delta$ small enough, for $w\in M_W$ there is a unique $\rho$-multivalued function $\gls{Srho}(.,w)$ vanishing at infinity s.t. $\rK S_\rho(.,w)=\delta_w$. Moreover, we have
$$S_\rho(z,w)=\frac 12R_B\left(\frac{e^{i\nu(w)}}{\pi(z-w)}\prod_{j=1}^n\left(\frac{z-x_j}{w-x_j}\cdot\frac{w-y_j}{z-y_j}\right)^{s_j}\right)
+\frac 12\bar R_B\left(\frac{e^{-i\nu(w)}}{\pi\overline{(z-w)}}\prod_{j=1}^n\overline{\left(\frac{z-x_j}{w-x_j}\cdot\frac{w-y_j}{z-y_j}\right)}^{-{s_j}}\right)+O(\delta^{1-2\max_j|s_j|})$$
uniformly on compact sets of $\Sigma^2\setminus\Delta_\Sigma$, $\Sigma=\C\setminus\{x_1,\dots,y_n\}$.
\end{Lem}
\begin{proof}
Existence and uniqueness of $S_\rho$ were obtained in the previous surgery argument.

First we estimate $S_\rho(b,w)$ for $w$ near a singularity, $b$ in a compact set of $\Sigma$. From 
$$\uK_\chi^{-1}(b,w)=O(|w|^{-s})$$
for $\chi=e^{2i\pi s}$, $s\in (0,\frac 12)$ (with singularity at $0$, see Lemma \ref{Lem:chirkernel0}), and the replication identity:
\begin{equation}\label{eq:LemSrhoquantrepl}
S_\rho(.,w)=\uK_{\chi_1,x_1}^{-1}(.,w)\ind_{B(x_1,r)}+S_\rho(\rK(\uK_{\chi_1,x_1}^{-1}(.,w)\ind_{B(x_1,r)})-\delta_w)
\end{equation}
we get the estimate: $S_\rho(b,w)=O(|w|^{|s_1|-1})$ for $w$ close to $x_1$, $b$ in a compact subset of $\Sigma$ (note that $-|s_1|>|s_1|-1$).

With similar (and simpler) arguments (using Theorem \ref{Thm:Kencrit} instead of Lemma \ref{Lem:chirkernel0}) , we obtain the estimate: $S_\rho(b,w)=O(|b-w|^{-1})$ for $b,w$ in a compact subset of $\hat\C\setminus\{x_1,\dots,y_n\}$.

Now fix $w_0$ away from the singularities; we wish to estimate $S_\rho(.,w_0)$. Let us set 
$$\tilde S_\rho(.,w_0)=\uK^{-1}(.,w_0)+\frac 12R_B(a+bz)+\frac 12\bar R_B(a'+b'\bar z)$$ 
in $B(w_0,r)$; and
$$\tilde S_\rho(z,w_0)=\frac 12R_B\left(\frac{e^{i\nu(w)}}{\pi(z-w_0)}\prod_{j=1}^n\left(\frac{z-x_j}{w_0-x_j}\cdot\frac{w_0-y_j}{z-y_j}\right)^{s_j}\right)
+\frac 12\bar R_B\left(\frac{e^{-i\nu(w)}}{\pi\overline{(z-w_0)}}\prod_{j=1}^n\overline{\left(\frac{z-x_j}{w_0-x_j}\cdot\frac{w_0-y_j}{z-y_j}\right)}^{-{s_j}}\right)$$
elsewhere ($a,b,a',b',r$ to be fixed). We have (see below)
\begin{equation}\label{eq:LemSrhoquant1}
S_\rho(z,w_0)-\tilde S_\rho(z,w_0)=\sum_{w\neq w_0}(\rK(\tilde S_\rho(.,w_0))(w)S_\rho(z,w)
\end{equation}
On $\partial B(w_0,r)$, we have $\rK((\tilde S_\rho(.,w_0))=O(\delta r^2)$ for an appropriate choice of $a,b,a',b'$ (simply by matching Taylor polynomials). For $r<|w-w_0|\ll 1$, we have
$\rK((\tilde S_\rho(.,w_0))(w)=O(\delta^4|w-w_0|^{-4})$ by \eqref{eq:findiff}. For $w$ in a compact subset of $\Sigma\setminus\{w_0\}$, we have $\rK((\tilde S_\rho(.,w_0))(w)=O(\delta^4)$, also by \eqref{eq:findiff}. 

For $|w-w_0|\gg1$, $\rK((\tilde S_\rho(.,w_0))(w)=O(\delta^4|w-w_0|^{-4})$. Remark that this ensures that the RHS in \eqref{eq:LemSrhoquant1} is convergent. Using the basic bound $S_\rho(z,w)=O(|z-w|^{-1})$ for $z$ and $w$ large enough, it follows that the RHS of \eqref{eq:LemSrhoquant1} goes to zero as $z\rightarrow\infty$. Together with the unique characterisation of $S_\rho$, this justifies \eqref{eq:LemSrhoquant1}.

Finally, for $w$ close to $x_j$,
$\rK((\tilde S_\rho(.,w_0))(w)=O(\delta^4|w-x_j|^{-|s_j|-3})$ (and similarly near $y_j$). This leads to:
\begin{align*}
S_\rho(z,w_0)-\tilde S_\rho(z,w_0)&=O((r\delta^{-1})(\delta r^2)r^{-1})+O(\delta^{-2}\delta^4)+O(\sum_{k=\delta^{-1}}^\infty k(\delta^4(k\delta)^{-4})(k\delta)^{-1})\\
&+\sum_{j=1}^nO(\sum_{k=1}^{\delta^{-1}} k\delta^4(k\delta)^{-|s_j|-3}(k\delta)^{-|s_j|})
\end{align*}
(the error terms coming respectively from $w$ near $w_0$, away from singularities, near $\infty$, and near a singularity in \eqref{eq:LemSrhoquant1}). Thus, by setting $r=\sqrt\delta$, we get
$$S_\rho(z,w)-\tilde S_\rho(z,w)=O(\delta^{1-2\max_j |s_j|})$$
as stated.
\end{proof}

\paragraph{Estimates near singularities.\\}

Our estimates so far (Lemma \ref{Lem:Srhoquant}) describe the asymptotics of the kernel $S_\rho$ when its arguments are away of each other and of singularities. We now need to analyse what happens in the most singular situation, i.e. when both arguments are adjacent to the same singularity. Lemma \ref{Lem:monzero} allows to transfer information from macroscopic to microscopic scale (w.r.t. the singularity).

If $x\in\{x_1,\dots,y_n\}$ is one of the singularities and $s\in(-\frac 12,\frac 12)$ is the corresponding exponent, we expect that the leading behaviour of $S_\rho$ near $x$ is given by $\rK^{-1}_{\chi,x}$ (with $\chi=e^{2i\pi s}$). This leading term depends only on the local data (position of the singularity and exponent); the subleading term, which we will estimate, accounts for the global data (position of the other singularities and values of the other exponents). Compare e.g. with \eqref{eq:diagestplane} or \eqref{eq:torusdiagest}.

In the continuum, let us consider the ``Robin kernels":
\begin{align*}
r_\rho(w)&=\lim_{z\rightarrow w}\left(\frac{1}{\pi(z-w)}\prod_{j=1}^n\left(\frac{z-x_j}{w-x_j}\cdot\frac{w-y_j}{z-y_j}\right)^{s_j}-\frac{1}{\pi(z-w)}\right)&=\frac 1\pi\sum_{j=1}^n\left(\frac{s_j}{w-x_j}-\frac{s_j}{w-y_j}\right)\\
\bar r_\rho(w)&=
\lim_{z\rightarrow w}\left(\frac{1}{\pi\overline{(z-w)}}\prod_{j=1}^n\overline{\left(\frac{z-x_j}{w-x_j}\cdot\frac{w-y_j}{z-y_j}\right)}^{-s_j}-\frac{1}{\pi\overline{(z-w)}}\right)
&=\frac 1\pi\sum_{j=1}^n\left(\frac{-s_j}{\overline{w-x_j}}-\frac{-s_j}{\overline{w-y_j}}\right)
\end{align*}
(so that $\bar r_\rho(w)=\overline{ r_{\bar\rho}(w)}$).

\begin{Lem}\label{Lem:monrobin}
If $b,w$ are adjacent to the singularity $x\in\{x_1,\dots,y_n\}$, $s=\pm s_k$ the corresponding exponent ($s=s_k$ for $x=x_k$ and $s=-s_k$ for $x=y_k$), and $b\in\Gamma$, we have
$$S_\rho(b,w)-\uK_{\chi,x}^{-1}(b,w)=\frac{2s}{1-\bar\chi}\left(\frac{x-b}{w-x}\right)^s\Re\left(ie^{i\nu(w)}\xi\right)%
+O(\delta^\eps)$$
where $\eps=\eps(s)$ is positive for $s_1,\dots,s_n$ small enough and
$$\xi=-\frac{s_k}{x_k-y_k}+\sum_{j:j\neq k}\left(\frac{s_j}{x_k-x_j}-\frac{s_j}{x_k-y_j}\right)=\lim_{w\rightarrow x}\left(\pi r_\rho(w)-\frac s{w-x}\right)$$
\end{Lem}
\begin{proof}
Let $w$ be adjacent to the singularity $x=x_k$ (which is in a face of $M$, see Figure \ref{fig:monloc}); by translating we may assume $x_k=0$. Let $r$ be of order 1 and small enough so that other singularities are outside of $\overline{B(0,r)}$. Also set $s=s_k$, $\chi=e^{2i\pi s}$. We consider $f(b)=S_\rho(b,w)-\uK_{\chi}^{-1}(b,w)\ind_{B(0,r)}$ which is discrete holomorphic except near $\partial B(0,r)$ and is $\rho$-multivalued. The goal is to estimate $f(b)$ for $b=O(\delta)$ (in particular for $b$ adjacent to the singularity). We have the Cauchy integral formula
$$f(b)=\sum_{w'}S_\rho(b,w')\rK(f)(w')=\sum_{w'\neq w}S_\rho(b,w')\rK(\uK_\chi^{-1}(.,w)\ind_{B(0,r)})(w')$$
for $b\in B(0,r)$ (e.g. from uniqueness in Lemma \ref{Lem:Srhoquant}). From Lemma \ref{Lem:monCauchy} and estimates on $S_\rho$ (Lemma \ref{Lem:Srhoquant}), we obtain for $b\in A(\frac r4,\frac r2)$:
\begin{align*}
f(z)&=\frac{\Gamma(1-s)e^{i\nu(w)}}2R_B\frac{1}{\pi(z-w)}\left(
\prod_{j=1}^n\left(\frac{z-x_j}{w-x_j}\cdot\frac{w-y_j}{z-y_j}\right)^{s_j}
-\left(\frac{z-x_k}{w-x_k}\right)^{s_k}\right)\\
&\hphantom{=}+\frac{\Gamma(1+s)e^{-i\nu(w)}}2\bar R_B
\frac{1}{\pi\overline{(z-w)}}\left(
\prod_{j=1}^n\overline{\left(\frac{z-x_j}{w-x_j}\cdot\frac{w-y_j}{z-y_j}\right)}^{-s_j}
-\overline{\left(\frac{z-x_k}{w-x_k}\right)}^{-s_k}
\right)\\
&\hphantom{=}+O(\delta^{2-|s|}+O(\delta^{1-|s|-2\max_j|s_j|})\\
&=R_B(\psi_1(z))+\overline R_B(\psi_2(z))+O(\delta^{1-3\max_j|s_j|})
\end{align*}
where $\psi_1(z)=z^s\sum_{n\geq 0}a_n z^n$, $\psi_2(z)=\bar z^{-s}\sum_{n\geq 0}b_n \bar z^n$,
\begin{align*}
a_0&=\frac{\Gamma(1-s)e^{i\nu(w)}}{2\pi w^s}\left(-\frac{s_k}{x_k-y_k}+\sum_{j\neq k}\left(\frac{s_j}{x_k-x_j}-\frac{s_j}{x_k-y_j}\right)\right)(1+O(\delta))\\
b_0&=\frac{\Gamma(1+s)e^{-i\nu(w)}}{2\pi \bar w^{-s}}\left(\frac{s_k}{\overline{x_k-y_k}}+\sum_{j\neq k}\left(\frac{-s_j}{\overline{x_k-x_j}}-\frac{-s_j}{\overline{x_k-y_j}}\right)\right)(1+O(\delta))
\end{align*}
(Recall that Lemma \ref{Lem:monCauchy} is written for $\delta=1$, $s\in(0,\frac 12)$. Here, on $\partial B(0,r)$, $\rK^{-1}_\chi(.,w)=O(\delta^{-|s|})$; the approximation error on $S_\rho$ is $O(\delta^{1-2\max_j|s_j|})$. This gives the main contribution to the error term $O(\delta^{1-3\max_j|s_j|})$, with smaller contributions coming from the approximation error on $\rK_\chi^{-1}$, and the Riemann sum approximation.)

By Lemma \ref{Lem:monzero}, we conclude:
$$f(b)=\frac{b_0}{\bar\tau 2^{-s}\Gamma(s)}f_\chi(b)+\frac{a_0}{\tau 2^s\Gamma(-s)}h_\chi(z)+O(\delta^{1-4\max_j|s_j|-\eta})$$
for $b=O(1)$, $\eta>0$ fixed. Taking into account Lemma \ref{Lem:monholom}, we have for $b$ adjacent to the singularity:
\begin{align*}
f_\chi(b)&=\frac{2i\pi e^{i\nu(b)}}{\bar \chi-1}\overline{2(x-b)}^{1-s}\\
h_\chi(b)&=\overline{f_{1,\bar\chi}(b)}=-\frac{2i\pi e^{-i\nu(b)}}{\bar \chi-1}{(2(x-b))}^{1+s}
\end{align*}
and $\tau=\frac 2\delta(b-x)$ if $b\in M_V$ adjacent to $x$. (Recall that $u\Gamma(u)=\Gamma(u+1)$).
\end{proof}

\subsection{Variational analysis}\label{ssec:monovariat}

Here we apply the technical results of the previous subsections (culminating in Lemma \ref{Lem:monrobin}) to the asymptotics of dimer electric correlations. As mentioned earlier, there are two distinct types of variation we will consider: displacing a puncture while fixing the character $\chi$ (which, given our estimates, is enough to handle the case where $\chi$ is close enough to $1$); and varying $\chi$ for a fixed position of the punctures. Finally we consider a more general situation with several pairs of punctures, building in particular on Section \ref{Sec:surgery}.

\subsubsection{Displacing the punctures}

The previous local analysis (Lemma \ref{Lem:monrobin}) may be used to estimate the pair ``electric correlator" 
$$\langle \chi^{h(y)-h(x)}\rangle=\langle\exp(2i\pi s(h(y)-h(x)))\rangle$$
for $x,y\in M^\dg$ at large distance, where $h$ is the height function associated with a matching ${\mf m}$ of $M$ (see Section \ref{ss:height} - the choice of normalisation is crucial here) and $\chi=e^{2i\pi s}$ is fixed ($s\in (0,\frac 12)$). This is a natural discrete analogue of \eqref{eq:electcorrGFF}.

More precisely, we are interested in estimating $\langle \chi^{h(y')-h(x)}\rangle/\langle\chi^{h(y)-h(x)}\rangle$ within $o(|y'-y|/|x-y|)$, i.e. logarithmic first differences. 

\paragraph{Set-up.}

Fix $x\in M^\dg$ and a simple path $\gamma'$ from $x$ to $y'$ on $M^\dg$; and that the penultimate vertex of $\gamma'$ is $y$: $\gamma'=(x,\dots,y,y')$. We also denote $\gamma=(x,\dots,y)$, the path stopped at $y$. Let $E_r=E_r(\gamma)$ be the set of edges of $M$ crossed by $\gamma$, with the black vertex on the righthand side of $\gamma$; and $E_l=E_l(\gamma)$ the edges crossed by $\gamma$, with the black vertex on the lefthand side. Then (see Section 5 in \cite{KPW})
$$h(y)-h(x)=\frac1{2\pi}{\rm wind}(\gamma)+\sum_{e\in E_r}\ind_{e\in{\mf m}}-\sum_{e\in E_l}\ind_{e\in{\mf m}}
$$ 
In order to define the winding number of $\gamma$, we draw it as a succession of smooth arcs connecting midpoints of segments $(bb')$ (ie centers of faces of $M$), where $b\in\Gamma$ and $b'\in\Gamma^\dg$, in such a way that $\gamma$ crosses these segments normally (see Figure \ref{Fig:winding}).\\
\begin{figure}[htb]
\begin{center}
\leavevmode
\includegraphics[width=0.7\textwidth]{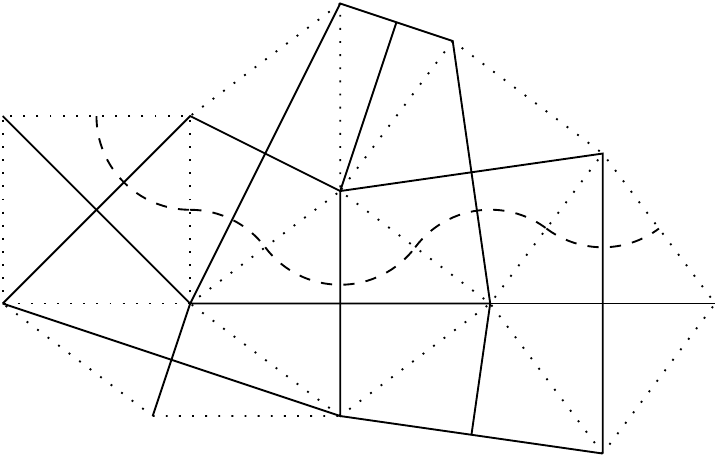}
\end{center}
\caption{Dotted: (portion of a) lozenge tiling $\Lambda$. Solid: graph $M$. Dashed, curved: path $\gamma$ crossing normally the edges of $\Lambda$.
}
\label{Fig:winding}
\end{figure}
Let $K_\gamma:M_B\rightarrow M_W$ be the operator defined by $K_\gamma(b,w)=\bar\chi \rK(b,w)$ if $(bw)\in E_r(\gamma)$; $K_\gamma(b,w)=\chi \rK(b,w)$ if $(bw)\in E_l(\gamma)$; and $K_\gamma(b,w)=\rK(b,w)$ otherwise.\\
Let us remark that if $\tilde\gamma$ is another path from $x$ to $y$, then $K_{\tilde\gamma}$ is conjugate to $K_\gamma$ by diagonal matrices involving the winding number of the loop obtained by concatenating $\gamma$ with $\tilde\gamma$ (in reverse orientation). Moreover, by using $\gamma$ as a branch cut (lifting functions on $M$ to functions in $(\C^M)_\rho$, where $\rho$ is the unitary character of $\pi_1(\C\setminus\{x,y\})$ defined as in Section \ref{sss:moninv}), one can identify $K_\gamma$ with $\rK$ operating on $(\C^{M_B})_\rho$. Correspondingly, we get an inverting kernel $S_\gamma:\C^{M_W}\rightarrow\C^{M_B}$. 

\begin{Lem}\label{Lem:fredelectr}
$K_\gamma\uK^{-1}:M_W\rightarrow M_W$ is a finite rank perturbation of the identity and
$$\langle \chi^{h(y)-h(x)}\rangle=e^{is.{\rm wind}(\gamma)}\det(K_\gamma\uK^{-1})$$
Moreover, for $|s|<\frac 14$, $\langle  \chi^{h(y)-h(x)}\rangle\neq 0$.
\end{Lem}
\begin{proof}
For the first statement, see the proof of Lemma \ref{Lem:detplane}.\\
We may write
$$|\langle \chi^{h(y)-h(x)}\rangle|^2=\langle\chi^{(h-\tilde h)(y)-(h-\tilde h)(x)}\rangle$$
where $\tilde h$ is the height field of an independent sample of the dimer model. The superimposition of the two dimer configurations (corresponding to $h,\tilde h$) is a so-called {\em double-dimer configuration}. A classical argument of ``cycle rotation" shows that
$$ |\langle \chi^{h(y)-h(x)}\rangle|^2=\E\left((\Re(\chi)^{N_{xy}}\right)$$
where $N_{xy}$ is the number of double-dimer loops separating $x$ from $y$, which is bounded by the graph distance between $x$ and $y$ on $M^\dg$. (In order to avoid technical difficulties for infinite volume dimer measures, one can apply the arguments to dimers in a large box and then take a weak limit).

We give an alternative argument, for $\chi\neq 1$ arbitrary and $|y-x|$ large enough, based on the existence and uniqueness of $S_\rho$, see Lemma \ref{Lem:Srho2pts} (this is in a sense more self-contained and extends directly to the case with $n$ pairs of singularities).

Let $\gamma_W$ be the (finite) set of vertices of $M_W$ adjacent to $\gamma$ (i.e. those for which $K_\gamma(w,.)\neq\rK(w,.)$). We may regard $\det(K_\gamma\uK^{-1})$ as the determinant of an operator on $\C^{\gamma_W}$:
$$D_1:f\longmapsto (K_\gamma(\sum_{w\in \gamma_W}f(w)\uK^{-1}(.,w)))_{|\gamma_W}$$
It is enough to construct an inverse of this operator. Consider
$$D_2:g\longmapsto (\rK(\sum_{w\in \gamma_W}g(w)S_\gamma(.,w)))_{|\gamma_W}$$
By construction $S_\gamma$ is a right inverse of $K_\gamma$ in that $K_\gamma S_\gamma(.,w)=\delta_w$ for all $w\in M_W$. It is also a left inverse: $\sum_wS_\gamma(.,w)K_\gamma(w,b)=\delta_b$ for all $b\in M_B$ (otherwise the difference would give a $\rho$-multivalued discrete holomorphic function vanishing at infinity).

Remark that for $w\in\gamma_W$, $K_\gamma(\uK^{-1}(.,w))=0$ on $M_W\setminus\gamma_W$. Then $S_\gamma(K_\gamma(\uK^{-1}(.,w)))=\uK^{-1}(.,w)$, and it follows that $D_2$ is a left inverse of $D_1$.
\end{proof}

\paragraph{Asymptotics.\\}

Here we (finally) apply the analysis of discrete holomorphic functions with monodromy (Lemmas \ref{Lem:chirharmest} to \ref{Lem:monrobin}) to dimers.

\begin{Prop}\label{Prop:electrsmalls}
For $s$ small enough, there exists a $c=c(\Lambda,s)> 0$  and $\eps>0$ such that 
$$|\langle \chi^{h(y)-h(x)}\rangle|=|x-y|^{-2s^2}(c+O(|x-y|^{-\eps}))$$
\end{Prop}
\begin{proof}
We analyse the effect of taking the last step along $\gamma'$ (notations are as before Lemma \ref{Lem:fredelectr}). As mentioned earlier, the goal is to estimate precisely enough
$$\frac{\langle \chi^{h(y')-h(x)}\rangle}{\langle \chi^{h(y)-h(x)}\rangle}=
\frac{\det(K_{\gamma'}\uK^{-1})}{\det(K_{\gamma}\uK^{-1})}
$$
(by Lemma \ref{Lem:fredelectr}).

Let $(bw)$ be the edge of $M$ separating $y,y'\in M^\dg$. Recall (from Lemma \ref{Lem:Srho2pts}) that $S_{\gamma}(.,w)$ can be characterised as the only function in $\C^{M_B}$ vanishing at infinity such that $K_\gamma S_{\gamma}(.,w)=\delta_w$; $S_{\gamma'}(.,w)$ is characterised similarly. Since $K_{\gamma'}-K_\gamma$ is supported on $w$ and 
$$(K_{\gamma}S_{\gamma'}(.,w))(w)=1+(1-\chi^{\pm 1}) \rK(w,b)S_{\gamma'}(b,w)$$
- here $\chi^{\pm 1}=\chi$ (resp. $\chi^{-1}$) if $b$ is to the left (resp. to the right) of $\gamma'$ - we deduce
$$S_{\gamma'}(.,w)=(1+(1-\chi^{\pm 1}) \rK(w,b)S_{\gamma'}(b,w))S_\gamma(.,w)$$

Let us observe that $K_{\gamma'} S_{\gamma}=\Id+(K_{\gamma'}-K_{\gamma})S_{\gamma}$ is a rank 1 perturbation of the identity and that
$$(K_{\gamma'} S_{\gamma})(K_{\gamma}\uK^{-1})=K_{\gamma'} \uK^{-1}$$
as bounded operators in $L^2(M_W,\mu)$ where $\mu\{w\}=(1+|w|)^{-\eps}$ for some $\eps>0$. This follows from $S_{\gamma}K_{\gamma}=\Id$ (indeed, $(S_{\gamma}K_{\gamma})\delta_b-\delta_b$ is in the kernel of $K_{\gamma}$ and vanishes at infinity, hence is identically zero). Consequently,
$$\frac{\det(K_{\gamma'}\uK^{-1})}{\det(K_{\gamma}\uK^{-1})}
=\det(K_{\gamma'}S_\gamma)=\det(\Id+(K_{\gamma'}-K_\gamma)S_\gamma)=1+(K_{\gamma'}-K_\gamma)(w,b)S_\gamma(b,w)=1+(\chi^{\pm 1}-1)\rK(w,b)S_\gamma(b,w)
$$
since $K_{\gamma'}-K_\gamma$ has rank 1. Then we have the key estimate (Lemma \ref{Lem:monrobin})
$$S_\rho(b,w)-\uK_{\bar\chi,y}^{-1}(b,w)=\frac{2s^2}{1-\chi}\left(\frac{w-y}{y-b}\right)^s\Re\left(ie^{i\nu(w)}/(y-x)\right)+O(|x-y|^{-1-\eps})$$
(Recall that Lemma \ref{Lem:monrobin} is written for fixed singularities and mesh $\delta\searrow 0$. Here we may set $M_\delta=(y-x)^{-1}M$, $\delta=|y-x|^{-1}$. Scaling by $\delta$ multiplies matrix elements of $\rK$ by $\delta$ and those of $S_\rho$ by $\delta^{-1}$).

Together with $\pm i\rK(w,b)=y'-y$, $\frac{w-y}{y-b}=\frac{y-w}{y'-w}$ (Section \ref{sss:critgraphs}) and the exact result (Lemma \ref{Lem:monCauchy})
$$\rK(w,b)\uK^{-1}_{\bar\chi,y}(b,w)=\frac 1{1-\chi}\left(1-\left(\frac{y'-w}{y-w}\right)^{-s}\right)$$
this yields
$$\frac{\det(K_{\gamma'}\uK^{-1})}{\det(K_{\gamma}\uK^{-1})}
=\left(\frac{y'-w}{y-w}\right)^{-s}\left(1-2s^2\Re\left(\frac{y'-y}{y-x}\right)\right)+O(|x-y|^{-1-\eps})$$
From the previous computation we get
$$\left|\frac{\langle\chi^{h(y')-h(x)}\rangle\cdot|y-x|^{-2s^2}}
{\langle\chi^{h(y)-h(x)}\rangle\cdot|y'-x|^{-2s^2}}\right|=1+O(|x-y|^{-1-\eps})$$
with $\eps>0$ for $s$ small enough.

To wrap up the argument, one can e.g. fix $x$ and set
$$\varphi_x(y)=\log|\langle\chi^{h(y)-h(x)}\rangle|+2s^2\log|y-x|$$
for $y\in M^\dg\setminus\{x\}$. Then we have 
$$\varphi_x(y')-\varphi_x(y)=O(|y-x|^{-1-\eps})$$
which trivially implies that $\varphi_x(y)$ has a finite limit $\log(c(x,\Lambda))$ as $y\rightarrow\infty$. (Indeed, if $y_n\rightarrow\infty$ with $y_n=n+O(1)$, then $|\varphi_x(y_{m+n})-\varphi_x(y_m)|=O(m^{-\eps})$; and $|\varphi_x(y)-\varphi_x(y_n)|=O(n^{-\eps})$ if $y=n+O(1)$, since the graph distance on $M^\dg$ and the Euclidean distance are comparable). More precisely,
$$\varphi_x(y)=\log(c(x,\Lambda))+O(|y-x|^{-\eps})$$
Remark that the error term is uniform in $(\Lambda,x)$ under $(\spadesuit)$ (see \eqref{eq:spade}), since Lemma \ref{Lem:monrobin} holds along any sequence of tilings $(\Lambda_{\delta_n})$, $\delta_n\searrow 0$. Moreover, for $x'\sim x$,
$$\varphi_x(y)-\varphi_{x'}(y)=O(|y-x|^{-1})$$
(exchanging the roles of $x$ and $y$) and consequently $c(x,\Lambda)=c(x',\Lambda)$ does not depend on $x$, which concludes.
\end{proof}

More generally, if $\{z_1,\dots,z_{2n}\}=\{x_1,\dots,y_n\}$ are marked points, $s_1,\dots,s_{2n}>0$ are small enough exponents with $s_1+s_2=\cdots=s_{2n-1}+s_{2n}=0$, we have (by ``integrating" Lemma \ref{Lem:monrobin}, as in Proposition \ref{Prop:electrsmalls}; a more general statement will be given in Theorem \ref{Thm:electr})
$$\left\langle\exp(2i\pi\sum_{j=1}^{2n} s_jh(z_j))\right\rangle=\prod_{i<j}|z_i-z_j|^{2s_is_j}(c(\Lambda)+O(R^{-\eps}))$$
where the pairwise distances $|z_i-z_j|$ are all comparable to $R\gg 1$.

This in agreement with the electric vertex correlator heuristic (see \eqref{eq:electcorrGFF}), i.e. what we would get by replacing $h(z_j)$ by the average of a scalar free field $\phi$ on a {\em microscopic} ball centred at $z_j$.  

The pairing of insertions $(z_1,z_2)$,\dots, $(z_{2n-1},z_{2n})$ is somewhat restrictive. For instance, it is not obvious how to treat correlators such as 
$$\langle\exp(\frac{2i\pi}3(h(x)+h(y)+h(z))\rangle$$
or even what to expect (as compactification may start to play a role). Indeed, up to a phase, one can write
$$\langle\exp(\frac{2i\pi}3(h(x)+h(y)+h(z))\rangle
=\langle\exp(\frac{2i\pi}{3}(h(x)-h(z)+h(y)-h(z)))\rangle
$$
and if $z'$ is close to, but at macroscopic distance of, $z$, we have asymptotically
$$\langle\exp(\frac{2i\pi}{3}(h(x)-h(z)+h(y)-h(z')))\rangle
\sim |(x-z)(x-z')(y-z)(y-z')|^{-\frac 29}|(x-y)(z-z')|^{\frac 29}
$$
Taking $z'\rightarrow z$ yields the heuristic
$$\langle\exp(\frac{2i\pi}3(h(x)+h(y)+h(z))\rangle
\stackrel{??}{\sim} |(x-z)(y-z)|^{-\frac 49}|x-y|^{\frac 29}
$$
which is plainly incorrect (as it is not symmetric in $x,y,z$).

\subsubsection{Varying the exponent}

We turn to the case where $s$ is not small, and for simplicity we first discuss the case of two marked points. Remark that all previous estimates are uniform in $s$ for $s$ in a compact interval of $(0,\frac 12)$. Again we use a variational argument; now $x,y$ are fixed and the exponent $s$ is varying.
\begin{Prop}\label{Prop:electrpairs}
If $s\in(0,\frac 12)$, there is a constant $c(\Lambda,s)$ such that 
$$|\langle\chi^{h(y)-h(x)}\rangle|=\exp(2c(\Lambda,s))|x-y|^{-2s^2}(1+o(1))$$
as $|x-y|\rightarrow\infty$.
\end{Prop}
\begin{proof}
Starting from $|\langle \chi^{h(y)-h(x)}\rangle|=|\det(K_\gamma\uK^{-1})|$ (see Lemma \ref{Lem:fredelectr}), where $\chi=e^{2i\pi s}$ and $K_\gamma$ is implicitly a function of $s$, we get
$$\frac{d}{ds}\log|\langle \chi^{h(y)-h(x)}\rangle|=\Re\Tr(\dot{K_\gamma}S_{\gamma})$$
and $\dot{K_\gamma}(w,b)=\pm 2i\pi\chi^{\pm 1}\rK(w,b)$ if $\gamma$ crosses the edge $(wb)$. Hence we need to evaluate $S_\gamma(w,b)$ for these edges. We have (uniformly in $(z,w)$ in a compact subset of $\{(z,w):z\neq x,y;w\neq x,y;z\neq w\}$; and in $s$ in a compact subset of $(0,\frac 12)$)
$$S_\gamma(z,w)=\frac 12R_B\left(\left(\frac{z-x}{z-y}\right)^s\left(\frac{w-x}{w-y}\right)^{-s}\frac{e^{i\nu(w)}}{\pi(z-w)}\right)
+\frac 12\bar R_B\left(\overline{\left(\frac{z-x}{z-y}\right)^{-s}}\overline{\left(\frac{w-x}{w-y}\right)^{s}}\frac{e^{-i\nu(w)}}{\pi\overline{(z-w)}}\right)+o(|x-y|^{-1})$$
by Lemma \ref{Lem:Srhoquant}. (Recall that that Lemma \ref{Lem:Srhoquant} is written for small mesh and fixed punctures; by scaling we get an estimate for fixed mesh and large separation of punctures). Here $\left(\frac{u-x}{u-y}\right)^{s'}=\exp(s'\log((u-x)/(u-y)))$ and $u\mapsto\log((u-x)/(u-y))$ is a determination of with branch cut along $\gamma$. 

Since
$$\left(\frac{z-x}{z-y}\right)^s\left(\frac{w-x}{w-y}\right)^{-s}\frac{1}{z-w}=\frac{1}{z-w}+s\left(\frac 1{w-x}-\frac 1{w-y}\right)+O(|z-w|/|w-x|^2+|z-w|/|w-y|^2)$$ 
we obtain that, for $b\sim w$ (by \eqref{eq:dirprobcauchy})
\begin{equation}\label{eq:Sgammadiag}
S_\gamma(b,w)-\uK^{-1}(b,w)=\frac{s}{2\pi} R_B\left(e^{i\nu(w)}\left(\frac{1}{w-x}-\frac{1}{w-y}\right)\right)-\frac{s}{2\pi}\bar R_B\left(e^{-i\nu(w)}\left(\frac{1}{\overline{w-x}}-\frac{1}{\overline{w-y}}\right)\right)+o(1/|x-y|)
\end{equation}
except if $(bw)$ crosses $\gamma$, in which case the LHS is replaced with $\chi^{\pm 1}(S_\gamma(b,w)-\uK^{-1}(b,w))$.\\
Let $\gamma'$ be a subpath of $\gamma$ running from $x'$ to $y'$, and $E(\gamma')$ the set of edges of $M$ crossed by $\gamma'$. If $(vv')=(wb)^\dg$ (as oriented edges), we have $K(w,b)=i(v-v')$ and $\rK(w,b)=e^{-i\nu(w)-i\nu(b)}i(v-v')$ (see Section \ref{sss:critgraphs}). Thus if $v_i,v_{i+1}$ are consecutive points on $\gamma'$ corresponding to the unoriented edge $(wb)$,
$$\pm \rK(w,b)(R_B(e^{i\nu(w)}\alpha)+\bar R_B(e^{-i\nu(w)}\beta))=
i\alpha(v_i-v_{i+1})-i\beta\overline{(v_i-v_{i+1})}=
-i\alpha\int_{v_i}^{v_{i+1}}dz+i\beta\int_{v_i}^{v_{i+1}}d\bar z$$
where $\pm=+$ if $(bw)\in E_r(\gamma)$ and $-$ if $(bw)\in E_r(\gamma)$. Then
\begin{align*}
\Re\sum_{(bw)\in E(\gamma')}\dot{K_\gamma}(b,w)S_\gamma(b,w)&=-s\Re\int_{x'}^{y'}\left(\frac{1}{z-x}-\frac{1}{z-y}\right)dz-s\Re\int_{x'}^{y'}\left(\frac{1}{\overline{z-x}}-\frac{1}{\overline{z-y}}\right)d\bar z+o(1)\\
&=2s\Re\log\left(\frac{(y'-x)(x'-y)}{(x'-x)(y'-y)}\right)+o(1)
\end{align*}
(one may also keep track of the imaginary part, since $\rK(w,b)\uK^{-1}(b,w)=p(w,b)$ is expressed in terms of the local gometry). We still need to address the logarithmic singularities at the endpoints of $\gamma$. A discrete Cauchy formula (see \eqref{eq:LemSrhoquantrepl}) together with estimates on $S_\rho(z,w)$ for $z$ close to $x$ and $|w-x|$ comparable to $|x-y|$ (Lemma \ref{Lem:Srhoquant} and Lemma \ref{Lem:monzerogr}) leads to the estimate
$$S_\rho(z,w)-\uK^{-1}_{\chi,x}(z,w)=O(\frac{|x-y|^{s-1}}{|w-x|^s}\cdot\frac{|x-y|^s}{|w-x|^s})=O(|x-y|^{2s-1}|w-x|^{-2s})$$
for $|w-x|\leq \frac 14|x-y|$, $\frac 12|w-x|\leq |z-x|\leq 2|w-x|$. The same estimate holds when interverting $x$ and $y$. Consequently, if we choose $\gamma$ such that the lengths of its segments are comparable to their Euclidean length, and $\gamma_i$ is the initial segment of $\gamma$ (between $x$ and $x'$), we have
$$\sum_{(bw)\in E(\gamma_i)} \dot{K_\gamma}(b,w)S_\gamma(b,w)=
\sum_{(bw)\in E(\gamma_i)} \dot{K_\gamma}(b,w)\uK_{\chi,x}^{-1}(b,w)+O((|x'-x|/|x-y|)^{1-2s})$$
From the estimate (Lemma \ref{Lem:monCauchy})
$$\uK_{\chi,x}^{-1}(z,w)=\frac 12R_B\left(\left(\frac{z-x}{w-x}\right)^s\frac{e^{i\nu(w)}}{\pi(z-w)}\right)
+\frac 12\bar R_B\left(\overline{\left(\frac{z-x}{w-x}\right)^{-s}}\frac{e^{-i\nu(w)}}{\pi\overline{(z-w)}}\right)+O(|w-x|^{2s-2})$$
for $\frac 12|w-x|\leq |z-x|\leq 2|w-x|$, $|z-w|$ comparable to $|w-x|$, we obtain 
$$\uK_{\chi,x}^{-1}(b,w)-\uK^{-1}(b,w)=\frac{s}{2\pi}R_B\left(\frac{e^{i\nu(w)}}{w-x}\right)-\frac{s}{2\pi}\bar R_B\left(\frac{e^{-i\nu(w)}}{\overline{w-x}}\right)+O(|w-x|^{2s-2})$$
for $b\sim w$ (this corresponds to letting $y\rightarrow\infty$ in the previous estimate \eqref{eq:Sgammadiag}). Since the error term is summable, it follows that
\begin{equation}\label{eq:Prop:electrpairs0}
\Re\sum_{(bw)\in E(\gamma_i)} \dot{K_\gamma}(b,w)\uK_{\chi,x}^{-1}(b,w)=-2s\log|x'-x|+\tilde c(x,\Lambda,\gamma,s)+O(|x-x'|^{2s-1})
\end{equation}
Moreover, the lefthand side is unchanged if $\gamma$ is another path started at $x$ going through $x'$, as the difference can be written as the trace of a commutator. Thus $\tilde c(x,\Lambda,\gamma,s)=\tilde c(x,\Lambda,s)$. Consequently
$$\Re\Tr(\dot{K_\gamma}S_\gamma)=-2s\Re\log|y-x|^2+\tilde c(x,\Lambda,s)+\tilde c(y,\Lambda,s)+o(1)$$
Fix a small $s_0>0$ and set $\chi_0=e^{2i\pi s_0}$. By Proposition \ref{Prop:electrsmalls}, we have
$$\log|\langle\chi_0^{h(y)-h(x)}\rangle|=-2s_0^2\log|y-x|+c_0(\Lambda,s_0)+O(|y-x|^{-\eps})$$
and the previous argument gives
\begin{align*}
\log\left|\frac{\langle\chi^{h(y)-h(x)}\rangle}{\langle\chi_0^{h(y)-h(x)}\rangle}\right|&=\int_{s_0}^s\Re\Tr(\dot{K_\gamma}K_\gamma^{-1})du\\
&=-2(s^2-s_0^2)\log|y-x|+c(x,\Lambda,s)+c(y,\Lambda,s)+o(1)
\end{align*}
Finally, let us now show that $c(x,\Lambda,s)=c(\Lambda,s)$, which we already know for small $s$ (Proposition \ref{Prop:electrsmalls}). We need only show
$$\left|\frac{\langle\chi^{h(y)-h(x')}\rangle}{\langle\chi^{h(y)-h(x)}\rangle}\right|\underset{y\rightarrow\infty}{\longrightarrow} 1$$
if $x'\in M^\dg$ is a neighbour of $x$; in turn (as in Proposition \ref{Prop:electrsmalls}), this follows from
$$S_\rho(b,w)-\uK_{\chi,x}^{-1}(b,w)=o(1)$$
as $y\rightarrow\infty$, where $(bw)$ is the edge of $M$ separating $x$ from $x'$. From Lemma \ref{Lem:monCauchy}, estimates for $S_\rho$ on the macroscopic scale (Lemma \ref{Lem:Srhoquant}), and Lemma \ref{Lem:monzero}, we get
$$S_\rho(b,w)-\uK_{\chi,x}^{-1}(b,w)=O(|y-x|^{s-1+s+\eta})$$
for $\eta>0$. This is enough to conclude that $c(x,\Lambda,s)=c(\Lambda,s)$.
\end{proof}

\subsubsection{More general correlations}\label{sss:generelectr}

We now discuss the ``general" case, where the problem is to estimate 
$$\langle\exp(2i\pi\sum_{j=1}^ns_j(h(y_j)-h(x_j))\rangle$$
where $s_1,\dots,s_n\in (0,\frac 12)$ and the pairwise distances between the $x_j,y_k$'s are large and comparable.

\begin{Thm}\label{Thm:electr}
For weights $s_1,\dots,s_{2n}\in (-\frac 12,\frac 12)$, $s_1+s_2=\cdots=s_{2n-1}+s_{2n}=0$, and singularities $\{z_1,\dots,z_{2n}\}$ with pairwise distances of order $R$, we have as $R$ goes to infinity:
$$\left\langle\exp(2i\pi\sum_{j=1}^{2n} s_jh(z_j))\right\rangle\sim \exp\left(\sum_{j=1}^{2n}c(\Lambda,|s_j|)\right)\prod_{1\leq i<j\leq 2n}|z_i-z_j|^{2s_is_j}$$
\end{Thm}
\begin{proof}
The line of reasoning is the same as in the two-point case (Propositions \ref{Prop:electrsmalls} and \ref{Prop:electrpairs}), so we will simply record the needed changes in the computations. This problem is associated to a character $\rho$ of $\pi_1(\C\setminus\{x_1,\dots,y_n\})$, and the kernel $S_\rho$ inverting $\rK$ on sections of the associated line bundle (more concretely, $\rho$-multivalued functions, of functions with a jump across branch cuts running from $x_j$ to $y_j$). The key estimates are Lemmas \ref{Lem:monrobin} and \ref{Lem:Srhoquant}.

We start with the case where the exponents are small enough and reason as in Proposition \ref{Prop:electrsmalls}. Set 
$$\varphi(z_1,\dots,z_{2n})=\log\left|\left\langle\exp(2i\pi\sum_{j=1}^{2n} s_jh(z_j))\right\rangle\prod_{1\leq i<j\leq 2n}|z_i-z_j|^{-2s_is_j}\right|$$
(the correlator is non-zero for $R$ large enough, see the argument in Lemma \ref{Lem:fredelectr}). Then from Lemma \ref{Lem:monrobin}, if $z'_j-z_j=O(1)$, then
$$\varphi(z'_1,\dots,z'_{2n})-\varphi(z_1,\dots,z_{2n})=O(R^{-1-\eps})$$
(with a uniform error term under $(\spadesuit)$) and then
$$\varphi(z_1,\dots,z_{2n})=O(R^{-\eps})+\lim_{z'_2,\dots,z'_{2n}\rightarrow\infty}\varphi(z_1,z'_2,\dots,z'_{2n})$$
where the limit is taken along any sequence where all the pairwise distances are comparable and go to infinity. The limit does not depend on $z_1$ (just on $\Lambda$), which gives the result for small exponents.

In the general case (exponents not necessarily close to $0$), we reason as in Proposition \ref{Prop:electrpairs}. We start from the estimate for $S_\rho$ (see Lemma \ref{Lem:Srhoquant}):
$$S_\rho(z,w)=\frac 12R_B\left(e^{i\nu(w)}S(z,w)\right)
+\frac 12\bar R_B\left(e^{-i\nu(w)}\bar S(z,w)\right)+o(R^{-1})$$
where all pairwise distances in $\{x_1,\dots,y_n,z,w\}$ are of order $R$, with
\begin{align*}
S(z,w)&=\frac{1}{\pi(z-w)}\prod_{i=1}^n\left(\frac {(z-x_j)(w-y_j)}{(z-y_j)(w-x_j)}\right)^{s_j}=
\frac{1}{\pi(z-w)}+r_\rho(w)+O(z-w)\\
\bar S(z,w)&=\frac{1}{\pi\overline{(z-w)}}\prod_{i=1}^n\overline{\left(\frac {(z-x_j)(w-y_j)}{(z-y_j)(w-x_j)}\right)}^{-s_j}=
\frac{1}{\pi\overline{(z-w)}}+\bar r_\rho(w)+O(z-w)
\end{align*}
where the estimate is for $z,w$ away from the singularities, and the Robin kernels (see before Lemma \ref{Lem:monrobin}) are given by 
\begin{align*}
r_\rho(w)&=
\pi^{-1}\sum_{j=1}^n\left(\frac{s_j}{w-x_j}-\frac{s_j}{w-y_j}\right)\\
\bar r_\rho(w)&=\pi^{-1}\sum_{j=1}^n\left(\frac{-s_j}{\overline{w-x_j}}-\frac{-s_j}{\overline{w-y_j}}\right)=-\overline {r_\rho(w)}
\end{align*}
Let $\gamma_j$ be a simple path from $x_j$ to $y_j$ (at macroscopic distance of the other singularities), and $\gamma'_j$ the subpath from $x'_j$ to $y'_j$. Since $r_\rho(w)\sim \frac{\pi^{-1}s_j}{w-x_j}$ near $x_j$ and $r_\rho(w)\sim -\frac{\pi^{-1}s_j}{w-y_j}$ near $y_j$, we have
$$\int_{\gamma'_j}r_\rho(w)dw=-\frac{s_j}\pi\log\left((y'_j-y_j)(x'_j-x_j)\right)+O(1)$$
Set
$$\int_{\gamma_j}^{reg} r_\rho(w)dw=\lim_{x'_j\rightarrow x_j,y'_j\rightarrow y_j}\left(\int_{\gamma'_j}r_\rho(w)dw+\frac{s_j}\pi\log\left((y'_j-y_j)(x'_j-x_j)\right)\right)$$
There is no determination (branch of logarithm) issue for the real part of this regularised integral, which is what we need. The variational argument of Proposition \ref{Prop:electrpairs} shows that
\begin{align*}
\frac{\partial}{\partial s_j}\log|\langle \exp(2i\pi\sum_{j=1}^ns_j(h(y_j)-h(x_j)))\rangle|&=-s_j\Re\int^{reg}_{\gamma_j} \pi r_\rho(w)dw+ 
s_j\Re\int^{reg}_{\gamma_j} \pi \bar r_\rho(w)d\bar w\\
&\hphantom{=}+\partial_{s} c(x,\Lambda,s_j)+\partial_{s}c(y,\Lambda,s_j)+o(1)
\end{align*}
and a direct computation yields
$$\int^{reg}_{\gamma_j}\pi r_\rho(w)dw=s_j\log(-(x_j-y_j)^2)+\sum_{k\neq j}s_k\log\left(\frac{(y_j-x_k)(x_j-y_k)}{(y_j-y_k)(x_j-x_k)}\right)
$$
as needed. Finally, $c(x,\Lambda,s)=c(\Lambda,s)$ as in Proposition \ref{Prop:electrpairs}.
\end{proof}

\paragraph{Locality.\\}

For, say, a periodic (non isoradial) dimer model with height function converging to the free field (as in Section 4.4 of \cite{KOS}), one would expect the following asymptotics for electric correlators:
$$\langle\chi^{h(y)-h(x)}\rangle\sim\exp(c(x,\Lambda,s)+c(y,\Lambda,s))|y-x|^{-2s^2}$$
where $c(x,\Lambda,s)$ depends on the graph around $x$ (rather like the invariant measure of a random walk). In the isoradial case, we saw that $c(x,\Lambda,s)=c(\Lambda,s)$, which is non-trivial. Let us further comment on that fact.

A by-product of Kenyon's construction (Section 4.2 in \cite{Ken_isoradial}) of the kernel $\uK^{-1}$ is its remarkable {\em locality} property; let us phrase the corresponding statement for the chiral kernel $\uK_\chi^{-1}$.

\begin{Lem}
The kernel $\uK_{\chi}^{-1}$ (with singularity at $x$) is {\em local} in the sense that $\uK_{\chi}^{-1}(b,w)$ depends only on $\Lambda$ restricted to a ball centred at $x$ containing $b,w$.
\end{Lem}
\begin{proof}
The discrete exponentials are by construction local. If $w$ is adjacent to $x$, then $\uK_{\chi}^{-1}(.,w)$ has an integral representation in terms of discrete exponentials (in the variable $b$) and consequently $\uK_\chi^{-1}(.,w)$ is also local (see Lemma \ref{Lem:monpole}). For a general $w$, consider a simple path $(x_0=x,x_1,\cdots,x_n=x')$ from $x$ to $x'$ on $M^\dg$, where $x'$ is a face adjacent to $w$. Let $(w_ib_i)=(x_ix_{i+1})^\dg$. By induction on the length of this path, it is easy to see that $\uK_\chi^{-1}(.,w)$ can be expressed as a linear combination (with local coefficients) of the $\uK_{\chi,x_i}^{-1}(.,w_i)$, which are local. Thus $\uK_\chi^{-1}(.,w)$ is itself local.
\end{proof}

Consider the space $\{(\Lambda,x)\}$ of lozenge tilings rooted at a face (fixing the scale $\delta=1$). As is customary for rooted graphs (\cite{BenSch_rec}), we may define a distance by:
$${\rm dist}((\Lambda,x),(\Lambda',x'))=\inf\{\eps>0:(\Lambda,x)_{|B(x,\eps^{-1})}{\rm\ isomorphic\ to\ }(\Lambda',x')_{|B(x',\eps^{-1})}\}$$
where graph isomorphisms are required to preserve roots. We need to take into account embedding data, i.e. lozenge angles. Thus we may refine the distance as follows:
\begin{align*}
{\rm dist}((\Lambda,x),(\Lambda',x'))&=\inf\{\eps>0:(\Lambda,x)_{|B(x,\eps^{-1})}{\rm\ isomorphic\ to\ }(\Lambda',x')_{|B(x',\eps^{-1})}\\
&\hphantom{=\inf\{\eps>0:}{\rm\ and\ corresponding\ angles\ differ\ by\ at\ most\ }\eps\}
\end{align*}
Then the set $\{(\Lambda,x):\Lambda{\rm\ satisfies\ }(\spadesuit)\}$ is compact. Indeed, under the condition $(\spadesuit)$ (see \eqref{eq:spade}), degrees of vertices are bounded and the number of vertices in a subgraph is comparable to its area; thus there are finitely many isomorphism classes as rooted planar graphs
for restrictions $(\Lambda,x)_{|B(x,\eps^{-1})}$ (for fixed $\eps>0$). Each isomorphism class is parameterised by finitely many angles taking values in compact intervals. 

Using locality of the kernels $\uK_\chi^{-1}$ and their continuous dependence on lozenge angles for a given graph type, we conclude that $(\Lambda,x)\mapsto c(x,\Lambda,s)$ is a continuous function for fixed $s$ (see \eqref{eq:Prop:electrpairs0}). In particular it is bounded under the condition $(\spadesuit)$ for $\Lambda$. Moreover, they are bounded for $s$ in a compact interval of $(0,\frac 12)$.

Let us now discuss the dependence on $\Lambda$ (or lack thereof). The locality and continuity arguments for $(\Lambda,x)\mapsto c(x,\Lambda,s)=c(\Lambda,s)$ show that, for $\Lambda$ satisfying $(\spadesuit)$, $c(\Lambda,s)$ can be approximated within $\eps$ by looking at {\em any} ball of $\Lambda$ of large enough radius $R(\eps)$.

Consider $\Lambda_\theta$ a periodic tiling of the plane by identical lozenges with angles $\theta$, $\pi-\theta$. One may glue $\Lambda_\theta$ and $\Lambda_{\theta'}$ along a line to obtain a lozenge tiling $\Lambda_{\theta,\theta'}$ (as in \cite{GriMan_inhom}). Then for small $s$,
$$c(\Lambda_\theta,s)=c(\Lambda_{\theta,\theta'},s)=c(\Lambda_{\theta'},s)$$
by continuity in $(\Lambda,x)$. 

More generally, consider the following condition for two rhombus tilings $\Lambda_1$, $\Lambda_2$ in the class ($\spadesuit$): for some $\theta_0'>0$, there exists arbitrarily large balls $B_1\subset\Lambda_1$, $B_2\subset\Lambda_2$ and a rhombus tiling $\Lambda_3$  with angles $\geq\theta'_0$ containing copies of $B_1$ and $B_2$. This generates an equivalence relation $\dot\sim$ on rhombus tilings; again by continuity we have (at least for small $s$): $c(\Lambda_1,s)=c(\Lambda_2,s)$ if $\Lambda_1\dot\sim\Lambda_2$. 

The previous argument shows that $\Lambda_\theta\dot\sim\Lambda_{\theta'}$. As another example, set $\Gamma$ to be the triangular lattice and $\Lambda_T$ the corresponding rhombus tiling; it contains an hexagonal lattice as a subgraph. Let $\Lambda'$ be any rhombus tiling obtained from $\Lambda_T$ by star-triangle transformations inside these hexagons. Then $\Lambda_T\dot\sim\Lambda'$; we thus obtain an equivalence class with infinitely many elements which are distinct as graphs. 

It is also easy to see that $\Lambda_T\dot\sim \Lambda_{\pi/3}$, since one may glue a half-space of $\Lambda_T$ and a half-space space $\Lambda_{\pi/3}$ to obtain a planar tiling $\Lambda$.

The triangular lattice has a two-parameter (up to isometry) family of embeddings with isometric faces, leading to a two-parameter family of rhombus tilings. In two steps (as in \cite{GriMan_inhom}), one sees that they all belong to the same equivalence class. Similarly, the isoradial embeddings of the square lattice (under ($\spadesuit)$), which are parameterised by two bi-infinite sequences of angles, fall in the same equivalence class.

It would be of some independent interest to decide whether $\dot\sim$ is trivial (coarse - and thus $c(\Lambda,s)=c(s)$) and if not describe its equivalence classes (building on the structure result of \cite{KenSch}).

\section{Monomers and the Fisher-Stephenson conjecture}\label{sec:mono}

\subsection{Introduction}\label{ssec:FSintro}

We now consider a planar graph $M$, derived from a lozenge tiling as illustrated in Figure \ref{fig:loztil}, with ``defects" (or {\em monomers}) consisting in a pair of missing vertices: a black vertex $b_0$ and a white vertex $w_0$. In this section, we consider the regime where the mesh is fixed ($\delta=1$) and the separation between defects goes to infinity.

\paragraph{Infinite volume limit.\\}
Specifically, consider a sequence $(\Xi_n)_{n\geq 1}$ of subgraphs of $M$ bounded by a simple cycle on $M^\dg$. We assume that for all $n$, $b_0$ and $w_0$ are in $\Xi_n$ and that $\Xi_n$ has a perfect matching; and that the inradius of $\Xi_n$, say seen from $w_0$, goes to infinity as $n\rightarrow\infty$. Moreover we assume that $\rK^{-1}_n(b,w)\rightarrow \uK^{-1}(b,w)$ for all $b\in M_B,w\in M_W$ (here $\rK_n$ denotes the restriction of $\rK$ to $\Xi_n$); this is the case for ``Temperleyan" boundary conditions (see Theorem 13 in \cite{Ken_domino_conformal} for the square lattice; the general case follows easily from e.g. the maximum principle \eqref{eq:maxprincip} and the Harnack inequality for discrete harmonic functions, see \eqref{eq:harnack}). It follows that the sequence of the perfect matching measures on $\Xi_n$ converges weakly to the measure $\P$ on matchings on $M$ described earlier (see \eqref{eq:locstats}). 

Following Fisher-Stephenson \cite{FisSte2}, we consider the monomer correlation 
$$\Mon_{\Xi_n}(b_0,w_0)\stackrel{def}{=}\frac{{\mc Z}(\Xi_n\setminus\{b_0,w_0\})}{{\mc Z}(\Xi_n)}$$
where ${\mc Z}(\Xi)$ is the partition function of perfect matchings of the graph $\Xi$ (with edge weight $|\rK(b,w)|$ for the edge $(bw)$, see \eqref{eq:dimerpartfun}). 

In general, a Kasteleyn orientation of $\Xi_n$ does not restrict to a Kasteleyn orientation of $\Xi_n'\stackrel{def}{=}\Xi_n\setminus\{b_0,w_0\}$ (see Section \ref{sss:critgraphs}). Indeed, if $b_0$ and $w_0$ are not neighbours in $M$, each correspond to a finite face of $\Xi'_n$ where the clockwise odd condition is violated. In order to 
obtain a Kasteleyn orientation of $\Xi'_n$, one may introduce a ``defect line" $\gamma$, i.e. a simple path $\gamma$ on $M^\dg$ running from a face of $M$ adjacent to $b_0$ (denoted by $x$) to a face adjacent to $w_0$ (denoted by $y$, see Figure \ref{Fig:trimer}). Reversing orientations of all edges crossing $\gamma$ yields  a Kasteleyn orientation of $\Xi'_n$. Thus, let us fix such a defect line $\gamma$ and define $\rK':\R^{M_B\setminus\{b_0\}}\rightarrow\R^{M_W\setminus\{w_0\}}$ by 
\begin{equation} \label{eq:Kprime}
\begin{split}
\rK'(w,b) & =-\rK(w,b){\rm\ \ \ \  if\ }\gamma{\rm\ crosses\ }(bw) \\
 & =\rK(w,b){\rm\ \ \ \ \ \ otherwise}
\end{split}
\end{equation}
We denote by $\rK_{n}$, $\rK'_{n}$ the restrictions of $\rK,\rK'$ to $\Xi_n,\Xi'_n$ respectively (throughout this section, the prime recalls the sign change across the defect line). Then \cite{FisSte2}:
$$\Mon_{\Xi_n}(b_0,w_0)=\frac{{\mc Z}(\Xi'_n)}{{\mc Z}(\Xi_n)}=\pm\frac{\det(\rK'_n)}{\det(\rK_{n})}
$$
In order to compare $\rK'_{n}$, $\rK_{n}$, it is rather convenient to extend $\rK'_{n}$ to an operator $\R^{\Xi_B}\rightarrow\R^{\Xi_W}$ by setting $\rK'(w_0,b_0)=1$, with all other new matrices elements $\rK'(w_0,b),\rK'(w,b_0)$ vanishing (one may think of an edge ``handle" connecting directly $b_0,w_0$). This does not change the determinant of $\rK'_{n}$ (up to sign). Then 
$$\Mon_{\Xi_n}(b_0,w_0)=\pm\det(\rK'_n\rK_{n}^{-1})$$
Clearly $\rK'_n-\rK_n$ has bounded support (vertices adjacent to $\gamma$), and does not depend on $n$ (for $n$ large enough) on its support. It follows from the assumption of pointwise convergence of $\rK_n^{-1}$ that: 
$$\gls{MonM}(b_0,w_0)=\left|\lim_{n\rightarrow\infty}\det(\rK'_n\rK_n^{-1})\right|=|\det(\rK'\uK^{-1})|$$  
This defines the infinite volume monomer correlation $\Mon_M(b_0,w_0)$ (which does not depend on the choice of approximating sequence $(\Xi_n)$, under the assumption $\rK_n^{-1}\rightarrow\uK^{-1}$ pointwise). 

\paragraph{Positivity.\\}

In order to study logarithmic first differences of monomer correlations, it will be convenient to justify {\em a priori} positivity of $\Mon_M(b_0,w_0)$. 

Let us sketch a direct argument. Consider a simple path from $b$ to $w$ on $M$: $(b_0=b,w_0,b_1,\dots,b_n,w_{n}=w)$. Then matchings of $\Xi'_n$ containing $(w_0,b_1)$,\dots,$(w_{n-1},b_n)$ are in bijection with matchings of $\Xi_n$ containing $(b_0,w_0)$, \dots, $(b_n,w_{n+1})$ (see Section 7.1 in \cite{Ken_Lap} for related arguments), which gives the lower bound:
\begin{equation}\label{eq:monopos}
\Mon_M(b,w)\geq\frac{|\rK(w_0,b_1)\dots\rK(w_{n-1},b_n)|}{|\rK(w_0,b_0)\dots\rK(w_{n},b_n)|}\P((b_0,w_0)\in\mfm,\dots,(b_n,w_{n})\in\mfm)
\end{equation}
One may choose the path so that $(b_0,\dots,b_{n+1})$ is a simple path on $\Gamma$ and $b_{n+1}$ is connected to infinity in $\Gamma\setminus\{b_0,\dots,b_{n}\}$. Then (see \eqref{eq:locstats})
$$\P((b_0,w_0)\in\mfm,\dots,(b_n,w_{n})\in\mfm)=\left|\rK(w_0,b_0)\dots\rK(w_{n},b_n)\det(\uK^{-1}(b_{i},w_j))_{0\leq i,j\leq n}\right|$$
This is positive; indeed, otherwise one could find a non-trivial $f\in (\C^{M_B})$ vanishing at infinity such that $\rK f$ is supported on $\{w_0,\dots,w_n\}$, and vanishes at $\{b_0,\dots,b_n\}$. It is then easy to see that $f_{|\Gamma}$ is discrete harmonic except at $\{b_0,\dots,b_{n+1}\}$ (see \eqref{eq:KstarK}). Since the harmonic measure of $b_{n+1}$ seen from infinity in $\Gamma\setminus\{b_0,\dots,b_{n+1}\}$ is positive, this forces $f(b_{n+1})=0$, then $f_{|\Gamma}=0$, then $f=0$, yielding a contradiction.

\paragraph{Fisher-Stephenson conjecture.\\}

The monomer correlation question bears on the asymptotic behaviour of $\Mon_M(b_0,w_0)$ as $|b_0-w_0|\rightarrow\infty$, eg with $w_0$ fixed. The Fisher-Stephenson conjecture \cite{FisSte2} proposes:
$$\Mon_{\Z^2}(b,w)\sim c|b-w|^{-\frac 12}$$

In \cite{Ken_Lap} (Section 7), Kenyon analyses the case of a white defect in the bulk and a black defect on the boundary of a simply connected portion of the square lattice. This leads to the length exponent for the loop-erased random walk.

More generally, if $\{b_1,\dots,b_p\}$ and $\{w_1,\dots,w_p\}$ are $p$-tuples of black and white vertices of $M$, one may consider the $2p$-point monomer correlation:
$$\Mon_{M}(b_1,\dots,b_p;w_1,\dots,w_p)=\lim_{n\rightarrow\infty}\frac{{\mc Z}(\Xi_n\setminus\{b_1,\dots,b_p,w_1,\dots,w_p\})}{{\mc Z}(\Xi_n)}$$
where the limit is justified as in the $p=1$ case.

Let us give a heuristic interpretation of these monomer correlators. If we apply the construction of the height function (Section \ref{ss:height}) to dimer-monomer configurations on $M$, with monomers located at the fixed $\{b_1,\dots,b_p,w_1,\dots,w_p\}$, we obtain an additively multivalued function on $M^\dg$, which increases by $1$ (resp. $-1$) when cycling counterclockwise around a white (resp. black) monomer. In the absence of defects, the height field converges to a free field (in the plane, see \cite{dT_isoradial} or e.g. Corollary \ref{Cor:FFplane}). Interpreting the monomers as discrete versions of magnetic operators for the free field, one may expect
$$\Mon_{M}(b_1,\dots,b_p;w_1,\dots,w_p)\sim c\langle:{\mc O}_{1}(w_1)\dots{\mc O}_{1}(w_p){\mc O}_{-1}(b_1)\dots{\mc O}_{-1}(b_p):\rangle_\C$$
in some asymptotic regime, for instance as the pairwise distance between insertions goes to infinity (alternatively as the lattice mesh goes to zero). For the planar free field, the magnetic correlator is explicitly
$$\langle :\prod_j{\mc O}_{m_j}(z_j):\rangle_\C=\prod_{j<k}|z_j-z_k|^{g_0m_jm_k}$$
(see \eqref{eq:GFFmagncorr}) and here $g_0=\frac 12$, in agreement with eg Corollary \ref{Cor:FFplane} (this is the coupling constant for the limit of $2\pi h$; 
recall that magnetic charges are measured as multiples of $2\pi$). This yields the natural extension of the Fisher-Stephenson conjecture (see \cite{Ciucu_PNAS}):
$$\Mon_{M}(z_1,\dots,z_{2p})\sim c\prod _{i<j}|z_i-z_j|^{\eps_i\eps_j/2}$$
where $\eps_i=1$ (resp. -1) if $z_i$ is black (resp. white) and $\sum_i\eps_i=0$. This is supported in particular by Ciucu's work (see \cite{Ciucu_PNAS,Ciucu_TAMS} and references therein). In the case of the square lattice, in view of the correspondence with the 6-vertex model (see e.g. Section 5 in \cite{Dub_abelian}), this may also be seen as part of the classical Coulomb Gas heuristics (e.g. \cite{NieKno_potts,Nien_Coulomb,DiFSalZub_coulomb}).

The (somewhat stronger) variational form of the statement we shall prove here consists in estimating $\log(\Mon_M(b,w)/\Mon_M(b',w))$ within $O(|b-w|^{-1-\eps})$ for $b',b$ black vertices on the same face of $M$ (or similar quantities for correlations with $2n$ monomers), as we did for electric correlators (Proposition \ref{Prop:electrsmalls}).

The analysis (estimates and local computations near singularities) is rather involved (and tedious) and will be carried out in details in Section \ref{subsec:monopair} in the (classical) case of $p=1$ pair of monomers (Proposition \ref{Prop:FS}). In Section \ref{subsec:generalmono}, we simply explain what are the modifications needed for the general case (Theorem \ref{Thm:FS}).

As we shall see, the relevant family of holomorphic line bundles are the line bundles over $\Sigma=\hat\C\setminus\{z_1,\dots,z_{2p}\}$ with holomorphic sections in $U$ written as:
$$s_U(z)=\prod_{i=1}^{2p}(z-z_i)^{\eps_i/2}g(z)$$
where $g$ is holomorphic in $U\setminus\hat\C$ and vanishing at infinity, and $\eps_i=1$ for singularities corresponding to black monomers and $-1$ for white monomers.

\subsection{Monomer pairs}\label{subsec:monopair}

Here, we analyse asymptotics for monomer pair correlations, in the framework laid out in Section \ref{ssec:FSintro}.

\paragraph{Plan.\\}

Let us observe that one may define a dimer measure $\P_{M'}$ on matchings of $M'=M\setminus\{b_0,w_0\}$ as the weak limit of the (weighted) dimer measures on $\Xi'_n$ as $n\rightarrow\infty$. It suffices (\cite{Ken_locstat}) to observe that $(\rK'_n)^{-1}=\rK_n^{-1}(\rK'_n\rK_n^{-1})^{-1}$, and as $n\rightarrow\infty$, $\rK_n^{-1}$ converges pointwise to $\uK^{-1}$ and $(\rK'_n\rK_n^{-1})$ converges pointwise to an invertible (since $\Mon_M(b_0,w_0)>0$), bounded support perturbation of the identity. Let us denote $(\rK')^{-1}$ the limiting kernel (or $(\rK'_{b_0,w_0})^{-1}$ if we need to emphasise the position of monomers). It satisfies $\rK'(\rK')^{-1}(.,w)=\delta_w$ for all $w\in M_W\setminus\{w_0\}$, and $(\rK')^{-1}(b,w)=O(|b-w|^{-1})$.

Denote by $(b_0,w_1,b_1)$ three consecutive vertices (in counterclockwise order, say) on a face $x$ of $M$. We have:
\begin{align*}
{\mc Z}(\Xi_n\setminus\{b_0,w_1,b_1,w_0\})&=
{\mc Z}(\Xi_n\setminus\{b_0,w_0\})|\rK(w_1,b_1)|^{-1}\P_{\Xi_n\setminus\{b_0,w_0\}}\{(b_1,w_1)\in\mfm\}\\
&=
{\mc Z}(\Xi_n\setminus\{b_1,w_0\})|\rK(w_1,b_0)|^{-1}\P_{\Xi_n\setminus\{b_1,w_0\}}
\{(b_0,w_1)\in\mfm\}
\end{align*}
where we consider two defects: the monomer $w_0$ (on the boundary of $y\in M^\dg$), and the ``trimer" $\{b_0,w_1,b_1\}$ (see Figure \ref{Fig:trimer}). 
\begin{figure}[htb]
\begin{center}
\leavevmode
\resizebox{10cm}{!}{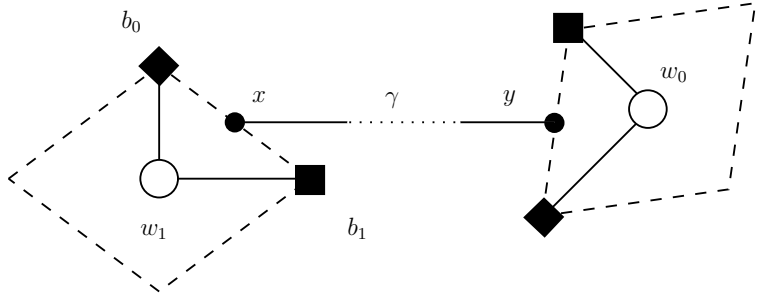}
\end{center}
\caption{Local geometry and notations near the monomers.}
\label{Fig:trimer}
\end{figure}

Consequently,
$$\frac{\Mon_{\Xi_n}(b_1,w_0)}{\Mon_{\Xi_n}(b_0,w_0)}=\left|\frac{\rK(w_1,b_0)}{\rK(w_1,b_1)}\right|\frac{\P_{\Xi_n\setminus\{b_0,w_0\}}\{(b_1,w_1)\in\mfm\}}{\P_{\Xi_n\setminus\{b_1,w_0\}}\{(b_0,w_1)\in\mfm\}}=\pm\frac{(\rK'_{\Xi_n\setminus\{b_0,w_0\}})^{-1}(b_1,w_1)}
{(\rK'_{\Xi_n\setminus\{b_1,w_0\}})^{-1}(b_0,w_1)}
$$
and taking the limit as $n\rightarrow\infty$,
\begin{equation}\label{eq:trimerlog}
\frac{\Mon_M(b_1,w_0)}{\Mon(b_0,w_0)}=\pm\frac{(\rK'_{b_0,w_0})^{-1}(b_1,w_1)}
{(\rK'_{b_1,w_0})^{-1}(b_0,w_1)}
\end{equation}
Thus we simply need to estimate $S_{b_0,w_0}=(\rK'_{b_0,w_0})^{-1}$ precisely enough near the singularity $b_0$. The argument is at times somewhat technical; let us briefly sketch the line of reasoning, which is roughly parallel to the one in Section \ref{sss:moninv}.
\begin{enumerate}
\item Random walk arguments give an estimate on $\hat S_\rho$, an inverting kernel for $\rK'_{b_0,w_0}$ which is bounded (rather than vanishing) at infinity. 
\item Bounded functions in the kernel of $\rK'_{b_0,w_0}$ are classified.
\item The correct inverting kernel $S_{b_0,w_0}$ is constructed, and estimated in the macroscopic scale.
\item A priori estimates for $S_{b_0,w_0}$ when one or both arguments are near singularities are given.
\item $\hat g$, an approximation of $S_{b_0,w_0}(.,w_1)$ is constructed, using as building blocks discrete holomorphic functions mesoscopically near singularities, and continuous holomorphic functions elsewhere.
\item $\hat g(b_1)-S_{b_0,w_0}(.,w_1)$ is evaluated within $O(|b_0-w_0|^{-1-\eps})$ (this requires controlling the leading term, which depends only on the local geometry around $x$, and the first subleading correction which carries the ``global" information; compare e.g. with Lemma \ref{Lem:monrobin}).
\item This is applied to monomer correlations via \eqref{eq:trimerlog}.
\end{enumerate}

Recall (from \eqref{eq:Kprime}) that $\rK'_{b_0,w_0}$ is obtained from $\rK$ by deleting the row and column corresponding to the pair of monomers and changing the sign of matrix entries corresponding to edges crossing the defect line $\gamma$ which runs  from $x\in M^\dg$ adjacent to $b_0$ to $y\in M^\dg$ adjacent to $w_0$:
$$\rK'_{b_0,w_0}:\R^{M_B\setminus\{b_0\}}\longrightarrow \R^{M_W\setminus\{w_0\}}$$
Displacing $\gamma$ (with endpoints fixed) results in composing $\rK_{b_0,w_0}$ with diagonal matrices (with $\pm 1$ diagonal coefficients). Equivalently, one may consider multivalued functions on $M$ corresponding to the character $\rho:\pi_1(\C\setminus\{x,y\})\rightarrow\{\pm 1\}$ with monodromy $-1$ around $x,y$. One may identify elements of $\R^{M_B\setminus\{b_0\}}$ with functions on $M_B$ vanishing at $b_0$, and elements of $\R^{M_W\setminus\{w_0\}}$ as functions on $M_W$ modulo $\delta_{w_0}$.\\

\paragraph{Bounded inverting kernel via harmonic functions.\\}

We start with by constructing an inverting kernel $S_{b_0,w_0}$ for $\rK'_{b_0,w_0}$, at least for $|b_0-w_0|$ large enough.

Assume without loss of generality that $b_0\in\Gamma^\dg$. We observe that the kernel $\hat S_\rho$ obtained from the Green kernel for the random walk with monodromy $-1$ at $b_0,w_0$ on $\Gamma$ vanishes at $b_0$ (see Lemma \ref{Lem:Green2pts} and afterwards). Hence $\rK'_{b_0,w_0} \hat S_\rho(.,w)=\delta_w$. Moreover we have
\begin{align*}
\hat S_\rho(z,w)&=\frac 12R_B\left(\sqrt{\frac{(z-x)(z-y)}{(w-x)(w-y)}}\cdot \frac{e^{i\nu(w)}}{\pi(z-w)}\right)+(cc)+o(1/|x-y|)\\
&=\Re R_B\left(\sqrt{\frac{(z-x)(z-y)}{(w-x)(w-y)}}\cdot \frac{e^{i\nu(w)}}{\pi(z-w)}\right)
+o(1/|x-y|)
\end{align*}
for $(z,w)$ in a compact set of $\{z\neq x,y,w;w\neq x,y,z\}$. (Here $(cc)$ denotes the complex conjugate of the previous term). Note that in this case, a closed form expression for the continuous Green kernel $G_\rho$ is easily obtained by going to a double cover of $\hat\C\setminus\{b_0,w_0\}$. We consider a branch of $z\mapsto\sqrt{z-x}$ (resp. $z\mapsto\sqrt{z-y}$) which is single-valued on $\C\setminus\gamma$.

\paragraph{Bounded discrete holomorphic functions on $M'$.\\}

The kernel $\hat S_\rho$ does not necessarily vanish at infinity; on the other hand we have some freedom since $S_{b_0,w_0}(.,w)$ is not required to be holomorphic at $w_0$. In order to correct $\hat S_\rho$ and obtain an inverting kernel vanishing at infinity, we need to study bounded functions in the kernel of $\rK'_{b_0,w_0}$ (this is a point where the argument differs from the construction for electric correlators, see Section \ref{sss:moninv}). 

We claim that $\dim\{f:\rK'_{b_0,w_0}f=0, \|f\|_\infty<\infty\}\leq 2$. Indeed, $\rK'_{b_0,w_0}f=0$ implies that $\Lap_\Gamma f_{|\Gamma}$ is supported on the two vertices of $\Gamma$ abutting $w_0$ (see \eqref{eq:KstarK}), and there are no non-zero bounded $\rho$-multivalued harmonic functions on $\Gamma$ ($f_{|\Gamma^\dg}$ is specified by harmonic conjugation and $f(b_0)=0$). We simply need to construct two linearly independent bounded functions in the kernel of $\rK'_{b_0,w_0}$.

We have constructed (Lemma \ref{Lem:monpole}) $g$ such that $\rK_{-1,y}g=0$ outside of $w_0$ and 
$$g(z)=R_B((z-y)^{-1/2})+O(|z-y|^{-3/2})$$
We then consider $f=g\ind_B-\hat S_\rho(g\ind_B)$ where $B=B(y,r)$, $r\leq \frac 12|x-y|$. Plainly, $f$ is bounded and $\rK_{b_0,w_0}'f=0$. Since
$$\frac 1{2i\pi}\oint_{C(y,r)}\frac{1}{\sqrt{w-y}}\cdot\sqrt{\frac{(z-x)(z-y)}{(w-x)(w-y)}}\cdot\frac{dw}{z-w}=\sqrt\frac{z-x}{(z-y)(y-x)}-\frac{\ind_B(z)}{\sqrt{z-y}}$$
and the contribution from the conjugate term in $\hat S_\rho$ vanishes (since $(z-y)^{-1/2}\propto \overline{(z-y)^{1/2}}$ on the contour), we get
$$f(z)=R_B\left(\sqrt{\frac{z-x}{(z-y)(y-x)}}\right)+o(|x-y|^{-1/2})$$
for $z$ in a compact subset of $\C\setminus\{x,y\}$. (This is the same construction as in \eqref{eq:Lem:Srho2pts}). 

We know a priori that $f_{|\Gamma}$ (resp. $f_{|\Gamma^\dg}$) has a limit as $z\rightarrow\infty$ (as it is bounded and harmonic near $\infty$). Plainly, $\lim_{z\rightarrow\infty}f_{|\Gamma}(z)=1/\sqrt{y-x}+o(|x-y|^{-1/2})$ and $\lim_{z\rightarrow\infty}f_{|\Gamma^\dg}(z)=i/\sqrt{y-x}+o(|x-y|^{-1/2})$.

We observe that $\bar f$ is also in the kernel of $\rK'_{b_0,w_0}$ (a real operator); examining the limits of $f_{|\Gamma}$, $f_{|\Gamma^\dg}$ at infinity shows that $f,\bar f$ are linearly independent. Hence they constitute a basis of bounded functions in the kernel of $\rK'_{b_0,w_0}$.

Moreover (at least for $b_0,w_0$ far enough), there is no non-vanishing function in the kernel of $\rK'_{b_0,w_0}$ going to zero at infinity.

\paragraph{Inverting kernel vanishing at infinity.\\}

As a consequence, there is a unique kernel $S_{b_0,w_0}$ inverting $\rK'_{b_0,w_0}$ (on the right) and vanishing at infinity. By uniqueness, this kernel is real. It may be realised as
$$S_ {b_0,w_0}(b,w)=\hat S_\rho(b,w)-\Re(a(w)f(b))$$
where $a(w)$ is fixed by the behaviour as $b\rightarrow\infty$: $a(w)=\frac{e^{i\nu(w)}}2\sqrt{\frac{y-x}{(w-x)(w-y)}}+o(|x-y|^{-1/2})$. Consequently, we obtain the estimate
\begin{equation}\label{eq:magnpairker}
S_ {b_0,w_0} (z,w)=\frac 12R_B\left(\sqrt{\frac{(z-x)(w-y)}{(z-y)(w-x)}}\cdot \frac{e^{i\nu(w)}}{\pi(z-w)}\right)+(cc)+o(|x-y|^{-1})
\end{equation}
for $(z,w)$ in a compact set of $\{z\neq x,y,w;w\neq x,y,z\}$. (Compare with Lemma \ref{Lem:Srho2pts}). Remark also that the error $o(|x-y|^{-1})$ is uniform over lozenge tilings under $(\spadesuit)$.

\paragraph{Estimates for $S_{b_0,w_0}(b,w)$ for $b$ close to $x$.\\}

Let us review some a priori estimates on $S_{b_0,w_0}(b,w)$ for $b$ close to $x$, and $w$ close to $x,y,\infty$ respectively. 
Remark that, for the purpose of these estimates, it is enough to know that $S_{b_0,w_0}(b,w)=O(1)$ on macroscopic scale (i.e. when all pairwise distances are comparable and large). In particular they apply {\em mutatis mutandis} to the kernel $S_{\underline b,\underline w}$ of Section \ref{subsec:generalmono}.

\begin{Lem}\label{Lem:magn_apriori}
There is $\eps>0$ such that:
\begin{enumerate} 
\item If $\rK_{-1,x} f=0$ in $B(x,r)$, $f(b_0)=0$, then $f(b)=O((|b-x|/r)^\eps\|f_{\partial B(x,r)}\|_\infty)$ for $b\in B(x,r)$.
\item There is a unique kernel $S_{b_0}$ s.t. $S_{b_0}(.,w)$ has monodromy $(-1)$ around $x$, $\rK_{-1,x} S_{b_0}(.,w)=\delta_w$, and $S_{b_0}(.,w)$ vanishes at $b_0$ and infinity.
\item $S_{b_0,w_0}(b,w)=O(|w-x|^{-1-\eps})$ if $b$ adjacent to $x$, $|w-x|\leq\frac 12|x-y|$
\item $S_{b_0,w_0}(b,w)=O(|w-y|^{1/2-\eps}|x-y|^{-3/2})$ if $b$ adjacent to $x$, $|w-y|\leq\frac 12|x-y|$
\item $S_{b_0,w_0}(b,w)=O(|w-x|^{-1}|x-y|^{-\eps})$ if $b$ adjacent to $x$, $|w-x|\geq\frac 12|x-y|$,$|w-y|\geq\frac 12|x-y|$.
\end{enumerate}
\end{Lem}
\begin{proof}
\begin{enumerate}
\item This follows from observing that $f_{|\Gamma}$ is harmonic, with monodromy $(-1)$ around $x$, and writing $f_{|\Gamma^\dg}$ as the harmonic conjugate of $f_{|\Gamma}$ vanishing at $b_0$ (see Lemma \ref{Lem:chirharmest}).
\item Uniqueness follows from the fact that the space of $(-1)$-multivalued bounded holomorphic functions is spanned by a single function $f_{-1}$, which does not vanish at $b_0$ (Lemma \ref{Lem:monholom}). By uniqueness this kernel is real. 
As before (Lemma \ref{Lem:chirkernel0}), one can write $S_{b_0}(.,w)_{|\Gamma}$ as a first difference of the Green kernel on $\Gamma$ (with monodromy $-1$ around $x$) and extend it to $\Gamma^\dg$ by harmonic conjugation. This leads to the a priori estimate
$$S_{b_0}(b,w)=O(|w-x|^{-1}((|b-x|/|w-x|)^\eps\wedge(|w-x|/|b-x|)^\eps))$$
e.g. if $|b-w|\geq \frac 12|w|$.
\item Let us write 
$$S_{b_0,w_0}(.,w)=\ind_BS_{b_0}(.,w)-S_{b_0,w_0}(\rK'_{b_0,w_0}(\ind_BS_{b_0}(.,w)))$$
where $B=B(x,\frac 34|x-y|)$. We have $\rK'_{b_0,w_0}(\ind_B S_{b_0}(.,w))(w')=O(|w-x|^{-1}(|w-x|/R)^\eps)$ for $w'\in\partial B$ (where $R=|x-y|$), and $S_{b_0,w_0}(b,w')=O(R^{-1-\eps})$ by 1. Thus 
$$S_{b_0,w_0}(b,w)=O(|w-x|^{-1-\eps})+O(|w-x|^{\eps-1}R^{-2\eps})=O(|w-x|^{-1-\eps})$$
\item Here the reference kernel $S_y$ is constructed from the Green kernel which has monodromy $(-1)$ at $y$ and is such that $S_y(b,w)=O(|w-y|^{-1})$ if $|b-y|$ and $|b-w|$ are comparable to $|w-y|$. We can construct (Lemma \ref{Lem:monpole}) two functions $g_y$, $\bar g_y$ with monodromy $(-1)$ at $y$ which are holomorphic except possibly at $w_0$ and with expansion $g_y=R_B((b-w_0)^{-1/2})+O((b-w_0)^{-3/2})$. Correcting  $S_y(.,w)$ by an appropriate linear combination of $g_y,\bar g_y$, we obtain a kernel $\hat S_y$  such that $\hat S_y(b,w)=O(|w-y|^{-1}(|b-y|/|w-y|)^{-3/2+\eps})$ for $|b-y|\geq 2|w-y|$, say (see Lemma \ref{Lem:moninf}).\\
Then we may represent $S_{b_0,w_0}(.,w)$ as 
$$S_{b_0,w_0}(.,w)=\ind_B\hat S_y(.,w)-S_{b_0,w_0}(\rK'_{b_0,w_0}(\ind_B\hat S_y(.,w)))$$
with $B=B(y,\frac 12|x-y|)$ to obtain for $b$ close to $x$:
$$S_{b_0,w_0}(b,w)=O(|x-y|^{-\eps}|w-y|^{-1}(|w-y|/|x-y|)^{3/2-\eps})$$
\item We simply write
$$S_{b_0,w_0}(.,w)=\ind_B\uK^{-1}(.,w)-S_{b_0,w_0}(\rK'_{b_0,w_0}(\ind_B\uK^{-1}(.,w)))$$
with $B=B(x,\frac 12|x-y|)\cup B(y,\frac 12|x-y|)$, and use 1.
\end{enumerate}
In 3.,4.,5., the discrete Cauchy integral formulae are justified by the absence of non-trivial functions in the nullspace of $\rK'_{b_0,w_0}$ vanishing at infinity.
\end{proof}

Let us denote by $(b_0,w_1,b_1,w_2)$ the vertices of the face $x$ of $M$, listed counterclockwise. Recall that we eventually want to estimate $S_ {b_0,w_0}(w_1,b_1)$ within $O(|w_0-b_0|^{-1-\eps})$. We shall need basic discrete holomorphic functions with monodromy $\chi=-1$ around the singularity $x\in M^\dg$. Set $g=\uK_{\chi}^{-1}(.,w_1)$, $f_{-1}=\Re f_\chi$ at $\chi=-1$; $h=\overline{f_{1,-1}}/\tau$ where $\tau=(b_1-x)/|b_1-x|$. We list some useful evaluations (specialised from Lemmas \ref{Lem:monholom}, \ref{Lem:monpole}). 
\begin{equation}\label{eq:magnevals}
\begin{split}
g(b)&=\frac{\Gamma(\frac 12)}{2}R_B\left(\frac{e^{i\nu(w_1)}}{\pi(b-w_1)}\sqrt{\frac{b-x}{w_1-x}}\right)+O(|b-x|^{-3/2})\\
g(b_0)&=\frac{e^{i\theta_0}-1}{2\sin\theta_0}\\
g(b_1)&=\frac{e^{i\theta_1}-1}{2\sin\theta_1}\\
h(b)&=\sqrt 2\Gamma(-\frac 12)R_B(\sqrt{b-x})+O(|b-x|^{-3/2})\\
h(b_0)&=-i\pi(2(b_0-x))^{1/2}\\
h(b_1)&=-\pi (2(b_1-x))^{1/2}=-h(b_0)
\end{split}\end{equation}
where $\theta_i=\arg(\frac{b_i-w_1}{x-w_1})$, and we assume that the defect line $\gamma$ does not cross $(b_0w_1)$, $(b_1w_1)$. 

\paragraph{Pole adjacent to the singularity.\\}

We now want to construct an approximation of $S_{b_0,w_0}(.,w_1)$ (within $o(|b_0-w_0|^{-1-\eps})$ near the singularity). We are looking for such an approximation $\hat g$ of the form: $\hat g$ is a linear combination of $g,\bar g,h,\bar h$ in $B(x,r)$; 
\begin{equation}\label{eq:magnhatgmacr}
\hat g(b)=\Re R_B\left(\frac{\mu}{\pi(b-w_1)}\sqrt{\frac{(b-b_0)(w_1-w_0)}{(w_1-b_0)(b-w_0)}}\right)
\end{equation}
in $(B(x,r)\cup B(y,\tilde r))^c$. %
In $B(y,\tilde r)$ we simply set $\hat g=0$ (rather crude but sufficient for our purposes). 
Set $R=|x-y|$; the mesoscopic scales $1\ll r,\tilde r\ll R$ are yet to be specified.

We need: 
\begin{equation}\label{eq:monapproxcond}
\hat g(b_0)=0,{\rm\ \ }(\rK'_{b_0,w_0}\hat g)(w_1)=1,
\end{equation}
and the values of $\hat g$ across $\partial B(x,r)$ differ by $o(\sqrt r/R)$: this defines uniquely $\hat g$ as a linear combination of $g,\bar g,h,\bar h$ and fixes $\mu$. Since $\bar{\hat g}$ satisfies the same conditions, we may set $\hat g=\Re(\alpha g+\beta h)$ in $B(x,r)$. From \eqref{eq:magnevals}, we get 
$$\mu=\frac\alpha 2 e^{i\nu(w_1)}\Gamma(1/2)\sqrt{\frac{w_1-b_0}{w_1-x}}$$
We need to fit $\alpha,\beta$ in order to make the error term small near $b_0$. Since
\begin{equation}\label{eq:monapproxcond1}
\begin{split}
\frac{1}{b-w_1}\sqrt{\frac{(b-b_0)(w_1-w_0)}{(w_1-b_0)(b-w_0)}}
&=\sqrt{\frac{b-b_0}{w_1-b_0}}\left(\frac{1}{b-w_1}+\frac 12\cdot\frac
{1}{w_0-w_1}+O(|b-w_1|/R^2)
\right)\\
&=\sqrt{\frac{b-b_0}{w_1-b_0}}\left(\frac{1}{b-w_1}+\xi+O(|b-w_1|/R^2)
\right)\\
&=\frac{1}{\sqrt{(w_1-b_0)(b-w_0)}}\left(1+O(|b-w_0|/R)\right)
\end{split}\end{equation}
if we set
$\xi=\frac 12\cdot\frac
{1}{w_0-b_0}$. As e.g. in Lemma \ref{Lem:monrobin}, the leading term in the expansion is local (depends on the singularity at $b_0,b_1,w_1$), while the subleading term $\xi$ captures the global information (position of the other monomer $w_0$) that we will need.

We require: 
\begin{equation}\label{eq:monapproxcond2}
\Re(\alpha)=1,{\rm\ \ }\Re(\alpha g(b_0)+\beta h(b_0))=0
\end{equation}
(forced by \eqref{eq:monapproxcond}), and
$$\beta\sqrt 2\Gamma(-\frac 12)=\alpha\frac{\Gamma(\frac 12)e^{i\nu(w_1)}}{2\pi\sqrt{w_1-x}}\cdot\xi$$
i.e. $\beta=\alpha\xi'$ where
\begin{equation}\label{eq:xiprime}
\xi'=-\frac{e^{i\nu(w_1)}}{4\pi\sqrt{2(w_1-x)}}\cdot\xi=-\frac{e^{i\nu(w_1)}}{4\pi\sqrt{2(w_1-x)}}\cdot\frac{1}{2(w_0-b_0)}=O(R^{-1})
\end{equation}
This is to ensure to ensure that the definitions of $\hat g$ inside and outside of $\partial B(x,R)$ match up to the subleading term.

Let us remark that $\Im(g(b_i))=\frac 12$ does not depend on the local geometry (while $\Re(g(b_i))$ does), see \eqref{eq:magnevals}. From $\Re(\alpha g(b_0)+\beta h(b_0))=0$, we get (see \eqref{eq:magnevals})
\begin{equation}\label{eq:monapproxcond3}
\frac 12\Im(\alpha)=\Re(g(b_0)+\beta h(b_0))=\Re(g(b_0)+\alpha\xi' h(b_0))
\end{equation}
and since $\Re(\alpha)=1$ (see \eqref{eq:monapproxcond2})
$$\alpha=1+2i\Re(g(b_0))=2i\overline{g(b_0)}+O(1/R)$$
Plugging this in the RHS of \eqref{eq:monapproxcond3}
$$\alpha=2i\overline{g(b_0)}+4i\Re(i\overline{g(b_0)}\xi'h(b_0))+O(1/R^2)$$
and we have (taking also into account \eqref{eq:magnevals})
\begin{equation}\label{eq:monapproxcond4}
\begin{split}
\hat g(b_1)&=\Re(\alpha g(b_1)+\beta h(b_1))=
2\Re(i\overline{g(b_0)}g(b_1))+4\Re(i\overline{g(b_0)}\xi'h(b_0))(-\frac 12)+\Re(2i\overline{g(b_0)}\xi'h(b_1))+O(1/R^2)\\
&=2\Re(i\overline{g(b_0)}g(b_1))+\Re(2i\overline{g(b_0)}\xi'(h(b_1)-h(b_0))
\end{split}\end{equation}

\paragraph{Approximation error near the singularity.\\}

Finally we estimate the approximation error, for $b$ within $O(1)$ of $b_0$, based on
$$ERR=\hat g-S_{b_0,w_0}(.,w_1)=S_{b_0,w_0}(\rK'_{b_0,w_0}\hat g-\delta_{w_1})$$
and the estimates of Lemma \ref{Lem:magn_apriori}. This may be broken down according to $\rK'_{b_0,w_0}\hat g$ on $\partial B(x,r)$, $\partial B(y,\tilde r)$; on mesoscopic scale (around $x$ and $y$); and on macroscopic scale. Recall that we are aiming for an error of order $ERR=O(R^{-1-\eps})$ for $b$ close to $b_0$.
\begin{itemize}
\item {\em On $\partial B(x,r)$.}\\ 
We have 
$$\rK'_{b_0,w_0}\hat g=O(r^{-3/2}+r^{3/2}/R^2+r^{-1/2}/R)$$%
on $\partial B(x,r)$ (with errors coming from \eqref{eq:magnevals}, \eqref{eq:monapproxcond1}, and a - crude - first-order Taylor's formula estimate for $\hat g$ in \eqref{eq:magnhatgmacr}). From Lemma \ref{Lem:magn_apriori}, this contributes $O(r^{-3/2-\eps}+r^{3/2-\eps}R^{-2}+r^{-1/2-\eps}R^{-1})$ towards the estimation error $ERR$.
\item {\em On mesoscopic scale.\\}
In $B(x,R)\setminus B(x,r)$, we have $(\rK'_{b_0,w_0}\hat g)(w)=O(|w-x|^{-1/2-3})$ (from \eqref{eq:findiff}). From Lemma \ref{Lem:magn_apriori}, this contributes an error of order $O(\sum_{k=r}^R k^{-1/2-3-\eps})=O(r^{-1/2-2-\eps})$ to $ERR$.

Similarly, in $B(y,R)\setminus B(y,\tilde r)$, we have $(\rK'_{b_0,w_0}\hat g)(w)=O(|w-y|^{-1/2-3})$, contributing to an error of order $O(\sum_{k=\tilde r}^R k^{-1/2-3}k^{3/2-\eps} R^{-3/2})=O(\tilde r^{-1-\eps}R^{-3/2})$ to $ERR$.

\item{\em On $\partial B(y,\tilde r)$.\\} 
On $\partial B(y,\tilde r)$, we estimate $\rK'_{b_0,w_0}\hat g=O(\tilde r^{-1/2})$ (from the discontinuity). From Lemma \ref{Lem:magn_apriori}, this leads to
a contribution of
 $O(\tilde r^{1-\eps}R^{-3/2})$ to $ERR$.
 
 \item{\em On macroscopic scale.\\} 
Finally, outside of $B(x,R)\cup B(y,R)$ we have $(\rK'_{b_0,w_0}\hat g)(w)=O(|w|^{-4}\sqrt R)$ (again by \eqref{eq:findiff}) and by Lemma \ref{Lem:magn_apriori} a contribution of order $O(\sum_{k=R}^\infty k^{-4}\sqrt R.R^{-\eps})=O(R^{1/2-\eps-3})$.\\
\end{itemize}

Summing up, by setting e.g. $r=R^{2/3}$, $\tilde r=R^{1/3}$, we obtain:
\begin{equation}\label{eq:monoapproxerr}
S_{b_0,w_0}(b_1,w_1)=\hat g(b_1)+O(R^{-1-\eps})
\end{equation}

\paragraph{Application to monomer correlations.\\}
We may exchange the roles of $b_0,b_1$. From \eqref{eq:monapproxcond4} and \eqref{eq:monoapproxerr}, we have
\begin{equation}\label{eq:monapproxfinal}
\begin{split}
S_{b_0,w_0}(b_1,w_1)&=2\Re(i\overline {g(b_0)}g(b_1))+\Re\left(2i\overline{g(b_0)}\xi'(h(b_1)-h(b_0))\right)+O(R^{-1-\eps})\\
S_{b_1,w_0}(b_0,w_1)&=2\Re(i\overline {g(b_1)}g(b_0))+\Re\left(2i\overline{g(b_1)}\xi'(h(b_0)-h(b_1))\right)+O(R^{-1-\eps})
\end{split}
\end{equation}
Observe that $2\Re(i\overline {g(b_0)}g(b_1)))=g(b_1)-g(b_0)$, which is real (since $\Im(g(b_0))=\Im(g(b_1))=\frac 12$), and that $h(b_1)-h(b_0)=-2\pi (2(b_1-x))^{1/2}$ (see \eqref{eq:magnevals}). An examination of the local geometry yields 
$$ie^{i\nu(w_1)}\frac{h(b_1)-h(b_0)}{4\pi\sqrt{2(w_1-x)}}=-\frac 12(b_1-b_0)$$ Consequently (see \eqref{eq:trimerlog} and \eqref{eq:xiprime}),
\begin{align*}
\frac{\Mon_M(b_1,w_0)}{\Mon_M(b_0,w_0)}=-\frac{S_{b_0,w_0}(b_1,w_1)}{S_{b_1,w_0}(b_0,w_1)}&=1+\Re(2i\xi'(h(b_1)-h(b_0)))+O(|b_0-w_0|^{-1-\eps})\\
&=1+\Re\left(\xi(b_1-b_0)\right)+O(|b_0-w_0|^{-1-\eps})\\
&=1-\frac 12\Re\left(\frac{b_1-b_0}{b_0-w_0}\right)+O(|b_0-w_0|^{-1-\eps})\\
&=\left|\frac{b_1-w_0}{b_0-w_0}\right|^{-1/2}+O(|b_0-w_0|^{-1-\eps})
\end{align*}
and thus we have obtained (wrapping up the argument as in the proof of Proposition \ref{Prop:electrsmalls}):

\begin{Prop}\label{Prop:FS}
There is $c(\Lambda,w)>0$ such that
$$\Mon_M(b,w)\sim c(\Lambda,w)|b-w|^{-1/2}$$
as $b\rightarrow\infty$.
\end{Prop}

Moreover $c$ is a continuous function of the face-rooted rhombus tiling $(\Lambda,w)$ (in the sense of Section \ref{sss:generelectr}), and is bounded away from $0$ and $\infty$ under $(\spadesuit)$. Presumably, $c$ depends only on $\Lambda$ (and thus is constant on a large class of rhombus tiling, see the discussion after Theorem \ref{Thm:electr}). To see this, it is enough to check that 
$$\frac{\Mon(b,w')}{\Mon(b,w)}\rightarrow 1$$
as $b\rightarrow\infty$, if $w,b',w'$ are consecutive vertices on the boundary of a face of $M$. In turn this is equivalent to
$$\left|\frac{S_{b,w}(b',w')}{S_{b,w'}(b',w)}\right|\rightarrow 1$$
as $b\rightarrow\infty$. This seems to require another {\em ad hoc} local computation (considering a ``trimer" with two white and one black vertex around $x$, instead of two black and one white vertex), which we leave to the dedicated reader.

Notice however that if $M$ is the square (or rectangular) lattice, one can interchange the roles of black and white vertices (or use periodicity) and consequently $c(\Z^2,w)=c(\Z^2)$.

\subsection{General monomer correlators}\label{subsec:generalmono}

In this subsection we indicate how to extend the previous argument to the general case of $2p$ monomers $b_1,\dots,b_p,w_1,\dots,w_p$ in the plane, which does not involve any substantial additional difficulty. We denote $\underline b=\{b_1,\dots,b_p\}$, $\underline w=\{w_1,\dots,w_p\}$.
 
Let $x_i$ (resp. $y_i$) be the face of $M$ adjacent to $b_i\in M_B$ (resp. $w_i\in M_W$). Let $\gamma_i$ be a defect line running from $x_i$ to $y_i$ on $M^\dg$; we assume that the $\gamma_i$'s are disjoint and that the pairwise distances between singularities are of order $R\gg1$. Let
$$\rK':\R^{M_B\setminus\{b_1,\dots,b_m\}}\longrightarrow
\R^{M_W\setminus\{w_1,\dots,w_m\}}$$
be the operator obtained from $\rK'$ by removing rows and columns corresponding to monomers and changing signs of entries corresponding to edges of $M$ crossing one of the defect lines. Equivalently, one may think of $\rK'$ as mapping $\{f\in\R^{M_B}:f(b_1)=\cdots=f(b_p)=0\}$ to $\R^{M_W}/\R^{\{w_1,\dots,w_p\}}$. 

By \eqref{eq:magnpairker}, surgery (Lemma \ref{Lem:surgery}) and induction on $m$ (as we did for multiple electric correlators leading to \eqref{eq:Srhomult}), we get that for $R$ large enough, there exists a unique inverting kernel $S_{\underline{b},\underline{w}}$ for $\rK'$ vanishing at infinity and that
$$S_{\underline{b},\underline{w}}(z,w)=\frac 12R_B\left(e^{i\nu(w)}S(z,w)\right)+(cc)+o(1)$$
where all pairwise distances are of order $R$ and $S$ is the continuous kernel 
$$S(z,w)=
\prod_{i=1}^p\left(\frac{(z-x_i)(w-y_i)}{(z-y_i)(w-x_i)}\right)^{1/2}\cdot\frac{1}{\pi(z-w)}$$
Moreover, the $o(1)$ (as $R\rightarrow\infty$) is uniform on lozenge tilings under $(\spadesuit)$. 

The corresponding ``Robin kernel" (compare with Lemma \ref{Lem:monrobin}) is 
$$r(w)=\lim_{z\rightarrow w}\left(S(z,w)-\frac{1}{\pi(z-w)}\right)=\frac {1}{2\pi}\sum_{i=1}^p\left(\frac{1}{w-x_i}-\frac 1{w-y_i}\right)$$
so that
$$\hat r_k\stackrel{def}{=}\lim_{x\rightarrow x_k}\left(r(w)-\frac 1{2\pi (w-x_k)}\right)=\frac{1}{2\pi(y_k-x_k)}+\frac {1}{2\pi}\sum_{i:i\neq k}\left(\frac{1}{x_k-x_i}-\frac 1{x_k-y_i}\right)$$
For $w$ within $O(1)$ of $x_k$ and $z$ away from other singularities, we have
\begin{equation*}
\pi S(z,w)=\frac{1}{z-w}\sqrt{\frac{z-x_k}{w-x_k}}\left(1+(z-w)\pi\hat r_k+O((z-w)^2R^{-2})\right)
\end{equation*}
which plays the role of \eqref{eq:monapproxcond1} with $\pi\hat r_k$ replacing $\xi$. 

If $(b_k,w,b'_k)$ are on the boundary of the face $x_k\in M^\dg$, a straightforward extension of the two-point argument (leading to \eqref{eq:monapproxfinal})yields
\begin{align*}
S_{\underline b,\underline w}(b'_k,w)&=2\Re(i\overline {g(b_k)}g(b_k'))+\Re\left(\frac{i\overline {g(b_k)}(h(b_k)-h(b_k'))e^{i\nu(w)}}{2\sqrt{2(w-x)}}\hat r_k\right)+O(R^{-1-\eps})\\
&=g(b'_k)-g(b_k)+\Re\left(\overline{g(b_k)}(b_k-b'_k)\pi\hat r_k\right)+O(R^{-1-\eps})
\end{align*}
(compare also with Lemma \ref{Lem:monrobin}). Let $\underline b'=(b_1,\dots,b_p)$ with $b'_k$ substituted for $b_k$. Then taking into account $g(b_k')-g(b_k)\in\R$ (see \eqref{eq:magnevals}), we get (by exchanging the roles of $b_k$ and $b'_k$)
$$\frac{S_{\underline b,\underline w}(b'_k,w)}
{S_{\underline b',\underline w}(b_k,w)}
=1+\Re\left((b'_k-b_k)\pi\hat r_k\right)+O(R^{-1-\eps})
$$
Reasoning as in \eqref{eq:trimerlog} we get
$$\frac{\Mon_M(\underline b',\underline w)}{\Mon_M(\underline b,\underline w)}=
\pm \frac{S_{\underline b,\underline w}(b'_k,w)}
{S_{\underline b',\underline w}(b_k,w)}$$
Positivity of the correlators can be justified as in \eqref{eq:monopos}.

Recall (from \eqref{eq:GFFmagncorr}) the expression for the magnetic correlators for the free field (with $g_0=\frac 12$)
$$\langle:{\mc O}_1(w_1)\dots{\mc O}_1(w_p){\mc O}_{-1}(b_1)\dots {\mc O}_{-1}(b_p):\rangle_\C=\left|\frac{\prod_{i<j}(b_i-b_j)(w_i-w_j)}{\prod_{i\neq j}(b_i-w_j)}\right|^{1/2}$$
from which we get the following variation when $b_k$ is replaced with $b_k'$: 
\begin{align*}
\frac{\langle:{\mc O}_1(w_1)\dots{\mc O}_1(w_p){\mc O}_{-1}(b_1)\dots{\mc O}(b'_k)\dots{\mc O}_{-1}(b_p):\rangle_\C
}
{\langle:{\mc O}_1(w_1)\dots{\mc O}_1(w_p){\mc O}_{-1}(b_1)\dots{\mc O}(b_k)\dots{\mc O}_{-1}(b_p):\rangle_\C}
&=1+\Re\left((b'_k-b_k)\pi\hat r_k)\right)+O(R^{-2})\\
&=\frac{\Mon_M(\underline b',\underline w)}{\Mon_M(\underline b,\underline w)}+O(R^{-1-\eps})
\end{align*}
where all pairwise distances are comparable to $R$, and the $O(R^{-1-\eps})$ is uniform over tilings under $(\spadesuit)$. Wrapping up the argument as in Theorem \ref{Thm:electr}, we obtain the

\begin{Thm}\label{Thm:FS}
Let $b_1,\dots,b_p$, $w_1,\dots,w_p$ be marked black (resp. white) vertices on $M$, with pairwise distances of order $R\gg1$. Then there is $c_p(\Lambda,\underline w)>0$ such that
$$\Mon_M(\underline b,\underline w)=c_m(\Lambda,\underline w)\left|\frac{\prod_{i<j}(b_i-b_j)(w_i-w_j)}{\prod_{i\neq j}(b_i-w_j)}\right|^{1/2}
(1+o(1))$$
\end{Thm}
Again, in the case of the square lattice one may switch colours, so that $c_p(\Lambda,\underline w)=c_p(\Z^2)$. Presumably $c_p(\Z^2)=c_1(\Z^2)^p$, though that requires an additional argument.

Let us point out a by-product of the method. In the course of the argument, we have considered trimers of type $\{b_0,w_0,b_0'\}$, three consecutive vertices on the boundary of a face $f$ of $M$. Then we have for example the following expression for the trimer-monomer correlator:
$$\Mon_M(\{b_0,w_0,b'_0\},w_1)=\Mon_M(b_0,w_1)|S_{b_0,w_1}(b'_0,w_0)|=\Mon_M(b'_0,w_1)|S_{b_0,w_1}(b'_0,w_0)|$$
and $|S_{b_0,w_1}(b'_0,w_0)|=|g(b'_0-g(b_0)|(1+o(1))$ as $|b_0-w_1|\rightarrow\infty$, where $g(b'_0)-g(b_0)$ depends only on the local geometry of the face $x\in M^\dg$:
$$g(b'_0)-g(b_0)=\frac 12\cot\arg\left(\frac{b'_0-w_0}{x-w_0}\right)-\frac 12\cot\arg\left(\frac{b_0-w_0}{x-w_0}\right)\neq 0$$
where $x=\frac{b'_0+b_0}2$ (see \eqref{eq:magnevals}). In the general case, we have
$$\Mon_M(\{b_1,w'_1,b'_1\},\dots,\{b_p,w'_p,b'_p\},w_1,\dots,w_p)\sim c_p(\Lambda,\underline w)\prod_{i=1}^p\left(\frac{\cot(\theta_i)+\tan(\theta_i)}2\right)\left|\frac{\prod_{i<j}(b_i-b_j)(w_i-w_j)}{\prod_{i\neq j}(b_i-w_j)}\right|^{1/2}$$
where $\theta_i\in(0,\frac\pi 2)$ is an angle of the right-angled triangle 
$\{b_i,w'_i,b'_i\}$.

Following Ciucu \cite{Ciucu_PNAS} (note also the relation with the Coulomb Gas formalism, see e.g. \cite{NieKno_potts,Nien_Coulomb,DiFSalZub_coulomb}), one can state a more general conjecture. For simplicity we consider the case $M=\Z^2$. An {\em islet} $I$ is a finite subset of vertices of $\Z^2$ bounded by a simple loop on $(\Z^2)^\dg$. Its {\em charge} is $\eps(I)=|I\cap M_W|-|I\cap M_B|$. The conjecture is that an islet $I$ is a discrete version of a magnetic operator ${\mc O}_{\eps(I)}$ in the sense that 
$$\Mon_M(I_1+x_1,\dots,I_n+x_n)\sim c(I_1,\dots,I_n)\prod_{1\leq i<j\leq n}|x_i-x_j|^{\eps(I_i)\eps(I_j)/2}(1+o(1))$$
as the pairwise distances $|x_i-x_j|$ go to infinity (we require here $\sum_i\eps(I_i)=0$). At the technical level, there is a significant difference between the cases of even and odd charges.

Lastly let us point out that the surgery argument (Lemma \ref{Lem:surgery}) enables to analyse at essentially no additional technical cost mixed magnetic-electric correlators, the simplest of which is
$$\langle{\mc O}_{1}(w){\mc O}_{-1}(b)\exp(2i\pi s(h(f')-h(f)))\rangle$$
where $s\in (0,\frac 12)$ and the pairwise distances between insertions $b,w,f',f$ go to infinity (and ${\mc O}_1(w)$, ${\mc O}_{-1}(b)$ represent monomer defects). Of some interest are the coincident magnetic-electric operators (which may be thought of as spinor variables, see e.g. Section 5 in \cite{Dub_abelian}). At the lattice level, one may consider the above correlator with $w$ on the boundary of $f$ and $b$ on the boundary of $f'$. For general $s$, their analysis seems to require additional arguments. Note however that in the coincident case, for $s=\frac 12$, 
$$\langle {\mc O}_{1}(w){\mc O}_{-1}(b)\exp(i\pi (h(f')-h(f)))\rangle=\pm \uK^{-1}(b,w)$$
the asymptotics of which underpin the analysis of all other correlators.

{\bf Acknowledgments.} I wish to thank Rick Kenyon, C\'edric Boutillier and B\'eatrice de Tili\`ere for interesting conversations during the preparation of this article, as well as anonymous referees for helpful comments.

\printglossary

\bibliographystyle{abbrv}
\bibliography{biblio}

-----------------------

Columbia University

Department of Mathematics

2990 Broadway 

New York, NY 10027

\end{document}